\documentclass[a4paper,twoside,10pt]{article}
\usepackage{amsmath,amssymb,graphics,epsfig,color}
\usepackage{amsmath,amsthm,amsfonts,latexsym,amscd,amssymb,enumerate,wasysym,stmaryrd}

\usepackage{graphicx}
\usepackage[all]{xy}
\usepackage{fancyhdr}
\usepackage[T1]{fontenc}
\usepackage{ifthen}
\usepackage{amsmath,amsfonts,amssymb}
\newcommand{\pleinepage}%
{\setlength{\oddsidemargin}{0in}\setlength{\textwidth}{6.26in}\setlength{\topmargin}{0in}\setlength{\textheight}{8.7in}}
\newcommand{\grossepage}%
{\setlength{\oddsidemargin}{-0.5cm}\setlength{\textwidth}{17.5cm}\setlength{\topmargin}{-1.5cm}\setlength{\textheight}{24cm}}
\newcommand{\defit}[1]%
{{\em #1}}

\def\Re{\mathop{\plainRe\mkern -2mu\mit e}\nolimits}
\def\Im{\mathop{\plainIm\mkern -2mu\mit m}\nolimits}
\def\surl#1_#2{\mathrel{\mathop{\kern 0pt #1}\limits_{#2}}}

\newcommand{\fleche}[1]%
{\rTo^{#1}}
\newcommand{\fonction}[5]%
{\begin{diagram}
#2 & {} &\rTo^{#1} & {} & #3 \\
#4 & {} &\rMapsto & {} & #5 
\end{diagram} }
\newcommand{\sfonction}[5]%
{$\begin{array}{ccc}#2 & {\buildrel #1 \over \rightarrow} & #3 \\#4 & \mapsto & #5 \\ \end{array}$ }
\newcommand{\accolade}[1]%
{\begin{cases}  #1 \end{cases}}

\newcounter{nbre}
\newcommand{\entete}[6]%
{{\large \noindent%
\mbox{\begin{tabular}{c} #1 \\ #2 \end{tabular}}\hspace{\fill}\mbox{\begin{tabular}{c} #3 \\ #4 \end{tabular}}\vspace{1cm}\begin{center}{\Huge \textsc{#5}} \\ \vspace{0.5cm}\begin{tabular}{c} #6 \\ \hline \end{tabular}\end{center}\bigskip}}
\newcommand{\defifont}{ \sc }

\newenvironment{dfns}[1][]%
{\refstepcounter{compteur} \par \vspace{0.3cm}\ \\ \noindent {{  \textsc{\textbf{Definitions}}} 
    \textbf{\thecompteur} \ 
    }---\ \sffamily\renewcommand{\em}{\normalfont\itshape}}{\par
    \vspace{0.3 cm}}

\newcommand{\gloss}[1]%
 {\index{{#1}@{#1}}{\em #1}\relax}
\newcommand{\xgloss}[2]%
 {\index{{#1}@{#1}!{#2}@{#2}}{\em #1\relax #2}\relax}
 \newcommand{\glossref}[2]%
 {\index{{#2}@{#1}}{\em #1}\relax}
 \newcommand{\xglossref}[4]%
 {\index{{#3}@{#1}!{#4}@{#2}}{\em #1 \relax #2 }\relax}
\newcounter{compteur}
\renewcommand{\thecompteur}{\thesection.\arabic{compteur}}
\newenvironment{dfn}[1][]%
{\refstepcounter{compteur} \par \vspace{0.3cm}\ \\ \noindent {{  \textsc{\textbf{Definition}}} 
    \textbf{\thecompteur} \ 
    }---\ \sffamily\renewcommand{\em}{\normalfont\itshape}}{\par
    \vspace{0.3 cm}}
{\refstepcounter{compteur} \par \vspace{0.3cm}\ \\ \noindent {{  \textsc{\textbf{Definitions}}} 
    \thecompteur \ 
    }---\ \sffamily\renewcommand{\em}{\normalfont\itshape}\begin{enumerate}}{\end{enumerate}
    \par \vspace{0.5 cm}}
\newcounter{theonum}

\newenvironment{thm}[1][]%
{\refstepcounter{theonum} \par \vspace{0.3cm}\ \\ \noindent {{  \textsc{\textbf{Theorem}}} 
    \textbf{\thecompteur} \ 
    }---\ \sffamily\renewcommand{\em}{\normalfont\itshape}}{\par
    \vspace{0.3 cm}}

\newcommand\addpage[2]{#2, page #1}
\makeatletter
\renewcommand\p@theonum{\protect\addpage{\thepage}}
\makeatother
 
\newenvironment{prop}[1][]%
{\refstepcounter{compteur} \par \vspace{0.2cm}\ \\ \noindent {{  \textsc{\textbf{Proposition}}} 
    \textbf{\thecompteur} \ 
    }---\ \sffamily\renewcommand{\em}{\normalfont\itshape}}{\par
    \vspace{0.2 cm}}

\newenvironment{prop*}[1][]%
{ \par \vspace{0.2cm}\ \\ \noindent {{\textsc{\textbf{Proposition}}} 
    #1 \ 
    }---\sffamily\renewcommand{\em}{\normalfont\itshape}}{\par \vspace{0.3 cm}}
{\refstepcounter{compteur} \par \vspace{0.3cm}\ \\ \noindent {{  \textsc{\textbf{Properties}}} 
    \thecompteur \ 
    }---\ \sffamily\renewcommand{\em}{\normalfont\itshape}}{\par
    \vspace{0.3 cm}}

\newenvironment{cor}[1][]%
{\refstepcounter{compteur} \par \vspace{0.3cm}\ \\ \noindent {{ \textsc{\textbf{Corollary}}}
    \textbf{\thecompteur} 
    \ }---\ \sffamily\renewcommand{\em}{\normalfont\itshape}}{\par  \vspace{0.3 cm}}

\newenvironment{lem}[1][]%
{\refstepcounter{compteur} \par \vspace{0.2cm}\ \\ \noindent {{  \textsc{\textbf{Lemma}}} 
    \textbf{\thecompteur} \ 
    }---\ \sffamily\renewcommand{\em}{\normalfont\itshape}}{\par
    \vspace{0.2 cm}}    
    
\renewenvironment{proof}%
{\par \vspace{0.2cm}\ \\ \noindent{ { \textsc{Proof}}\,---\ } }{\hfill{$\Box$} \par \vspace{0.2 cm}}

\newenvironment{oss}%
{\par \vspace{0.2cm}\ \\ \noindent{\sc \textbf{Remark}\ }---\ }{\par \vspace{0.2cm}}
{\refstepcounter{compteur}\par \vspace{0.2cm}\ \\ \noindent{\sc \textbf{Conjecture} \textbf{\thecompteur}\ }---\ }{\par \vspace{0.2cm}}

\newenvironment{exemple}
{\par\vspace{0.3cm}\noindent{\defifont
    Example\ }---\  \par \vspace{0.3 cm}}

\begin{document}

\NoCompileMatrices
% fibrations
\def\ds{\displaystyle}
\def\pn{\pi_{n}}
\def\pnu{\pi_{n-1}}
\pagestyle{fancy}
% with this we ensure that the chapter and section
% headings are in lowercase.
%\renewcommand{\chaptermark}[1]{\markboth{#1}{}}
\renewcommand{\sectionmark}[1]{\markright{\thesection\ #1}}
\fancyhf{} % delete current setting for header and footer
\fancyhead[LE,RO]{\;\thepage}

\fancyhead[CO]{\textsc{}}
\fancyhead[CE]{\textsc{Paolo Antonini}}

\renewcommand{\headrulewidth}{0.16pt}
\renewcommand{\footrulewidth}{0pt}
\addtolength{\headheight}{0.7pt} % make space for the rule
\fancypagestyle{plain}{%
\fancyhead{} % get rid of headers on plain pages
\renewcommand{\headrulewidth}{0pt} % and the line
}

% quando il footnote e' della stessa dimensione:
\newcommand{\foot}[1]{\footnote{\begin{normalsize}#1\end{normalsize}}}

% o p e r a t o r i     m a t e m a t i c i    

\def\bX{\partial X}
\def\dim{\mathop{\rm dim}}
\def\Re{\mathop{\rm Re}}
\def\Im{\mathop{\rm Im}}
\def\I{\mathop{\rm I}}
\def\Id{\mathop{\rm Id}}
\def\grad{\mathop{\rm grad}}
\def\vol{\mathop{\rm vol}}
\def\SU{\mathop{\rm SU}}
\def\SO{\mathop{\rm SO}}
\def\Aut{\mathop{\rm Aut}}
\def\End{\mathop{\rm End}}
\def\GL{\mathop{\rm GL}}
\def\Cinf{\mathop{\mathcal C^{\infty}}}
\def\Ker{\mathop{\rm Ker}}
\def\Coker{\mathop{\rm Coker}}
\def\dom{\mathop{\rm Dom}}
\def\Hom{\mathop{\rm Hom}}
\def\Ch{\mathop{\rm Ch}}
\def\sign{\mathop{\rm sign}}
\def\SF{\mathop{\rm SF}}
\def\loc{\mathop{\rm loc}}
\def\AS{\mathop{\rm AS}}
\def\spec{\mathop{\rm spec}}
\def\Ric{\mathop{\rm Ric}}
\def\ch{\mathop{\rm ch}}
\def\Ch{\mathop{\rm Ch}}

\def\ev{\mathop{\rm ev}}
\def\id\textrm{Id}
\def\dd{\mathcal{D}(d)}
\def\Cli{\mathbb{C}l(1)}
\def\kerd{\operatorname{ker}(d)}

%                  l e t t e r e     g r e c h e 
\def\Fi{\Phi}

\def\de{\delta}
\def \dl{\partial L_x^0}
\def\e{\eta}
\def\ep{\epsilon}
\def\ro{\rho}
\def\a{\alpha}
\def\o{\omega}
\def\O{\Omega}
\def\b{\beta}
\def\la{\lambda}
\def\th{\theta}
\def\s{\sigma}
\def\t{\tau}
\def\g{\gamma}
\def\D{\Delta}
\def\G{\Gamma}
\def \fol{\mathcal F}
\def\R{\mathbin{\mathbb R}}
\def\Rn{\R^{n}}
\def\C{\mathbb{C}}
\def\Cm{\mathbb{C}^{m}}
\def\Cn{\mathbb{C}^{n}}
\def\gr{\mathcal{G}}
% Some special symbols
\def\Kahler{{K\"ahler}}
\def\w{{\mathchoice{\,{\scriptstyle\wedge}\,}{{\scriptstyle\wedge}}
{{\scriptscriptstyle\wedge}}{{\scriptscriptstyle\wedge}}}}
% Calligraphic and bold abbreviations
\def\cA{{\cal A}}\def\cL{{\cal L}}
\def\cO{{\cal O}}\def\cT{{\cal T}}\def\cU{{\cal U}}
\def\cD{{\cal D}}\def\cF{{\cal F}}\def\cP{{\cal P}}\def\cH{{\cal H}}\def\cL{{\cal L}}
\def\cB{{\cal B}}

% N O R M A   D I    U N    V E T T O R E 

\newcommand{\n}[1]{\left\| #1\right\|}% grande palla B di raggio R

%%%%%%%%%%%%%%%%%%%%%%%%%%%%%%%%%%%%%%%%%%%%%%%%%%%%%%%%%%%%%%%%%%%%%%%%%%%%%%%%
%%%%%%%%%%%%%%%%%%%%%%%%%%%%%%%%%%%%%%%%%%%%%%%%%%%%%%%%%%%%%%%%%%%%%%%%%%%%%%%5
\def\Z{\mathbb{Z}}
\def\cgs{C^{*}(\Gamma,\sigma)}
\def\bcgs{C^{*}(\Gamma,\bar{\sigma})}
\def\cgsr{C^{*}_{red}(,\sigma)}
\def\Mt{\tilde{M}}
\def\Et{\tilde{E}}
\def\Vt{\tilde{V}}
\def\Xt{\tilde{X}}
\def\N{\mathbb{N}}
\def\Nbs{\N^{\bar{\s}}}
\def\rcab{\ro^{[c]}_{\a-\b}}
\def\rc{\ro^{[c]}}
\def\Cd{\mathbb{C}^{d}}
\def\tr{\mathop{\rm tr}}\def\tralg{\tr{}^{\text{alg}}}       

%%%%%%        T R A C C E   %%%%%%%%%%%%%%%
\def\TR{\mathop{\rm TR}}\def\trace{\mathop{\rm trace}}
\def\STR{\mathop{\rm STR}}
\def\trG{\mathop{\rm tr_\Gamma}}
\def\TRG{\mathop{\rm TR_\Gamma}}
\def\Tr{\mathop{\rm Tr}}
\def\Str{\mathop{\rm Str}}
\def\Cl{\mathop{\rm Cl}}
\def\Op{\mathop{\rm Op}}
\def\supp{\mathop{\rm supp}}
\def\scal{\mathop{\rm scal}}
\def\ind{\mathop{\rm ind}}
\def\Ind{\mathop{\mathcal I\rm nd}\,}
\def\Diff{\mathop{\rm Diff}}
\def\T{\mathcal{T}}
\def\dn{\textrm{dim}_{\Lambda}}
\def \lke{\textrm L^2-\textrm{Ker}}

%Connessione

\newcommand{\wt}{\widetilde}
\newcommand{\go}{\mathcal{G}^{0}}
\newcommand{\dii}{(d_x^{k-1})^\ast}
\newcommand{\di}{d_x^{k-1}}
\newcommand{\ra}{\operatorname{range}}
\newcommand{\rb}{\rangle}
\newcommand{\lb}{\langle}
\newcommand{\re}{\mathcal{R}}
\newcommand{\vo}{\operatorname{End}_{\Lambda}(E)}
\newcommand{\mt}{\mu_{\Lambda,T}}
\newcommand{\tru}{\operatorname{tr}_{\Lambda}}
\newcommand{\buno}{B^1_{\Lambda}(E)}
\newcommand{\bdue}{B^2_{\Lambda}(E)}
\newcommand{\clis}{H^{2k}_{(2),dR}(L_x^0)}
\newcommand{\cali}{L^2(\Omega^{2k}(\partial L_x^0))}
\newcommand{\binf}{B^{\infty}_{\Lambda}(E)}
\newcommand{\bif}{B^{f}_{\Lambda}(E)}
\newcommand{\vn}{\operatorname{End}_{\mathcal{R}}}
\newcommand{\ho}{\operatorname{Hom}_{\Lambda}}
\newcommand{\spc}{\operatorname{spec}_{\Lambda,e}}
\newcommand{\ix}{\operatorname{Ind}_{\Lambda}}
\newcommand{\cic}{C^{\infty}_c(L_x;E_{|L_x})}
\newcommand{\ci}{C^{\infty}_c(L_x;E_{|L_x})}
\newcommand{\tx}{\{T_x\}_{x\in X}}
\newcommand{\cc}{C^{\infty}_c(X)}
\newcommand{\rom}{\underline{\mathcal{R}_0}}
\newcommand{\roma}{(\mathcal{R}_0)_{|\partial X_0}     }
\newcommand{\dfo}{D^{\mathcal{F}_{\partial}}}
\newcommand{\deu}{D_{\epsilon,u}}
\newcommand{\deupp}{D^{+}_{\epsilon,u}}
\newcommand{\deum}{D^{-}_{\epsilon,u}}
\newcommand{\deuf}{D_{\epsilon,u}^{\mathcal{F}_{\partial}}}
\newcommand{\deufo}{D_{\epsilon,u,x_0}^{\mathcal{F}_{\partial}}}
\newcommand{\pie}{\Pi_{\epsilon}}
\newcommand{\pal}{\partial L_x}
\newcommand{\pr}{\partial_r}
\newcommand{\inbl}{\int_{\partial L_x} }
\newcommand{\pkp}{\chi_{\{0\}}(D^+_x)}
\newcommand{\deup}{D_{\epsilon,x}^{\pm}}
\newcommand{\dext}{D_{\epsilon,\mp u,x}^{\pm}}
\newcommand{\dex}{D_{\epsilon,\pm u,x}^{\pm}}
\newcommand{\ext}{\operatorname{Ext}(D_{\epsilon,x}^{\pm})}
\newcommand{\eppu}{0<|u|<\epsilon}
\newcommand{\hdeupx}{e^{-tD^2_{\epsilon,u,x}}}
\newcommand{\udif}{\operatorname{UDiff}}
\newcommand{\ki}{L^2(\Omega^kL_x^0)}
\newcommand{\uc}{\operatorname{UC}}
\newcommand{\op}{\operatorname{Op}}
\newcommand{\deux}{D_{\epsilon,u,x}}
\newcommand{\pk}{\phi_k}
\newcommand{\hdeupsx}{e^{-sD^2_{\epsilon,u,x}}}
\newcommand{\hdeups}{e^{-sD^2_{\epsilon,u}}}
\newcommand{\hdeut}{e^{-tD^2_{\epsilon,u}}}
\newcommand{\indu}{\operatorname{ind}_{\Lambda}}
\newcommand{\stru}{\operatorname{str}_{\Lambda}}
\newcommand{\deuq}{D^2_{\epsilon,u}}
\newcommand{\intk}{\int_{\sqrt{k}}^{\infty}}
\newcommand{\dmd}{d\mu_{\Lambda,D_{\epsilon,u}}(x)}
\newcommand{\defox}{D_{x}^{\mathcal{F}_{\partial}}}
\newcommand{\mun}{\mu_{\Lambda,D_{\epsilon,u}}(x)}
\newcommand{\tsi}{\int_{-\sigma}^{\sigma}}
\newcommand{\ak}{\lim_{k\rightarrow \infty}\operatorname{LIM}_{s\rightarrow 0}}
\newcommand{\deus}{D_{\epsilon,u}e^{-tD_{\epsilon,u}^2}}

\newcommand{\deuss}{D_{\epsilon,u}^2e^{-tD_{\epsilon,u}^2}}
\newcommand{\pkd}{\phi_k^2}
\newcommand{\eup}{e^{-tD^{+}_{\epsilon,u}D^{-}_{\epsilon,u}}}
\newcommand{\eum}{e^{-tD^{-}_{\epsilon,u} D^{+}_{\epsilon,u} }}
\newcommand{\deussx}{D_{\epsilon,u,x}e^{-tD_{\epsilon,u,x}^2}}
\newcommand{\dessx}{S_{\epsilon,u,x}e^{-tS_{\epsilon,u,x}^2}}
\newcommand{\clib}{c(\partial_r)\partial_r \phi_k^2}
\newcommand{\sk}{\int_s^{\sqrt{k}}}
\newcommand{\esm}{S_{\epsilon,u}e^{-tS_{\epsilon,u}^2}}
\newcommand{\deussxo}{D_{\epsilon,u,z_0}e^{-tD_{\epsilon,u,x_0}^2}}
\newcommand{\dessxo}{S_{\epsilon,u,z_0}e^{-tS_{\epsilon,u,z_0}^2}}
\newcommand{\deusszo}{D^{\mathcal{F}_{\partial}}_{\epsilon,u,z_0}e^{-t(D^{\mathcal{F}_{\partial}}_{\epsilon,u,x_0})2}}
\newcommand{\desszo}{S_{\epsilon,u,x_0}e^{-tS_{\epsilon,u,x_0}^2}}
\newcommand{\essp}{S_{\epsilon,u}^2}
\newcommand{\esspo}{S_{0,u}^2}
\newcommand{\dotto}{\dot{\theta}}
\newcommand{\piep}{\Pi_{\epsilon}}
\newcommand{\ome}{\Omega}
\newcommand{\deffo}{D^{\mathcal{F}_{\partial}}}
\newcommand{\nablal}{\nabla_x^l}
\newcommand{\nablak}{\nabla_y^k}
\newcommand{\kerk}{[f(P)]_{(x_0,\bullet)} }
\newcommand{\kepp}{\operatorname{Ker} (D^{\mathcal{F}_0^+})}
%PECETTA
\newcommand{\ty}{\infty}
\definecolor{light}{gray}{.95}
\newcommand{\pecetta}[1]{
$\phantom .$
\bigskip
\par\noindent
\colorbox{light}{\begin{minipage}{13.5 cm}#1\end{minipage}}
\bigskip
\par\noindent
}

%D I R A C :
\newcommand\Di{D\kern-7pt/}
%%%%%%%%%%%%%%%%%%%%%%%%%%%%%%%%%%%%%%%%%%%%%%%%%%%%%%%%%%%%%%%%%%%%%%%%%%%%%%
%%%%%%%%%%%%%%%%%%%%%%%%%%%%%%%%%%%
%         E n d    o f    T o p m a t t e r   %%%%%%%%%

\title{The Atiyah Patodi Singer signature formula for measured foliations}
\author{\Large Paolo Antonini\\
antonini@mat.uniroma1.it\\
paolo.anton@gmail.com}

\maketitle

\begin{abstract}
\noindent Let $(X_0,\mathcal{F}_0) $ be a compact manifold with boundary endowed with a  foliation 
$\mathcal{F}_0$ which is assumed to be measured and transverse  to the boundary. 
We denote by $\Lambda$  a holonomy invariant transverse measure on  $(X_0,\mathcal{F}_0) $ and by 
$\mathcal{R}_0$ the equivalence relation of the foliation.
Let $(X,\mathcal{F})$ be the corresponding manifold with 
 cylindrical end  and extended foliation with equivalence relation $\mathcal{R}$.\\ In the first part of this work we prove a  formula for the $L^2$-$\Lambda$ index of a 
 longitudinal 
 Dirac-type operator $D^{\mathcal{F}}$ on $X$ in the spirit of Alain Connes' non commutative geometry  \cite{Cos}
$$\operatorname{ind}_{L^2,\Lambda}(D^{\mathcal{F},+})=\langle\hat{A}(T\mathcal{F})\operatorname{Ch}(E/S),C_{\Lambda}\rangle +1/2[\eta_{\Lambda}(D^{\mathcal{F}_{\partial}})-h^+_{\Lambda}+h^{-}_{\Lambda}].$$

%Here $\Lambda$ is a holonomy invariant transverse measure,
%$\eta_{\Lambda}(D^{\mathcal{F}_{\partial}}$ is the measured eta invariant of the boundary operator defined by Ramachandran \cite{???},
%and the $\Lambda$--dimension $h^{\pm}_{\Lambda}$ of the space of extended solution is  defined using square integrable representations of the equivalence relation with values in the weighted $L^2$ spaces of the leaves. 

\noindent In the second part we specialize to the signature operator. 
\noindent We define three  types of signature for the pair (foliation, boundary foliation): the analytic signature, denoted
 $\sigma_{\Lambda,\operatorname{an}}(X,\partial X_0)$ is the $L^2$-$\Lambda$-index of the signature operator on the cylinder;
 the Hodge signature $\sigma_{\Lambda,\operatorname{Hodge}}(X,\partial X_0)$, defined using the natural representation of $\mathcal{R}$ on the field of square integrable harmonic forms on the leaves and
the de Rham signature, $\sigma_{\Lambda,\operatorname{dR}}(X,\partial X_0)$, defined using the natural representation  
of $\mathcal{R}_0$ on the field of  relative de Rham spaces of the leaves.
We prove that 
these three signatures coincide 
$$\sigma_{\Lambda, \operatorname{an}}(X_0,\partial X_0)=
\sigma_{\Lambda,\operatorname{Hodge}}(X,\partial X_0)=\sigma_{\Lambda,\operatorname{dR}}(X,\partial X_0).$$ 
As a consequence of these equalities and of the index formula we finally obtain the main result of this work,
the Atiyah-Patodi-Singer
signature formula for measured foliations:
$$\sigma_{\Lambda,\operatorname{dR}}(X,\partial X_0)=\langle L(T\mathcal{F}_0),C_{\Lambda}\rangle +1/2[\eta_{\Lambda}(D^{\mathcal{F}_{\partial}})]$$

\end{abstract}

\begin{small}
\tableofcontents
\end{small}
\section{Introduction}
% IN ROSSO ERA LA PRIMA INTRODUZIONE SCRITTA COME CONSEGNATA A PIAZZA
%The aim of this PHD thesis is to prove an Atiyah Patodi Singer index formula, in the flavor of Alain Connes non--commutative geometry (integration theory) \cite{Cos}, for a Dirac type operator on a manifold with a cylindrical end $$X=X_0\cup [0,\ty)\times \partial X_0$$
%which is foliated by an even dimensional foliation that respects the cylindrical structure and is equipped with a holonomy invariant transverse measure. 

%\noindent Is our intention to look at this as an $L^2$--index problem,
%corresponding in the original A.P.S. paper \cite{AtPaSi1} to the manifold with cylinder attached where the A.P.S. boundary condition becomes a natural $L^2$ condition. 

%\noindent Then inspired by the case of $\Gamma$--coverings \cite{lusc}, using the Von Neumann algebras associated to various representations of the equivalence relation of the foliation induced the compact $\partial$--manifold $X_0$ we define three different $L^2$--measured signatures for the pair (foliation,foliation to the boundary) on $X_0$ and use the index theorem to prove their are equal together with a Hirzebruch type formula.

Let $X_0$ be a $4k$--dimensional oriented manifold without boundary. One can give four different definitions of the signature.
\begin{itemize}
\item The \underline{topological signature} $\sigma(X_0)$ is defined as the signature of the intersection form in the middle degree cohomology;
$(x,y):=\langle x \cup y,[X_0]\rangle, \,x,y \in \operatorname{H}^{2k}(X_0,\R).$
\item The \underline{de Rham signature} $\sigma_{\operatorname{dR}}(X_0)$ is the signature of the Poincar\'e intersection form in the middle de Rham cohomology;
$([\omega],[\phi]):=\int_{X_0}\omega\wedge \phi;\,\,\omega,\phi \in \operatorname{H}^{2k}_{dR}(X_0).$
\item The \underline{Hodge signature}, $\sigma_{\operatorname{Hodge}}(X_0)$ is the signature of the Poincar\'e intersection form defined in the space of $2k$ Harmonic forms with respect to some choosen Riemannian structure
$(\omega,\phi):=\int_{X_0}\omega\wedge \phi;\,\,\omega,\phi \in \mathcal{H}^{2k}(X_0).$
\item The \underline{analytical signature} is the index of the chiral signature operator\footnote{this is the differential operator $d+d^*$ acting on the complex of differential forms, odd w.r.t. the natural chiral grading $\tau:=(-1)^k\ast (-1)^{\frac{|\cdot|(|\cdot|-1)}{2}}$, $D^{\operatorname{sign}}=\left(\begin{array}{cc}0 & D^{\operatorname{sign},-} \\D^{\operatorname{sign},+} & 0\end{array}\right)$}
$$\sigma_{\operatorname{an}}(X_0):=\operatorname{ind}(D^{\operatorname{sign},+}).$$
\end{itemize} One can prove that all these numbers coincide,
\begin{equation}\label{allequal}
\sigma(X_0)=\sigma_{\operatorname{dR}}(X_0)=\sigma_{\operatorname{Hodge}}(X_0)=\sigma_{\operatorname{an}}(X_0).
\end{equation} 
 The Hirzebruch formula $$\sigma(X_0)=\int_{X_0}L(X_0)$$ can be proven using cobordism arguments as in the original work of Hirzebruch or can be seen as a consequence of the Atiyah--Singer index formula  \cite{BeGeVe} and the chain of equalities \eqref{allequal}.

\bigskip
\noindent If $\widetilde{X_0}\longrightarrow X_0$ is a Galois covering with deck group $\Gamma$ with $X_0$ as above Atiyah \cite{At-e} used the Von Neumann algebra associated to the regular right representation of $\Gamma$ to normalize the signature on $L^2$ middle degree harmonic forms on the total space. This signature $\sigma_{\Gamma}(X_0)$ again enters in a Hirzebruch type formula
$$\sigma_{\Gamma}(X_0)=\int_{X_0}L(X_0).$$ This is the celebrated Atiyah $L^2$--signature theorem.

\bigskip
\noindent The Atiyah $L^2$--signature theorem was extended by Alain Connes \cite{Cos} to the situation in which the total space $X_0$ is foliated by an even dimensional foliation. This is the realm of \emph{non--commutative geometry}. We shall have the occasion to describe extensively this kind of result.

\bigskip
\noindent
What can one say if $X_0$ has non empty boundary ?
\bigskip

\noindent So let now $X_0$ be an oriented compact manifold with boundary and suppose the metric is product type near the boundary. Attach an infinite cylinder across the boundary to form the manifold with cylindrical ends,
$$X=X_0\bigcup_{\partial X_0}\Big{[} \partial X_0\times [0,\infty)\Big{]}$$
In the seminal paper by Atiyah Patodi and Singer \cite{AtPaSi1} is showen that the  Fredholm index of the generalized boundary value problem
with the pseudodifferential APS boundary condition
 on $X_0$ for the signature operator (or a chiral Dirac type operator) is connected to the $L^2$--index of the extended operator on $X$. Indeed this Fredholm index is the $L^2$ index on $X$ plus a defect depending on the space of extended solutions on the cylinder. More precisely the operator on the cylinder acting on the natural space of $L^2$--sections is no more Fredholm (in the general case in which the boundary operator is not invertible) but its kernel and the kernel of its formal adjoint are finite dimensional and this difference is given  by the formula\footnote{opposite orientation w.r.t. APS} $$\operatorname{ind}_{L^2}(D^+)=\int_{X_0} \hat{A}(X,\nabla)\operatorname{Ch}(E)+\dfrac{\eta(0)}{2}+\dfrac{h_{\infty}(D^-)-h_{\infty}(D^+)}{2};$$
  where $h_\infty(D^{\pm})$ are the dimensions of the limiting values of the extended $L^2$ solutions and  $\eta(0)$ is the eta invariant of the boundary operator.
  
\noindent Then, in the case of the signature operator, in dimension $4k$ the authors investigate the relationship between the APS index of the operator on $X_0$, the signature of the pair $(X_0,\partial X_0)$, the $L^2$ index on $X$ and the space of harmonic forms on $X$.
The conclusion is that the signature $\sigma(X_0)$ is exactly the $L^2$--index on the cylinder i.e. the difference of the dimensions $h^\pm$ of positive/negative square integrable harmonic forms\footnote{indeed the intersection form is passes to be non--degenerate to the image of the relative cohomology into the absolute one. This vector space is naturally isomorphic to the space of $L^2$--harmonic forms on $X$.} 
on $X$,
$$\sigma(X_0)=h^+-h^-=\operatorname{ind}_{L^2}(D^{\operatorname{sign},+})$$ while $h_{\infty}(D^{\operatorname{sign},-})=h_{\infty}(D^{\operatorname{sign},+})$ by specific simmetries of the signature operator.
In particular the APS signature formula becomes 
$$\sigma(X_0)=\int_{X_0}L(X_0,\nabla)+\eta\big{(}D^{\operatorname{sign}}_{|\partial X_0}\big{)}.$$

\noindent In the case of $\Gamma$--Galois coverings of a manifold with boundary
 with a cylinder attached, $\widetilde{X}\longrightarrow X$ this program is partially carried on by Vaillant \cite{Vai} in his Master thesis. More specifically he estabilishes a Von Neumann index formula in the sense of Atiyah \cite{At-e} for a Dirac type operator
 and relates this index with the $\Gamma$--dimensions of the harmonic forms on the total space. The remaining part of the story i.e. the relation with the topologically defined $L^2$--signature is carried out by L{\"u}ck and Schick \cite{lusc}. 
Call the index of Vaillant the analytical $L^2$--signature of the compact piece $X_0$, in symbols $\sigma_{\operatorname{an},(2)}(X_0)$ while Harmonic $\sigma_{Ho}(X_0)$ is the $L^2$ signature defined using harmonic forms on $\tilde{X}$. Then Vaillant proves that
 $$\sigma_{\textrm{an},(2)}(X_0)=\int_{X_0}L(X_0,\nabla)+\eta_{\Gamma}\big{(}D^{\operatorname{sign}}_{|\partial \tilde{X}}\big{)}=\sigma_{\operatorname{Hodge}}(X_0).$$
Luck and Schick define other different types of $L^2$ signatures, de Rham $\sigma_{\operatorname{dR},(2)}(X_0)$ and simplicial $\sigma_{\operatorname{top},(2)}(X_0)$ and prove that they are all the same and coincide with the signatures of Vaillant. To be more precise they prove
$$\sigma_{\operatorname{Hodge}}(X_0)=\sigma_{\operatorname{dR},(2)}(X_0)=\sigma_{\operatorname{top},(2)}(X_0).$$
None of these steps are easy adaptations of the closed case since in the classical proof a fundamental role is played by the existence of a gap around the zero in the spectrum of the boundary operator. This situation fails to be true in non compact (also cocompact) ambients.

\bigskip
\noindent I this thesis I carry out this program for a foliated manifold with cylindrical ends  endowed with a holonomy invariant measure $\Lambda$ \cite{Cos}. The framework is that explained by Connes in his seminal paper on non commutative integration theory \cite{Cos} in particular I use in a crucial way the semifinite Von Neumann algebras associated to a measured foliation. 
Working with the Borel groupoid defined by the equivalence relation $\mathcal{R}$  I first extend   the index formula of Vaillant.
  \begin{thm}
The Dirac operator has finite dimensional $L^2-\Lambda$--index and the following formula holds
\begin{align}\label{2111}\operatorname{ind}_{L^2,\Lambda}(D^+)=\langle\hat{A}(X)\operatorname{Ch}(E/S),[C_{\Lambda}]\rangle +1/2[\eta_{\Lambda}(D^{\mathcal{F}_{\partial}})-h^+_{\Lambda}+h^{-}_{\Lambda}].\end{align} 
\end{thm}
\noindent The dimensions of the spaces of extended solutions, $h^\pm_{\Lambda}$  are suitably defined using Von Neumann algebras associated to square integrable representations of $\mathcal{R}$, the foliation eta invariant is defined by Ramachandran \cite{Rama} and the usual integral in the APS formula is changed into the  distributional pairing with a tangential distributional form with the Ruelle--Sullivan current \cite{MoSc}. In the proof of \eqref{2111}
a significative role is played by the introduction of a notion of \underline{$\Lambda$--essential spectrum} of an operator, relative to the trace defined by $\Lambda$ in the leafwise Von Neumann algebra of the foliation. This is stable by $\Lambda$--compact perturbations (in the sense of Breuer \cite{Br2}) and translate to the foliation contest the general philosophy of Melrose \cite{Me}
 stating that the operator is Fredholm iff is invertible at the boundary. Then we define a two parameter perturbation with invertible family at the boundary. This is a Breuer--Fredholm perturbation. The strategy is to prove first the formula for the perturbation then let the parameters go to zero.

\bigskip

\noindent In the second part, inspired by the definitions of L{\"u}ck and Schick \cite{lusc} I pass to the study of three different representations of $\mathcal{R}_0$ (the equivalence relation of the foliation on the compact piece $X_0$) in order to define the \underline{Analytical Signature}, $\sigma_{\Lambda,\operatorname{an}}(X_0,\partial X_0)$
(i.e. the measured index of the signature operator on the cylinder), the \underline{de Rham signature} 
$\sigma_{\Lambda,\operatorname{dR}}(X_0,\partial X_0)$
(i.e the one induced by the representation which is valued in the relative de Rham spaces of the leaves) and the \underline{Hodge signature}, $\sigma_{\Lambda,\operatorname{Hodge}}(X_0,\partial X_0)$ (defined in terms of the representation of $\mathcal{R}_0$ in the harmonic forms on the leaves of the foliation on $X$).

\noindent Combining a generalization of the notion of the $L^2$ long exact sequence of the pair (foliation,boundary foliation), in the sense of sequences of Random Hilbert complexes (the analog of the homology $L^2$ long sequence of Hilbert $\Gamma$--modules in Cheeger and Gromov \cite{ChGr}) together with the analysis of boundary value problems of \cite{sch2}, one shows that the methods in \cite{lusc} can be generalized and give the following

\begin{thm}
The above three notions of $\Lambda$--signature for the foliation on $X_0$ coincide,
$$\sigma_{\Lambda,\operatorname{dR}}(X,X_0)=\sigma_{\Lambda,\operatorname{Hodge}}(X,X_0)=\sigma_{\Lambda,\operatorname{an}}(X,X_0)$$
 and the following APS signature formula holds true
$$\sigma_{\Lambda, \operatorname{an}}(X_0,\partial X_0)=\langle L(X),[C_{\Lambda}]\rangle +1/2[\eta_{\Lambda}(D^{\mathcal{F}_{\partial}})]$$
\end{thm}

A more detailed description of the various sections follows.

\bigskip

\noindent{\bf{Geometric setting}}

\noindent In this section the whole geometric structure is introduced. We speak about cylindrical foliations and all the data needed to define the longitudinal Dirac operator associated to a Clifford bundle. Every cylindrical foliation arises from a gluing process, starting from a foliated manifold with boundary and foliation transverse to the boundary. The first geometrical invariant of a foliation is its holonomy. It enters into index theory in an essential way providing a natural desingularization of the leaf space. The various holonomy covers glue all together into a manifold, the holonomy groupoid $G$ where one can speak about smooth functions and apply the usual analytical techniques. Following Ramachandran we work at level of the equivalence relation $\mathcal{R}$ of being on the same leaf. This is the most elementary level of desingularization where boundary value problems can be set up without ambiguity.
\bigskip

\noindent {\bf{Von Neumann algebras, foliations and index theory}}

\bigskip
\noindent {\bf{Von Neumann algebras and Breuer Fredholm theory with traces.}} 
\noindent In this section generalities about Von Neumann algebras are given. 
These are particular $*$--subalgebras of all bounded operators acting on an Hilbert space.
We specialize to Von Neumann algebras that can be equipped with a semi--finite normal faithful trace like Von Neumann algebras arising from foliations admitting a holonomy invariant transverse measure. 
Indeed Connes \cite{Co} has shown that holonomy invariant transverse measure correspond one to one to semifinite normal faithful traces on the Von Neumann algebra of the foliation. More generally a transverse measure gives a weight. This weight is invariant under interior automorphisms, i.e. is a trace, iff the measure is invariant under holonomy. 
Then in some sense holonomy is represented in the Von Neumann algebra by the interior automorphisms.
One can also use the language of foliated current. A transverse measure selects a transverse current. This current is closed iff the transverse measure is hol. invariant.
 
So let $M$ be a Von Neumann agebra with a trace $\tau:M^+\longrightarrow [0,\infty]$ one has a natural notion of dimension of a closed subspace affiliated to $M$, i.e. a subspace $V$ whose projection $\operatorname{Pr}_V$ belongs to $M$. This is by definition the relative dimension $\tau(\operatorname{Pr}_V)$. Relative dimension is the cornerstone of a theory of Fredholm operators inside $M.$ This story goes back to the seminal work of Breuer \cite{Br1,Br2}. For this reason relatively Fredholm operators are called Breuer--Fredholm. A Breuer--Fredholm operator has a finite real index with some stability properties as in the classical theory. 

\bigskip
\noindent{\bf{Transverse measures and Von Neumann algebras.}}

\noindent In the spirit of Alain Connes non commutative geometry Von Neumann algebras stand for measure spaces while $C^*$--algebras describes topological spaces.
In the seminal work \cite{Cos} he has shown that 
a foliation with a given transverse measure gives rise to a Von Neumann algebra whose properties reflect the properties of the measure. 
First we define transverse measures as measures on the sigma ring of all Borel transversals. This is acted by the holonomy pseudogroup. When the measure is  invariant w.r.t. this action one has a holonomy invariant measures. 

\noindent If a holonomy invariant measure exists then the associated $W^*$-- algebra is type $I$ or type $II$ (the first type appears only in the ergodic case). In particular there's a natural trace whose definition is explicitly given as an integral of suitable objects living along leaves against the transverse measure. 

\noindent Then transverse measures can be considered as some kind of measure on the space of the leaves.

\noindent In this section we define the Von Neumann algebra associated to the transverse measure and a square representation of the Borel equivalence relation $x \mathcal{R} y$ iff $x$ and $y$ are in the same leave. For a vector bundle $E$ this is the algebra of uniformly bounded fields of operators $x\longmapsto A_x:L^2(L_x;E)\longrightarrow L^2(L_x;E)$ ($L_x$ is the leave of $x$) acting on the Borel field of Hilbert spaces $x\longmapsto L^2(X;E)$
suitably identified using the transverse measure. Thinking of an operator as a family of leafwise operators the trace has a natural meaning, it is the integral against the transverse measure of a family of leafwise measures called local traces.
\bigskip

\noindent For self adjoint intertwining operators, using the spectral theorem and the trace on $M$ (coming from a transverse measure $\Lambda$) one can define a measure on $\R$ called the spectral measure (depending on the trace). Breuer--Fredholm properties of the operator are easily described in terms of this spectral measure. In particular one can define some kind of essential spectrum called the $\Lambda$--essential spectrum. Belonging of zero to the essential spectrum is equivalent for the operator to be Breuer--Fredholm.
We show also that for elliptic operators the essential spectrum is governed by the behavior of the operator outside compact subsets on the ambient manifold. 
Notice that if one fix a compact set $K$ on $X$ every leaf can intersect $K$ infinitely many times then our notion of "lying outside $K$" must be explained with care. We call this result the Splitting principle. It will be useful in the study of the Dirac operator.

\bigskip
\noindent{ \bf{{Analysis of the Dirac operator.}}}
Consider the leafwise Dirac operator on $X$ associated to the geometric datas of the first section. This is obtained from the collection of all Dirac operators $\{D_x\}_x$ one for each leave $L_x$. 
If the foliation is assumed even dimensional this is $\mathbb{Z}_2$--graded $D=D^+\oplus D^-$ with respect to a natural involution on the bundle $E$. This is called the Chiral Dirac operator.

\noindent This leafwise family of operators gives an operator affiliated to the Von Neumann algebra $M$ (the transverse measure gives the glue to join all the operators together). In particular each spectral projection of $D$ defines a projection in $M$.
If the foliated manifold is compact Connes han shown that this is a Breuer--Fredholm operator and the index, the relative dimension of Kernel minus CoKernel is related to topological invariants of the foliation by the Connes index formula.
$$\operatorname{ind_{\Lambda}}(D^+)=\langle\operatorname{Ch}(D^+)\operatorname{Td}(X),[C_{\Lambda}]\rangle$$ On the right handside one finds the distributional coupling between a longitudinal characteristic class and the homology class of a closed current $C_{
\Lambda}$ 
associated to the transverse measure by the Ruelle--Sullivan method.

\bigskip
\noindent{\bf{Finite dimensionality of the index problem.}}

\noindent In our cylindrical case, the operator is in general non Breuer--Fredholm. As a general philosophical principle for manifolds with cylindrical ends and product--structure operators, Fredholm properties of the operator on the natural $L^2$ space are essentially captured by the spectrum at zero of the operator on the cross section (the base of the cylinder).  
\noindent Thanks to the splitting principle the Philosophy 
$$
\textrm{invertibility at boundary}\quad 
\Longleftrightarrow\quad \textrm{Freholm property }
$$
\noindent carries on to the foliated case if one looks at the $\Lambda$--essential spectrum of the leafwise operator on the foliation induced on the transverse section of the cylinder (this is to be thought of as the foliation at infinity).

\noindent Now it's a well known fact that lots of Dirac type operators of capital importance in Physics and Geometry are not invertible at the boundary (infinity). One example for all is the Signature operator, our main application here.

\noindent However some work on elliptic regularity and the use of the generalized eigenfunction expansion of Browder and G{\aa}rding shows that the $\Lambda$--dimension of the projection on the $L^2$ kernel of $D^+$ and $D^-$ are finite projections of the V.N. algebra $M$. In particular we can define the $L^2$ chiral index of $D^+$ as
$$\operatorname{ind}_{L^2,\Lambda}(D^+)=\operatorname{dim}_{\Lambda}\operatorname{Ker}_{L^2}(D^+)-\operatorname{dim}_{\Lambda}\operatorname{Ker}_{L^2}(D^-).$$
\noindent On a compact foliated manifold, if a family of operators is implemented by a family of leafwise uniformly smoothing schwartz kernels the finite trace property follows immediately from the remarkable fact that integrating a longitudinal Radon measure against a transverse measure gives a finite mass measure on the ambient. Now the ambient is a manifold with a cylinder, hence Radon longitudinal measures do not give finite measures in general. Our strategy to prove the finite dimensionality of the $L^2$ index problem is to show that the field of $L^2$ projections on the kernel of $D^+_x$ enjoys the additional property to be locally traceable with respect to a bigger family of Borel sets. To be more precise we prove that for every compact set $K$ on the boundary of the cylinder of a leave the operator 
$\chi_{K\times \R^+}{\Pi}_{\operatorname{Ker}_{L^2}(D^+)}\chi_{K \times \R+}$ is trace class on $L^2(L_x)$. This is sufficient (by the integration process) to assure finite dimensionality.

\bigskip
\noindent{\bf{Breuer--Fredholm perturbation.}}

\noindent Once finite dimensionality of kernels is proven we perform a perturbation argument to change the Dirac operator into a Breuer--Fredholm one. This is done following very closely Boris Vaillant paper \cite{vai} where the same problem is studied for Galois coverings of manifolds with cylindrical ends. Since we are working with Von Neumann algebras the possibilty to use Borel functional calculus gives a great help in a way that we can define our two parameters perturbation essentially by boundary operator minus the boundary operator restricted to some small spectral interval near zero
$$D\leadsto D_{\epsilon,u},\quad D_{\epsilon,0}:=D_{\epsilon}$$
\noindent Next we prove (through the splitting principle) that $D_{\epsilon,u}$ is Breuer--Fredholm for small parameters and its index approximates the chiral index. Actually we have to consider separately the two parameters limits. 

\noindent The analysis of the relation between the perturbed Fredholm index and the chiral $L^2$ index requires the introduction of weighted $L^2$ spaces along the leaves, $e^{u\theta}L^2$ for $u>0$ ($r$ is the cylindrical coordinate). Smooth solutions belonging to each weighted space are called Extended Solutions, $\operatorname{Ext}(D^{\pm})$. They enter naturally into the A.P.S index formula but do not form a closed subspace in $L^2$. Some care is needed in defining their $\Lambda$--dimension and prove that this is finite. 

\noindent The remaining part of the section is devoted to the proof of the fundamental asymptotic relations
$$\lim_{\epsilon\rightarrow 0}\operatorname{ind}_{L^2,\Lambda}(D^+_{\epsilon})=\operatorname{ind}_{L^2,\Lambda}(D^+),\quad \quad \lim_{\epsilon\rightarrow 0}\operatorname{dim}_{\Lambda}\operatorname{Ext}(D^{\pm}_{\epsilon})=\operatorname{Ext}(D^{\pm}).$$

\noindent {\bf{Cylindrical finite propagation speed and Cheeger Gromov Taylor type estimates.}}

\noindent To prove the index formula we need some pointwise estimates on the Schwartz kernels of functions of the leafwise Dirac operator. Our perturbation on the cylinder has the shape $D+Q$ where $Q$ is some selfadjoint order zero pseudodifferential operator on the base of the cylinder (actually $Q$ is just a sum of a uniformly smoothing operator and $u\operatorname{Id}$) in particular one can repeat the proof of energy estimates as in the Book by John Roe for example \cite{Roel} for the wave equation no more on a small geodesic ball but on a strip $\partial L_x\times (a,b)$ ($\partial L_x$ is the base of the cylinder) finding out that unitary cylindrical diffusion speed holds i.e. if $\xi_0$ is supported in $\partial L_x\times (a,b)$ then the solution of the wave equation $e^{iQ}\xi_0$ is supported in $\partial L_x\times (a-|t|,b+|t|).$ This is sufficient to extimate kernels of class schwartz spectral functions of $D$ and $Q$ following the method of Cheeger, Gromov and Taylor \cite{ChGrTa} obtaining decaying estimates for the heat kernel \begin{equation}\label{stim}|\nabla_{z_1}^l \nabla_{z_2}^k[T^me^{-tT^2}](z_1,z_2)|\leq C(k,l,m,T)e^{(|s_1-s_2|-r_1)^2/6t}.\end{equation} With the notation $[\cdot]$ one denotes the Schwartz kernel,
and $r_1$ is some positive number while $z_i=(x_i,s_i)$ are two points on the cylinder 
with $|s_1-s_2|>2r_1$. It is clear why one is brought to call these the \emph{Chegeer Gromov Taylor estimates} in the cylindrical direction.
\noindent There is also an extremely useful relative version of estimate \eqref{stim} where one can estimate the difference of the kernels of spectral functions of two operators that agree on some open subset $U$ of the cylinder. This is an estimate of the form
$$|\nablal \nablak ([f(P_1)]-[f(P_2)])_{(x,y)}|\leq \mathcal{C}(P_1,k,l,r_2)\sum_{j=0}^{2\bar{n}+l+k}\int_{\R-J(x,y)}|\hat{f}^{(j)}(s)|ds, $$ where $l$ is a leaf, $r_2>0$, \underline{$x,y  \in U$},
$Q(x,y):=\max\{\min \{d(x,\partial U);d(y,\partial U) \}-r_2;0\}$,\\$J(x,y):=\Big{(}\dfrac{-Q(x,y)}{c},\dfrac{Q(x,y)}{c}\Big{)}$
and $f\in \mathcal{S}(\R)$ is a Schwartz function to apply to the operator using the spectral theorem.

\noindent In practice we shall collect all these estimates, one for each leaf. Thanks to the uniformly bounded geometry of the leaves that run trough a compact manifold (with boundary) the constants can be made independend. This is an extremely important fact.
\bigskip

\noindent{\bf{The foliated eta invariant.}}

\noindent Since its first apparition in \cite{AtPaSi1} the eta invariant of a Dirac operator as the difference between the local and global term on the Atiyah Patodi Singer index formula 
$$1/2\eta(D_0)=\int_X\omega_D-\{\operatorname{ind}(D^+)+1/2\operatorname{dim Ker}(D_0)\}$$
 or the spectral asimmetry defined as the regular value at zero of the meromorphic function (summation over eigenvalues)
 \begin{equation}\label{etasumm}\eta_{D_0}(s):=\sum_{\lambda\neq 0}
 \dfrac{\operatorname{sign }\lambda}{|\lambda|^s},\quad \Re(s)>\operatorname{dim }\partial X \end{equation} has become a key concept a in Spectral geometry and modern Physics.
 
 \noindent The foliation eta invariant on a compact manifold (when a transverse invariant  measure is fixed) was defined independently and essentially in the same way by Peric \cite{Peric} and Ramachandran \cite{Rama} and enters into our A.P.S index formula exactly in the way it enters classically.  
It should be remarked that Peric and Ramachandran numbers are not the same. The reason is simple. Peric uses the holonomy groupoid to desingularize the space of the leaves while Ramachandran works directly on the Borel equivalence relation. Due to their global nature the eta invariants obtained are not the same. As a striking consequence one gets that on a cylindrical foliated manifold every choice of desingularization from the equivalence relation to the holonomy (or the monodromy groupoid ) leads to different index formulas with different eta invariants. This is a genuine new feature of the boundary (cylindrical) case.
 
 \noindent Since we work with the Borel equivalence relation our eta--invariant is that of Ramachandran.
 
\noindent So consider the base Dirac operator $D^{\mathcal{F}_{\partial}}$ 
the eta function\footnote{the relation with \eqref{etasumm} comes from the identity $\operatorname{sign}(\lambda)|\lambda|^{-1}=\Gamma\big(\frac{s+1}{2}\big)^{-1}\int_{0}^{\infty}t^{\frac{s-1}{2}}\lambda e^{-t\lambda^2}dt$  } of $D^{\mathcal{F}_{\partial}}$ is defined for $\Re(s)\leq 0$ by
$$\eta(D^{\mathcal{F}_{\partial}},s):=\dfrac{1}{\Gamma((s+1)/2)}\int_0^{\infty}t^{\frac{s-1}{2}}\tru(D^{\mathcal{F}_{\partial}}e^{-D^{\mathcal{F}_{\partial}}})dt,\quad |\lambda|>0,\quad s>-1.$$ It can be shown that this is meromorphic for $\Re(\lambda)\leq 0$ with simple poles at $(\operatorname{dim}\mathcal{F}_{\partial}-k)/2,\quad k=0,1,...$ and a regular value at zero. 

\noindent In this section we describe this result and we extend it to some classes of perturbations of the operator needed in the proof of the index formula. We shall consider perturbations of the form 
$Q=D^{\mathcal{F}_{\partial}}+K$
 with $K$ some uniformly smoothing spectral function 
 $K=f(D^{\mathcal{F}_{\partial}})$, $f:(-a,a)\longrightarrow \R$. For 
$f=\chi_{(-\epsilon,\epsilon)}$ more can be said about the family 
$Q_u:=D^{\mathcal{F}_{\partial}}+D^{\mathcal{F}_{\partial}}f(D^{\mathcal{F}_{\partial}})+u$ in fact we can define
$$\eta_{\Lambda}(Q_u)={\operatorname{LIM}}_{\delta \rightarrow 0}\int_{\delta}^{k}\dfrac{t^{-1/2}}{\Gamma(1/2)}\tru(Q_ue^{-tQ_u^2})dt+\int_{k}^{\infty}\dfrac{t^{-1/2}}{\Gamma(1/2)}\tru(Q_ue^{-tQ_u^2})dt$$
where ${\operatorname{LIM}}$ is the constant term in the asymptotic development in powers of 
$\delta$ near zero of the function 
$\delta \longmapsto \int_{\delta}^{k}$. Moreover two important formulas hold true \begin{itemize}
\item $\eta_{\Lambda}(Q_u)-\eta_{\Lambda}(Q_0)=\operatorname{sign}(u)\tru(f({D}^{\mathcal{F}_{\partial}})$
\item \begin{equation}\label{etasplit}\eta_{\Lambda}(Q_0)=1/2(\eta_{\Lambda}(Q_u)+\eta_{\Lambda}(Q_{-u})).\end{equation}
\end{itemize}
This only requires a minimal modification of the proof given by Vaillant \cite{Vai} for Galois coverings. The point is that the relevant Von Neumann algebras are closed under the Borel functional calculus. 
\bigskip

\noindent{\bf{The index formula.}}

\noindent Finally we prove the index formula 
\begin{align*}\operatorname{ind}_{L^2,\Lambda}(D^+)=\langle\hat{A}(X)\operatorname{Ch}(E/S),[C_{\Lambda}]\rangle +1/2[\eta_{\Lambda}(D^{\mathcal{F}}_0)-h^+_{\Lambda}+h^+_{\Lambda}]\end{align*} where
$h^{\pm}_{\Lambda}:=\operatorname{dim}_{\Lambda}(\operatorname{Ext}(D^{\pm})-\operatorname{dim}_{\Lambda}(\operatorname{Ker}_{L^2}(D^{\pm}).$ 
Our proof is a modification of Vaillant proof that in turn is inspired by M\"uller proof of the $L^2$--index formula on manifolds with corners of codimension two \cite{Mu}. This is a (of course) a proof based on the heat equation. 

\noindent The starting point is the identity
\begin{equation}\label{indicelimite}\operatorname{ind}_{L^2,\Lambda}(D^+_{\epsilon})=\lim_{u\downarrow 0}
1/2\{\operatorname{ind}_{\Lambda}(D^+_{\epsilon,u})+\operatorname{ind}_{\Lambda}(D^+_{\epsilon,-u})+h^-_{\Lambda,\epsilon}-h^+_{\Lambda,\epsilon}\}\end{equation}
where 
$$h^{\pm}_{\Lambda,\epsilon}:=\operatorname{dim}_{\Lambda}\operatorname{Ext}(D^{\pm}_{\epsilon})-\operatorname{dim}_{\Lambda}\operatorname{Ker}_{L^2}(D^{\pm}_{\epsilon})$$ definition also valid for $\epsilon=0$.

\noindent Next we prove
\begin{equation}\label{indpiu}\operatorname{ind}_{\Lambda}(D_{\epsilon,u}^+)=\langle \hat{A}(X)\operatorname{Ch}(E/S),[C_{\Lambda}]\rangle +1/2 \eta_{\Lambda}(D^{\mathcal{F}_{\partial}}_{\epsilon,u})+g(u)\end{equation} with $g(u)\longrightarrow_u 0$.

\noindent Equation \eqref{indpiu} combined with \eqref{indicelimite} and \eqref{etasplit} becomes, after the $u$--limit
$$
\operatorname{ind}_{L^2,\Lambda}(D^+_{\epsilon})=\langle\hat{A}(X)\operatorname{Ch}(E/S),[C_{\Lambda}]\rangle+1/2\eta_{\Lambda}(D^{{\mathcal F}_0})+h^-_{\epsilon}-h^-_{\epsilon}.
 $$
\noindent The last step is to assure that under $\epsilon \rightarrow 0$ each $\epsilon$--depending object in the above equation goes to the corresponding value for $\epsilon=0$.
\bigskip

\noindent Some words about the proof of \eqref{indpiu}. This is inspired from the work of M\"uller \cite{Mu}. We start from the convergence into the space of leafwise smoothing kernels of 
$[\operatorname{exp}({-t D^2_{\epsilon,u}})]$ to 
$[\operatorname{Ker}_{L^2}(D_{\epsilon,u})]$. The choice of cut off functions $\phi_k$ supported in 
$X_{k+1}$ ($X_m$ is the manifold truncated at $r=m$) gives an exaustion of $X$ into compact pieces. Consider the equation
\begin{eqnarray}\nonumber\operatorname{ind}_{\Lambda}(D^{+}_{\epsilon,u})=\operatorname{str}_{\Lambda}\chi_{\{0\}}(D_{\epsilon,u})=\lim_{k\rightarrow +\infty}\lim_{t\rightarrow +\infty}\operatorname{str}_{\Lambda}(\phi_k e^{-tD_{\epsilon,u}^2}\phi_k)=\\ \label{finalc}
\lim_{k\rightarrow \infty}\operatorname{str}_{\Lambda}(\phi_ke^{-sD^2_{\epsilon,u}}\phi_k)-\int_s^{\infty}\operatorname{str}_{\Lambda}(\phi_kD^2_{\epsilon,u}e^{-tD^2_{\epsilon,u}}\phi_k)dt.\end{eqnarray}
\noindent The $t$--integral is splitted into $\int_s^{\sqrt{k}}+\int_{\sqrt{k}}^{\infty}$ the second one going to zero thanks to the Breuer--Fredholm property of 
$D_{\epsilon,u}$. More work is needed in the study of the first one, the responsible for the presence of the eta invariant in the formula. Using heavily the relative version of the Cheeger--Gromov--Taylor estimate \eqref{stim} one shows that 
$$\lim_{k\rightarrow \infty}\operatorname{LIM}_{s\rightarrow 0}\int _{s}^{\sqrt{k}}=1/2 \eta_{\Lambda}(D_{\epsilon,u}^{{\mathcal{F}}_0}).$$ \noindent The first summand in \eqref{finalc}
will lead to the well known local term
$$\lim_{k\rightarrow \infty}\operatorname{LIM}_{s\rightarrow 0} \operatorname{str}_{\Lambda}(\phi_ke^{-sD^2_{\epsilon,u}}\phi_k)=\langle\hat{A}(X)\operatorname{Ch}(E/S),[C_{\Lambda}]\rangle.$$
This requires some work in developing the appropiate asymptotic expansion. Again we have to consider three pieces of $X$ separately making use of relative kernel estimates to compare the perturbed operator with the original one.
\bigskip

\noindent{\bf{Comparison with Ramachandran index formula.}}

\noindent There's in the literature anothere A.P.S. formula for a Dirac Operator; this the Ramachandran index theorem \cite{Rama}. This corresponds exactly to the first point of view of A.P.S. index theorem as a solution of a boundary value problem with non--local boundary condition. This section is devoted to show the compatibility of our index formula with that of Ramachandran. This is an important aspect since it represents, for foliations the passage 
from incomplete closed case to te cylindrical one as in A.P.S.
\bigskip

\noindent{\bf{The signature formula.}}

\noindent The main application of our index formula should be a Hirzebruch type formula for the signature.
First we review the A.P.S. version of the Hirzebruch formula, 
$$\sigma(X)=\int_XL-\eta(B^{\operatorname{ev}}).$$
in particular the cohomological interpretation they give to the index of their boundary value problem of the signature operator, in fact harmonic forms on the elonged manifolds are naturally isomorphic to the image of the relative 
cohomology into the absolute one. This is exactly the reason why A.P.S. have to attach an infinite cylinder. Notice that while the Hirzebruch formula for a closed manifold can proved by only topological methods (Hirzebruck used Thom's theory of cobordism) the formula for a manifold with boundary is proved up to now only with analytical methods. 
 
 \noindent In the case of Galois coverings the theory of $\Gamma$--Hilbert modules with their formal dimension permitted to Luck and Schick \cite{lusc} to define (at least) three equivalent type of  $L^2$ signatures for a regular covering $\Gamma-\tilde{M} \longrightarrow M$ of a manifold with boundary (cylindrical end).
This signatures are: \underline{analytical signature} based upon the index of the signature operator, \underline{harmonic signature} looking at the harmonic forms on the cylinder, \underline{de Rham signature}, based on the relative de Rham $L^2$ cohomology of the covering. The proof of their equivalence is very tricky, the first, harmonic=analytical is made by Vaillant \cite{Vai}, the second is by Luck and Schick and reminds of course the cohomological interpretation A.P.S give but uses strongly the weakly exact long sequence in $L^2$ together with some essential properties of the $\Gamma$ dimension (again a Von Neumann dimension). 

\noindent We show that the theory of Random Hilbert spaces of Connes is strong enough to generalize all the properties of the $\Gamma$--dimension, so we are able to generalize the long weakly exact sequence as proved by Cheeger and Gromov \cite{ChGr} valid in the sense of the various Von Neumann algebras defined by leafwise differential forms and the transverse measure. 

\noindent So we prove that the three defined notions of signature coincide and the following Hirzebruch signature is valid,
$$\sigma_{\Lambda,\,\operatorname{dR}}(X_0,\partial X_0)=\langle L(X),[C_{\Lambda}]\rangle +1/2[\eta_{\Lambda}(D^{\mathcal{F}_{\partial}})].$$
\bigskip
\newpage
\noindent {\bf{Acknowledgements}}

\noindent The author wishes to thank his thesis advisor Paolo Piazza for having suggested such a beautiful research topic and for a lot of interesting discussions; his own family for the continuous support,
 the equipe of the
projet Alg\'ebres d'op\'erateurs at the
institut de Math\'ematiques de Jussieu in Paris, in particular Georges Skandalis and Eric Leitchnam for their hospitality and mathematical advices.
Also a thank for useful discussions on index theory is due to Moulay T. Benameur.
\newpage

\section{Geometric Setting}\label{geom}
\begin{dfn}\label{12340}
\noindent A $p$--dimensional foliation $\mathcal{F}$ on a manifold with boundary $ X_0$ is transverse to the boundary if it is given by a foliated atlas $\{U_{\alpha}\}$ with homeomorphisms $\phi_{\alpha}:U_{\alpha}\longrightarrow V_{\alpha}\times W_{\alpha}$ with $V_{\alpha}$ open in ${\mathbb{H}}^p:=\{(x_1,...,x_p)\in \R^p:x_1\geq 0\}$ and $W^{q}$ open in $\R^q$ with change of coordinated $\phi_{\alpha}(u,v)$ of the form
 \begin{equation}\label{9009}v'=\phi(v,w),\quad w'=\psi(w)\end{equation} ($\psi$ is a local diffeomorphism).
 Such an atlas is assumed to be maximal among all collections of this type. 
 The integer $p$ is the dimension of the foliation, $q$ its codimension and $p+q=\operatorname{dim}(X_0).$
\end{dfn}
In each foliated chart, the connected components of subsets as $\phi_{\alpha}^{-1}(V_{\alpha}\times \{w\})$ are called \underline{plaques}. The plaques coalesce (thanks to the change of coordinate condition \eqref{9009}) to give maximal connected injectively immersed (not embedded !) submanifolds called \underline{leaves}. One uses the notation $\mathcal{F}$ for the set of leaves. Note that in general each leaf passes infinitely times trough a foliated chart so a foliation is only locally a fibration.
Taking the tangent spaces to the leaves one gets an integrable subbundle $T\mathcal{F}\subset TX_0$ that's transverse to the boundary i.e $T\partial X_0 + T\mathcal{F}=TX_0$ in other words the boundary is a submanifold that's transverse to the foliation.
\subsection{Holonomy}
\noindent We skip the definition of a foliation on a manifold without boundary and only recall that is defined by foliated charts as in the definition \ref{12340} above with local models $U\times V$ where $U$ is an open set in $\R^p$
Let $X$ a manifold equipped with a $(p,q)$--foliation. If $X$ has boundary the foliation is assumed transverse to the boundary according to definition \ref{12340}.

\begin{dfn}
A map $f:X\longrightarrow \R^q$ is called \underline{distinguished} if each point $x$ is in the domain of a foliated chart $U\overset{\phi}{\longrightarrow}V\times W$ such that $f_{|U}=\phi \circ \operatorname{Pr}_V$ where $\operatorname{Pr}_V:U\times V\longrightarrow V$ is the projection on the second factor.
\end{dfn}\noindent Let $\mathcal{D}$ the collection of all the germs of distinguished maps
with the obvious projection 
$\sigma:\mathcal{D}\longrightarrow X$ sending the germ of $f$ at $x$ onto $x$. Consider a foliated chart $(U,\phi)$ and $P$ a plaque of $U$, then $P$ individuates the set $\tilde{P}\subset \mathcal{D}$ of the \underline{distinguished germs} 
$\{[\phi \circ \operatorname{Pr}_V]_x\}_{x\in P}$. 
When $P$ varies over all the possible foliated charts these sets form the base of a topology of a $p$--dimensional manifold on $\mathcal{D}$ called the \underline{leaf topology}. The mapping 
$\sigma:\mathcal{D}\longrightarrow \mathcal{F}$ is a covering (\cite{hae}) where  $\mathcal{F}$ is the non paracompact manifold equal to the disjoint union of all the leaves (equivalently use the plaques to give $X$ a topology where the connected components are exactly the leaves with their natural topology).
Let $\gamma:x \longrightarrow y$ a continuous leafwise path. Since $\sigma$ is a covering map there's a holonomy map 
$h_{\gamma}:\sigma^{-1}(x)\longrightarrow \sigma^{-1}(y)$ sending the point $\pi \in \sigma^{-1}(x)$ into the endpoint of the unique lifting $\tilde{\gamma}$ of $\gamma$ starting from $\pi$.
\begin{dfn}
A $q$--dimensional submanifold $Z\subset X$ is a \underline{transversal} if for every $z\in Z$ there exists a distinguished map $\pi:U\longrightarrow \R^q$ such that $\pi_{|Z\cap U}$ is an homeomorphism.
\end{dfn}
\noindent There are many equivalent definitions of transverse submanifold for example at infinitesimal level, one can ask, $T_zZ\oplus T_z\mathcal{F}=T_zX$. The definition given here makes possible to realize that  holonomy acts in a natural way on the disjoint union of all transversals \cite{Pla}. 

\noindent First we give a slight different version of holonomy. For a continuous leafwise path $\gamma:x\longrightarrow y$ we can choose a path of foliated charts $(U_0,\phi_1),...,(U_k,\phi_m)$ associated to a decomposition $0=s_0,...,1=s_m$ of $[0,1]$ such that $\gamma_{|[s_l,s_{l+1}]}\subset U_l$ and each plaque of $U_l$ meets only a plaque of $U_{l+1}$. Following the plaques along $\gamma$ one obtain a mapping of the plaques of $U_0$ to the plaques $U_m$ hence, composing with the distinguished maps associated a germ of diffeomorphism of $\R^q$. Since the inclusion of a transversal compose with a distinguished mapping to give coordinates on the transversal this is also a germ of diffeomorphism $H_{T_0T_1}(\gamma)$ of transversals $T_0$ around $x$ and $T_1$ around $y$.

\noindent The connection with the holonomy map given before in terms of the holonomy covering is given as follows. Let $\pi \in \sigma^{-1}(x)$ and $f$ a distinguished map defined around $x$. The diffeomorphism $H_{T_0T_1}(\gamma)$ allows to define a local coordinate system on $T_1$ defined around $y$ and in turn a distinguished map $f_1:V\longrightarrow \R^q$ defined around $y$. Then the germ of $f_1$ at $y$ coincides with $h_{\gamma}(\pi)\in \sigma^{-1}(y).$

\noindent It is clear that the relation 
\begin{equation}\label{holonomyrel}\gamma\sim \tau\quad  \operatorname{iff}\quad h_{\gamma}= h_{\gamma}(\tau)\end{equation} 
is weaker than homotopy (obvious by the definition in terms of lifting).
\begin{dfn}The \underline{holonomy groupoid} $G$ of the foliation is the quotient of the homotopy groupoid (the set of all equivalence fixed points homotopy classes of leafwise continuous paths) under the relation \eqref{holonomyrel}.
\end{dfn}
\noindent One can show that this procedure gives a finite dimensional reduction of the homotopy groupoid. In fact in the case 
$\partial X=\emptyset$ $G$ is a smooth, in general non--Hausdorff $2p+q$--dimensional manifold where the local coordinates are given by mappings in the form of $(U \times V)\times_{h_{\gamma}} (U'\times V')$ where $x\in U\times V$, $y\in U'\times V'$, $\gamma:x\longrightarrow y$ is a leafwise path and one uses the graph of the holonomy $h_{\gamma}:V\longrightarrow V'$ (\cite{Va,Cos,MoSc}). 
 Finally \begin{dfn} A \underline{pseudogroup} of a manifold $X$ is a family $\Gamma$ of diffeomorphisms defined on open subsets of $X$ such that
 \begin{enumerate}
\item if $\Phi \in \Gamma$  then $\Phi^{-1} \in \Gamma$
\item $\Gamma$ is closed under composition when possible (depending on domains and ranges).
\item If $\Phi:U\longrightarrow W$ is in $\Gamma$ then every restriction of $\Phi$ to open subsets $ V\subset U$ is in $\Gamma$.
\item If $\Phi:U\longrightarrow W$ is a diffeomorphism such that every point in $U$ has a neighborhood on which $\Phi$ restricts to an alement of $\Gamma$ then $\Phi \in \Gamma$.
\item The identity belongs to $\Gamma$.
\end{enumerate}
The \underline{holonomy pseudogroup} of a foliation is the pseudogroup $\Gamma$ acting on the disjoint union of all (regular) whose germs at every point are germs of holonomy mappings defined by some  leafwise path.
\end{dfn}
\begin{exemple}\label{foliatedflat}{\bf{Foliated flat bundles (with boundary)}}. A huge class of foliations comes from the theory of the foliated flat bundles.
So let $Y$ an $n$--dimensional compact Riemannian manifold with boundary where the metric is product type near the boundary. Look at its universal cover $\widetilde{Y}\longrightarrow Y$ as a principal bundle with discrete structural group $\Gamma:=\pi_1(Y)$ acting on the right as deck trasformations. Lift the metric upside, so $\tilde{Y}$ becomes a manifold with boundary and bounded geometry (definition in section \ref{bbg}). For every representation $\Phi:\Gamma \longrightarrow \operatorname{Diffeo}(T)$ of $\Gamma$ as a group of diffeomorphisms of a closed manifold (no boundary) $T$ one can form the associated flat bundle $V\longrightarrow Y$ with fiber $T$ simply composing a cocycle of transition functions of $\widetilde{Y}$ with $\Phi$. More explicitely consider the \underline{right}, \underline{free} and \underline{proper} action $\gamma \circlearrowright \widetilde{Y}\times T$ defined by 
$$(\widetilde{y},\theta)\cdot \gamma:=(\widetilde{y},\gamma \theta)$$ so the quotient space is a compact manifold with boundary $V$ and we have a diagram
$$\xymatrix{
{}&\widetilde{Y}\times T\ar[d]\ar[dl]^-{q}\\
V\ar[dr]^-{p}&\widetilde{Y}\ar[d]^-{\pi}\\
{}&Y
}$$
where $V\rightarrow Y$ is a bundle with fiber $T$. Consider $\widetilde{Y}\times T$, it has the trivial foliation $\widetilde{\mathcal{F}}$ with leaves $\widetilde{Y}_{\theta}:=\widetilde{Y}\times\{\theta\}$ that's normal to the boundary and pushes
down with $q$ to an $n$--dimensional foliation $\mathcal{F}$ on $V$ where each leaf is diffeomorphic to $\widetilde{Y}_{\theta}/\Gamma(\theta)$, where $$\Gamma(\theta):=\{\gamma \in \Gamma:\Phi(\gamma)\theta=\theta\}$$ is the stabilizer of $\theta$. Note also that each leave covers $Y$ trough the composition $$\widetilde{Y}_\theta \longrightarrow \wt{T}_\theta/\Gamma(\theta)\longrightarrow^{p}Y.$$ 

\noindent It is a well known fact \cite{MoSc} that here the holonomy group $G_x^x$ is the image of the holonomy mapping $$\pi^1(L_\theta)\simeq \Gamma(\theta)\longrightarrow \operatorname{Homeo}(F,x)$$ where $\operatorname{Homeo}(F,x)$ is defined as the group of germs of homeomorphisms of $F$ keeping $x$ fixed.

\noindent So as an explicit example one can take a closed Riemann surface $\Sigma$ of genus $g>1$ and $\Gamma=\pi_1(\Sigma)$ is a discrete subgroup of $PSL(2,\R)$. 
Choose points $\{p_1,...,p_k\}$ and $D_j$ a small disc around $p_j$. Let $X:=\Sigma \setminus \cap_{j=1,...,k}D_j$ be the base manifold and
the Galois cover is $\Gamma\longrightarrow \widetilde{X}\longrightarrow X$ induced by the universal cover $\mathbb{H}^2\longrightarrow \Sigma$, $T:=\mathbb{S}^1$, with $\Gamma$ acting on $\mathbb{S}^1$ by fractional linear transformations.
\end{exemple}

\subsection{Longitudinal Dirac operator}
\noindent Let $X=X_0\cup Z$ be a connected manifold with cylindrical end meaning that $X_0$ is a compact manifold with boundary and $Z=\partial X_0 \times [0,\infty)_r$ is the cylindrical end. 
Suppose that $X$ has a Riemannian metric $g$ that is product type on the cylinder $g_{|Z}=g_{\partial X_0}+ dr \otimes dr$. 

\noindent Let given on $X$ a smooth oriented foliation $\mathcal{F}$ with leaves of dimension $2p$ respecting the cylindrical structure i.e. 
\begin{enumerate}
\item The submanifold $\partial X_0$ is transversal to the foliation and inherits a $(2p-1,q)$ foliation $\mathcal{F}_{\partial}=\mathcal{F}_{|\partial X_0}$ with foliated atlas given by $\phi_{\alpha}:U_{\alpha}\cap \partial X_0 \longrightarrow \partial V_{\alpha} \times W_{\alpha}$. Note that the codimension is the same.
\item The restriction of the foliation on the cylinder is product type $\mathcal{F}_{|Z}=\mathcal{F}_{\partial}\times [0,\infty).$\end{enumerate}
\noindent Note that these conditions imply that the foliation is \underline{normal} to the boundary.
\noindent The orientation we choose is the one given by $(e_1,..,e_{2p-1},\partial_r)$ where $(e_1,..,e_{2p-1})$ is a positive leafwise frame for the induced boundary foliation. As explained in \cite{AtPaSi1} this is a way to fix the boundary Dirac type operator.
\noindent Let $E\longrightarrow X$ be a leafwise Clifford bundle with leafwise Clifford connection $\nabla^E$ and Hermitian metric $h^E$. Suppose each geometric structure is of product type on the cylinder meaning that if $\rho:\partial X_0\times [0,\infty) \longrightarrow \partial X_0$ is the base projection
$$E_{|Z}\simeq \rho^*(E_{|\partial X_0}),\quad h^E_{|\partial X_0}=\rho^*(h^{E}_{|\partial X_0}),\quad \nabla^E_{|Z}=\rho^*(\nabla^{E}_{|\partial X_0}).$$
\noindent Each geometric object restricts to the leaves to give a longitudinal Clifford module that's canonically $\mathbb{Z}_2$ graded by the leafwise chirality element. One can check immediately that the positive and negative boundary eigenbundles $E^+_{\partial X_0}$ and $E^-_{\partial X_0}$ are both modules for the Clifford structure of the boundary foliation (see Appendix \ref{cliffo} for more informations).
\noindent Leafwise Clifford multiplication by $\partial_r$ induces an isomorphism of leafwise Clifford modules between the positive and negative eigenbundles
$$c(\partial_r):E_{\partial X_0}^+\longrightarrow E_{\partial X_0}^-.$$ Put $F=E^+_{|\partial X_0}$ the whole Clifford module on the cylinder $E_{|Z}$ can be identified with the pullback $\rho^*(F\oplus F)$ with the following action:
tangent vectors to the boundary foliation $v\in T\mathcal{F}_{\partial}$ acts as
 $c^{E}(v)\simeq c^F(v)\Omega$ with $\Omega=\left(\begin{array}{cc}0 & 1 \\1 & 0\end{array}\right)$ while in the cylindrical direction $c^F(\partial_r)\simeq \left(\begin{array}{cc}0 & -1 \\1 & 0\end{array}\right)$. 
\noindent In particular one can form the longitudinal Dirac operator assuming under the above identification the shape\footnote{we choose to insert $-\partial_r$ the inward pointing normal to help the comparison with the orientation of A.P.S  }
\begin{equation}
\label{diracco}
D=c(\partial_r)\partial_r+c_{|\mathcal F_0}\nabla^{E_{|\mathcal{F}_{\partial}}}=
c(\partial_r)\partial_r+\Omega D^{\mathcal{F}_{\partial}}=
c(-\partial_r)[-\partial_r-c(-\partial_r)\Omega D^{\mathcal{F}_{\partial}}].\end{equation}
Here 
$D^{\mathcal{F}_{\partial}}$ is the leafwise Dirac operator on the boundary foliation.  
\noindent In the following, these identifications will be omitted letting $D$ act directly on $F\oplus F$ according to
\begin{align}\nonumber&\left(\begin{array}{cc}0 & D^- \\D^+ & 0\end{array}\right):F\oplus F\longrightarrow F\oplus F\\ \nonumber
&\left(\begin{array}{cc}0 & D^- \\D^+ & 0\end{array}\right)=\left(\begin{array}{cc}0 & -\partial_r+D^{\mathcal{F}_{\partial}} \\ 
 \partial_r+D^{\mathcal{F}_{\partial}}
& 0\end{array}\right)=\left(\begin{array}{cc}0 & \partial_u+D^{\mathcal{F}_{\partial}} \\ 
 -\partial_u+D^{\mathcal{F}_{\partial}}
& 0\end{array}\right)
\end{align} where $u=-r$, $\partial_u=-\partial_r$ (interior unit normal)
note this is the opposite of A.P.S. notation.
\noindent We are using the notation $X=X_k\cup Z_k$ with $Z_k=\partial X_0\times [k,\infty)$ and $X_k=X_0\cup (\partial X_0\times [0,k])$ also $Z_a^b:=\partial X_{0}\times [a,b]$ and where there's no danger of confusion $Z_x$ is the cylinder of the leaf passing trough $x$, $Z_x=L_x\cap Z_0.$

\section{The Atiyah Patodi Singer index theorem}\label{aaps}
\noindent We are going to recall the classical Atiyah--Patodi--Singer index theorem in \cite{AtPaSi1}
So let $X_0$ a compact $2p$ dimensional manifold with boundary $\partial X_0$ and consider a Clifford bundle $E$ with all the geometric structure  as in the previous section. We take here the opposite orientation of A.P.S i.e. we use the exterior unit normal to induce the boundary operator instead of the interior one; as pointed out by A.P.S
this is a way to declare what is the positive eigenbundle for the natural splitting. In other words 
$$D^{+}_{\textrm{here}}=D^{-}_{\textrm{APS}}.$$ 
The operator writes in a collar around the boundary $$
\left(\begin{array}{cc}0 &D^- \\ 
 D_+
& 0\end{array}\right)=
\left(\begin{array}{cc}0 & -\partial_r+D_0 \\ 
 \partial_r+D_0
& 0\end{array}\right)$$ where $\partial_r$ is the exterior unit normal and $D_0$ is a Dirac operator on the boundary. It is shown in \cite{AB} that the $K$--theory of the boundary manifold contains topological obstructions to the existence of elliptic boundary value conditions of local type (for the signature operator they are always non zero). If one enlarges the point of view to admit global boundary conditions a Fredholm problem can properly set up. More precisely, consider the boundary operator $D_0$ acting on the boundary manifold $\partial X_0$. This is a first order elliptic differential operator with real discrete spectrum on $L^2(\partial X_0;F)$. Let $P=\chi_{[0,\ty)}(D_0)$ be the spectral projection on the non negative part of the spectrum. This is a pseudo--differential operator (\cite{AtPaSi1}). Atiyah Patodi and Singer prove the following facts 
\begin{itemize}
\item The (unbounded) operator $D^+:C^{\ty}(X;E^+,P)\longrightarrow C^{\ty}(X,E^-)$ with domain $$C^{\ty}(X;E^+,P):=\{s\in C^{\ty}(X;E^+):P(s_{|\partial X_0})=0\}$$ is Fredholm and the index is given by the formula
$$\operatorname{ind}_{\operatorname{APS}}(D^+)=\int_{X_0} \hat{A}(X,\nabla )\operatorname{Ch}'(E,\nabla)-h/2+\eta(0)/2$$ with the Atiyah--Singer 
$\hat{A}(X,\nabla)$ 
differential form\footnote{as explained in the introduction; due to the presence of the boundary one does not have here a cohomological pairing, for this reason the notation $\hat{A}(X,\nabla)$ stresses the dependence from the metric trough the connection $\nabla$} with the twisted Chern character $\operatorname{Ch}'(E,\nabla)$ \cite{BeGeVe,Me} and two correcting terms:
\begin{enumerate}
\item $h:=\operatorname{Ker}(D_0)$ is the dimension of the kernel of the boundary operator
\item $\eta(0)$, the eta invariant of $D_0$ is a spectral invariant which 
gives a measure of the asymmetry of the spectrum of the boundary operator $D_0$. This is extensively explained in section \ref{eta}.
\end{enumerate}

\item The index formula can be interpreted as a natural $L^2$ problem on the  manifold with a cylinder attached, $X=X_0\cup_{\partial X_0}( \partial X_0 \times [0,\infty))$ with every structure pulled back. 
\noindent More precisely the kernel of $D^{+}:C^{\ty}(X;E^+,P)\longrightarrow C^{\ty}(X,E^-)$ is naturally isomorphic to the kernel of $D^{+}$ extended to an ubounded operator on $L^2(X)$ while to describe the kernel of its Hilbert space adjoint i.e. the closure of $D^-$ with the adjoint boundary condition
$D^{-}:C^{\ty}(X;E^-,1-P)\longrightarrow C^{\ty}(X,E^+)$ we have to introduce the space of extended $L^2$ solutions.

\noindent A locally square integrable solution $s$ of the equation $D^-s=0$ on $X$ is called an extended solution if for large positive $r$
\begin{equation}\label{ex2}
s(y,r)=g(y,r)+s_{\ty}(y)\end{equation} where $y$ is the coordinate on the base $\partial X_0$ and $g\in L^2$ while $s_{\ty}$ solves $D_0 s_{\ty}=0$ and is called \underline{the limiting value} of $s$.

\noindent APS prove that the kernel of $(D^+)^*$ (Hilbert space adjoint of $D^{+}:C^{\ty}(X;E^+,P)\longrightarrow C^{\ty}(X,E^-)$ ) is naturally isomorphic to the space of $L^2$ extended solution of $D^-$ on $X$.
Moreover \begin{equation}\label{ldueindex}\operatorname{ind}_{\textrm{APS}}(D^+)=\operatorname{dim}_{L^2}(D^+)-\operatorname{dim}_{L^2}(D^-)-h_{\ty}(D^{-})=\operatorname{ind}_{L^2}(D^+)-h_{\ty}(D^{-})\end{equation} where $\operatorname{ind}_{L^2}(D^+):=\operatorname{dim}_{L^2}(D^+)-\operatorname{dim}_{L^2}(D^-)$
and
 the number $h_{\ty}(D^{-})$ is the dimension of the space of limiting values of the extended solutions of $D^-$.
In this sense the APS index can be interpreted as an $L^2$--index. The number at right in \eqref{ldueindex} is called often the $L^2$ \underline{extended index}. Along the proof of \eqref{ldueindex} the authors prove that
\begin{equation}\label{hsp}h=h_{\ty}(D^+)+h_{\ty}(D^-)\end{equation} 
and conjecture that it must be true at level of the kernel of $D_0$ i.e. 

\noindent \emph{every section in} 
$\operatorname{Ker}(D_0)$ 
\emph{is uniquely expressible as a sum of limiting values coming from} 
$D^+$ 
\emph{and} 
$D^-$. 

\noindent The conjecture was solved by Melrose with the invention of the $b$--calculus, a pseudo--differential calculus on a compactification of $X$ that furnished a totally new point of view on the APS problem \cite{Me}. 

\noindent With \eqref{ldueindex} and \eqref{hsp} the index formula is
$$\operatorname{ind}_{L^2}(D^+)=\int_{X_0} \hat{A}(X)\operatorname{Ch}(E)+\dfrac{\eta(0)}{2}+\dfrac{h_{\ty}(D^-)-h_{\ty}(D^+)}{2}.$$
\end{itemize}
Finally a naive remark about the introduction of extended solutions in order to motivate our definition of $h_{\infty}(D^{\pm})$ (equation \eqref{deffh} and \eqref{1001}) in our Von Neumann setting. For a real parameter $u$ say that a distributional section $s$ on the cylinder is in the weighted $L^2$--space 
$e^{ur}L^2(\partial X_0\times [0,\ty);E^{\pm})$ if $e^{-ur}s\in L^2$. The operator $D^{\pm}$ trivially esxtends to act on each weighted space. Now it is evident from \eqref{ex2} that an $L^2$--extended solution of the equation $D^+s=0$ is in each $e^{ur}L^2$ for positive $u$. Viceversa let $s\in \bigcap_{u>0}\operatorname{Ker}_{e^{ur}L^2}(D^+)$. Keep $u$ fixed, then $e^{-ur}s\in L^2$ can be represented in terms of a complete eigenfunction expansion for the boundary operator $D_0$,
$$e^{-ur}s=\sum_{\lambda}\phi_{\lambda}(y)g(r).$$ Solving $D^+s=0$ together with the condition $e^{-ur}s\in L^2$ leads to the representation (on the cylinder)
$s(y,r)=\sum_{\lambda >-u}\phi_{\lambda}(y)g_{0\lambda}(y)e^{-\lambda r}.$ Since $u$ is arbitrary we see that $s$ should have a representation as a sum
$$s(y,r)=\sum_{\lambda \geq 0}\phi_{\lambda}(y)g_{0\lambda}e^{-\lambda r}$$
over the non negative eigenvalues of $D_0$, i.e. $s$ is 
an extended solution with limiting value $\sum_{\lambda=0}\phi_0(y)g_{00}.$
 We have proved that
$$\operatorname{Ext}(D^{\pm})=\bigcap_{u>0}\operatorname{Ker}_{e^ur L^2}(D^{\pm}).$$

\section{Von Neumann algebras, foliations and index theory}

\subsection{Non--commutative integration theory.}\label{paoletto}
The measure--theoretical framework of non--commutative integration theory is particular fruitful when applied to measured foliations. 
\noindent The non--commutative integration theory of Alain Connes \cite{Co} provides us a measure theory on every measurable groupoid $(G,\mathcal{B})$ with $G^{(0)}$ the space of unities. In our applications $G$ will be mostly the equivalence relation $\mathcal{R}$ or sometimes the holonomy groupoid of a foliation. Transverse measures in non--commutative integration theory sense will be defined from holonomy invariant transverse measures. Below a list of fundamental objects and facts. This very brief and simplified survey in fact the general theory admits the existence of a modular function that says, in the case of foliations how transverse measure of sets changes under holonomy(under flows generated by fields tangent to the foliation). Hereafter our modular function is everywhere 1, corresponding to the geometrical case of a foliation equipped with a \underline{holonomy invariant transverse measure} (this is a definition we give below).

\begin{description}
\item[Measurable groupoids]. A groupoid is a small cathegory $G$ where every arrow is invertible. The set of objects is denoted by $G^{(0)}$ and there are two maps 
$s,r:G\longrightarrow G^{(0)}$ where $\gamma:s(\gamma)\longrightarrow r(\gamma).$ Two arrows 
$\gamma_1,\gamma_2$ can be composed if 
$r(\gamma_2)=s(\gamma_1)$ and the result is $\gamma_1\cdot \gamma_2$. The set of composable arrows is $G^{(2)}=\{(\gamma_1,\gamma_2):r(\gamma_2)=s(\gamma_1)\}$. As a notation 
$G_x=r^{-1}(x)$,  $G^x=s^{-1}(x)$ for $x\in G^{(0)}.$
 An equivalence relation $\mathcal{R}\subset X\times X$ is a groupoid with $r(x,y)=x$ and $s(x,y)=y$, in this manner $(z,x)\cdot (x,y)=(z,y)$. The range of the map $(r,s):G\longrightarrow G^{(0)}\times G^{(0)} $ is an equivalence relation called {\underline{the principal groupoid associated}} to $G$. In this sense groupoids desingularize equivalence relations.
A \underline{measurable groupoid} is a pair 
$(G,\mathcal{B})$ where $G$ is a groupoid and 
$\mathcal{B}$ is a $\sigma$--field on $G$ making measurable the structure maps $r,s$, composition 
$\circ:G^{(2)}\longrightarrow G$ and the inversion $\gamma \longmapsto \gamma^{-1}$. 
\item[Kernels] are mappings $x\longmapsto \lambda^x$ where $\lambda^x$ is a positive measure on $G$, supported on the $r$--fiber $G^x=r^{-1}(x)$ with a measurability property i.e. for every set $A\in \mathcal{B}$ the function $y\longmapsto \lambda^{y}(A)\in [0,+\infty]$ must be measurable.

\noindent
A kernel $\lambda$ is called {\underline{proper}} if there exists an increasing family of measurable sets $(A_n)_{n\in \mathbb{N}}$ with $G=\cup_n A_n$ making the functions $\gamma\longmapsto \lambda^{s(\gamma)}(\gamma^{-1}(A))$ bounded for every $n\in \mathbb{N}$. The point here is that every element $\gamma:x\longrightarrow y$ in $G$ defines by left traslation a measure space isomorphism $G^x\longrightarrow G^y$ and calling \begin{equation}\label{horml} R(\lambda)_{\gamma}:=\gamma \lambda^x\end{equation} here $\gamma \lambda_x$ is push--forward measure under the $\gamma$--right traslation) one has a kernel in the usual sense i.e. a mapping with value measures. The definition of properness is in fact properness for $R(\lambda)$. 

\noindent The space of proper kernels is denoted by $\mathcal{C}^+$.
\item[Transverse functions] are kernels 
$(\nu^{x})_{x\in X}$ with the left invariance property 
$\gamma \nu^{s(\gamma)}=\nu^{r(\gamma)}$ for every 
$\gamma \in G.$ One checks at once that properness is equivalent to the existence of an increasing family of measurable sets 
$(A_n)_n$ with $G=\cup_n A_n$ such that the functions 
$x\longmapsto \nu^{x}(A_n)$ are bounded for every $n\in \mathbb{N}$. The space of proper transverse functions is denoted $\mathcal{E}^{+}.$ 

\noindent The {\underline{support}} of a transverse function $\nu$ is the measurable set $\operatorname{supp}(\nu)=\{x\in G^{(0)}:\nu^x\neq 0\}$. This is $\underline{saturated}$ w.r.t. the equivalence relation induced by $G$ on $G^{(0)}$, $ x\mathcal{R} y$ iff there exists $\gamma:x\longrightarrow y$. If $\operatorname{supp}(\nu) =G^{(0)}$ we say that $\nu$ is {\underline{faithful}}. 

\noindent When $G=\mathcal{R}$ or the holonomy groupoid this gives families of positive measures one for each leaf in fact in the first case the invariance property is trivial, in the second case we are giving a measure $\nu^x$ on each holonomy cover $G^x$ with base point $x$ but the invariance property says that these are invariant under the deck trasformations together with the change of base points then push forward on the leaf under $r:G^x\longrightarrow L_x.$ 
\item[Convolution.] The groupoid structure provides an operation on kernels. For fixed kernels $\lambda_1$ and, $\lambda_2$ on $G$ their convolution product is the kernel $\lambda_1\ast\lambda_2$ defined by
$$(\lambda_1 \ast\lambda_2)^y=\int(\gamma \lambda_2^x)d\lambda_1^y(\gamma),\quad y\in X.$$ It is a fact that if $\lambda$ is a kernel and $\nu$ is a transverse function then $\nu \ast \lambda$ is a transverse function. Clearly $R(\lambda_1\ast \lambda_2)=R(\lambda_1)\circ R(\lambda)$ the standard composition of kernels on a measure space. Here $R(\cdot)$ is that of equation \eqref{horml}
\item[Transverse invariant measures] \label{tras}(actually transverse measures of modulo $\delta=1$). These are linear mappings $\Lambda:\mathcal{E}^+\longrightarrow [0,+\infty]$ such that
{\begin{enumerate}
\item $\Lambda$ is normal i.e 
$\Lambda(\sup \nu_n)=\sup \Lambda(\nu_n)$ for every increasing sequence 
$\nu_n$ in $\mathcal{E}^+$  bounded by a transverse function. Since the sequence is bounded by an element of $\mathcal{E}^+$ the expression $\sup \nu_n$ makes sense in $\mathcal{E}^+$.
\item $\Lambda$ is invariant under the right traslation of $G$ on $\mathcal{E}^+$. This means that $$\Lambda(\nu)=\Lambda(\nu \ast \lambda)$$ for every $\nu\in \mathcal{E}^+$ and kernel $\lambda$ such that $\lambda^y(1)=1$ for every $y\in G^{(0)}$.
\end{enumerate}}
\noindent A transverse measure is called {\underline{semi--finite}} if it is determined by its finite values i.e $\Lambda(\nu)=\sup\{\Lambda(\nu'),\, \nu'\leq \nu,\, \Lambda(\nu')<\infty\}.$ We shall consider only semi--finite measures.

\noindent A transverse measure is 
$\sigma$--{\underline{finite}} if there exists a faithful transverse function 
$\nu$ of kind $\nu=\sup{\nu_n}$ with $\lambda(\nu_n)<\infty$.

\noindent The coupling of a transverse function $\nu\in \mathcal{E}^+$ and a transverse measure $\Lambda$ produces a positive measure $\Lambda_{\nu}$ on $G^{(0)}$ through the equation $\Lambda_{\nu}(f):=\Lambda((f\circ s)\nu$ the invariance property reflects downstairs (in the base of the groupoid) in the property $\Lambda_{\nu}(\lambda)=\Lambda(\nu \ast \lambda)$ for $\nu \in \mathcal{E}^+$ and $\lambda \in \mathcal{C}^+$.

\noindent Measures on the base $G^{(0)}$ that can be represented as $\Lambda_{\nu}$ are characterized by a theorem of disintegration of measures.

\begin{thm}\label{pocahontas}(Connes \cite{Cos})
Let $\nu$ be a transverse proper function with support $A$. 

\noindent The mapping $\Lambda \longmapsto \Lambda_{\nu}$ is a bijection between the set of transverse measures on $G_A^A=r^{-1}(A) \cup s^{-1}(A)$ and the set of positive measures $\mu$ on $G^{(0)}$ satisfying the following equivalent relations
\begin{enumerate}
\item $(\mu \circ \nu)^{\tilde{}}
=\mu \circ \nu$
\item $\lambda,\lambda'\in \mathcal{C}^+, \nu \ast \lambda=\nu \ast \lambda' \in \mathcal{\epsilon}^+ \Longrightarrow \mu(\lambda(1))=\mu(\lambda'(1)).$
\end{enumerate}
\end{thm}
Nex we shall explain this procedure of disintegration in a geometrical way for foliations.

\noindent We shall see that what is important here is the 
class of null--measure subsets of $G^{(0)}$. A saturated set $A\subset G^{(0)}$ is called $\Lambda$--{\underline{negligible}} if $\Lambda_{\nu}(A)=0$ for every $\nu\in \mathcal{E}^{+}$.

\item[Representations.] Let $H$ be a measurable field of Hilbert spaces; by definition this is a mapping $x\longmapsto H_x$ from $G^{(0)}$ with values Hilbert spaces. The measurability structure is assigned by a linear subspace $\mathcal{M}$ of the free product vector space of the whole family $ \Pi_{x\in G^{(0)}}H_x$ meaning that
\begin{enumerate}
\item For every $\xi \in \mathcal{M}$ the function $x\longmapsto \|\xi(x)\|$ is measurable.
\item A section $\eta \in  \Pi_{x\in G^{(0)}}H_x$ belongs to $\mathcal{M}$ if and only if the function $\langle \eta(x),\xi(x)\rangle$ is measurable for every $\xi \in \mathcal{M}$.
\item There exists a sequence $\{\xi_i\}_{i\in \mathbb{N}}\subset \mathcal{M}$ such that $\{\xi_i(x)\}_{i\in \mathbb{N}}\subset \mathcal{M}$ is dense in $H_x$ for every $x$. 
\end{enumerate} Elements of $\mathcal{M}$ are called measurable sections of $H$.

\noindent Suppose a measure $\mu$ on $G^{(0)}$ has been chosen. One can put together the Hilbert spaces $H_x$ taking their direct integral
$$\int H_x d\mu(x).$$ This is defined as follows, first select the set of square integrable sections in $\mathcal{M}$. This is the set of sections $s$ such that the integral $\int_{G^{(0)}}\|s(x)\|_{H_x}^2d\mu(x)<\ty$ then identitify two square integrable sections if they are equal outside a $\mu$--null set. The direct integral comes equipped with a natural Hilbert space structure with product $$\langle s,t\rangle:=\int_{G^{(0)}} \langle s(x),t(x)\rangle_{H_x}d\mu(x).$$ The notation $s=\int_{G^{(0)}}s(x)d\mu(x)$ for an element of the direct integral is clear. A field of bounded operators $x\longmapsto B_x \in B(H_x)$ is called \underline{measurable} if sends measurable sections to measurable sections. A mesurable family of operators with operator norms uniformely bounded $\operatorname{ess} \sup \|B_x\|<\ty$ defines a bounded operator called \underline{decomposable} $B:=\int_{G^{(0)}}B_xd\mu(x)$ on the direct integral in the simplest way $$Bs:=\int_{G^{(0)}}B_xd\mu(x)\, s=\int_{G^{(0)}}B_xs(x)d\mu(x).$$ For example each element of the abelian Von Neumann algebra $L^{\ty}_{\mu}(G^{(0)})$ defines a decomposable operator acting by pointwise multiplication.
One gets an involutive algebraic isomorphism of $L^{\ty}_{\mu}(G^{(0)})$ onto its image in $B(\int H_x d\mu(x))$ called the algebra of \underline{
diagonal operators}.
 One can ask when a bounded operator $T\in B(\int H_x d\mu(x))$ is decomposable i.e. when $T=\int T_x d\mu(x)$ for a family of uniformely bounded operators $(T_x)_x$.
The answer is precisely when it belongs to the commutant of the diagonal algebra.

\noindent A {\underline{representation}} of $G$ on $H$ is the datum of an Hilbert space isomorphism $U(\gamma):H_{s(\gamma)}\longrightarrow H_{r(\gamma)}$ for every $\gamma\in G$ with 
\begin{enumerate}
\item $U(\gamma_1^{-1}\gamma_2)=U(\gamma_1)^{-1}U(\gamma_2),\quad \forall \gamma_1,\gamma_2\in G,\quad r(\gamma_1)=r(\gamma_2).$
\item For every couple $\xi,\eta$ of measurable section the function defined on $G$ according to
$\gamma \longmapsto \langle \eta_{r(\gamma)},U(\gamma)\eta_{s(\gamma)}\rangle,$ is measurable.
\end{enumerate}
\noindent A fundamental example is given by the {\underline{left regular representation}} of $G$ defined by a proper transverse function $\nu \in \mathcal{E}^+$ in the following way. The measurable field of Hilbert space is $L^2(G,\nu)$ defined by $x\longmapsto L^2(G^x,\nu^x)$ with the unique measurable structure making measurable the family of sections of the kind $y\longmapsto f_{|G^x}$ obtained from every measurable $f$ on $G$ such that each $\int |f|^2d\nu^x$ is finite. For every $\gamma:x\longrightarrow y$ in $G$ one has the left traslation $L(\gamma):L^2(G^x,\nu^x)\longrightarrow L^2(G^y,\nu^y)$,\quad $(L(\gamma)f)(\gamma')=f(\gamma^{-1}\gamma')$, $\gamma'\in G^y$. 

\item[Intertwining operators] are morphisms between representations. If $(H,U)$, $(H',U')$ are representations of $G$ an intertwining operator is a measurable family of operators $(T_x)_{x\in G^{(0)}}$ of bounded operators $T_x:H_x\longrightarrow H'_x$ such that
\begin{enumerate}
\item Uniform boundedness; $\operatorname{ess}\sup \|T_x\|<\infty$.
\item For every $\gamma\in G$  $U'(\gamma)T_{s(\gamma)}=T_{r(\gamma)}U(\gamma).$
\end{enumerate}
\noindent Looking at a representation as a measurable functor, an intertwining operator gives a natural transformation between representations. The vector space of intertwing operators from $H$ to $H'$ is denoted by $\operatorname{Hom}_{G}(H,H').$
\item[Square integrable representations.] Fix some transverse function $\nu \in \mathcal{E}^+$. For a representation of $G$ the property of being equivalent to some sub--representation of the infinite sum of the regular left representation $L^{\nu}$ is independent of $\nu$ and is the definition of {\underline{square integrability}} for representations. Actually, due to measurability issues much care is needed here to define sub representations (see section 4 in \cite{Cos}) but the next fundamental remark assures that square integrable representations are very commons in applications.  

\item[Measurable functors and  representations.] Let $\tilde{\mathcal{R}}_+$ be the cathegory of (standard) measure spaces without atoms i.e. objects are triples $(\mathcal{Z},\mathcal{A},\alpha)$ where $(\mathcal{Z},\mathcal{A})$ is a standard measure space and $\alpha$ is a $\sigma$--finite positive measure.  

\noindent Measurability of a functor $F:G\longrightarrow \tilde{\mathcal{R}}_+$ is a measure structure on the disjoint union $Y=\bigcup_{x\in G^{(0)}}F(x)$ making the following structural mappings measurable
\begin{enumerate}
\item The projection $\pi:Y\longrightarrow G^{(0)}.$
\item The natural bijection $\pi^{-1}(x)\longrightarrow F(x).$
\item The map $x\longmapsto \alpha^x$, a $\sigma$--finite measure on $F(x)$.
\item The map sending $(\gamma,z)\in G\times X$ with $s(\gamma)=\pi(z)$ into $F(\gamma)z\in Y$.
\end{enumerate} 
\noindent Usually one assumes that $Y$ is union of a denumerable collection $(Y_n)_n$ making every function $\alpha^x(Y_n)$ bounded. With a measurable functor $F$ one has an associated representation of $G$ denoted by $L^2\bullet F$ defined in the following way: the field of Hilbert space is $x\longmapsto L^2(F(x),\alpha^x)$ and if $\gamma:x\longrightarrow y$ then define $U(\gamma):L^2(F(x),\alpha^x)\longrightarrow L^2(F(y),\alpha^y)$ by $f\longmapsto F(\gamma^{-1})\circ f.$ Proposition 20 in \cite{Cos} shows that this is a square--integrable representation.
\item[Random hilbert spaces and Von Neumann algebras.] We have seen that every fixed transverse measure $\Lambda$ defines a notion of $\Lambda$--null measure sets (for saturated sets) hence an equivalence relation on $\operatorname{End}_{G}(H_1,H_2)$
the vector space of all intertwining operators $T:H_1\longrightarrow H_2$ between two square integrable representations $H_i$. Each equivalence class is called a {\underline{random operator}} and the set of random operators is denoted by $\operatorname{End}_{\Lambda}(H_1,H_2)$. 
\noindent Also square integrable representations can be identified modulo $\Lambda$--null sets. An equivalence class of square integrable representations is by definition a {\underline{random hilbert space}}.  
\noindent 

\noindent Theorem 2 in \cite{Cos} says that $\operatorname{End}_{\Lambda}(H)$ is a Von Neumann algebra for every random Hilbert space. 

\noindent More precisely  choose some $\nu \in \mathcal{E}^+$ and put $\mu=\Lambda_{\nu}$ and $m:=\mu \bullet \nu$ to form the Hilbert space $\mathcal{H}=L^2(G,m)$. For a function $f$ on $G$ denote $Jf=f^{\sharp}(\gamma)=\bar{f(\gamma^{-1})}$, consider the space $\mathcal{A}$ of measurable functions $f$ on $G$ such that $f,f^{\sharp}\in L^2(G,m)$ and $\sup(\nu|f^{\sharp}|)<\infty$. Equip $\mathcal{A}$ with the product $f\ast_{\nu}g=f\nu\ast g$. The structure $\mathcal{A}$ has is that of an {\underline{Hilbert algebra}} (a left--Hilbert algebra in the modular case) i.e $\mathcal{A}$ is a $\ast$--algebra with positive definite (separable) pre--Hilbert structure such that
\begin{enumerate}
\item $\langle x,y\rangle=\langle y^*,x^*\rangle, \quad \forall x,y \in \mathcal{A}.$
\item The representation of $\mathcal{A}$ on $\mathcal{A}$ by left multiplication is bounded, involutive and faithful.
\end{enumerate}
\noindent With such structure one can speak about the left regular representation $\lambda$ of 
$\mathcal{A}$ on the Hilbert space completion $\mathcal{H}$ of $\mathcal{A}$ itself. The double commutant 
$\lambda''(\mathcal{A})$ of this representation is the Von Neumann algebra $W(\mathcal{A})$ associated to the Hilbert algebra 
$\mathcal{A}$. It is a remarkable fact that 
$W(\mathcal{A})$ comes equipped with a semifinite faithful normal trace $\tau$ such that 
$$\tau(\lambda(y^*)\tau(x))=\langle x,y\rangle \quad \forall x,y \in \mathcal{A}.$$
\noindent Furthermore one knows that the commutant of $\lambda(\mathcal{A})$ in ${\mathcal{H}}$ is generated by the algebra of right multiplications $\lambda'(\mathcal{A})=J\lambda(\mathcal{A})J$ for the conjugate--linear isometry $J:\mathcal{H}\longrightarrow \mathcal{H}$ defined by the involution in $\mathcal{A}$.
\noindent For every $\Lambda$--random Hilbert space $H$ one can use the measure $\Lambda_{\nu}$ on $G^{(0)}$ to form the {\underline{direct integral}} $\nu(H)=\int H_xd\Lambda_{\nu}(x)$. Remember that the direct integral is the set of equivalence classes modulo $\Lambda_{\nu}$ zero measure of square integrable measurable sections.
 Now, directly from the definition, an intertwining operator $T\in \operatorname{Hom}_{\Lambda}(H_1,H_2)$ is a decomposable operator defining a bounded operator $\nu(T):\nu(H_1)\longrightarrow \nu(H_2).$

Put $W(\nu)$ for the Von Neumann algebra associated to the Hilbert algebra $L^2(G,m)$, $m=\Lambda_{\nu}\bullet \nu$, $\nu\in \mathcal{E}^+$. 
\begin{thm}\label{equivalenza}(Connes)
Fix some transverse function $\nu \in \mathcal{E}^+$
\begin{enumerate}
\item For every $\Lambda$--random Hilbert space $H$ there exists a unique normal representation of $W(\nu)$ in $\nu(H)$ such that $U_{\nu}(f)=U(f\nu)$ $f\in \mathcal{A}_{\nu}.$ Here $U(f\nu)$ is defined by $(U(f\nu)\xi)_y=\int U(\gamma)\xi_xd(f\nu^y)(\gamma).$
\item The correspondence $H\longmapsto \nu(H)$, $T\longmapsto \nu(T)$ is a functor from the $(W^*)$--cathegory $\mathcal{C}_{\Lambda}$ of random Hilbert spaces and intertwining operators to the category of $W(\nu)$ modules.
\item If the transverse measure $\nu$ is faithful the functor above is an equivalence of cathegories.
\end{enumerate}
\end{thm}
\noindent Then in the case of faithful transverse measures one gets an isometry of $\operatorname{End}_{\Lambda}(H)$ on the commutant of $W(\nu)$ on the direct integral $\nu(H)$. In particular $\operatorname{End}_{\Lambda}(H)$ is a Von Neumann algebra.
\item[Transverse integrals.] The most important notion of non commutative integration theory is the integral of a {\underline{random variable}} against a transverse measure. A positive random variable on $(G,\mathcal{B},\Lambda)$ is nothing but a measurable functor $F$ as defined above. Let $X:=\bigcup_{x\in G^{(0)}}F(x)$ disjoint union measure space and $\bar{\mathcal{F}}^+$ the space of measurable functions with values in $[0,+\infty]$ while ${\mathcal{F}}^+$ 
is for functions with values on 
$(0,+\ty]$. Kernels $\lambda$ on $G$ acts as convolution kernels on  
$\bar{\mathcal{F}}^+$ according to $(\lambda \ast f)(z)=\int f(\gamma^{-1}z)d\lambda^{y}(\gamma),$ 
$y=\pi(z)\in G^{(0)}$. This is an associative operation $(\lambda_1\ast \lambda_2)\ast f=\lambda_1 \ast (\lambda_2\ast f)$.

\noindent Now to define the integral $\int Fd\lambda$ choose some $\nu$ faithful and put
$$\int F d\lambda=\sup\{\Lambda_{\nu}(\alpha(f)),\, f\in {\mathcal{F}}^+,\,\nu \ast f\leq 1\},$$ 
this is independent from $\nu$ and enjoys the following properties \begin{enumerate}
\item there exist random variables $F_1,F_2$ with $F=F_1\oplus F_2$ such that $\int F_1d\Lambda=0$ and a function $f_2\in \mathcal{F}^+(X_2)$ with $X_2=\bigcup_{x\in G^{(0)}}F_2(x)$ with $\nu \ast f_2=1.$
\item Monotony. If $f,f'\in \mathcal{F}(X)$ satisfy $\nu \ast f\leq \nu \ast f'\leq 1$ then $$\Lambda_{\nu}((\alpha(f))\leq \Lambda_{\nu}((\alpha(f'))$$ in particular for $F_2$ as in 1.
$$\int F_2 d\Lambda=\Lambda_{\nu}((\alpha(f')).$$
\end{enumerate}
\item[Traces.] 
Let $A$ be a Von Neumann algebra with the cone of positive elements $A^+$. 

\noindent A \underline{weight} on a $A$ is a functional $\phi:A^+\longrightarrow [0,\infty]$ such that
\begin{enumerate}
\item $\phi(a+b)=\phi(a)+\phi(b)$, $a,b\in A^+$
\item $\phi(\alpha a)=\alpha \phi(a)$, $\alpha \in \R^+$, $a\in A^+$.
\end{enumerate}
a weight is a 
\underline{trace} if 
$\phi(a^*a)=\phi(aa^*),\,a \in A^+$. 
A weight is called 
\begin{itemize}
\item \underline{faithful} if $\phi(a)=0\Rightarrow a=0$, $a\in A^+$.
\item \underline{normal} if for every increasing net $\{a_i\}_i$ of positive elements with least upper bound $a$ then
$$\phi(a)=\sup \{\phi(a_i)\}.$$
\item \underline{Semifinite} if the linear span of a the set of $\phi$--finite elements, $\{a\in A^+:\phi(a)<\infty\}$ is $\sigma$--weak dense. 
\end{itemize} Every V.N algebra has a semifinite normal faithful weight.

The Von Neumann algebra $\operatorname{End}_{\Lambda}(H)$ associated to a square integrable representation comes equipped with a bijection $T\longmapsto \Phi_T$ between positive operators and semifinite normal weights $\Phi_T:\operatorname{End}_{\Lambda}(H)\longrightarrow [0,+\ty]$ where $\Phi_T$ is faithful if and only if $T_x$ is not singular $\Lambda$--a.e. The construction of this correspondence uses the fact, for a faithful transverse function $\nu$ the
direct integral $\nu(H)=\int H_x d \Lambda_{\nu}(x)$ is a module over the
 Von Neumann algebra $W(\nu)$ associated to the Hilbert algebra $\mathcal{A}$ above described. 
 
 \noindent The notation of Connes is $$\Phi_T(1):=\int \operatorname{Trace}(T_x)d\Lambda(x)$$ i.e. the mapping $T\longmapsto \Phi_T(1)$ is the canonical trace on $\operatorname{End}_{\Lambda}(H)$. In fact this is related to the type $I$ Von Neumann algebra 
 $P$ of classes modulo equality $\Lambda_{\nu}$ almost everywhere of measurable fields $(B_x)_{x\in G^{(0)}},\,B_x\in B(H_x)$ of bounded operators. Remember that $P$ has a canonical trace $\rho(B)=\int \operatorname{Trace}(B_x)d\Lambda_{\nu}(x)$ hence we can define
 $$\rho_T(B):=\int \operatorname{Trace}(T_xB_x)d\Lambda_{\nu}(x).$$
 The next lemma will be important in our applications
 \begin{lem}\label{11010}
For a faithful transverse function $\nu$ there's a unique operator valued weight\footnote{see \cite{take} for the definition} $E_{\nu}$ from $P$ to $\operatorname{End}_{\Lambda}(H)$ such that the diagram
$$
\xymatrix{{P^+}\ar[d]_{E_{\nu}}\ar[dr]^{\,\,\,\,\,\,\,\,\,\,\,\rho_T(\cdot)=\int\operatorname{Trace }(T_x\cdot)d\Lambda_{\nu}(x)} & \\
\operatorname{End}_{\Lambda}(H)\ar[r]_{\Phi_T} & \mathbb{C}}
$$is commutative. Moreover $E_{\nu}$ is such that if $B=(B_x)_{x\in G^{(0)}}$, $B\in P^+$ if an operator making bounded the corresponding family $$C_y:=\int U(\gamma)B_xU(\gamma)^{-1}d\nu^y$$   then $E_{\nu}(B)=C$.
\end{lem}

Let $F$ be a random variable and put $H=L^2\bullet F$. The integration process above defines a semi--finite faithful trace on the Von Neumann algebra $\operatorname{End}_{\Lambda}(H)$. In fact, for $T\in\operatorname{End}^{+}_{\Lambda}(H)$ let ${F_T}$ the new random variable defined by $x\mapsto(F(x),\alpha_T(x))$ where $\alpha_T(x)$ is the measure on $F(x)$ such that $\alpha_T(x)(f)=\operatorname{Trace_{L^2}}(T_x^{1/2}M(f)T_x^{1/2})$ where $f$ is a bounded measurable function on $F(x)$ and $M(f)$ the corresponding multiplication operator on $L^2(F(x))$. The trace is 
$$\Phi_T(1)=\int F_Td \Lambda.$$ 

\noindent In the following we shall use often the notation $\operatorname{tr}_{\Lambda}(T)=\Phi_T(1)$ to emphasize the dependence on $\Lambda$.

\noindent With a trace one can develop a dimesion theory for square integrable representation i.e. a dimension theory for random Hilbert spaces that's very similar to the dimension theory of $\Gamma$--Hilbert modules.

\bigskip

\noindent The \underline{formal dimension} of the random Hilbert space $H$ is $$\operatorname{dim}_{\Lambda}(H)=\int \operatorname{Trace}(1_{H_x})d\Lambda(x)$$ here some fundamental properties
\begin{lem}\label{formaldimension}
\begin{enumerate}
\item  If $\operatorname{Hom}_{\Lambda}(H_1,H_2)$ contains an invertible element then $\operatorname{dim}_{\Lambda}(H_1)=\operatorname{dim}_{\Lambda}(H_2).$
\item $\operatorname{dim}_{\Lambda}(\oplus H_i)=\sum \operatorname{dim}_{\Lambda}(H_i).$
\item $\operatorname{dim}_{\sum \Lambda_i}(\oplus H)=\sum \operatorname{dim}_{\Lambda_i}(H).$
\end{enumerate}
 \end{lem}
 \item[Formal dimensions and projections]
  We need more properties of the formal dimension that are implicit in Connes work but not listed above.
 
\noindent Start to consider sub--square integrable representation.
Consider a Random Hilbert space $(H,U)$; if for every $x$ one choose in a mesurable way a closed subspace $K$ such that $U(\gamma):K_x\longrightarrow K_y$ for every $\gamma \in G$ we say that $(K,V)$, $V(\gamma):=U(\gamma)_{|i_x K_x}$ is a \underline{sub Random Hilbert space.} Once a faithful $\nu\in \mathcal{E}^+$ is keeped fixed, the functor $\nu$ in theorem \ref{equivalenza} displays $H$ and $K$ as submodules of the V.N. algebra associated to the Hilbert Algebra $\mathcal{A}$, hence there must be an injection $\operatorname{End}_{\Lambda}(K)\longrightarrow \operatorname{End}_{\Lambda}(H)$. In fact from the diagram 
$$
\xymatrix{{W(\nu)}\ar[r] \ar[dr]& B(\nu(H)) & \\
{}&B(\nu(K))\ar[u]_{\nu(i)}
}
$$
we see that multiplication by the bounded operator $\nu(i)=\int_{G^{(0)}}\nu(i_x)d\Lambda_{\nu}(x)$ sends the commutator of $W(\nu)$ in $B(\nu(K))$ into its commutator in $B(\nu(H))$. To check that the natural traces $\varphi^H\in P(\operatorname{End}_{\Lambda}(H))$ and $\varphi^K\in P(\operatorname{End}_{\Lambda}(K))$ are preserved by this inclusion we can examine with more detail the meaning of square integrability for a representation. So let us consider the subset of measurable vector fields
$$D(V,\nu):=\Big{\{} \xi:\exists c>0: \forall y \in G^{(0)},\forall \alpha \in K_y
\int_{G^y}|\langle \alpha, V(\gamma)\xi_x\rangle_{K_y} |^2d\nu(\gamma)\leq c^2 \|\alpha\|^2
 \Big{\}}.$$ The definition of square integrability is equivalent to the statement that $D(V,\nu)$  contains a denumerable total subset. 
 In other words the operation of assigning a coefficient $\alpha\longmapsto T_{\nu}(\xi)\alpha=(\alpha,\xi)$ defines an intertwining operator from $V$ to the left regular representation of $L^{\nu}$ of $G$, on the field of Hilbert spaces $L^2(R^x,\nu^x)_x$. This has the property $T_{\nu}(\xi)^*f=V(f\nu)\xi,\quad \xi\in D(V,\nu)$ if $\nu|f|$ is bounded. Then, for $\xi,\eta \in D(V,\nu)$ the operator $$\theta_{\nu}(\xi,\eta):=T_{\nu}(\xi)^*T_{\nu}(\eta) \in \operatorname{End}_{\Lambda}(K)$$ and the following interesting formula holds $(\theta_{\nu}(\xi,\eta)\xi',\eta')=(\xi',\eta)\ast_{\nu}(\eta',\xi)^{v}$ for bounded measurable sections of $K$. Furthermore the vector space $\mathcal{J}_{\nu}$ generated by couples $\xi,\eta \in D$ is a \underline{bilateral ideal} and respects ordening for transverse functions, $\mathcal{J}_{\nu}\subset \mathcal{J}_{\nu'}$ if $\nu\leq \nu'$. Since the measure $\nu$ is faithful this is also \underline{weakly dense} hence completely determines the trace by the simple formula
 \begin{equation}\label{novembre}
 \varphi^K(\theta_{\nu}(\xi,\xi))=\int_{G^{(0)}}\langle \xi_x,\xi_x\rangle d\Lambda_{\nu}(x),\quad \xi \in D(V,K).
 \end{equation} 
Now via $i$ 
we get an inclusion $D(V,\nu)\subset D(U,\nu)$ let's check this statement: let $\xi\in D(V,\nu)$, $y\in G^{(0)}$, $\alpha \in H_y$ then
\begin{align*}
\int_{G^y}|\langle \alpha,U(\gamma)i_x\xi_x\rangle_{K_x}|^2d\nu(\gamma)=\int_{G^y}|\langle \alpha,i_x V(\gamma)\xi_x\rangle |_{K_x}^2d\nu(\gamma)\\
=
\int_{G^y}|\langle P_{H_y}\alpha,V(\gamma)\xi_x\rangle|^2d\nu(\gamma)\leq c^2 \|\alpha\|^2.
\end{align*} 
It turns out that under the inclusion $\operatorname{End}_{\Lambda}(K)\longrightarrow \operatorname{End}_{\Lambda}(H)$ it is essential to check how a $\theta_{\nu}(\xi,\xi)$ acts on $H$ and to check that the two natural traces are equal. These two problems are very simple now since for $\xi \in D(V,K)$ the endomorphism $\theta_{\nu}(\xi,\xi)$ under the inclusion is sent in $\operatorname{End}_{\Lambda}(H)$ to the operator
 \begin{equation}\label{proppp}
 \theta_{\nu}(i\xi,i\xi)=T_{\nu}(i\xi)^*T(i\xi)=i^*T_{\nu}(\xi)T_{\nu}(\xi)i.
 \end{equation}
  We can prove the following
 \begin{lem}\label{comppppa}
\begin{enumerate}
\item The natural traces are compatible w.r.t. the inclusions, in other words we have a commutative diagram
$$
\xymatrix{{\operatorname{End}_{\Lambda}(K)}\ar[d]^{\varphi^K}\ar[r]&\operatorname{End}_{\Lambda}(H)\ar[dl]^{\varphi^H}\\
\R
}
$$
\item To get the formal dimension of $K$ as a Random Hilbert space is sufficient to trace the corresponding field of projections in $\operatorname{End}_{\Lambda}(H)$
\end{enumerate}
 \end{lem}
 \begin{proof}By the computation \eqref{proppp} above, the density result on the ideal $\mathcal{J}_{\nu}$ and formula \eqref{novembre}
 it is suffcient to check the next identity
 \begin{align*}\varphi^H(\theta_{\nu}(i\xi,i\xi))&=
 \varphi^H(T_{\nu}(i\xi)^*T_{\nu}(i\xi))=
 \int_{G^{(0)}}\langle i_x\xi_x,i_x\xi_x\rangle_{H}d\Lambda_{\nu}(x)\\&=\int_{G^{(0)}}\langle \xi_x,\xi_x\rangle_K d\Lambda_{\nu}(x)=\varphi^K(\theta_{\nu}(\xi,\xi)).
 \end{align*}
 \end{proof}
 \noindent Now we have tools to prove two crucial properties of the formal dimension similar to the properties of the dimension of $\Gamma$-- Hilbert modules (compare Chapter 1. of \cite{luk} )
 \begin{prop}\label{decrescenza}
 Let $\{(H^{(i)},U^{(i)})\}_{i\in I}$ a system of Random Hilbert subspaces of $(H,U)$ directed by $\subset$ then
$$\operatorname{dim}_{\Lambda}
(
\operatorname{closure}\Big{(}\bigcup_{i\in I}H^{(i)}\Big{)}
\Big{)}=
\sup\{\operatorname{dim}_{\Lambda}H_i,i\in I\}$$
\noindent If the system is directed by $\supset$ then
$$\operatorname{dim}_{\Lambda}
\Big{(}\bigcap_{i\in I}H^{(i)}
\Big{)}=
\inf\{\operatorname{dim}_{\Lambda}H_i,i\in I\}$$
\end{prop}
\begin{proof}
The choice of a faithful normal transverse function $\nu \in \mathcal{E}^+$ estabilishes the equivalence of categories described above between $C_{\Lambda}$ and the cat. of normal representations of the Von Neumann algebra associated with $W(\nu)$; the first statement then follows from the compatibility of the natural traces proved in \ref{comppppa} and the \underline{normality} (the passage to $\sup$) of the trace in the limit square integrable representation. 
The second statement follows from the first adopting a standard trick changing a decreasing system into an increasing one. It is in fact sufficient to consider $H^{(i)\bot}$ and observe $${\Big{(}\bigcup_{i\in I}H^{(i)\bot}\Big{)}^{\bot}}=\bigcap_{i\in I}H^{(i)}.$$ From the fact that the family is bounded by $H$ we can write the following equation with finite $\Lambda$--dimensions
$$\operatorname{dim}_{\Lambda}\Big{(}     H^{(i)\bot}       \Big{)}=\operatorname{dim}_{\Lambda}(H)-\operatorname{dim}_{\Lambda}(H^{(i)})$$
\end{proof}
 \end{description}
\subsection{Holonomy invariant transverse measures}\label{111111}
\noindent The main example of a non commutative measure space is the space of leaves of a foliation. It is in general impossible to look at the space of leaves as a quotient measure space. A famous example is the Cronecker foliation on the thorus $\mathbb{T}^2$ given by irrational flows (\cite{Co}).
This foliation is also a foliated flat bundle, so we can describe it, using the notations of example \ref{foliatedflat}
\begin{exemple}\label{kronecker}{\bf{The Kronecker foliation.}}
Let the group $\Gamma=\Z$ and the covering $\widetilde{Y}=\R\longrightarrow \mathbb{S}^1=Y$. Also let $T=\mathbb{S}^1$. The group $\Z$ acts on $\R\times \mathbb{S}^1$ according to $n\cdot(r,e^{i\theta}):=(r+n,e^{i(\theta+n\alpha)})$ where a real number $\alpha$ has been fixed. Then if $\alpha/2\pi$ is irrational each leaf of
 the corresponding foliation is a copy of the real line and wires densely around $\mathbb{T}^2.$
 
This foliation is ergodic i.e. a function almost everywhere  constant along the leaves must be constant on the ambient. In particular every Lebesgue space of classical analysis is one dimensional isomorphic to $\mathbb{C}$.
 \end{exemple}
A central concept is that of {\underline{holonomy invariant transverse measure}} introduced by Plante \cite{Pla} and  Ruelle and Sullivan \cite{RuSu}. According to Connes \cite{Cos} a transverse measure provides a measure on the space of leaves. Actually there exists a more general modular theory. Holonomy invariant measures correspond to the simplest case. 
\subsubsection{Measures and currents}
\bigskip Let $X$ be a manifold equipped with a foliation of dimension $p$ and codimension $q$. We suppose always that the foliation is ${\underline{oriented}}$ i.e. the bundle of degree $p$ leafwise forms $\Lambda^{p}_{\mathbb{C}}T\mathcal{F}$ is trivial. This is not truly a restrictive assumption, in fact in the non--orientable case one can make use of densities instead of forms to define currents. Currents are directly related to holonomy invariant transverse measures by the Ruelle--Sullivan isomorphism. The goal of this section is to introduce all these notions and prove the relations between them.

\noindent There is a weak version of the concept of a transversal 
\begin{dfn}
A Borel subset $T\subset X$ is called a \underline{Borel transversal} if the intersection of $T$ with each leaf is (finite) denumerable.
\end{dfn}
\noindent The set of all Borel transversals $\mathcal{T}$ is a \underline{$\sigma$--ring} i.e it is closed under the operation of relative complementation and denumerable union. Recall that a $\sigma$--ring is a \underline{$\sigma$--algebra} if contains the entire space. This is in general not the case for the set of all Borel transversals hence holonomy measures will be defined only on $\sigma$--rings.
\begin{dfn}\label{holm}
A holonomy invariant transverse measure is a $\sigma$--additive map $\mu:\mathcal{T}\longrightarrow [0,+\ty]$ such that
\begin{enumerate}
\item For a Borel bijection $\psi:B_1\longrightarrow B_2$ with $\psi(x)\sim x$ (the relation of being on the same leaf) then $\mu(B_1)=\mu(B_2).$
\item $\mu$ is Radon i.e. for every compact $K\subset B$ then $\mu(K)<\infty.$
\end{enumerate}
 \end{dfn}
 \begin{dfn}
 A holonomy invariant transverse distribution is the datum for every transverse submanifold $T$ of a linear and continuous\footnote{w.r.t. the usual topology of the direct limit i.e. a distribution in the usual sense} map $\delta_T:C^{\infty}_c(T)\longrightarrow \mathbb{C}$ such that if $\psi:T_1\longrightarrow T_2$ is the holonomy of a path $\gamma$ on $X$,
 $$\langle \delta_{T_1},f\rangle= \langle \delta_{T_2},f\circ \psi\rangle.$$
 \end{dfn}
\noindent Now let $\operatorname{Hom}_{\operatorname{cont}}(C^{\infty}_c(\wedge^d T^*_{\mathbb{C}}X),\mathbb{C})$ the space of $d$--dimensional currents on $X$. This is the dual space of the t.v.s. given by the compactly supported $d$--forms equipped with the topology of the direct limit of Frechet spaces. The operations of Lie derivative $L_V$ and contraction $i_V$ w.r.t. a vector field $V$ and the de Rham exterior derivative $d$ extends to distribution just by duality \cite{Co}.

\noindent Note that a $d$--differential form $\omega$ can be restricted to a subbundle $S$ of the tangent bundle just by evaluation of $\omega$ to the $d$--vectors belonging to $\wedge^d S^*_{\mathbb{C}}\subset \wedge^p T^*X_{\mathbb{C}}.$
\begin{dfn}\label{correntefoliata}
A $d$--dimensional current ($d$ is the dimension of the leaves) $C$ is said a {\underline{foliated current}} if it is invariant under the operation of restriction i.e
$\langle C,\omega\rangle =0$ for every $p$--form $\omega$ such that $\omega_{|T\mathcal{F}}=0.$
\end{dfn}\noindent Notice that for a $d$--dimensional foliated current $C$ the condition of being closed is equivalent to require $\partial_X C=0$ for every section $X\in C^{\infty}(X;T\mathcal{F})$.
\begin{prop}
For a manifold $X$ equipped with a $d$--dimensional foliation is equivalent to give
\begin{enumerate}
\item A holonomy invariant transverse distribution.
\item A closed foliated $d$--current.
\end{enumerate}
\end{prop}
\begin{proof}
We define first holonomy invariant transverse distributions relative to regular atlas and show that they define closed foliated $d$--currents. Since the definition of current does not depend on the atlas and every h.i.t. distribution restricts to a h.i.t. distribution relative to each regular atlas the proof will be complete. 
\noindent For a foliated chart $\Omega \longrightarrow V\subset \R^{n-d}\times \R^d$ the local transversal associated is the quotient space defined by the relation $x\sim y$ if $x,y$ belongs to the same plaque of $\Omega$. In particular a local transversal is the space of plaques in $\Omega$. We say that the inclusion $\Omega \hookrightarrow \Omega'$ of distinct open sets is regular and write $\Omega \vartriangleleft \Omega'$ if the inclusion mapping $i:\Omega \hookrightarrow \Omega'$ passes to the quotient to define a smooth mapping on the transversals. In particular each plaque of $\Omega$ meets only a plaque of $\Omega'$. 

\noindent We say that a foliated atlas $\{(\Omega_i,\phi_i)\}_i$ of 
$(X,\mathcal{F})$ of foliated charts $\Omega_i$ is a \underline{good cover} if 
\begin{enumerate}
\item 

$\{\Omega_i\}_i$ is locally finite 
\item  for every 
$i,j$ such that 
$\overline{\Omega}_i\cap \overline{\Omega}_j\neq \emptyset$ there exist a distinct open set 
$\Omega$ such that 
$\Omega_i \vartriangleleft \Omega$ and 
$\Omega_j \vartriangleleft \Omega.$
\end{enumerate}
Standard methods show that a regular atlas always exists. 

\noindent Now  define a \underline{transverse distribution relative to a regular cover} to be a distribution on every local transversal $T_{\Omega}$ of each finite intersection $\Omega=\Omega_1\cap...\cap \Omega_k$ with the property of (relative) holonomy invariance i.e the distribution associated to $T_{\Omega \cap \Omega'}$ is equal to the restriction of the distribution associated to $T_{\Omega}$ and the distribution associated to $T_{\Omega'}$.  

\noindent So let $C$ be a closed foliated current and $\{\Omega_i\}_i$ a regular atlas for $\mathcal{F}$. For every $i$ choose a differential $d$--form $\omega_i$ compactly supported in some neighborhood of 
$\Omega_i\simeq L_i\times T_i$ such that $\int_{L(t)}\omega_i=1$ for every $t\in T_i$. A transverse distribution $\delta_i$ on the local transversal $T_i$ is now defined by
$$\langle\delta_i,f\rangle:=\langle C,f \omega_i\rangle\quad f\in C^{\ty }_c(T_i).$$ This definition is independent of the choice of the forms $\omega_i$; in fact if $\int_{L(t)}\omega_i=\int_{L(t)}\omega'_i=1$ there must be some family $d+1$--forms $t\longmapsto \sigma(t)$ such that $d_{L(t)}\sigma(t)=\omega(t)-\omega'(t)$. This family can be extended to a form $\sigma$ on $\Omega_i$ using the trivial connection. But $C$ is foliated and closed then,
$$\langle C, \omega_i-\omega'_i\rangle=\langle C, df\sigma \rangle=0.$$ The independence from the choice of $\omega_i$ also proves the relative holonomy invariance in fact, for two distinct sets $\Omega_i \cup \Omega_j$
one can choose $\omega_{ij}$ such that $\int_{L_i(t)}\omega_{ij}=\int_{L_j(t)}\omega_{ij}=1$ for $t\in T_i\cap T_j.$

\noindent Viceversa let $\delta$ a holonomy invariant transverse distribution relative to a good cover. Define first a closed foliated $d$--current $C_{\Omega}$ on $\Omega$ for every $\Omega_i\simeq L_i\times T_i$ of the cover then patch together using a smooth partition of the unity. 

\noindent If $\omega$ is a  compactly supported $d$--form on $\Omega$ define $$\langle C_{\Omega},\omega\rangle :=\Big{\langle} \delta , \int_{L}\omega_{|F}\Big{\rangle},$$ in other words we let $\delta$ act on the function on $T$ defined by $t\longmapsto \int_{L(t)}\omega_{F}(l,t).$ This collection of local currents is coherent with intersections by means of the holonomy invariance in fact $C_{\Omega}=C_{\Omega'}$ on $\Omega \cap \Omega'$. Furthermore every $C_{\Omega}$ is closed since $$\langle C_{\Omega},d\omega \rangle=\langle \delta_{T}, \int_{L}d\omega_{|F}\rangle=\langle \delta_{T},0\rangle$$ The property of being foliated is immediate since by costruction they depend only on the values of the forms on the foliation.
\end{proof}
\begin{oss}
Actually there is also another interesting geometric definition of a holonomy invariant measure as a (Radon)  measure on $X$ that is invariant in the direction of the leaves i.e. a measures on the ambient manifold that is invariant under flows generated by vector fields tangent to the foliation. Also a notion of distribution invariant in the direction of the leaves can be defined (see \cite{Cos}).
\end{oss}
\noindent To complete the picture one has to speak about positivity. Recall that our foliation is oriented.
\begin{dfn}A closed $d$--current $C$ is \underline{positive} in the direction of the leaves if $\langle C,\omega \rangle \geq 0$ for every $d$--form that restricts to a positive form on the leaves.
\end{dfn}
\begin{thm}\label{ruellesullivan}
Is equivalent to give on an manifold $X$ with an oriented foliation 
\begin{enumerate}
\item A holonomy invariant transverse measure i.e. a (Radon) measure on the $\sigma$--ring of all transversals invariant under the action of the holonomy pseudogroup $\Gamma$.
\item A measure on $X$ invariant in the direction of the leaves.
\item A closed foliated current positive in the direction of the leaves.
\end{enumerate}
\end{thm}
\begin{proof}Apart for the case of invariant measures on $X$ that are positive in the direction of the leaves for whose we make reference to \cite{Cos} the only observation to do here is that a foliated current that is positive in the direction of the leaves defines a positive transverse distribution.
\end{proof}

\subsubsection{Tangential cohomology}
\noindent Let $\Lambda^{k}T^*\mathcal{F}$ the bundle of exterior $k$--powers of the cotangent bundle of the foliation. In the terminology of Moore and Schochet this is a \underline{tangential vector bundle} i.e. it has a canonical foliation compatible with the vector bundle structure. In a local trivialization over a foliated chart
$$\xymatrix{ & \wedge^{k}T^*\mathcal{F}\ar[d]_{\pi} \ar[r]& U\times \R^{p\choose  k}\ar[dl] \\ L\times T  \ar[r]^{{\simeq}} & U}$$ this foliation is given by the product foliation $\Big{(}L\times \R^{p\choose q}\Big{)}\times T$, in particular the bundle projection maps leaves into leaves.
\begin{dfn}
A continuous section of $\wedge^{k}T^*\mathcal{F}$ is called a tangential $k$ differential form if in every trivialization as above it restricts to be a smooth section on every plaque $L\times \{t\}$. The space of tangential $k$--differential forms is denoted with $\Omega^k_{\tau}(X)$ and $\Omega^k_{\tau,c}(X)$ is the subspace of the compactly supported ones.
\end{dfn}
\noindent In a foliated chart with leafwise cordinates $x_1,...,x_p$ and transversal coordinate $t$, a tangentially smooth differential form can be written \begin{equation}\label{loccal}\omega=\sum_{i_l}a_{i_l}(x_1,...,x_p,t)dx_{i_1}\wedge \cdot \cdot \cdot \wedge dx_{i_k}\end{equation} with $a_{i_l}$ and all of its derivatives w.r.t. $x_1,...,x_p$ continuous in all its variables. One can hence form the \underline{tangential} \underline{De} \underline{Rham} \underline{operator} $d_{\tau}:\Omega^k_{\tau c}(X)\longrightarrow \Omega^k_{\tau c}(X)$ just applying the standard de Rham operator plaque by plaque. We have defined the complex $(\Omega^*_{\tau c}(X),d_{\tau})$ of tangential forms with compact support ($d_{\tau}$ is an example of \underline{leafwise differential operator}, it decrease supports).
\begin{dfn}
The homology of the complex $(\Omega^*_{\tau c}(X),d_{\tau})$ is called the \underline{tangential} \underline{cohomology} \underline{with} \underline{compact} \underline{support} and denoted by $H^*_{\tau c}(X)$.
\end{dfn}
\noindent We can naturally define also tangential cohomology starting with forms without the condition of compactness of the support.
In general the tangential cohomology has infinite dimension this is due to the fact that the continuous transverse control is much more relaxing than smoothness in every direction. In fact there is an interesting question on how the dimension of these spaces changes passing from tangential continuity (also measurability) to smoothness. In Chapter III of \cite{MoSc} there are examples of these phenomena. In the case the foliation is given by the fibers of a trivially local fiber bundle $F\hookrightarrow M\longrightarrow X$ the tangential cohomology turns out to be naturally isomorphic to the space of continuous sections of the bundle $H\longrightarrow X$ where the fiber $H_x=H_{\operatorname{dR}}^{*}(M_x)$ is the de Rham cohomology of the fiber above $x$.

\noindent Let's topologize each space $\Omega^{\bullet}_{\tau c}(X)$ by requiring uniform convergence of every coefficient function $a_{i_l}$ in \eqref{loccal}
with its tangential derivatives in every compact subset of each foliated chart.
It often happens that the topological vector space $H^{\bullet}_{\tau c}(X)$ is not Hausdorff; this is the reason why it is convenient to take its maximal Haudorff quotient to define the \underline{closed tangential cohomology}\footnote{sometimes called the tangential reduced cohomology}
$$\overline{H}^k_{\tau}(X):=H^{k}_{\tau}(X)/\overline{\{0\}}={\operatorname{Ker}(d_{\tau}:\Omega^k_{\tau c}\longrightarrow \Omega^{k+1}_{\tau c})}/{\overline{\operatorname{Range}(d_{\tau}:\Omega^{k-1}_{\tau c}\longrightarrow \Omega_{\tau c}^{k}})}.$$
\noindent In general this leads to different spaces, for the irrational flow on the torus $\overline{H}_{\tau}^1(\mathbb{T},\R)\cong \R$ while $H^1_{\tau}(\mathbb{T},\R)$ is infinite dimensional (\cite{MoSc}).  
\begin{dfn}
\noindent Elements of the topological dual of $\Omega^{\bullet}_{\tau c}(X)$ i.e. continuous linear functionals
$C:\Omega^{\bullet}_{\tau c}(X)\longrightarrow \mathbb{C}$ are called \underline{tangential currents}. The space of tangential currents is denoted by $$\Omega^{\tau}_k:=\operatorname{Hom}_{\operatorname{
con.}}(\Omega^k_{\tau c}(X);\mathbb{C}).$$
\end{dfn}
 \noindent Note that  a foliated current of 
definition \ref{correntefoliata} is a current in the ordinary sense that passes to define a tangential current under the restriction morphism $(\cdot)_{|\mathcal{F}}:\Omega^k(X)\longrightarrow \Omega_{\tau}^k(X).$
The differential $d_{\tau}:\Omega^{\bullet}_{\tau }(X)\longrightarrow \Omega^{\bullet +1}_{\tau }(X)$ (we will omit the subscript $\tau$ by simplicity of notation) is continuous and extends by duality to currents, $d_*:\Omega^{\tau}_{\bullet }(X)\longrightarrow \Omega^{\tau}_{\bullet-1 }(X)$ according to the sign convention $\lb \omega,d_{*}\rb=(-1)^{k-1}\lb d_{\tau}\omega,c\rb$. There is an isomorphism
$$\operatorname{Hom}_{\operatorname{con.}}(H^{k}_{\tau c}(X);\R)\cong H^{\tau}_k(X;\R)$$ and theorem \ref{ruellesullivan} is essentially the \underline{Ruelle--Sullivan isomorphism}\footnote{at this level this is only a vector space iso. but one can consider the $\ast$--weak topology on the space of measures to force this to be a topological iso. However we don't need continuity.}
$$\operatorname{MT}(X)\longrightarrow \operatorname{Hom}_{\operatorname{con.}}(H^p_{\tau c},\R)$$ between the vector space of signed   holonomy invariant transverse Radon measures and the topological dual space of the top degree tangential homology. The tangential current defined by a measure $\Lambda$ is called the \underline{Ruelle Sullivan} current $C_{\Lambda}.$
\subsubsection{Transverse measures and non commutative integration theory}
\noindent Up to this point we have used the name \emph{transverse measure} for at least two objects; measures on the union of all transversals and transverse measures in the equivalence relation $\mathcal{R}$ (or the holonomy groupoid,  $G$) according to definition \ref{tras}. In the rest of the section we clarify the relationship between them.
 First we need a couple of additional definitions
 \begin{dfn}A transverse measure $\Lambda$ in the sense of non commutative integration theory for the equivalence relation $\mathcal{R}$ (or the holonomy groupoid $G$) is called \underline{locally finite} if $\Lambda(\nu)<\ty$ for every $\nu \in \mathcal{E}^+$ having the following two properties
 \begin{enumerate}
\item $\nu$ is \underline{locally bounded} i.e. $\sup \nu^x(K)<\infty$ for every $K$ compact in $\mathcal{R}$
\item $\nu$ is \underline{compactly supported} i.e. $\nu^x$ is supported in $s^{-1}(K)$ for a compact $K\subset X$.
\end{enumerate}
 \end{dfn}
\begin{dfn}
The \underline{characteristic function} $\nu_A$ of a subset $A\subset X$ is the transverse function defined by $\nu_A^x(B)=|s^{-1}(A)\cap G^x \cap B|$ or equivalently $\nu(f)(y):=\sum_{\gamma \in G^y,\,s(\gamma)\in A}f(\gamma)$ for a Borel function $f$ on $G$.
\end{dfn} \noindent Note that the characteristic function is nothing but the lift $s^{-1}(\mu_A)$ of the counting measure concentrated in $A$. This actually shows that $\gamma \nu^x_A=\nu^y_A$, $\gamma \in G_x^y$.
\begin{thm}(Connes \cite{Cos}) Let $\Lambda$ be a locally finite transverse measure for $\mathcal{R}$ ($G$). Let $Z$ a transverse submanifold; for a compact set $K\subset Z$ define $\tau(K):=\Lambda(\nu_K)$. This is the definition of a Positive Radon measure on $Z$ that is holonomy invariant. 

\noindent In other words the correspondence $\Lambda \longmapsto \tau$ is a bijection $$\{\mbox{Locally finite transverse measures on } \mathcal{R}\}\longrightarrow \{\mbox{Holonomy invariant transverse measures on }X\}.$$
\end{thm} \noindent Remember that there is a coupling between transverse measures $\Lambda$ on $\mathcal{R}$ and transverse functions $\nu$ to produce a measure on $X$ defined by $\Lambda_{\nu}(f)=\Lambda((s\circ f)\nu)$ then $$\Lambda_{\nu_K}(1)=\Lambda(\nu_k)=\tau(K).$$
\begin{dfn}
Choose some Radon measure $\alpha$ on the ambient manifold $X$; call the \underline{lift} of $\alpha$ the transverse measure $\nu^x:=s^*(\alpha)$ where $s:G^x\longrightarrow X$. We say that a lift is \underline{transversally measurable} if for every foliated chart $\Omega\cong U\times T$ it is represented as a weakly measurable mapping $T\longrightarrow \operatorname{Ra}(U)$ from $T$ to the space of Radon measures on $U$, bounded if $\Omega$ is relatively compact.
\end{dfn}
\begin{prop}(Connes \cite{Cos} )
The map $\alpha\longmapsto s^*(\alpha)$ is a bijection between transversally measurable Radon measures on $X$ and transverse functions $\nu$ suc that $\sup \nu(K)<\ty$ for every compact $K\subset G$.
\end{prop}
\begin{prop}\label{reddd}
Choose some Radon measure $\alpha$ on $X$ with support $X$. Let $\nu=s^*(\alpha)$. The mapping $\Lambda\longmapsto \Lambda_{\nu}$ is a bijection between locally finite transverse measures on $G$ and Radon measures $\mu$ on $X$ with the property:

\noindent for every disintegration of $\mu$ on a foliated chart along the fibers of the distinct mapping $\Omega\cong U\times T\longrightarrow T$ the conditional measures satisfy
$$d\mu_t=d\alpha_t.$$
\end{prop}
 \noindent In practice the above propositions furnishes a geometrical recipe to recognize the measure $\Lambda_{\nu}$ on the base $X$ if $\Lambda$ is a transverse measure on the foliation i.e. a measure on the $\sigma$--ring of all Borel transversals. In fact choose some foliated atlas $\Omega_{i}\simeq U_i\times T_i$ with the set of coordinates $(x,t)$ and a subordinate smooth partition of the unit ${\varphi}_i$. Then for a function $f$
 $$\Lambda_{\nu}(f)=\sum_i \int_{T_i}\int_{U_i}\varphi_i(x,t)f(x,t)d\nu_t(x)d\Lambda_{T_i}(t)$$ where $\nu_t(x)$ is the longitudinal measure $\nu$ restricted to the plaque $U_i\times \{t\}$. We shall refer to this Fubini type decomposition as to the \underline{integration process} according to the terminology of the book by Moore and Schochet \cite{MoSc}.

\subsection{Von Neumann algebras and Breuer Fredholm theory for foliations}

Let $\mathcal{R}$ the equivalence relation of the foliation.
For square integrable representations on the measurable fields of Hilbert spaces $H_i$ let $\operatorname{Hom}_{\mathcal{R}}(H_1,H_2)$ the vector space of all \underline{intertwining} \underline{operators}. The choice of a holonomy invariant measure ${\Lambda}$ on the foliation gives rise to a transverse measure on $\mathcal{R}$ in the sence of non commutative integration theory hence a quotient projection $$\operatorname{Hom}_{\mathcal{R}}(H_1,H_2)\longrightarrow \operatorname{Hom}_{\Lambda}(H_1,H_2)$$ given by identification modulo $\Lambda$-a.e. equality. Elements of $\operatorname{Hom}_{\Lambda}(H_1,H_2)$ are called \underline{Random}  \underline{operators}.
If $H_1=H_2=H$, then $\operatorname{Hom}_{\mathcal{R}}(H,H)=\operatorname{End}_{\mathcal{R}}(H)$ is an involutive algebra, the quotient via $\Lambda$ is a Von Neumann algebra\footnote{to be precise this is a $W^*$ algebra in fact it is not naturally represented on some Hilbert space. The choice of a longitudinal measure $\nu$ gives however a representation 
$\operatorname{End}_{\mathcal{R}}(H)\longrightarrow B(\int_X H_x d\Lambda_{\nu}(x))$ on the direct integral of the field $H_x$
}
$$\operatorname{Hom}_{\mathcal{R}}(H)\longrightarrow \operatorname{End}_{\Lambda}(H).$$

\bigskip
\noindent For a vector bundle $E\longrightarrow X$ 
let $L^2(E)$ 
be the Borel field of Hilbert spaces on $X$, fixed by the leafwise square integrable sections 
$\{L^2(L_x,E_{|L_x})\}_{x\in X}$. There is a natural square integrable representation of 
$\re$ on 
$L^2(E)$ the one given by 
$(x,y)\longmapsto \operatorname{Id}:L^2(L_x,E)\longrightarrow L^2(L_y,E)$. 
\noindent Denote 
$\operatorname{End}_{\mathcal{R}}(E)$ the vectorspace of  all intertwining operators and $\operatorname{Hom}_{\Lambda}(E)$ the corresponding Von Neumann algebra.

\noindent Since we need unbounded operators we have to define measurability for fields of closed unbounded operators. Remember that the polar decomposition $T= u|T|$ 
is determined by the couple of bounded operators $u$ and $(1+T^*T)^{-1}$.
\begin{dfn}\label{messa}We say that a field of unbounded closed operators $T_x$ is measurable if are measurable the fields of bounded operators $u_x$ and $(1+T_x^*T_x)^{-1}$.
\end{dfn}

\begin{oss}\label{misurabunbounded}
\noindent .In the paper \cite{nus} about unbounded reduction theory. An unbounded field of closed operators $A$ is said measurable if the family corresponding to the projection on the graph is measurable on $H\oplus H$ with the direct sum measure structure. Writing the projection on the graph as $$(\xi,\eta)\longmapsto ((1+A^*A)^{-1}(\xi+A^*\eta), A(1+A^*A)^{-1}(\xi+A^*\eta))$$ we can see that these definition is equivalent to the one given here
\end{oss}

\noindent Next, we review some ingredients from Breuer theory of Fredholm operators on Von Neumann algebras, adapted to our weight--theory case with some notions translated in the language of the essential $\Lambda$--spectrum, a straightforward generalization of the essential spectrum of a self--adjoint operator. Main references are \cite{Br1,Br2} and \cite{cp0} and \cite{cp}.

\bigskip
\noindent Remember that the set of projections $\mathcal{P}:=\{A\in \vo, A^*=A, A^2=A\}$ of a Von Neumann algebra, has the structure of a complete lattice i.e. for every family 
 $\{A_i\}_i$ of projections one can form their \underline{join} 
 $\vee A_i$ and their \underline{meet} 
 $\wedge A_i$. 
 Then for a random operator $A\in \vo$ we can define its projection on the range $R(A)\in \mathcal{P}(\vo)$ and the projection on its kernel $N(A)\in \mathcal{P}(\vo)$ according to 
 $R(A):=\vee\{P\in \mathcal{P}(\vo):PA=A\}$ and $N(A):=\wedge\{P\in \mathcal{P}(\vo):PA=P\}$. If $A$ is the class of the measurable field of operators $A_x$, it is clear that $R(A)$ and $N(A)$ are the classes of $R(A)_x$ and $N(A)_x$.
\begin{dfn}Let $H_i$, $i=1,..,3$ be square integrable representations of $\mathcal{R}$ define
\begin{enumerate}
\item $\Lambda$--finite rank random operators. $B^{f}_{\Lambda}(H_1,H_2):=\{A \in \operatorname{Hom}_{\Lambda}(H_1,H_2) :\tru R(A)<\infty\}$
\item $\Lambda$--compact random operators. $B^{\infty}_{\Lambda}(H_1,H_2)$ is the norm closure of finite rank operators.
\item $\Lambda$--Hilbert--Schmidt random operators $$B^2_{\Lambda}(H_1,H_2):=\{A \in \operatorname{Hom}_{\Lambda}(H_1,H_2): \tru(A^*A)<\infty\}.$$
\item $\Lambda$--trace class operators. $B^1_{\Lambda}(H)=B^2_{\Lambda}(H)B^2_{\Lambda}(H)^*= \{\sum_{i=1}^nS_iT_i^*:S_i,T_i\in B^2_{\Lambda}(H)\}$.\end{enumerate}
\end{dfn}
\begin{lem}Let $*=f,1,2,\infty$.
$B^{*}_{\Lambda}(H)$ is a $*$--ideal in $\vo$. An element $A\in B^*_{\Lambda}(H)$ iff $|A|\in B^*_{\Lambda}(H)$.
The following inclusion holds 
$$\bif \subset \buno \subset \bdue \subset \binf.$$
Furthermore $$\buno=\{A\in \vo:\tru |A|<\infty\}.$$
\end{lem}
\begin{proof}The proof is very similar to the standard case.\end{proof}
\noindent An important inequality is the following, take $A\in \buno$ and $C\in \operatorname{End}_{\Lambda}(H)$. We  have polar decompositions $A=U|A|$, $C=V|C|$ then $|A|=U^*A\in \buno$, $|A|^{1/2}\in \bdue$ and \begin{equation}\label{disugess}|\tru (CA)|\leq \|C\|\tru|A|.\end{equation} For the proof being a very standard calculation in Von Neumann algebras, can be found in chapter $V$ of \cite{take}.  

\begin{dfn}
A random operator $F\in \operatorname{Hom}_{\Lambda}(E_1,E_2)$ is $\Lambda$--Fredholm (Breuer--Fredholm) if there exist $G\in \operatorname{Hom}_{\Lambda}(E_2,E_1)$ such that $FG-\operatorname{Id}\in B^{\infty}_{\Lambda}(E_2)$ and $GF-\operatorname{Id}\in B^{\infty}_{\Lambda}(E_1)$.
\end{dfn}

\begin{dfn}
For an unbounded field of closed operators $T_x:H_1\longrightarrow H_2$ between two measurable fields of Hilbert spaces $H_i$ the field of bounded operators $$T_x:(\operatorname{Domain}(T_x),\|\cdot \|_{T_x})\longrightarrow H_2$$ where $\|\cdot \|_{T_x}$ is the graph norm is measurable by Remark \ref{misurabunbounded}. We say that $T$ is $\Lambda$--Breuer--Fredholm when this field of bounded operators is $\Lambda$--Breuer--Fredholm.
\end{dfn}
\begin{prop}
A random operator $F\in \operatorname{Hom}_{\Lambda}(H_1,H_2)$ is $\Lambda$--Fredholm if and only if $N(F)$ is $\Lambda$--finite rank and there exist some finite rank projection $S\in \operatorname{End}_{\Lambda}(H_2)$ such that $R(\operatorname{Id}-S)\subset R(F).$
\end{prop}
\noindent Hence from the proposition above $\Lambda$--Fredholm operators $F$ have a finite $\Lambda$--index. In fact $\tru(N(F))<\infty$ and $$\tru(1-R(F))\leq \tru(S)<\infty,$$ making clear the next definition.
\begin{dfn}
Let $F \in \ho(H_1,H_2)$ be $\Lambda$--Fredholm. The $\Lambda$ index of $F$ is defined by
$$\operatorname{ind}_{\Lambda}(F):=\tru(N(F))-\tru(1-R(F)).$$ 
\end{dfn}
\noindent The next result in The Shubin book by Shubin \cite{Shu}, motivates the definition of an useful instrument called the $\Lambda$--essential spectrum

\begin{lem}\label{lemshu}
Let  $M$ be a Von Neumann algebra endowed with a semi--finite faithful trace $\tau$, $S=S^*\in M$. Then $S$ is $\tau$--Breuer--Fredholm if and only if there exists $\epsilon>0$ such that $\tau(E(-\epsilon,\epsilon))<\infty$, where $E(\Delta)$ is the spectral projection of $S$ corresponding to a Borel set $\Delta$. Besides if $S=S^*$ is $\tau$--Breuer--Fredholm then $\operatorname{ind}_{\tau}S=0$.
\end{lem}
\bigskip
\noindent So consider a measurable field $T$ of unbounded intertwining operators.  
 If $T$ is selfadjoint 
 (every $T_x$ is self--adjoint a.e.) the parametrized (measurable) spectral Theorem (cf. Theorem XIII.85 in \cite{Reed}) shows that for every bounded Borel function $f$ the family 
 $x\longmapsto f(T_x)$ is a measurable field of uniformely bounded  intertwining operators
 defining a unique random operator. In other words
$$\{f(T_{x})\}_x\in \operatorname{End}_{\Lambda}(H).$$ 
\noindent For a Borel set $U\subset \R$ let $\chi_T(U)$ be the family of spectral projections $x\longmapsto \chi_{U}(T_x)$. Denote $H_T(U)$ the measurable field of Hilbert spaces corresponding to the family of the images $(H_T(U))_x=\chi_{U}(T)H_x$.  Let $\tru:\operatorname{End}^{+}_{\Lambda}(H)\longrightarrow [0,+\infty]$ the semifinite normal faithful trace defined by $\Lambda$. The formula $$\mt(U):=\operatorname{tr}_{\Lambda}(\chi_T(U))=\operatorname{dim}_{\Lambda}(H_U(T))$$ defines a Borel measure on $\R$.
\begin{dfn}\label{misuraspettrale}
We call the Borel measure defined above the \underline{$\Lambda$--spectral measure} of $T$.
\end{dfn} 
\begin{oss}
\noindent Clearly this is not in general a Radon measure (i.e. finite on compact sets). In fact
due to the non--compactness of the ambient manifold a spectral projection of a relatively compact set of an (even elliptic) operator is not trace class. 
In the case of elliptic self adjoint operators with spectrum bounded by below this is the Lebesgue--Stiltijes measure associated with the spectrum distribution function relative to the $\Lambda$--trace. This is the (not decreasing) function $\lambda\longmapsto \operatorname{tr}_{\Lambda}\chi_{(-\ty,\lambda)}(T)$.
A good reference on this subject is the work of Kordyukov  \cite{koo}.
\end{oss}
\noindent Notice the formula 
$$\int f d\mt=\operatorname{tr}_{\Lambda}(f(T))$$ for each bounded Borel function $f:\R \longrightarrow [0,\infty).$
The proof of this fact easily follows starting from characteristic functions. Here the normality property of the trace plays a fundamental role. A detailed argument can be found in \cite{Peric}.
\noindent Next we introduce, inspired by \cite{Vai}  the \emph{main character} of this section.
\begin{dfn}
The essential $\Lambda$--spectrum of the measurable field of unbounded self--adjoint operators $T$ is
$$\spc(T):=\{\lambda \in \R: \mu_{\Lambda,T}(\lambda-\epsilon,\lambda+\epsilon)=\infty, \forall \epsilon>0\}.$$
\end{dfn}
\begin{lem}\label{compactness}
For Random operators the $\Lambda$--essential spectrum is stable under compact perturbation. If $A\in \operatorname{End}_{\Lambda}(E)$ is selfadjoint $A=A^*$ and $S=S^*\in B^{\infty}_{\Lambda}(E)$ then 
$$\operatorname{spec}_{\Lambda,e}(A+S)=\operatorname{spec}_{\Lambda,e}(A).$$ Then if $\tru$ is infinite i.e. 
$\tru(1)=\infty$ 
we have
$\operatorname{spec}_{\Lambda,e}(A)=\{0\}$
for every $A=A^*\in B^{\infty}_{\Lambda}(E).$ 
\end{lem}
\begin{proof}
Let $\lambda \in \operatorname{spec}_{\Lambda,e}(A)$, by definition $\operatorname{dim}_{\Lambda}H_A(\lambda-\epsilon,\lambda+\epsilon)=\infty$. Then consider the field of Hilbert spaces 
$$G_{\epsilon,x}:=\big{\{}t\in \chi_{(-\lambda-\epsilon,\lambda+\epsilon)}(A_x)H_x;\,\|S_xt\|< \epsilon \|t\|\big{\}}=H_{S_x}(-\epsilon,\epsilon)\cap H_{A_x}(-\lambda-\epsilon,\lambda+\epsilon).$$ 
This actually shows that $G_{\epsilon}
$ is $\Lambda$--finite dimensional in fact  
$H_{A_x}(-\lambda-\epsilon,\lambda+\epsilon)$ is 
$\Lambda$--infinite dimensional while $H_{S_x}(-\epsilon,\epsilon)$ is $\Lambda$--finite codimensional. This shows that $\lambda\in \operatorname{spec
}_{\Lambda,e}(A+S)$. The second statement is immediate.
\end{proof}
\noindent There is a spectral characterization of $\Lambda$--Fredholm random operators as expected after Lemma \cite{lemshu}.
\begin{prop}
For a random operator $F\in \operatorname{Hom}_{\Lambda}(H_1,H_2)$ the following are equivalent
\begin{enumerate}
\item $F$ is $\Lambda$--Fredholm.
\item $0 \notin \spc(F^*F)$ and $0 \notin \spc(FF^*)$.
\item $0 \notin \spc\left(\begin{array}{cc}0 & F^* \\F & 0\end{array}\right)$
\item $N(F)$ is $\Lambda$--finite rank and there exist some finite rank projection $S\in \operatorname{End}_{\Lambda}(H_2)$ such that $R(\operatorname{Id}-S)\subset R(F).$
\end{enumerate} 
\end{prop}

\subsubsection{The splitting principle}Let $E\longrightarrow X$ be a vector bundle.
For every $x\in X$ and integer $k$ consider the Sobolev space $H^k(L_x,E)$ of sections of $E$, obtained by completion of 
$C^{\infty}_c(L_x,E)$ with respect to the $k$ Sobolev norm 
$$\|s\|^2_{H^k(L_x;E)}:=\sum_{i=0}^k\|\nabla^k s\|_{L^2(\otimes^k T^*L_x;E)}^2,$$ here the longitudinal Levi Civita connection w.r.t. the metric has been used. This is the definition of a Borel field of Hilbert spaces with natural Borel structure given by the inclusion into 
$L^2$. In fact, by Proposition 
$4$ of Dixmier \cite{Dix} p.167 to prescribe a measure structure on a field of hilbert spaces 
$H$ it is enough to give a countable sequence 
$\{s_j\}$ of sections with the property that for 
$x\in X$ the countable set 
$\{s_j(x)\}$ is complete orthonormal. 
In the appendix of Heitsch and Lazarov paper \cite{hl} is shown, making use of holonomy that a family with the property that each $s_j$ is smooth and compactly supported on each leaf can be choosen.

\bigskip
\begin{dfn}
\noindent Consider a field $T=\tx$ (not necessarily Borel by now) of continuous intertwining operators
$T_x:\cic \longrightarrow \ci$. 
\begin{itemize}
\item We say that $T$ is of order $k\in \mathbb{Z}$ if $T_x$ extends to a bounded operator $$H^m(L_x,E_{|Lx})\longrightarrow H^{m-k}(L_x,E_{|Lx})$$ for each $m\in \mathbb Z$ and for $x$ a.e.
\item We say that the $T$ is elliptic if each $T_x$ satisfies a G\"arding type inequality
$$\|s\|_{H^{m+k}_x}\leq C(L_x,m,k)[\|s\|_{H_x^m}+\|T_xs\|_{H_x^m}],$$
and the family $\{C(L_x,m,k)\}_{x\in X}$ is bounded outside a null set in $X$.
\end{itemize}
\end{dfn}
\noindent Since each leaf $L_x$ is a manifold with bounded geometry for a family of elliptic selfadjoint intertwining operators $\tx$ every $T_x$ is essentially selfadjoint with domain $H^k(L_x;E_{|L_x})$. It makes sense again to speak of measurability of such a family.
\begin{dfn}\label{eqout}For two fields of operators $P$ and $P'$ say that $P=P'$ outside a compact $K\subset X$ if for every leaf $L_x$ and every section $s\in C^{\infty}_c(L_x\setminus K;E)$ then $Ps=P's$. This property holding $x$ a.e in $X$ with respect to the standard Lebesgue measure class.
\end{dfn}

\begin{thm}\label{1}{The splitting principle.} Let $P$ and $P'$ two Borel fields of (unbounded) selfadjoint order $1$ elliptic intertwining operators. If $P=P'$  outside a compact set $K\subset X$ then
$$\spc(P)=\spc(P').$$
\end{thm}
\begin{proof}
Let $\lambda \in \spc(P)$, for each $\epsilon>0$ put $\chi^{\lambda}_{\epsilon}:=\chi_{(\lambda-\epsilon,\lambda+\epsilon)}$ and $G_{\epsilon}:=\chi_{\epsilon}^{\lambda}(P)$, then $\tru(G_{\epsilon})=\infty$. The projection $G_{\epsilon}$ amounts to the Borel field of projections 
$\{\chi_{\epsilon}^{\lambda}(P_x)\}_{x\in X}.$ 
By elliptic regularity on each Hilbert space $G_{\epsilon,x}$ every Sobolev norm is equivalent in fact the spectral theorem and G{\aa}rding inequality show that for $s\in G_{\epsilon,x}$ and $k\in \mathbb{N}$
$$\|s\|_{H^{k+2}_x}\leq C(P_1,k+2)\{\|s\|_{L^2_x}+\|(P_1-\lambda)^ks\|_{L^2_x}\}\leq (C+\epsilon^k)\|s\|_{L^2_x}$$ where $C(P_1,k+2)$ is a constant bigger than each leafwise G{\aa}rding constant. 

\noindent Now choose two cut--off functions $\phi, \psi \in \cc$ with 
$\phi_{K}=1$ and $\psi_{|{\operatorname{supp}\phi}}=1.$ Consider the following fields of operators 
\begin{equation}
{\xymatrix{B_{\phi}:L^2_x \ar[r]^{\chi_{\lambda}^{\epsilon}} &  G_{\epsilon,x} \ar[r]^{\phi} &L^2_x,
}}\end{equation}
\begin{equation}
{\xymatrix{C_{\psi}:L^2_x\ar[r]^{\chi_{\lambda}^{\epsilon}}&(G_{\epsilon,x},\|\cdot\|_{L^2}) \ar[r]&(G_{\epsilon,x},\|\cdot\|_{H^k})\ar[r]^{\psi} &  H^1_x
}}\end{equation}for a $k$ sufficiently big in order to have the Sobolev embedding theorem.
\noindent We declare that $C^*_{\psi}C_{\psi}\in \vo$ is 
$\Lambda$--compact. In fact consider by simplicity the case in which $\psi$ is supported in a foliation chart 
$U\times T$. The integration process shows that the trace of 
$C^*_{\psi}C_{\psi}$ is given by integration on 
$T$ of the local trace on each plaque 
$U_t=U \times \{t\}.$ Now the operator 
$C^*_{\psi,x}C_{\psi,x}$ is locally traceable by Theorem 1.10 in Moore and Schochet \cite{MoSc} since by Sobolev embedding the range of $C_{\Psi}$ is made of continuous sections (the fact that each sobolev norm is equivalent on $G_{\epsilon}$ makes the teorem appliable i.e don't care in forming the adjoint w.r.t. $H^1$ norm or $L^2$). These local traces are uniformly bounded in $U \times T$ from the uniformity of the G{\aa}rding constants for the family since we are multiplying by a compactly supported function $\psi$. 
Actually we have shown that $C_{\psi}^*C_{\psi}$ is $\Lambda$--trace class.
There follows from Lemma \ref{compactness}  about $\Lambda$--compact operators that the projection $\widetilde{G}_{\epsilon}:=\chi_{(-\epsilon^2,\epsilon^2)}(C^*_{\psi}C_{\psi})$ is $\Lambda$--infinite dimensional; in fact $$\spc(C^*_{\psi}C_{\psi})=\{0\}.$$ 
\noindent Now $1-B_{\phi}$ is $\Lambda$--Fredholm ($B_{\phi}$ is $\Lambda$--compact ) then its kernel has finite $\Lambda$--dimension. Also since $C^*_{\psi}C_{\psi}\chi_{\epsilon}^{\lambda}=C^*_{\psi}C_{\psi}$ then $\widetilde{{G}_{\epsilon}}\chi_{\epsilon}^{\lambda}=\tilde{G}_{\epsilon}$ hence $(1-B_{\phi})\tilde{G}_{\epsilon}=(1-\phi)\tilde{G}_{\epsilon}\subset \operatorname{domain}(P')$ is $\Lambda$--infinite dimensional.

\noindent Take $s\in \tilde{G}_{\epsilon}$, from the definition
$$\|\psi s\|^2_{H^1}=\langle C_{\psi}s,C_{\psi}s\rangle_{H^1}=\langle C^*_{\psi}C_{\psi}s,s\rangle_{L^2}\leq \epsilon^2 \|s\|^2_{L^2}$$ then
\begin{align*} 
\|(P'-\lambda)(1-\phi)s\|_{L^2}\leq \|[P,\phi]s    \|_{L^2} + \|(1-\phi)(P-\lambda)s   \|_{L^2}\leq C\|\psi s\|_{H^1}+\\
\|(P-\lambda)s\|_{L^2}
\leq 
 \epsilon(1+C)\|s\|_{L^2}.\end{align*} 
 The second chain of inequalities follows from \begin{align*}(P'-\lambda)(1-\phi)s=(P-\lambda)(1-\phi)s=([P-\lambda,1-\phi]-(1-\phi)(P-\lambda))s\\=-([P,\phi]+(1-\phi)(P-\lambda))s
 .\end{align*}
 
 \noindent Finally the spectral theorem for (unbounded) self adjoint operators shows that $$(1-\phi)\tilde{G}_{\epsilon}\subset \chi_{(\sigma,\tau)}(P')$$ with  
$\sigma=\lambda-\epsilon(1+C),\tau=\lambda+\epsilon(1+C)$. In particular $\lambda \in \spc(P')$.
\end{proof}
\begin{cor}
Consider two foliated manifolds $X$ and $Y$ (with cylindrical ends or bounded geometry) with holonomy invariant measures $\Lambda_1$, $\Lambda_{2}$ and   bounded geometry vector bundles $E_1\longrightarrow X$ and $E_2\longrightarrow Y$. Suppose there exist compact sets $K_1\subset X$ and $K_2\subset Y$ such that outside $X\setminus K_1$ and $Y\setminus K_2$ are isometric with an isometry that identifyies every geometric structure as the bundles and the foliation with the transverse measure. If $P$ and $P'$ are operators as in Theorem \ref{1} with $P=P'$ on $X\setminus K_1 \simeq Y\setminus K_2$ in the sense of definition \ref{eqout} then
$$\operatorname{spec}_{\Lambda_1,e}(P)=\operatorname{spec}_{\Lambda_2,e}(P').$$
\end{cor}
\begin{proof}
The proof of \ref{1} can be repeated word by word till the introduction of the element $(1-\phi)\tilde{G}_{\epsilon}$ that can be considered as an element of $\operatorname{End}_{\Lambda_2}(E_2)$ through the fixed isometry.
\end{proof}
\section{Analysis of the Dirac operator}\label{finitt}

\subsection{Finite dimensionality of the index problem}

\noindent Consider the leafwise Dirac operator $D$. This is a measurable field of unbounded first order differential operators $\{D_x\}_{x\in X}$. Its measurability property is easily checked observing that is equivalent to prove the measurability of the field of bounded operators $$(D_x+i)^{-1}:L^2(L_x;E)\longrightarrow H^1(L_x;E).$$ Here the field of natural Sobolev spaces has the canonical structure given by inclusion into $L^2$. Now, the self--adjointness of $D_x$ with domain $H^1(L_x;E)$ shows that $$(D_x+i):H^1(L_x;E)\longrightarrow L^2(L_x;E)$$ is a Hilbert space isomorphism. Choose two  sections $s,t$ of the domain and range respectively with the additional property that are smooth when restricted to each leaf then
$$\langle (D_x+i)s(x),t(x)\rangle_{L^2(L_x;E)}=\langle s(x),(D_{x}-i)t(x)\rangle_{L^2(L_x;E)}$$ and the measurability of the right--hand side is clear.
Now it remains to apply the Example 2. in Dixmier \cite{Dix} p-157 to have that the leafwise inverse family is measurable (Borel)\footnote{The example simply states that the family of the inverses of a Borel family of Hilbert space isomorphisms is Borel}.
\bigskip

\noindent Since the foliation is even dimensional there is a canonical involution 
$\tau=i^pc(e_1\cdot\cdot\cdot e_{2p})$ giving a parallel hortonormal $\pm 1$ eigenbundles splitting $E=E^+\oplus E^-$. Moreover the Dirac operator is odd with respect to this splitting. That's to say that $D$ anticommutes with $\tau$, giving a pair of first order leafwise differential elliptic operators $D_x^{\pm}:C^{\infty}_c(L_x;E^{\pm})\longrightarrow C^{\infty}_c(L_x;E^{\mp}).$ We continue to use the same notation for their unique 
$L^2$--closure and we have 
$D=D^+\oplus D^-$ with 
$D^+=(D^-)^*.$ 

\bigskip
\noindent 
The operator $D^+$ is called the chiral longitudinal Dirac operator, in general this is not a Breuer--Fredholm operator. In fact Fredholm properties are governed by its behavior at the boundary i.e its restriction to the base of the cylinder $\partial X_0$. In the one leaf situation $D^+$ is Fredholm in the usual sense if and only if $0$ is not in the continuous spectrum of $D^-D^+$  or equivalently if the continuous spectrum has a positive lower bound.
However what is still true in this case is that the $L^2$ kernels of $D^+$ and $D^-$ are finite dimensional and made of smooth sections. The difference $$\operatorname{dim}_{\Lambda}\operatorname{Ker}_{L^2}(D^+)-\operatorname{dim}_{\Lambda}\operatorname{Ker}_{L^2}(D^-)$$ is by definition the $L^2$--chiral index of $D^+$. It gives the usual fredholm index when the operator is Fredholm. Notice that in the non Fredholm case the $L^2$ index is not stable under compactly supported perturbations. This is one of the most difficulties in its computation.
\bigskip

\noindent We are going to show that in our foliation case the chiral index problem is $\Lambda$--finite dimensional in the following sense. 
\begin{itemize}
\item By an application of the parametrized measurable spectral theorem  the projections on the $L^2$--kernels of $D^{\pm}$ belong to the Von Neumann algebras of the corresponding bundles, 
$\chi_{\{0\}}(D^+)\in \operatorname{End_{\Lambda}(E^{\pm})}$ and
decompose as a Borel family of bounded operators $\{\chi_{\{0\}}(D^{\pm})_x\}_x$ corresponding to the projections on the $L^2$ kernels of $D^{\pm}_x$. Furthermore they are implemented by a Borel family of uniformly smoothing Schwartz kernels.

\item The family of projections above give rise to a longitudinal measure on the foliation. These measure are the local traces 
$U\longmapsto \operatorname{tr}_{L^2(L_x)}[\chi_U\cdot \chi_{\{0\}}(D^{\pm})_x \cdot \chi_U]$ where for a Borel $U\subset L_x$ the operator $\chi_U$ acts on $L^2(L_x)$ by multiplication.
In terms of the smooth longitudinal Riemannian density these measures are represented by the pointwise traces of the leafwise Schwartz kernels. We prove that these local traces has the following finiteness property completely analog to the Radon property for compact foliated spaces.

\bigskip

\noindent {\bf{Finiteness property for local traces of projections on the kernel.}} 

\noindent Consider a leaf $L_x$. This is a bounded geometry manifold with a cylindrical end $\partial L_x \times \R^+$. We claim that for every compact $K\subset \partial L_x$ 
\begin{equation}\label{latracciaefinita} \operatorname{tr}_{L^2(L_x)}[\chi_{K\times \R^+}\cdot \chi_{\{0\}}(D^{\pm})_x \cdot \chi_{K\times \R^+}]<\ty\end{equation} Since this list of items is aimed to the definition of the index, the (rather long) proof of inequality \eqref{latracciaefinita} statement is postponed immediately after. We limit ourselves here to say that is the relevant form of \emph{elliptic regularity} in our situation.
\item The integration process of a longitudinal measure against a transverse holonomy invariant measure immediately shows that the integrability condition above is sufficient to assure finite $\Lambda$--dimensionality of the $L^2$ kernels of $D^{\pm}$. Here is the proof.

First one has to choose a complete compact transversal $S$ and a Borel map $f:X\longrightarrow S$ that respects the leaf equivalence relation displaying $X$ as measure--theoretically fibering over $S$. Thanks to our assuptions on the foliation we can choose $S$ composed by two pieces $S_1$ and $S_2$ where $S_1=\partial X_0\times \{0\}$ on the cylinder while $S_2$ is an \emph{interior} transversal.
Since we are working in the Borel world we can surely think that $f$ restricts to $U$ with values on $S_1$ and outside $U$ with values on $S_2$.
Now the integral has two terms. The first integral, on $S_1$ is finite thanks to the finiteness property above in fact the situation here is a fibered integral of a standard Radon measure on the base times a finite measure. The interior term is finite thanks to proposition 4.22 in \cite{MoSc}.
\item \begin{dfn}Define the chiral $\Lambda$--$L^2$--index $$\operatorname 
{Ind}_{L^2,\Lambda}(D^+):=\tru (\chi_{\{0\}}(D^+))-\tru (\chi_{\{0\}}(D^-))\in \R.$$
\end{dfn}
\end{itemize}
\noindent {\bf{Proof of finiteness property of the local trace of kernel projections}}
\begin{proof}\label{2344}
It is clear that it suffices to prove the property for each operator $(\cdot)_{|\partial_x \times \R^+}\chi_{\{0\}}(D^+_x)$. 
Let us consider the operator $D^+$ on a fixed leaf $L_x$. This is a bounded geometry manifold with a cylindrical end $\partial L_x\times \R^+=\{y\in L_x:r(y)\geq 1\}$ where
the operator can be written in the form $B+\partial/{\partial t}$
acting on sections of $F\longrightarrow \partial L_x\times \R^+$.
The boundary operator $B$ is essentially selfadjoint on $L^2(\partial L_x;F)$ on the complete manifold $\partial{L_x}$ (see \cite{Ch} and \cite{ChGrTa} for a proof of self--adjointness using finite propagation speed tecniques). 

\noindent We are going to remind the Browder--G{\aa}rding type generalized eigenfunction expansion for $B$ (see \cite{Di} 11, 300--307, \cite{DuSc} and \cite{Rama} for an application to a A.P.S foliated and Galois covering index problems). 

\noindent According to Browder--G{\aa}rding there exist
\begin{enumerate}
\item  A sequence of smooth sectional maps 
$e_j:\R \times \partial L_x \longrightarrow {F}$ i.e. $e_j$ is measurable and for every $\lambda \in \R$, 
$e_j(\lambda,\cdot)$ is a smooth section of ${F}$ over $\partial L_x$ such that
$Be_j(\lambda,x)=\lambda e_j(\lambda,x).$
\item A sequence of measures $\mu_j$ on $\R$ such that the map
$V:C^{\infty}_c(\partial L_x;F)\longrightarrow \bigoplus_j L^2(\R,\mu_j)$ defined by 
$(Vs)_j(\lambda)=\langle s,e_j(\lambda,\cdot) \rangle_{L^2(\partial L_x)}$ (integration w.r.t Riemannian density) extends to an Hilbert space isometry 
$$V:L^2(\partial L_x; F)\longrightarrow \bigoplus_j L^2(\R,\mu_j)=:\mathcal{H}_B$$ which intertwines Borel spectral functions $f(B)$ with the operator defined by multiplication by $f(\lambda)$ with domain given by
$\operatorname{dom}f(B)=\Big \{ s:\sum_j \int_{\R}|f(\lambda)|^2|(Vs)_j(\lambda)|^2d\mu_j(\lambda)<\infty \Big \}.$ In particular beying an isometry means
$\displaystyle{\int}_{\partial L_x}|s(x)|^2dg=\sum_j \int_{\R}|(Vs)_j|^2d\mu_j(\lambda).$
\end{enumerate}\noindent Notice that $e_j(\lambda,\cdot)$ need not be square integrable on $L_x$. Taking tensor product with $L^2(\R)$ we have the isomorphism \begin{equation}\label{is}L^2(\partial L_x \times \R^+,F) \simeq L^2(\partial L_x,{F})\otimes L^2(\R) \xrightarrow{\sim} [\oplus_j L^2(\R,\mu_j)]\otimes L^2(\R^+)={\mathcal{H}}_B\otimes L^2(\R^+)\end{equation} where $R^+=(0,\infty)_r$. 
\noindent Under the identification $W:=V \otimes \textrm{Id}$ the operator $D^+$ is sent into $\lambda + \partial_r$ acting on the space ${\mathcal H}_B\otimes L^2({\R^+})$.
\noindent Now let $s$ be an $L^2$--solution of $D_xs=0$. By elliptic regularity it restricts to the cylinder as an element  $s(x,r)\in C^{\infty}(\R^+,H^{\infty}(\pal;F))\cap L^2(\R^+;L^2(\pal,F))$ solution of $(\partial_r+B)s=0$ then
{\begin{align}\label{decade}
\partial_r (Vs)_j(\lambda,t)&=\pr \inbl \langle s(x,r), e_j(x,r)\rangle dg=
\inbl \langle dr \, s(x,r),e_j(\lambda,x)\rangle dg 
\\ \nonumber
&=-\inbl \langle Bs(x,r),e_j(\lambda,x)\rangle dg= \inbl \langle s(x,r),Be_j(x,r)\rangle dg \\ \nonumber
&= -\lambda \inbl \langle Bs(x,r),
e_j(\lambda,x) \rangle dg=-\lambda (Vs)_j(\lambda,r).
\end{align}}
Equation \eqref{decade} says that all $L^2$ solutions of $D^+=0$ under the representation $V$ on the cylinder are zero $\mu_j(\lambda)$--a.e. for $\lambda\leq 0$ for every $j$. 
Decompose, for fixed $a>0$
\begin{equation}\label{decomp}L^2(\pal \times \R^+; F)=L^2(\R^+;\mathcal{H}_B([-a,a]))\oplus L^2(\R^+,\mathcal{H}_B(\R\setminus [-a,a]))\end{equation} 
where the notation is $\mathcal{H}_B(\Delta)$ for the spectral projection associated to $\chi_{\Delta}$.
Let $\Pi_{\leq a}$ and $\Pi_{>a}$ respectively be the hortogonal projections corresponding to \eqref{decomp}. Let $\pkp$ be the $L^2$ projection on the kernel, there's a composition $$\Pi^a:={ \Pi_{\leq a}} \circ(\cdot)_{|\pal\times \R^+}\circ \pkp$$ defined trough 
\begin{equation}
{
\xymatrix{ L^2(L_x) \ar[r] & \operatorname{Ker}_{L^2}(D_x^+) \ar[r]&L^2(\partial L_x\times \R^+) \ar[r]&L^2(\R^+;\mathcal{H}_{B}([-a,a])).
}}\end{equation} Thanks to the Browder--Garding expansion and equation \eqref{decade} we can see that elements $\xi$ belonging to the space $\Pi^aL^2(L_x)$ are of the form \begin{equation}\label{sha}\xi=\chi_{(0,\infty)}(\lambda)e^{-\lambda t}\zeta_0\end{equation} with $\zeta_0=\zeta_{0j} \in H^{\infty}(\pal;{F})$ to be univocally determined using boundary conditions. Formula \eqref{sha} allows to define\footnote{this is clearly inspired by Melrose definition \cite{Me} Chapter 6} the "boundary datas" mapping 
\begin{eqnarray*}
\textrm{BD}:\Pi^aL^2(L_x;F)\longrightarrow \mathcal{H}_B((0,a])\\ 
  W^{-1}(\chi_{(0,a]}(\lambda)\zeta_0 e^{-\lambda t})\longmapsto W^{-1}(\chi_{(0,a]}(\lambda)\zeta_0)
\end{eqnarray*}
This is continuous and injective in fact injectivity is obvious while continuity follows at once from 
\begin{align} \nonumber
\|\xi\|_{L^2(\pal\times \R^+)}&=\sum_j\int_{\R}\int_0^{\infty}e^{-2\lambda t}|\zeta_{0j}(\lambda)|^2dtd\mu_j(\lambda)\geq \sum_j\int_{[-a,a]}\int_0^{\infty}e^{-2\lambda t}|\zeta_{0j}(\lambda)|^2dtd\mu_j(\lambda)\\ \nonumber
&\geq \sum_j \int_{[-a,a]}\int_0^{\infty}e^{-2at}|\zeta_{0j}(\lambda)|^2dt d\mu_j(\lambda)=
1/{(2a)}\sum_j \int_{\R}|\chi_{[-a,a]}\zeta_{0j}(\lambda)|^2d\mu_j(\lambda)\\ \nonumber
&=1/(2a)\|\chi_{[-a,a]}\zeta_0 \|_{\mathcal{H}_B}.
\end{align} 
Now choose an orthonormal basis $s_m=f_m\otimes g_m \in L^2(\partial L_x\times \R^+,{F})$ and a compact set of the boundary $A\subset \partial L_x$, then put $\chi_{A^{\square}}=\chi_{A\times (0,\infty)}(x,r)$. Consider the operator $\chi_{A^{\square}}\Pi^{a}\chi_{A^{\square}}$ acting on $L^2(L_x;F)$, now notice that $\Pi^a$ acts on $s_m$ via the natural embedding $L^2(\partial L_x)\subset L^2(L_x)$ then
\begin{equation}\label{trac}\textrm{tr}(\chi_{A^{\square}}\Pi^{a}\chi_{A^{\square}})=
\sum_m\langle \chi_{A^{\square}}\Pi^{a}\chi_{A^{\square}} s_m,s_m\rangle_{L^2(\pal\times \R^+)}.\end{equation}
Write $\textrm{BD}[\Pi^a \chi_{A^{\square}} s_m]=W^{-1}[\chi_{(0,a]}(\lambda)\zeta_0^{(m)}]$ hence $[\Pi^a \chi_{A^{\square}} s_m]=\chi_{(0,a]}(\lambda)\zeta_0^{(m)}e^{-\lambda t}.$ 
By continuity of $\textrm{BD}$ the sequence 
$\chi_{(0,a]}\zeta_0^{(m)}$ is bounded. Then \eqref{trac} becomes
\begin{eqnarray}
\nonumber\label{conint}
\textrm{tr}(\chi_{A^{\square}}\Pi^a\chi_{A^{\square}})=
\sum_m\langle W^{-1}[\chi_{(0,a]}(\lambda)\zeta_0^{(m)}e^{-\lambda t}],\chi_{A^{\square}}s_m\rangle \\=  \label{conint}
\sum_m\langle \chi_{(0,a]}(\lambda)\zeta_0^{(m)}e^{-\lambda t},W(\chi_{A^{\square}}s_m)\rangle\\= \sum_m \int_{\R^+}\int_{\R \times \mathbb{N}}\chi_{(0,a]}(\lambda)\zeta_0^{(m)}e^{-\lambda t}\overline{\Big\{{W(\chi_{A^{\square}}s_m)}\Big\}}d\mu({\lambda})dt
\end{eqnarray} where $\mu$ is the direct sum of the $\mu_j$'s.

\noindent Last term of \eqref{conint} can be estimated using Cauchy--Schwartz inequality and the trivial identity 

\begin{align*}
&W(\chi_{A^{\square}}s_m)W(\chi_{A^{\square}}f_m\otimes g_m)=V(\chi_A(x)f_m(x))g_m(t).\\
&{\sum_m} \int_{\R^+}\int_{\R \times \mathbb{N}}\chi_{(0,a]}(\lambda)\zeta_0^{(m)}e^{-\lambda r}\overline{\Big\{{w(\chi_{A^{\square}}s_m)}\Big\}}d\mu({\lambda})dr
\\
&\leq \sum_m\Big \{\int_{\R^+}\int_{\R \times \mathbb{N}}|g_m(r)|^2|\zeta_0^{(m)}|^2d\mu(\lambda)dr \Big \}^{1/2}\cdot
\\
&\cdot \Big\{\int_{\R^+}e^{-2a  r}\int_{\R \times \mathbb{N}}\chi_{(0,a]}e^{-2(\lambda-a)}|V(\chi_A f_m)|^2d\mu(\lambda)dr\Big\}^{1/2}
\\
&\leq\sum_m C\Big\{\int_{\R \times \mathbb{N}}\chi_{(0,a]}|V(\chi_A f_m)|^2d\mu(\lambda)dr\Big\}^{1/2}\\&=
C \sum_m \| \chi_A \mathcal{H}_B((0,a]) \chi_A f_m \|_{L^2(\pal)}\\
&\leq C\sum_m \langle \chi_A \mathcal{H}_B((0,a])\chi_A f_m,f_m\rangle\\& =
C \textrm{tr}( \chi_A \mathcal{H}_B((0,a])\chi_A) <\infty.
\end{align*}

\noindent In the last step we used the fact that for a projection on a closed subspace $K$ one can compute its trace as, $\textrm{tr}(K)= \sum_m \langle Kf_m,f_m \rangle=\sum_m\| Kf_m \|$ together with the fact that $\mathcal{H}_B((0,a])$ is a spectral projection of $B$ hence uniformly smoothing.
\noindent Let us now pass to examine the operator 
$$\Pi_a:= \Pi_{\geq a} { \circ(\cdot)_{|\pal\times \R^+} }\circ \pkp$$ defined by
\begin{equation}
{
\xymatrix{ L^2(L_x) \ar[r] & \operatorname{Ker}_{L^2}(D_x^+) \ar[r]&L^2(\partial L_x\times \R^+) \ar[r]&L^2(\R^+;\mathcal{H}_{B}(\R\setminus [-a,a])).
}}\end{equation} 
arising from the second addendum of the splitting \eqref{decomp}.
Let $\varphi_k$ be the characteristic function of $r\leq k$ and
$$\Lambda_k:=\Pi_{\geq a}\circ \varphi_k \circ (\cdot)_{|\partial_x\times \R^+}\circ \chi_{\{0\}}(D^+_x).$$ Now
\begin{align}\label{final}
\|(\Pi_a-\Lambda_k)\xi\|&=
\|\Pi_{\geq a}(\varphi_k-1)(\cdot)_{|\partial L_x \times \R^+}\chi_{\{0\}}(D^+_x)\xi\|_{L^2(\partial L_x\times \R^+)}\\\nonumber 
&= \int_{k}^{\infty}\int_{(a,\infty)\times \mathbb{N}}e^{-2\lambda r}|\zeta_0|^2d\mu
(\lambda)dt\leq e^{-2ak}\int_{(a,\infty)\times \mathbb{N}}\int_0^{\infty}e^{-2\lambda r}|\zeta_0|^2d\mu(\lambda)dr\\ \nonumber
& \leq e^{-2ak}\|\xi\|_{L^2(\partial L_x \times \R^+)}.\end{align}
\noindent Finally choose a compact $A\subset \partial L_x$, estimate \eqref{final} shows that $S_k:=\chi_{A^{\square}}\Lambda_k \chi_{A^{\square}}$ converges uniformly to $\chi_{A^{\square}}\Pi_a \chi_{A^{\square}}.$ 
\noindent Observe that $S_k$ is compact by Rellich theorem and regularity theory in fact $\Pi_{\textrm{Ker}(T^+)}$ is obtained by functional calculus from a rapid Borel function hence has a uniformly smoothing Schwartz--kernel (see the appendix for more informations). Since
$\chi_{A^{\times}}\Lambda_k \Pi_{>a}\Pi_{\textrm{Ker}(T^+)}\chi_{A^{\times}}$ is norm--limit of compact operators is compact but a compact projection is finite rank. \end{proof}

\subsection{Breuer--Fredholm perturbation}

\noindent Our main application of the splitting principle is the construction of a $\Lambda$--Breuer Fredholm perturbation of the leafwise Dirac operator. We recall some some notations; $$X_k:=X_0\bigcup_{\partial X_0}\Big{(}X_0\times [0,k]\Big{)},\,\, Z_k=\partial X_0 \times [k,\ty).$$  
Let $\theta$ be a smooth function satisfying $\theta=\theta(r)=r$ on $Z_1$ while $\theta(r)=0$ on $X_{1/2}$, put $\dot{\theta}=d\theta/dr.$ Let $\Pi_{\epsilon}:=\chi_{I_{\epsilon}}(D^{\mathcal{F}_{\partial}})$ for $I_{\epsilon}:=(-\epsilon,0)\cup (0,\epsilon).$
\bigskip
\noindent Our perturbation will be the leafwise operator \pecetta{\begin{equation}\label{de2}D_{\epsilon,u}:=D+\dot{\theta} \Omega(u-\dfo \Pi_{\epsilon})\textrm{ for }\epsilon>0,\quad u\in \mathbb{R}\end{equation}} that is $\mathbb{Z}_2$ odd as $D$. \noindent We write $D_{\epsilon,u}=D^+_{\epsilon,u}\oplus D^-_{\epsilon,u}$ and $D_{\epsilon,u,x}$ for its restriction to $L_x$, also for brevity $D_{\epsilon,0
}:=D_{\epsilon}.$

\noindent Notice that the perturbed boundary operator is \begin{equation}\label{23dice}\deuf =\dfo(1-\Pi_{\epsilon})+u=D_{\epsilon,0}^{\mathcal{F}_{\partial}}+u.\end{equation} Since for $\epsilon>0$, $0$ is an isolated point in the spectrum of $D_{\epsilon,0}^{\mathcal{F}_{\partial}}$ we see\footnote{this is a simple application of the functional calculus; if the spectrum of the boundary operator $D_{\epsilon,0}^{\mathcal{F}_{\partial}}$ has an hole of around the zero then for $0<|u|<\epsilon$ also the operator $\deuf =D_{\epsilon,0}^{\mathcal{F}_{\partial}}+u$ has an hole around the zero and is inverted by 
 spectral  its function under $f(\lambda):=\chi_{(-\delta,\delta)}(\lambda)\lambda^{-1}$ with  some positive $\delta$} 
that $\deuf$ is invertible for $0<|u|<\epsilon$. 
\noindent For further application let us compute the essential spectrum of $B_{\epsilon,u}^2$ where $$B_{\epsilon,u}=D+\Omega(u-D^{\mathcal{F}_{\partial}}\Pi_{\epsilon})$$ on the foliated cylinder  $Z_0$ with product foliation $\mathcal{F}_{\partial}\times \R$. Since we deal with product structure operators we can surely think the Von Neumann algebra becomes $\operatorname{End}_{\Lambda_0}(E)\otimes B(L^2(\R))$ where $\operatorname{End}_{\Lambda_0}(E)$ is the Von Neumann algebra of the base i.e. the foliation induced on the transversal $X_0\times\{0\}$. The integration process shows that the trace is nothing but $\tru=\operatorname{tr}_{\Lambda_0}\otimes \operatorname{tr}$ where the second factor is the canonical trace on $B(L^2(\R))$. In section \ref{primac} we will say more on the relation betwenn the boundary algebra and the whole algebra.
 We can write 
\begin{align}\label{22dice}B_{\epsilon,u}^2=&\left(\begin{array}{cc}0 & -\partial_r+u+D^{\mathcal{F}_{\partial}}(1-\Pi_{\epsilon}) \\\partial_r+u+D^{\mathcal{F}_{\partial}}(1-\Pi_{\epsilon})&0\end{array}\right)^2\\=&
\left(\begin{array}{cc}-\partial_r^2 &  0\\0 & -\partial_r^2\end{array}\right)+\left(\begin{array}{cc}0 & u+D^{\mathcal{F}_{\partial}}(1-\Pi_{\epsilon}) \\ \nonumber u+D^{\mathcal{F}_{\partial}}(1-\Pi_{\epsilon}) & 0\end{array}\right)^{2}=-\partial_r^2 \operatorname{Id}+V^2.\end{align}

\noindent Consider the spectral measure $\mu_{\Lambda_0,V^2}$ of $V^2$ on the tranversal section $X_0\times \{0\}$. We claim the following facts
\begin{enumerate}
\item $\omega:=\inf \operatorname{supp}(\mu_{\Lambda_0,V^2})>0$
\item $\mu_{\Lambda,B^2_{\epsilon,u}}(a,b)=\infty,\quad 0\leq a<b,\quad \omega<b$
\item $\mu_{\Lambda,B^2_{\epsilon,u}}(a,b)=0,\quad 0\leq a<b\leq \omega.$
\end{enumerate} First of all 1. is obvious since \eqref{23dice} together with \eqref{22dice} imply  $$\operatorname{spec}(D_{\epsilon,u}^{\mathcal{F}_{\partial}})^2\subset [(\epsilon+u)^2,\infty).$$ To prove 2. one first observes that we can use the Fourier transform in the cylindrical direction. This gives a spectral representation of 
$-\partial_r^2$ as the multiplication by 
$y^2$ on 
$L^2(\R)$. Choose some 
$\gamma<(b-\omega)/2.$ We can prove the following inclusion for the spectral projections \begin{equation}\label{specin}
\chi_{(a,\gamma+\omega)}(V^2)\otimes \chi_{(0,\gamma)}(-\partial_r^2)\subset 
\chi_{(a,b)}(B^2_{\epsilon,u}).\end{equation} In fact one can also use a (leafwise) spectral representation for 
$V$ as the multiplication operator by $x$. Then \eqref{specin} is reduced to prove the implication
$$a<x^2<\gamma+\omega,\quad 0<y^2<\gamma \Rightarrow a<x^2+y^2<b.$$  From \eqref{specin} follows 
$$\mu_{\Lambda,B^2_{\epsilon,u}}(a,b)\geq \mu_{\Lambda_0,V^2}(a,\gamma+\omega)\cdot \operatorname{tr}_{B(L^2(\R))}\chi_{(0,\gamma)}(-\partial_r^2)=\infty$$
in fact the first factor is non zero and the second is clearly infinite. Finally the third statement is very similar in the proof. We have shown that 
$$\spc(B_{\epsilon,u}^2)=[\omega,\infty).$$
\begin{prop}\label{312}
The operator $\deu$ is $\Lambda$--Breuer--Fredholm if $0<|u|<\epsilon$.
\end{prop}
\begin{proof}
The splitting principle (actually for order 2 operators but it makes no difference) says that the essential spectrum is determined by the operator on the cylinder for $r>1$. The above calculation ends the proof.
\end{proof}
\noindent In what follows we shall investigate the relations between the Breuer--Fredholm index of the perturbed operator and the $L^2$--index of the unperturbed Dirac operator. To this end we remark the use of weighted $L^2$--spaces is fruitful as Melrose shows in \cite{Me}.

\begin{dfn}
For $u\in \R$, denote 
$e^{u\theta}L^2$ the Borel field of Hilbert spaces (with obvious Borel structure given by $e^{u\theta}\cdot L^2-\textrm{Borel structure}.$
$\{e^{u\theta}L^2(L_x;E)\}_x$ where, for 
$x\in X$, $e^{u\theta}L^2(L_x;E)$ is the space of distributional sections $w$ such that $e^{-u\theta}\omega\in L^2(L_x;E)$. Analog definition for weighted Sobolev spaces $e^{u\theta}H^k$ can be written.
\end{dfn}

\noindent Notice that $e^{u\theta}L^2(L_x;E)=L^2(L_x;E,e^{-2u\theta}dg_{|L_x})$ where $dg$ is the leafwise Riemannian density so these Hilbert fields correspond to the representation of $\mathcal{R}$ with the longitudinal measure $x\in X \longmapsto e^{-2u\theta}dg_{|L_x}=r^*(e^{-2u\theta}dg)$ (transverse function, in the language of the non commutative integration theory \cite{Cos}).

\noindent The operators $D$ and its perturbation $D_{\epsilon,u}$ extend to a field of unbounded operators 
$e^{u\theta}L^2\longrightarrow e^{u\theta}L^2$ with domain 
$e^{u\theta}H^1$. Put $$e^{\infty \theta}L^2_x:=\cup_{\delta>0}e^{\delta \theta}L^2_x.$$

\noindent In what follows we will use, for brevity the following notation: $\partial L_x:=L_x \cap (\partial X_0\times \{0\})$ and 
$$Z_x:=\partial L_x \times [0,\infty)$$ for the cylindrical end of the leaf $L_x$.

\bigskip

\noindent For a smooth section 
$s^{\pm}$ such that 
$D_{\epsilon,u,x}^{\pm}s^{\pm}=0$ we have 
$(D_{\epsilon,u,x}^{\pm})_{|{\partial L_x\times \R^+}}(s^{\pm})_{|\partial L_x \times \R^+}=0$ that can be easily seen choosing smooth $r$--functions  
$\phi,\psi$ with $\phi_{X_0}=1$, $\psi_{Z_{1/4}}=1$, $\operatorname{supp}(\psi\subset Z_{1/8})$ and evaluating $[D_{\epsilon,u,x}^{\pm}(\phi(1-\psi)s+\phi\psi s)=0]_{|\partial L_x\times \R^+}.$

\noindent The isomorphism $W$ defined in \eqref{is} used in the proof of finiteness property for the kernel projection, can be defined also as an isomorphism  $e^{u\theta}L^2(\partial L_x \times \R^+,F) \simeq {\mathcal{H}}_B\otimes e^{u\theta} L^2(\R^+)$ in a way that solutions of $D^{\pm}_{\epsilon,u,x}s^{\pm}=0$ with conditions $s^{\pm}\in e^{\infty \theta}\cap L^2_x$ can be represented as solutions of $[\pm \partial_r+\lambda+\dot{\theta}(r)(u-\chi_{\epsilon}(\lambda)\lambda)]Ws^{\pm}=0$ with $\chi_{\epsilon}(\lambda)=\chi_{(-\epsilon,0)\cup(\epsilon,0)}(\lambda)$ acting as a multiplier on $\bigoplus_j L^2(\R,\mu_j).$ In particular (forgetting for brevity the restriction symbol)
\begin{equation}
\label{rappa}
Ws^{\pm}=\zeta_j^{\pm}(\lambda)\operatorname{exp}\{{\mp u \theta(r) \mp\lambda[r-\theta(r)\chi_{\epsilon}(\lambda)]}\}\end{equation} with suitable choosen $\zeta_j^{\pm}(\lambda)
\in L^2(\mu_j).$

\begin{prop}\label{rappacon}
Let $\epsilon>\delta>0$ and $\delta'\in \R$ then
\begin{enumerate}
\item $\xi \in \operatorname{Ker}_{e^{\delta' \theta}L^2}
(D_x^+)\Longmapsto\xi_{Z_x}=e^{-r\defox}{h}$ 
with ${h}\in \chi(\defox)_{(-\delta',\ty)}L^2_x.$
\item $\xi \in \operatorname{Ker}_{L^2}(D^+_{\epsilon,x})\Longmapsto \xi_{Z_x}=e^{-r \defox+\theta(r)\defox \Pi_{\epsilon,x}}h$, with $h\in \chi(\defox)_{(\epsilon,\ty)}L^2_x$
\item $\xi \in \operatorname{Ker}_{e^{\delta \theta}L^2}(D^+_{\epsilon,x})\Longmapsto \xi_{|Z_x}=e^{-r\defox+\theta(r)\defox \Pi_{\epsilon,x}}h,$ $h\in \chi(\defox)_{(-\epsilon,\ty)}L^2_x,$
\end{enumerate} recall that $\Pi_{\epsilon,x}=\chi_{(-\epsilon,\epsilon)-\{0\}}(\defox).$ Moreover the following identity (as fields of operators) holds true
$$D^{\pm}e^{\mp\theta(r)D^{\mathcal{F}_{\partial}}\Pi_{\epsilon}}=e^{\mp \theta(r) D^{\mathcal{F}_{\partial}} \Pi_{\epsilon}}D_{\epsilon}^{\pm}.$$
\end{prop}
\begin{proof}
\begin{enumerate}
\item From the representation formula \eqref{rappa} of formal solutions for $u=0,$ $\epsilon=0$ it remains $\xi=\xi_j(\lambda)e^{-\lambda r}$. Then $e^{-\delta' \theta}\xi$ must be square integrable hence clearly $\xi_j(\lambda)=h_j(\lambda)\in \chi_{(-\delta',\ty)}(\defox)$. 
\end{enumerate} The remaining are proved in a very similar way. The last statement is merely a computation.
\end{proof}
\bigskip

\noindent Solutions of $D^{\pm}_{\epsilon,x}s^{\pm}=0$ belonging to the space 
$\bigcap_{u>0}e^{u\theta}L^2(L_x;E^{\pm})$ are called 
$L^2$--\emph{extended solutions}, in symbols 
$\operatorname{Ext}(D_{\epsilon,x}^{\pm}).$ 
Next we study this space of solutions as $x$ varies.
\begin{prop}\label{exx}
For every $x\in X$ and $0<u<\epsilon$
\begin{align}\label{primain}
&1.\,\operatorname{Ker}_{L^2}(\deup)=\operatorname{Ker}_{e^{-u\theta}L^2}(\deup)=\operatorname{Ker}_{L^2}(D^{\pm}_{\epsilon,\mp u,x})\\ &2.\,
\operatorname{Ext}(\deup)=\operatorname{Ker}_{e^{u\theta}L^2}(\deup)=\operatorname{Ker}_{L^2}(\dex).\\& 3.\, \label{terx}
\operatorname{Ker}_{L^2}(D^{\pm}_{\epsilon,x})\subset \operatorname{Ext}(D^{\pm}_{\epsilon,x})
\end{align}
\end{prop}
\begin{proof}
We show only the first equality of \eqref{primain} the others being very similar. This is a simple application of equation \eqref{rappa}. In fact, for $u=0$, $Ws^{\pm}=\zeta_j^{\pm}(\lambda)\operatorname{exp}\{{ \mp\lambda[r-\theta(r)\chi_{\epsilon}(\lambda)]}\}.$ The condition of being square integrable in $(\R,\mu_j)\otimes (\R^+,dr)$ is easily seen to be equivalent to $\zeta_j^+(\lambda)=0$ $\lambda<\epsilon$, $\lambda$--a.e and $\zeta_j^-(\lambda)=0$ $\lambda>-\epsilon$ in particular, for $r\geq 1$
$Ws^{\pm}=\zeta^{\pm}_j(\lambda)e^{\mp\lambda r}\chi_{\pm \lambda \geq \epsilon}(\lambda)$ then $e^{u\theta}s^{\pm}\in L^2$ if $u<\epsilon$. For the reverse inclusion the proof is the same. For the third stament note that $e^{u\theta}L^2\subset e^{v\theta}L^2$ for every $u,v\in \R$ with $u\leq v$ then $\operatorname{Ker}_{L^2}\subset \operatorname{Ext}$.
\end{proof}
\noindent Proposition \ref{exx} shows that the mapping $x\longmapsto \ext$ gives a Borel field of closed subspaces of $L^2$. No difference in notation between the space $\operatorname{Ext}$ and $\operatorname{Ker}$ and the corresponding projection in the Von Neumann algebra will be done in the future.
Inclusion \eqref{terx} together with \ref{312} and the finiteness property of the $L^2$--kernel projection says that the difference 
\begin{equation}
\label{deffh}
h^{\pm}_{\Lambda,\epsilon}=\operatorname{dim}_{\Lambda}(\operatorname{Ext}(D_{\epsilon}^{\pm}))-\operatorname{dim}_{\Lambda}(\operatorname{Ker}_{L^2}(D_{\epsilon}^{\pm}))=\tru(\operatorname{Ext}(D_{\epsilon}^{\pm}))-\tru(\operatorname{Ker}_{L^2}(D_{\epsilon}^{\pm}))\in \R\end{equation} is a finite number.  
\begin{lem}\label{111}For $\epsilon>0$
\begin{enumerate}
\item $\operatorname{dim}_{\Lambda}{Ker}_{L^2}(D^{\pm}_{\epsilon})=\lim_{u\downarrow 0}\operatorname{dim}_{\Lambda}{Ker}_{L^2}(D^{\pm}_{\epsilon,\mp u})=\lim_{u\downarrow 0}\operatorname{dim}_{\Lambda}{Ker}_{L^2}(D^{\pm}_{\epsilon,\pm u})-h^{\pm}_{\Lambda,\epsilon},$
\item
$\operatorname 
{Ind}_{L^2,\Lambda}(D^+_{\epsilon})=\lim_{u\downarrow 0}\operatorname{Ind}_{\Lambda}(D^{+}_{\epsilon,u})-h^{+}_{\Lambda,\epsilon}=\lim_{u\downarrow 0}\operatorname{Ind}_{\Lambda}(D^{+}_{\epsilon,-u})+h^{-}_{\Lambda,\epsilon} $
\end{enumerate}
\end{lem}
\begin{proof}Nothing to prove here, proposition \ref{exx} says that the limit is constant for $u$ sufficiently small, the second one in the statement follows from the first by summation.
\end{proof}
\noindent Now define the extended solutions $\operatorname{Ext}(D^{\pm}_x)$ in the same way i.e. distributional solution of the differential operator $D^{\pm}_x:C^{\infty}_c(L_x;E^{\pm})\longrightarrow C^{\ty}_c(E^{\mp};E)$ belonging to each weighted $L^2$--space with positive weights,
$$\operatorname{Ext}(D^{\pm}_x)=\bigcap_{u>0}\operatorname{Ker}_{e^{u\theta}L^2}(D^{\pm})=\{s\in C^{-\infty}(L_x;E^{\pm});\,D ^{\pm}s=0;\, e^{-u\theta}s\in L^2\, \forall u>0\}.$$ Here we have made use of the longitudinal Riemannian density to to identify sections with sections with values on density and the Hermitian metric on $E$, in a way that one has the isomorphism 
$C^{-\infty}(L_x;E^{\pm})\simeq C^{\infty}_c(L_x;(E^{\pm})^*\otimes \Omega(L_x))$ to simplify the notation with distributional sections of the bundle $E$. 

\noindent It is clear by standard elliptic regularity that extended solutions of $D^{\pm}$ are smooth on each leaf. In fact $D^{\pm}$ a first order differential elliptic operator and one can construct a parametrix i.e. an inverse of $D^{\pm}$ modulo a smoothing operator i.e. an operator sending each Sobolev space onto each Sobolev space (of the new, weighted metric). 

\bigskip
\begin{oss}\label{dimma}
\noindent By definition $\operatorname{Ext}(D^{\pm})\subset e^{u\theta}L^2$ for every $u>0$, define $\operatorname{dim^{(u)}_{\Lambda}}(\operatorname{Ext})$ as the trace in $\operatorname{End}_{\Lambda}(e^{u\theta}L^2)$ of the projection on the closure of $\operatorname{Ext}$, now we must check that under the natural inclusion $e^{u\theta}L^2\subset e^{u'\theta}L^2$ if $u<u'$, these dimensions are preserved. This is done at once in fact the inclusion $\operatorname{Ext}(D^{\pm})\subset e^{u\theta}L^2 \hookrightarrow   \operatorname{Ext}(D^{\pm})\subset e^{u'\theta}L^2$ is bounded and extends to a bounded mapping 
$$\overline{\operatorname{Ext}(D^{\pm})}^{ e^{u\theta}L^2} \longrightarrow   \overline{\operatorname{Ext}(D^{\pm})}^{ e^{u'\theta}L^2}$$ with dense range. Now the unitary part of its polar decomposition is an unitary isomorphism then the $\Lambda$ dimensions are the same by 1. in \ref{formaldimension}.
\end{oss}
\begin{dfn}\label{170}
The $\Lambda$--dimension of the space of extended solution is $$\operatorname{dim}_{\Lambda}\operatorname{Ext}(D^{\pm}):=\operatorname{dim}_{\Lambda} \overline{\operatorname{Ext}(D^{\pm})}^{e^{u\theta}L^2}$$ for some $u>0$.
\end{dfn}
\begin{prop}\label{hj}
\begin{enumerate}
\item $\lim_{\epsilon \downarrow 0} \operatorname{dim}_{\Lambda}\operatorname{Ker}_{L^2}(D_{\epsilon}^{\pm})=\operatorname{dim}_{\Lambda}\operatorname{Ker}_{L^2}(D^{\pm})$
\item $\lim_{\epsilon\downarrow 0} \operatorname{Ind}_{L^2,\Lambda}D^+_{\epsilon}=\operatorname{Ind}_{L^2,\Lambda}D^+$
\item $\lim_{\epsilon \downarrow 0} \operatorname{dim}_{\Lambda}\operatorname{Ext}(D_{\epsilon}^{\pm})=\operatorname{dim}_{\Lambda}\operatorname{Ext}(D^{\pm})$
\end{enumerate}
\end{prop}
\begin{proof}
\begin{enumerate}
\item Let $\xi \in \operatorname{Ker}_{L^2}(D^{+}_{\epsilon,x})$ thanks to Proposition \ref{rappacon} $$\xi_{Z_x}=e^{-r \defox+\theta(r)\defox \Pi-{\epsilon,x}}h,\,\, h\in \chi_{(\epsilon,\ty)}(\defox)$$  from $\Pi_{\epsilon,x}h=0$ we get
\begin{align}\nonumber
D^{+}_x\xi_{|Z_x}&=(D^+_{\epsilon,x}+\theta(r)D^{\mathcal{F}_{\partial}}_x\Pi_{\epsilon,x})\xi_{|Z_x}=\theta(r)\defox \Pi_{\epsilon,x}(\xi_{|Z_x})\\ \nonumber &= \theta(r) \defox \Pi_{\epsilon,x}(e^{-r\defox+\theta(r)\defox \Pi_{\epsilon,x}}h)=0
\end{align} meaning that $\operatorname{Ker}_{L^2}(D^+_{\epsilon,x})\subset \operatorname{Ker}_{L^2}(D^+).$
Moreover 
\begin{align}\nonumber
D^{+}_{\epsilon}&(\operatorname{Ker}_{L^2}(D^+))\\ \nonumber&= 
\dotto \defox \Pi_{\epsilon,x}(\operatorname{Ker}_{L^2}(D^+)\subset -\dotto \defox e^{-r \defox} \chi_{(-\epsilon,\epsilon)}(\defox)(L^2(\partial L_x \otimes L^2(\R^+)). 
\end{align}Note that clearly $\operatorname{dim}_{\Lambda}\Big{[} \dotto \defox e^{-r \defox} \chi_{(-\epsilon,\epsilon)}(\defox)(L^2(\partial L_x \otimes L^2(\R^+))\Big{]}\longrightarrow_{\epsilon \rightarrow 0 } 0$ by the normality of the trace.
 Then the family of operators $${D^{+}_{\epsilon}}_{|\operatorname{Ker}_{L^2}(D^+)}:\operatorname{Ker}_{L^2}(D^+)\longrightarrow L^2$$ has kernel  $\operatorname{Ker}_{L^2}(D^+_{\epsilon,x})$ and range with $\Lambda$ dimension going to zero, 1. follows by looking at an hortogonal decomposition 
 $\operatorname{Ker}_{L^2}(D^+)=\operatorname{Ker}_{L^2}(D^+_{\epsilon})\oplus \operatorname{Ker}_{L^2}(D^+)/\operatorname{Ker}_{L^2}(D_{\epsilon}^+).$
 \item Follows immediately from $1.$
 \item Consider the following commutative diagram
$$\xymatrix{\operatorname{Ker}_{e^{\delta\theta}L^2}(D^+)\ar[r]\ar[dr]^{\Psi^{+}_{\epsilon}} &\operatorname{Ker}_{e^{(\delta+\epsilon)\theta}L^2}(D^+)\\ 
&  \operatorname{Ker}_{e^{\delta\theta}L^2}(D^+_{\epsilon})\ar[u]^{\Psi^{-}_{\epsilon}}        }$$ where 
$\Psi^{\pm}_{\epsilon}=e^{\pm\theta \Pi_{\epsilon}D^{\mathcal{F}_{\partial}}}$. It is easily seen  thanks to the representation of solutions in proposition \ref{rappacon} that each arrow is injective and bounded with respect to the inclusions 
$$\xymatrix{
e^{\delta\theta}L^2\ar[r]\ar[dr] &e^{(\delta+2\epsilon)\theta}L^2\\
& e^{(\delta+\epsilon)\theta}L^2\ar[u]}.$$ Then joining together the two diagrams,
$$\xymatrix{
\operatorname{Ker}_{e^{\delta\theta}L^2}(D^+)\ar[r]\ar[dr]^{\Psi^+_{\epsilon}} & \operatorname{Ker}_{e^{(\delta+\epsilon)\theta}L^2}(D^+)\ar[r] &
e^{(\delta+2\epsilon)\theta}L^2\ar[d]\\
 & \operatorname{Ker}_{e^{\delta\theta}L^2}(D^+_{\epsilon})\ar[u]^{\Psi^-_{\epsilon}}\ar[r]&e^{(\delta+\epsilon)\theta}L^2
}$$ and using the last column to measure dimensions one gets the inequality
$$\operatorname{dim}_{\Lambda}\operatorname{Ker}_{e^{\delta\theta}L^2}(D^+)\leq \operatorname{dim}_{\Lambda}\operatorname{Ker}_{e^{\delta\theta}L^2}(D^+_{\epsilon})\leq
\operatorname{dim}_{\Lambda}\operatorname{Ker}_{e^{(\delta+\epsilon)\theta}L^2}(D^+)$$ from which $3.$ immediately follows.
  \end{enumerate}
\end{proof}

\section{Cylindrical finite propagation speed and Cheeger Gromov Taylor type estimates.}
\subsection{The standard case}
\noindent A very important property of the Dirac operator on a manifold of bounded geometry $X$ is finite propagation speed for the associated wave equation.
Let $P\in \udif^1(X,E)$ uniformly elliptic first order (formally) self--adjoint operator.
\begin{dfn}
The diffusion speed of $P$ in $x$ is the norm of the principal symbol $$\sup_{v\in S^*_x}|\sigma_{\textrm{pr}}(P)(x)|$$ ($S_x^*)$ is the fibre of cosphere bundle at $x$). Taking the supremum on $x$ in $M$ one gets the \underline{maximal diffusion speed} $c=c(P)$.

\noindent We say that an operator has \underline{finite propagation speed} if its maximal diffusion speed is finite.
\end{dfn}
\begin{oss}
A (generalized) Dirac operator associated to bounded geometry datas (manifold and clifford structure) has finite propagation speed in fact its principal symbol is Clifford multiplication.
\end{oss}
\noindent The starting point is an application of the spectral theorem to show that for every initial data $\xi_0 \in C^{\infty}_c(X,E)$ there is a unique solution $t\mapsto \xi(t)$ of the Cauchy problem for the wave equation associated with $P$,
\begin{equation}\label{waveproblem} 
\left\{\begin{array}{l} {\partial \xi}/\partial t-iP\xi=0, \\
\xi(0)=\xi_0,
\end{array}
\right.
\end{equation}
this solution is given by the application of the one parameter group of unitaries $\xi(t)=e^{itP}\xi_0$. By the Stone theorem the domain of $P$ is invariant under each unitary $e^{itP}$ and $e^{itP}$is bounded from each Sobolev space $H^s$ into itself. In particular the domain of $P$ is invariant under each unitary $e^{itP}$.
\begin{lem}
For $\theta$ suitably small and $x\in M$, $\|\xi(t)\|_{L^{2}B(x,\theta-ct)}$ is decreasing in $t$. In particular 
$\textrm{supp}(\xi_0)\subset B(x,r) \Longmapsto \textrm{supp}(e^{itP}\xi_0)\subset B(x,r+ct). $
\end{lem}
\begin{proof}
The proof is in J. Roe's book \cite{Roel} Prop. 5.5 and lemma 5.1. Next we shall prove something similar in the cylindrical end. First one proves that for a small geodesic ball of radius $r$ the function $\|e^{itP}\xi_0\|_{L^2(B(x,r-ct))}$ is decreasing. This is called \underline{energy estimate}; then the second step follows easily.
\end{proof}
\noindent Finite propagation speed techniques provide us with the construction of a \underline{functional calculus}; a morphism of algebras 
$\mathcal{S}(\R)\longrightarrow B(L^2(X,E)), \, f\longmapsto f(P)$ with properties
\begin{itemize}
\item Continuity, $\|f(P)\|\leq \textrm{sup} |f|$ hence it can be extended to $C_0(\R)$, the space of continuous functions vanishing at infinity.
\item If $f(x)=xg(x)$ then $f(P)=Pg(P)$.
\item We have the representation formula in terms of the inverse Fourier transform 
\begin{equation}\label{fourier rep}f(P)=\int_{\R}\hat{f}(t)e^{itP}{dt}/{2\pi},\end{equation} here ${\hat{\cdot}}$ is Fourier transform and the integral converges in the weak operator topology, namely
$\langle f(P)x,y\rangle=\int  \hat{f}(t)\langle e^{itP}x,y\rangle {dt}/{2\pi},$ for every $x,y\in L^2(X;E)$. If $X=S^1$ this is just Poisson summation formula.
\end{itemize}
The representation \eqref{fourier rep} gives indeed further informations, as an example we recount how John Roe, using ideas contained in \cite{ChGrTa} used to build a pseudodifferential calculus.

\noindent Let $S^m(\R)$ be the space of symbols of order $\leq m$ on the real line i.e. smooth functions such that $|f^{\lambda}(k)|\leq C_k(1+|\lambda|)^{m-k}$. This is a Fr\'echet space with best constants $C_k$ as seminorms and $\mathcal{S}(R)=\bigcap S^m(\R)$. 

\noindent Roe proves in \cite{Roe1} that for a bounded geometry Dirac operator $D$ every spectral function $f(D)$ with $f$ a symbol of order $\leq m$ is a uniform pseudodifferential operator of order $m$. The proof of this fact uses formula \eqref{fourier rep} together with a convolution smoothing technique.

\bigskip
\noindent Now formula \eqref{fourier rep} leads us to an easy method to obtain pointwise extimates of the Schwartz kernel $[f(P)]$ for a class Schwartz function $f$. In fact due to the ellipticity of $P$, $f(P)$ is a uniformly smoothing operator and $[f(P)] \in UC^{\infty}(X\times X;\operatorname{End}(E))$ (see the appendix \ref{oper}) here we have used the Riemannian density to remove the density coefficient in the Schwartz kernels.
\begin{prop}
Take some section $\xi \in L^2(X;E)$ supported into a geodesic ball $B(x,r)$
then the following estimate holds true
\begin{equation}\label{base}
\|f(P)\xi\|_{L^2(X- B(x,R))}\leq (2\pi)^{-1/2}\|\xi\|_{L^2(X)}\int_{\R-I_R}|\hat{f}(s)|ds,\end{equation}
where $I_R:=(-\frac{r-R}{c},\frac{r-R}{c})$ with the convention that $I_R=\emptyset$ if $R\leq r.$
\end{prop}
\begin{proof}
 From the finite propagation speed \begin{equation}\label{suppp}\operatorname{supp}(e^{itP}\xi)\subset B(x,r+c|t|).\end{equation} From the identity \eqref{fourier rep},
\begin{align}\nonumber
\|f(P)\xi\|_{L^2(X- B(x,R))}=&\Bigg{\|}(2\pi)^{-1/2}\int_{\R}\hat{f}(s)e^{isP}\xi ds\Bigg{\|}_{L^2(X- B(x,R))}
\\  &\nonumber
\leq \Bigg{\|}(2\pi)^{-1/2}\int_{\R- I_R}\hat{f}(s)e^{isP}\xi ds\Bigg{\|}_{L^2(X)}\\ &\nonumber
 \leq (2\pi)^{-1/2}\|\xi\|_{L^2(X)}\int_{\R-I_R}|\hat{f}(s)|ds
\end{align} where $I_R:=(-\frac{r-R}{c},\frac{r-R}{c})$ with the convention that $I_R=\emptyset$ if $R\leq r$. In fact $$\|f(P)\xi\|^2_{L^2(X-B(x,R))}=(2\pi)^{-1}\int_R|\hat{f}(s)|^2\|e^{isP}\xi\|^2_{L^2(X-B(x,R))}ds$$ and the function $s\longmapsto \|e^{isP}\xi\|^2_{L^2(X-B(x,R))}$ is zero if $|s|<\frac{r-R}{c}$ from \eqref{suppp}.
\end{proof}
\noindent So the point of view is the following;
\begin{enumerate}
\item Mapping properties of $f(D)$ will lead to pointwise estimates on the Schwartz kernel of $f(D)$ \cite{ChGrTa}. More precisely; start with a compactly supported section $s$, suppose we can extimate the $L^2$ norm of the image 
$f(D)s$ on a small ball $B$ at some distance $d$ from the support, then by elliptic regularity (G\"arding inequality) and Sobolev embeddings we can extimate the kernel $[f(D)]$ pointwisely.
\item This $L^2$ norm, $\|f(D)s\|_{L^2(B)}$ is extimated in terms the $L^1$ norm of the Fourier transform $\|\hat{f}\|_{L^1(\R)}$. As $d$ increases we can cut large and large intervals around zero in $\R$. This means that the relevant norm becomes $\|\hat{f}\|_{L^1(\R-I_d)}$ where $I_d$ is an interval containing zero.
 The limit case of this phenomenon says that spectral functions made by functions with compactly supported Fourier transforms will produce \underline{properly supported operators} i.e. operators whose kernel lies within a $\delta$--neighborhood of the diagonal. For a good application of finite propagation speed in Foliations one can look at the paper \cite{Roeff} where is showen that spectral functions $f(D)$ where $f$ has compactly supported Fourier transform have some special properties (they belong to the $C^*$--algebra of the foliation).
\end{enumerate}

\noindent Estimate \eqref{base} is the starting point. The following proposition shows how to work out pointwise estimates on the kernel from this mapping properties. This is a very rough version of the ideas contained in \cite{ChGrTa}
\begin{prop}\label{est}
Let $r_1>0$ sufficiently small, $x,y\in X$ put $$R(x,y):=\max\{0,d(x,y)-r_1\}$$ and $\bar{n}:=[n/2+1]$, $n=\operatorname{dim}X$, $I(x,y):=(-R(x,y)/c,R(x,y)/c)$. For a Schwartz class function $f\in \mathcal{S}(\R)$
\begin{equation}\label{schw}\Big{|}\nabla_x^l\nabla_y^k[f(P)]_{(x,y)}\Big{|}\leq \mathcal{C}(P,l,k,r_1)\sum_{j=0}^{2\bar{n}+l+k}\int_{\R-I(x,y)}\Big{|}\hat{f}^{(j)}(s)\Big{|}ds.\end{equation}
\end{prop}
\begin{proof}
$|\nablal \nablak [f(P)]_{(x_0,y_0)}|
\leq C_0 \| \nablal \kerk \|_{H^{\bar{n}+k}(B(y_0,r_1/3))}$ where 
$C_0$ is the constant\footnote{if preferable one can suppose $B(y_0,r_1)$ a geodesic ball and multiply by a cut off $\varphi$ supported within distance $r_1/3$ from $y_0$ and use the global Sobolev embedding. In that case the constant depends on $\varphi$ but using normal coordinates $\varphi$ can be used well for each $y_0$} 
of the Sobolev embedding $H^{\bar{n}+k}(B(y_0,r_1/3))\longrightarrow UC^{k}(B(y_0,r_1/3))$ applied to the function $\nablal \kerk.$ 

\noindent Then we have to apply the G{\aa}rding inequality of $P$ 
\begin{align}\nonumber \|\nablal \kerk \|_{H^{\bar{n}+k}(B(y_0,r_1/3))}\leq C_1 \sum_{j=0}^{\bar{n}+k}\| \nablal P_y \kerk \|_{L^2(B(y_0,r_1/2))}\\ \nonumber=C_1 \sum_{j=0}^{\bar{n}+k}\|\nablal [f(P)P]_{(x_0,\bullet)}\|_{L^2(B(x_0,r_1/2))}\end{align}
 in fact by self adjointness $P_y[f(P)]_{(x_0,\bullet)}=[f(P)P]_{(x_0,\bullet)}.$ No problem here in \emph{localizing} the G{\aa}rding inequality we can choose in fact for each $y_0$ a function $\chi$ supported in $B(y_0,r_1)$ with $\chi_{|B(y_0,r_1/2)}=1$. Then since the coefficients of $P$ in normal coordinates are uniformly bounded,  each $[P,\chi]$ is uniformly bounded.
 Let $\xi_j(y):=\chi_{B(y_0,r_1/2)}(y)\nablal [P^jf(P)]_{(x_0,y)}$ the inequality becomes $|\nablal \nablak [f(P)]_{(x_0,y_0)}|
\leq C_0 C_1 \sum_{j=0}^{\bar{n}+k}\|\xi_j\|_{L^2(X)}.$
 
 \noindent Now \begin{align}\nonumber \|\xi_j\|^2_{L^2(X)}=\int \chi_{B(y_0,r_1/2)}\nabla [P^jf(P)]_{(x_0,\bullet)}\xi_j(y)dy=|(\nablal P^jf(P)\xi_j)(x_0)|\\ \nonumber
 \leq C_2\|P^jf(P)\xi_j\|_{H^{\bar{n}+l}(B(x_0,r_1/3))}\leq C_2C_3 \sum_{i=0}^{\bar{n}+l}\|P^{j+i}f(P)\xi_j\|_{L^2(B(x_0,r_1/2))}
 \end{align}again by Sobolev embedding and G{\aa}rding inequality. The choice to keep every constant is motivated to control their dependence in order to apply these extimates leaf by leaf.
 
 \noindent Finally putting everything together
 \begin{align}\nonumber
 |\nablal \nablak [f(P)]_{(x_0,y_0)}|&\leq \mathcal{C}\sum_{j=0}^{\bar{n}+k}\sum_{i=0}^{\bar{n}+l}\|P^{j+i}f(P)\|_{L^2(B(x_0,r_1/2)),L^2(B(y_0,r_{1/2}))} \\ \nonumber
& \underbrace{\leq \mathcal{C}  }_{\eqref{base}}
 \sum_{j=0}^{2\bar{n}+l+k}\int_{\R-I(x_0,y_0)}|\hat{f}^{(j)}(s)|ds
 \end{align}
\end{proof}
\noindent For the \underline{heat} \underline{kernel} $[f(P)]=[e^{-tP^2}]$ when $f(x)=e^{-tx^2}$, $\hat{f}(s)=(2t)^{-1/2}e^{-s^2/4t}$,
\begin{align}\nonumber \hat{f}(s)^{(k)}=\dfrac{1}{(2t)^{1/2}(4t)^{k/2}}((4t)^{1/2}\partial_s)^ke^{-\Big{(}\dfrac{s}{(4t)^{1/2}}\Big{)}^2}
\\ \nonumber=
\dfrac{C(k)}{t^{(k+1)/2}}H_k\Big{(}\dfrac{s}{(4t)^{1/2}}\Big{)}e^{-\Big{(}\dfrac{s}{(4t)^{1/2}}\Big{)}^2},\end{align}
where $H_k$ is the $k$--th Hermite polynomial. Then using the simple inequalities 
\begin{align}\nonumber &\int_u^{\ty}e^{-x}dx\leq e^{-u^2},\,y^se^{ay^2}\leq \Big{(}\dfrac{s}{2ae}\Big{)}^{s/2},\,a,s,u,y \in \R^+,\\ \nonumber &
\int_{u}^{\ty}y^se^{-y^2}dy=\int_u^{\ty}y^se^{-\epsilon y^2}e^{-(1-\epsilon)y^2}dy\leq C(s,\epsilon)e^{-(1-\epsilon)u^2}
\end{align}
with $R=R(x,y)$ and $\eta=2\bar{n}+l+k$
\begin{align}
\label{heats}
|\nablal \nablak [P^me^{-tP^2}]_{(x,y)}| &\leq \mathcal{C}\sum_{j=m}^{\eta+n}t^{-j/2}\int_{R/c}^{\ty}\Big{|}H_j\Big{(}\dfrac{s}{(4t)^{1/2}}\Big{)}\Big{|}e^{-\Big{(}\dfrac{s}{(4t)^{1/2}}\Big{)}^2}(4t)^{-1/2}ds \\ \nonumber &
\leq \mathcal{C}\sum_{j=m}^{\eta+m}t^{-j/2}
\int_{R/2c\sqrt{t}}^{\ty}|H_j(x)|e^{-x^2}dx \\ 
\nonumber & \leq \mathcal{C}\sum_{j=m}^{\eta+m}t^{-j/2}\int_{R/2c \sqrt{t}}^{\ty}(1+x^j)e^{-x^2}dx\\ 
\nonumber & \leq \mathcal{C}e^{-R^2/5c^2 t}\sum_{j=m}^{\lambda+m}t^{-j/2}\\ \nonumber &  
\leq 
\Bigg{\{}{\begin{array}{c}\mathcal{C}(k,l,m,P)t^{-m/2}e^{-R^2/6c^2t},\,\,\,t>T \\   
\mathcal{C}(k,l,m,P)e^{R^2/6c^2t},\,\,\,d(x,y)>2r_1
\end{array}t\in \R^{+}.}
\end{align}
There's also a relative version of Proposition \ref{est} in which two \underline{differential}, formally self--adjoint uniformly elliptic operators $P_1$ and $P_2$ are considered. More precisely relative means that $P_1$ acts on $E_1\longrightarrow X_1$ and $P_2$ acts on $E_2\longrightarrow X_2$ with open sets $U_1 \subset X_1$, $U_2\subset X_2$ and isometries $\varphi,\Phi$ 
$$
\xymatrix{{E_1}_{|U_1}\ar[d]\ar[r]^{\Phi}& {E_2}_{|U_2}\ar[d]  \\
U_1 \ar[r]^{\varphi} &U_2}
$$ making possible to identify $P_1$ with $P_2$ upon $U=U_1=U_2$ i.e.
$$\Phi( P_1 s)=P_2( \Phi s),\quad s\in C^{\infty}_c(U_1;E_1)$$ where $\Phi$ is again used to denote the mapping induced on sections $$\Phi:C^{\ty}_c(U_1;E_1)\longrightarrow C^{\ty}_c(U_2;E_2),\,(\Phi s)(y):=\Phi_{\varphi^{-1}(y)}s(\varphi^{-1}(y)).$$ Thanks to the identification one calls $P=P_1=P_2$ over $U$.
\noindent Then the relative version of the estimate \eqref{schw} is contained in the following proposition.
\begin{prop}\label{asd}
Choose $r_2>0$ and let \underline{$x,y$ be in $U$}. Set
$$Q(x,y):=\max\{\min \{d(x,\partial U);d(y,\partial U) \}-r_2;0\},\,\, J(x,y):=\Big{(}\dfrac{-Q(x,y)}{c},\dfrac{Q(x,y)}{c}\Big{)}.$$
For $f\in \mathcal{S}(\R)$,
$$|\nablal \nablak ([f(P_1)]-[f(P_2)])_{(x,y)}|\leq \mathcal{C}(P_1,k,l,r_2)\sum_{j=0}^{2\bar{n}+l+k}\int_{\R-J(x,y)}|\hat{f}^{(j)}(s)|ds. $$ More precisely the reason of the dependence of the constant only to $P_1$ is that it depends upon ${P_1}_{|U}$ where the operators coincide.
\end{prop}
\begin{proof}
This is very similar to the proof of \ref{base}. Choose $x_0,y_0\in U$ then
\begin{align}
|\nablal \nablak([f(P_1)]-[f(P_2)])_{(x_0,y_0)}|\leq {C} \| \nablal([f(P_1)]-[f(P_2)]_{(x_0,\bullet)}\|_{H^{\bar{n}+k}(B(y_0,r_2/3))}\\\nonumber 
\leq C \sum_{j=0}^{\bar{n}+k}\|\nablal([P_1^jf(P_1)]-[P_2^jf(P_2)])_{(x_0,\bullet)}  \|_{L^2(B(y_0,r_2/2))}.\end{align} Where the first step is Sobolev embedding $H^{\bar{n}+k}\longrightarrow UC^k$, again no problem in reducing the Sobolev norm to be computed on the ball $B(y_0,r_2/3)$ in fact one can suppose $r_2$ is smaller than the injectivity radius and build a cut off function $\chi$. The Sobolev embedding is applied then to the section
$\chi \nabla_x^k[f(P_1)-f(P_2)]$ and the resulting constant ${C}$ will be depending also on $\chi$ but uniform geometry assumption makes $\chi$ universal in that can be used on each normal coordinate.
For example for the Sobolev order one the argument one applies is
$$\|\nabla_y \chi t\|_{H^1}\leq \| (\nabla_y \chi)t\|_{L^2}+\|\chi \nabla_y t\|_{L^2}\leq {D}(\chi,1)\|t\|_{H^1(B(y_0,r_2/3))}$$
if $\chi$ is supported in $B(y_0,r_2/3)$.

The second step is G{\aa}rding inequality of $P_1$ and $P_2$ together with the fact that they coincide on $U_1$. The same argument with a cut off function $\chi_2$ also works well with G{\aa}rding inequality. 
\noindent Let $\xi_j(y):=\chi_{B(y_0,r_2/2)}(y)\nablal \{[P_1^jf(P_1)]_{(x_0,y)}-[P_2^jf(P_2)]_{(x_0,y)}\}$ then
\begin{align}
\|\xi_j\|^2_{L^2(B(y_0,r_2/2))}=&|(\nablal (P_1^jf(P_1)-P_2^jf(P_2))\xi_j)(x_0)|\\
\nonumber &\leq C\|P_1^jf(P_1)-P_2^jf(P_2)\xi_j\|_{H^{\bar{n}+l}(B(x_0,r_2/3))}\\ \nonumber & 
\leq C\sum_{i=0}^{\bar{n}+l}\|P_1^{j+i}f(P_1)-P_2^{j+i}f(P_2)\|_{L^2(B(x_0,r_2/2))}
\\ \nonumber &\leq  C\|\xi_j\|_{L^2(U)}\sum_{i=0}^{\bar{n}+l}\int_{\R-J(x,y)}|\hat{f}^i(s)|ds
\end{align}in fact for a class Schwartz function $g$,
\begin{align}\nonumber
\|(g(P_1)-g(P_2))\xi_j\|_{L^2(B(x_0,r_2/2))}=&\Big{\|}(2\pi)^{-1/2}\int_{\R}\hat{g}(s)(e^{isP_1}-e^{isP_2})\xi_j ds\Big{\|}_{L^2(B(x_0,r_2/2))}\\ \nonumber
=&\Big{\|}(2\pi)^{-1/2}\int_{\R-J(x_0,y_0)}\hat{g}(s)(e^{isP_1}-e^{isP_2})\xi_j ds \Big{\|}_{L^2(B(x_0,r_2/2))}
\end{align}\noindent since $\operatorname{supp}(e^{iP_is}\xi_j)\subset B(y_0,r_2/2+c|s|)$ then $e^{isP_1 \xi_j}$ and $e^{isP_2 \xi_j}$ remain supported in $U$ then $e^{isP_1}\xi_j=e^{isP_2}\xi_j$ by the uniqueness of the solution of the Cauchy problem for the wave equation.
\end{proof}
\begin{prop}\label{rel12}
\noindent The relative version of \eqref{heats} is
\begin{align*}
|\nablal \nablak ([P_1^me^{-tP_1^2}]-[P_2^m &e^{-tP_2^2}])_{(x,y)}|  \nonumber \leq \Bigg{\{}{\begin{array}{c}\mathcal{C}(k,l,m,P_1)t^{-m/2}e^{-Q(x,y)^2/6c^2t},\,t>T \\   
\mathcal{C}(k,l,m,P_1)e^{-Q(x,y)^2/6c^2t},\,\,\,
\end{array}.}
\end{align*}
for $x,y\in U$, $d(x,\partial U)$, $d(y,\partial U)>r_2$ and $t\in \R^{+}$.
\end{prop}
\subsection{The cylindrical case}
In this section our manifold $L$ will be the generic leaf of the foliation i.e.  start with a manifold with bounded geometry $L_0$ with boundary $\partial L_0$ composed of possibly infinite connected components and a product type Riemannian metric near the boundary. Glue an infinite cylinder $Z_0=\partial L_0 \times [0,\infty)$ with product metric and denote $L:= L_0 \cup_{\partial L_0} Z_0$. Let $E\longrightarrow L$ be an Hermitian Clifford bundle. Every notation of section \ref{geom} is keeped on with the slight abuse that $Z_0$ is the cylinder here and in $X$. Recall that $E_{|Z_0}=F\oplus F$.
\begin{dfn}
We say that a first order uniformly elliptic (formally) selfadjoint operator $T\in \operatorname{Op}^1(L;E)$ has product structure if
\begin{enumerate}
\item $T$ restricts to $L_0$ and $Z_0$ i.e. $\operatorname{supp}(T s)\subset L_0 (Z_0)$ if $s$ is supported on $L_0\,(Z_0).$
\item $T_{|L_0}$ is a uniformly elliptic \underline{differential} operator.
\item $T$ restricts to the cylinder to have the form 
$$T_{|Z_0}=c(\partial_r)\partial_r+\Omega B(r)=\left(\begin{array}{cc}0 & B(r)-\partial_r  \\B(r)+\partial(r) & 0 \end{array}\right)$$
 for a smooth\footnote{
 Some words about the smoothness condition on the mapping $B$. Here we shall make use only of pseudodifferential operators with uniformly bounded symbols, (almost everywhere they will be smoothing operators) hence the smoothness condition of the family is the usual one. In particular this is the smoothness of the family of operators acting on the fibers of $\partial L_0\times \R^+\longrightarrow \R^+$, $B(t)\in \operatorname{Op}^1(\partial L_0\times \{t\};E)$. If $U$ is a coordinate set for $\partial L_0$ such a family is determined by a smooth mapping $p:\R^+\longrightarrow \operatorname{S}_{\operatorname{hom}}^1(U)$ in the space of polihomogeneous symbols. Here smooth means that each derivative $t\longmapsto {d^k \sigma}/{dt^k}$ is continuous as a mapping with values in the space of symbols (with the symbols topology, see \cite{Va})} mapping $B:\R^+\longrightarrow \operatorname{Op}^1(\partial L_0;E)$ with values on the subspace of uniformly elliptic and selfadjoint operators. Furthermore suppose that $B(r)\cong B$ is constant for $r\geq 2$.
\end{enumerate} 
\end{dfn}
\noindent  However this is only a model embracing our Breuer--Fredholm perturbation of the Dirac operator in fact 
\begin{equation}\label{asc}(D_{\epsilon,u,x})_{|\partial_x \times \R^+}=c(\partial_r)\partial_r+\Omega\underbrace{(\dot{\theta}u-\dot{\theta}D^{\mathcal{F}_{\partial}}\Pi_{\epsilon}+D^{\mathcal{F}_{\partial}})}_{B(r)}.\end{equation} In this sense every result from here to the end of the section has to be thought applied to $D_{\epsilon,u}.$

\noindent Again the spectral theorem shows that for a compactly supported section $\xi_0\in C^{\ty}_c(L;E)$ there is a unique solution $t\mapsto \xi(t)$ of the Cauchy problem \eqref{waveproblem} for the wave equation associated with $T$. This solution is
given by the application of the wave one parameter group $e^{itT}$ 
with the same properties written above in the standard case.
\begin{prop}\label{59}{\bf{ Cylindrical finite propagation speed.}} Let $U=\partial L_0\times (a,b)$ $0<a<b$ and $B(U,l)=\{x\in L:d(x,U)<l\}$. For $\xi_0\in C^{\ty}_c(L;E)$ let $\xi(t)=e^{itT}\xi_0$ the solution of the wave equation. If $\alpha<a$ the function $\|\xi(t)\|_{L^2(B(U,\alpha-t))}$ is not increasing in $t$. In particular $$\operatorname{supp}(\xi_0)\subset U \Longmapsto \operatorname{supp}(\xi(t))\subset B(U,t).$$
\end{prop}
\begin{proof}
The product structure of the operator makes us possible to repeat the standard proof of the energy estimates and finite propagation speed that can be found in John Roe's book \cite{Roel}. So let us consider
\begin{align}\nonumber
\dfrac{d}{dt} \|\xi(t)\|^2&_{L^2(B(U,\alpha-t))}=\dfrac{d}{dt}\int_{B(U,\alpha-t)}|\xi(t)|^2(z)dz 
\\\leq \nonumber &
\Bigg{|}\int_{B(U,\alpha-t)}\Big{(}\langle \xi(t),i T\xi(t)\rangle+\langle i T \xi(t),\xi(t)\rangle\Big{)} (z)dz\Bigg{|}-\int_{\partial B(U,\alpha-t)}|\xi(t)|^2(z)dz.\end{align} Since the operator $T$ has product structure, the integration domain is a product and the operator $B(t)$ is selfadjoint on the base $$\int_{B(U,\alpha-t)}\langle \xi(t),iT\xi(t)\rangle+\langle i T \xi(t),\xi(t) \rangle dz =\int_{B(U,\alpha-t)}\langle \xi(t),ic(\partial_r)\partial_r\xi(t)\rangle+\langle i c(\partial_r)\partial_r \xi(t),\xi(t) \rangle dz.$$
Here the fact that the function $$r\longmapsto \int_{\partial L_0}\Big{(}\langle i\Omega B(r)\xi(t)_{|\partial L_0 \times \{r\}},\xi(t)_{|\partial L_0 \times \{r\}}\rangle + \langle \xi(t)_{|\partial L_0 \times \{r\}},i\Omega B(r) \xi(t)_{|\partial L_0 \times \{r\}} \rangle \Big{)}(x)dx$$ is identically zero by the self--adjointness of $B$ has been used. Note that $\xi(t)_{|\partial L_0 \times \{r\}}$ is in the domain of $B(r)$ by the theorem (however it is certainly true for operators in the form of our perturbation \eqref{asc}).
\noindent Finally 
\begin{align}\nonumber &\dfrac{d}{dt} \|\xi(t)\|^2_{L^2(B(U,\alpha-t))}
\leq 
\Bigg{|}\int_{B(U,\alpha-t)}\Big{(}\langle \xi(t),i c(\partial_r)\partial_r \xi(t)\rangle
+ \langle i c(\partial_r)\partial_r\xi(t),i c(\partial_r)\partial_r \xi(t)\rangle (z)\Big{)}dz \Bigg{|}
\\
\nonumber
&
-\int_{\partial B(U,\alpha-t)}|\xi(t)|^2(z)=
\Bigg{|}\int_{B(U,\alpha-t)}
\partial_r \langle \xi(t),c(\partial_r)\xi(t) \rangle (z)dz \Bigg{|}
  -\int_{\partial B(U,\alpha-t)}|\xi(t)|^2(z)\leq 0
\end{align}
\end{proof}
\noindent As a notation for a subset $H\in L$ and $t\geq0$ put $H\ast t:=B(H,t)\cup \partial L_0\times (\alpha-t,\beta+t)$ where 
$\alpha:=\inf\{r(z):z\in H\cap Z_0\}$ and 
$\beta:=\max\{r(z):z\in H\cap Z_0\}$ in other words $H\ast t$ is the set of points at distance $t$ from $H$ in the cylindrical direction.

\noindent It is clear from \eqref{59} that the support of the solution of the wave problem satisfies $$\operatorname{supp}(e^{itT}\xi)\subset \operatorname{supp}(\xi)\ast |t|.$$ Then the cylindrical basic Cheeger--Gromov--Taylor estimate similar to \eqref{base} is obtained in the following way:

\noindent first note that proposition \ref{59} is certainly true if the propagation speed is $c$, for a section $\xi$ supported into a ball $B(x,r_0)$ and $f\in \mathcal{S}(\R)$ let
$I_R:=(-(R-r_0)/c,(R-r_0)/c)$ if $R>r_0$ and $I_R=\emptyset$ if $r\leq R$ then, \begin{align}
\|f(P)\xi\|_{L^2(L-B(x,r_0)\ast R)}=&\Big{\|}(2\pi)^{-1/2}\int_{\R}\hat{f}(s)e^{isP}\xi ds\Big{\|}_{L^2(L-B(x,r_0)\ast R)}\\ \nonumber &
\leq \Big{\|}(2\pi)^{-1/2}\int_{\R-I_R}\hat{f}(s)e^{isP}\xi ds \Big{\|}_{L^2(L)}\\ \label{11} &\leq (2\pi)^{-1/2}\|\xi\|_{L^2(L)}\int_{\R-I_R}|\hat{f}|ds,
\end{align} since $\operatorname{supp}{e^{isP}\xi}\cap (L-B\ast R)=\emptyset$ for $|t|<(R-r_0)/c$.
\begin{prop}Choose two points on the cylinder $z_1=(x_1,s_1)$ and $z_2=(x_2,s_2)$ with $s_i>r_1$, $|s_1-s_2|>2r_1$, put $I(z_1,z_2):=\Big{(}-\dfrac{|s_1-s_2|+r_1}{c},\dfrac{|s_1-s_2|-r_1}{c}\Big{)}$ then for $f\in \mathcal{S}(\R)$,
$$|\nabla_{z_1}^l\nabla_{z_2}^k[f(P)]_{(z_1.z_2)}|\leq \mathcal{C}(P,l,k)\sum_{j=0}^{2\bar{n}+l+k}\int_{R-I(z_1,z_2)}|\hat{f}(s)^{(j)}|ds$$ with $\bar{n}:=[n/2+1]$
\end{prop}
\begin{proof}

Imitate the proof of \ref{est} till the estimate
$$|\nablal \nablak [f(P)]_{(x,y)}|\leq C\sum_{j=0}^{\bar{n}+k}\|\xi_j\|_{L^2(L)}$$
where $\xi_j:=\chi_{B(y,r_1=2)}\nablal[P^jf(P)]_{(x,\bullet)}$ and $x,y \in L$ .

\noindent There is a subtle point we need to reckle, it is when one let $P^j$ act on $[f(P)]_{(x,\bullet)}$. This is perfectly granted by the smoothing properties of $f(P)$ in fact, let the bundle be $L\times \R$ and identify distributions with functions through the Riemannian density.
The operator $f(P)$ extends to and operator from compactly supported distributions to distributions (actually takes values on smooth functions).
 Consider the family of Dirac masses $\delta_y(\cdot)$ concentrated at $y$,
first note that \begin{equation}\label{nucleodelta}[f(P)]_{(x,y)}=(f(P)\delta_y(\cdot))(x)\end{equation} in fact by selfadjointness $$\langle f(P)\delta_y,s\rangle=\langle \delta_y,f(P)s\rangle=\int [f(P)]_{(z,y)}t(z)dz,$$ that's to say \eqref{nucleodelta}. Now the Sobolev embedding theorem says that $\delta_y\in H^k(X)$ with $k<-n/2$ with norms uniformly (in $y$) bounded. Since $f(P)$ maps every Sobolev space into each other Sobolev space, every section $[f(P)]_{(x,\bullet)}$ (and the symmetric one by selfadjointness) is in the domain of $P^j$.

\noindent Again \begin{align}\nonumber
\|\xi_j\|^2_{L^2(L)}&=\|\chi_{B(y,r_1/2)}\nablal [P^jf(P)]_{(x,\bullet)}\|^2_{L^2(B(y,r_1/2))}\\\nonumber &=|\nablal P^jf(P)\xi_j(x)|\leq C\|P^jf(P)\xi_j\|_{H^{\bar{n}+l}B(x,r_1/3)}\\  \label{ultimas}& \leq C\sum_{i=0}^{\bar{n}+l}\|P^jf(P)\xi_j\|_{L^2(B(x,r_1/2))}.
\end{align}
\noindent It's time to move on to the cylindrical end, so let $x=(x_2,s_2)$, $y=(x_1,s_1)$ with $s_i>r_1$ and $|s_1-s_2|>2r_1$, then last term in \eqref{ultimas} can be estimated by $$\sum_{i=0}^{\bar{n}+l}\|P^{j+i}f(P)\xi_j\|_{L^2(V)}$$ with $V=L-B(y,r_1/2)\ast c(|s_1-s_2|-r_1)/2$ so we can conclude by application of \eqref{11}.
\end{proof}
\begin{cor}
With the notations of the proposition above
\begin{enumerate}
\item If $|s_1-s_2|>2r_1$, $s_i>r_1$
\begin{equation}\label{stil}|\nabla_{z_1}^l\nabla_{z_2}^k [P^me^{-tP^2}]_{(z_1,z_2)}|\leq \mathcal{C}(k,l,m,P)e^{-\dfrac{(|s_1-s_2|-r_1)^2}{6t}}\end{equation}
\item Let $\psi_1,\psi_2$ compactly supported with supports at $r$--distance $d$ on the cylinder, then for the operator norm and $t>0$
\begin{equation}\label{1333}
\|\psi_1 P^me^{-tP^2}\psi_2\|\leq \mathcal{C}(m,\psi_1,\psi_2)e^{-d^2/6t}.
\end{equation}
\item The relative version of \eqref{stil} is 
\begin{equation}\label{asdl}
|\nabla_{z_1}^l\nabla_{z_2}^k [P^me^{-tP^2}-T^me^{-tT^2}]_{(z_1,z_2)}|\leq \mathcal{C}(k,l,m,P)e^{\{{-(\min\{s_1,s_2\}-r_2)^2/6t}\}}.\end{equation}
\end{enumerate}
\end{cor}
\begin{proof}
The second statement follows immediately from the first one while the third can be proven exactly in the way proposition \ref{asd} is proven.
\end{proof}
\section{The eta invariant}\label{eta}
\subsection{The classical eta invariant} The eta invariant of Atiyah Patodi and Singer appears for the first time in the following theorem that we write in the cylindrical case.
\begin{thm}
\noindent Let $X$ a compact manifold with boundary $Y$ and product type metric on a collar $Y\times [0,1]$, attach an infinite cylinder $Y\times [-\infty,0]$ to get the elongated manifold $\hat{X}:=X\cup Y\times [-\infty,0]$. Let $D:C^{\infty}(X;E)\longrightarrow C^{\infty}(X;F)$ a first order differential elliptic operator with product structure near the boundary i.e.
$$D=\sigma(\partial_u+A)$$ where $\sigma E_{|Y}\longrightarrow F_{|Y}E$ is a bundle isomorphism, $\partial_u$ is the normal interior coordinate and $A$ is the boundary self--adjoint elliptic operator. Then the operator $D$ extends to sections of the bundles extended to $\hat{X}$ and has a finite $L^2$ index i.e the space of $L^2$ solutions of the equations $Ds=0$ and $D^*s=0$ are finite dimensional and $$\operatorname{ind}(D)=\operatorname{dim}_{L^2(\hat{X},E)}(D)-\operatorname{dim}_{L^2(\hat{X},E)}(D^*)=\int_X\alpha_0(x)dx-\eta(0)/2-\dfrac{h_{\ty}(E)-h_{\ty}(F)}{2}$$ where
\begin{enumerate}
\item $h_{\ty}(E)$ is the dimension of the space of limiting values of the extended $L^2$ solutions. More precisely one says that $s$ is an $L^2$ extended solution of the equation $Ds=0$ with limiting value $s_{\ty}$ if $s$ is locally square integrable and for large $u<0$
$$s(y,u)=g(y,u)+s_{\infty}(y),\, s_{\ty}(y)\in \operatorname{Ker}(A).$$ Similar definition for $h_{\ty}(F)$.
\item $\alpha_0(x)$ is the constant term in the asymptotic expansion as $t\rightarrow 0$ of
\begin{equation}\label{asieta}\operatorname{tr}\Big(e^{-tD^-D^+}\Big{)}-\operatorname{tr}\Big(e^{-tD^+D^-}\Big{)}
=\sum e^{-t\mu'}|\phi'_{\mu}(x)|^2-\sum e^{-t\mu^{''}}|\phi^{''}_{\mu}(x)|^2=\end{equation} where $\mu',\phi'_{\mu}$ are the eigenvalues and eigenfunctions of $D^*D$ on the double of $X$ and $\mu^{''},\phi^{''}_{\mu}$  are the corresponding objects for $DD^*$.
\item The number $\eta(0)$, is called the \underline{spectral asymmetry} or the \underline{eta invariant} of $A$ and is obtained as follows:

\noindent the summation on the non negative eigenvalues of $A$, $$\eta(s):=\sum_{\lambda\neq 0}\operatorname{sign}(\lambda)|\lambda|^{-s}$$ converges absolutely for $\operatorname{Re}(s)>>0$ extends to a meromorphic function on the whole $s$--plane with regular value at $s=0$. Moreover if the asymptotic expansion at \eqref{asieta} has no negative powers of $t$ then $\eta(s)$ is holomorphic for $\operatorname{Re}(s)>-1/2$. That's the case of the Dirac operator of a Riemannian manifold.
\end{enumerate}
\end{thm}
\subsection{The foliation case}
\noindent The existence of the eta invariant for the leafwise Dirac operator on a closed foliated manifold was shown by Peric \cite{Peric} and Ramachandran \cite{Rama}. In fact they build different invariants, Peric works with the holonomy groupoid of the foliation and Ramachandran with the equivalence relation but the methods are essentially the same. So consider a compact manifold $Y$ with a foliation and a longitudinal Dirac structure i.e. every geometrical structure needed to form a longitudinal Dirac--type operator acting on the tangentially smooth sections of the bundle $S$, $D:C^{\ty}_{\tau}(Y;S)\longrightarrow (Y;S)$. In our index formula $Y$ will be a transverse section of the cylinder sufficiently far from the compact piece and $D$ is the operator at infinity. Suppose also that a transverse holonomy invariant measure $\Lambda$ is fixed.

\noindent Here the first issue to solve is to pass from the summation $\eta(s)=\sum_{\lambda}\operatorname{sign}(\lambda)|\lambda|^{-s}$ which deals with the discrete spectrum to a continuous spectrum and family version. The link is offered by the definition of Euler gamma function 
$$\operatorname{sign}(\lambda)|\lambda|^{-s}=\dfrac{1}{\Gamma(\frac{s+1}{2})}\int_0^{\ty}t^{\frac{s-1}{2}}\lambda e^{-t\lambda^2}dt.$$ Each bounded spectral function of $D$ belongs to the Von Neumann algebra of the foliation arising from the regular representation of the equivalence relation on the Borel field of $L^2$ spaces of sections of $S$. 
Replace the summation by integration w.r.t. the spectral measure of $D$ (definition \ref{misuraspettrale}) and (formally) change the integration to define the eta function of $D$ as
\begin{align}\label{etafun}
\eta_{\Lambda}(D;s):=\int_{-\infty}^{\infty}\operatorname{sign}(\lambda)|\lambda|^{-s}d\mu_{D}(\lambda)=\dfrac{1}{\Gamma(\frac{s+1}{2})}\int_0^{\ty}t^{\frac{s-1}{2}}\operatorname{tr}_{\Lambda}(De^{-tD^2})dt.
\end{align}
\noindent We shall use also the notation $$\eta_{\Lambda}(D;s)_k:=\int_k^{\ty}t^{\frac{s-1}{2}}\operatorname{tr}_{\Lambda}(De^{-tD^2})dt,\,\,\,\eta_{\Lambda}(D;s)^k:=\int_0^{k}t^{\frac{s-1}{2}}\operatorname{tr}_{\Lambda}(De^{-tD^2})dt$$

\begin{thm}(Ramachandran)
The eta function \eqref{etafun} is a well defined meromorphic function for $\operatorname{Re}(s)\leq 0$ with eventually simple poles at $(\operatorname{dim}{\mathcal{F}}-k)/2$, $k=0,1,2,....$, $\eta_{\Lambda}(D;s)$ is regular at $0$ and its value $\eta_{\Lambda}(D;0)$ is called the foliated eta invariant of $D$.
\end{thm}
\begin{proof}
Here a sketch of the proof.
\begin{itemize}
\item[First step.]
For every $s\in \mathbb{C}$ with $\Re(s)\leq 0$ the integral \begin{equation}\label{eta1}\int_1^{\infty}t^{\frac{s-1}{2}}\operatorname{tr}_{\Lambda}(De^{-tD^2})dt\end{equation} is convergent then in some sense the most important piece of the eta function is the integral $\int_0^1$. 

\noindent This is reminescent of the remark in the paper of Atiyah Patodi and Singer \cite{AtPaSi1} where they define the function $K(t)$ to be the integral on the cylinder of the difference of the heat kernels $e^{-t\Delta_1}-e^{-t\Delta_2}$ of $D$ and $D^*$,
$$K(t)=\int_0^{\ty}\int_{\partial Y}K(t,y,u)dydu=-\sum_{\lambda}\operatorname{sign}(\lambda)/2\operatorname{erf}(|\lambda|\sqrt{t})\sim_{t\rightarrow 0}\sum_{k\geq -n}a_kt^{k/2}$$ where $\partial Y$ is the boundary manifold of dimension $n$. The remark they do is that the asymptotic expansion is the same replacing the integral with an integral on $\int_{[0,\delta]}$.

\noindent The convergence of \eqref{eta1} is proven by simple estimates and the use of the spectral measure. In particular here, by compactness the spectral measure $\mu_{\Lambda,D}$ is tempered i.e. $$\int \dfrac{1}{(1+|x|^l)}d\mu_{\Lambda,D}<\ty$$ for some positive $l$. In fact this measure corresponds to 
a positive functional  \cite{Rama} $$I:\mathcal{S}(\R)\longrightarrow \R,\, I(f):=\operatorname{tr}_{\Lambda}(f(D)).$$ 
The same is obviously true for the square $D^2=|D|^2.$  
Start with $|t^{(s-1)/2}|\leq t^{\operatorname{(Re(s)-1)/2}}\leq t^{-1/2},\,0\leq t\leq \ty$ then
\begin{align}\nonumber
\int_1^{\ty}|t^{(s-1)/2}\operatorname{tr}_{\Lambda}(De^{-tD^2})dt| &\leq 
\int_1^{\ty}|t^{-1/2}|\operatorname{tr}_{\Lambda}(De^{-tD^2})dt| \\ \nonumber &\leq \int_1^{\ty}t^{(s-1)/2}\operatorname{tr}_{\Lambda}(|D|e^{-tD^2})dt. 
\end{align}
\noindent The last integral is equal to
$$\int_1^{\ty}t^{-1/2}dt\int_0^{\ty}\lambda^{1/2}e^{-t\lambda}d\mu_{D^2}(\lambda)$$
hence
\begin{align}\label{integralone}
\int_0^{\ty}\lambda^{1/2}d\mu_{D^2}(\lambda)\int_1^{\ty}t^{-1/2}e^{-t\lambda}dt 
 &=\int_0^{\ty}\lambda^{1/2}e^{-\lambda}d\mu_{D^2}(\lambda)\int_1^{\ty}t^{-1/2}e^{-\lambda(t-1)}dt\\
\nonumber& =\int_0^{\ty}e^{-\lambda}d\mu_{D^2}(\lambda)\int_0^{\ty}(u+\lambda)^{-1/2}e^{-u}du \\ \nonumber &
 \leq \int_0^{\ty}e^{-\lambda}d\mu_{D^2}(\lambda)\int_0^{\ty}u^{-1/2}e^{-u}du \\
\nonumber &=\pi^{1/2}\int_0^{\ty}e^{-\lambda}d\mu_{D^2}(\lambda)=\pi^{1/2}\operatorname{tr}_{\Lambda}(e^{-D^2})<\ty
\end{align}

\item[Second step.]
The examination of the finite piece \begin{equation}\label{eta2}\int_0^{1}t^{\frac{s-1}{2}}\operatorname{tr}_{\Lambda}(De^{-tD^2})dt\end{equation}
is done using the expansion of the Schwartz kernel of the leafwise operator $De^{-tD^2}$ in fact one can prove that there exists a family of tangentially smooth and locally computable functions $\{\Psi_m\}_{m\geq 0}$ \footnote{in the case of the holonomy groupoid the $\Psi_m$ are locally bounded i.e. bounded on every set in the form of $r^{-1}K$ for $K$ compact in $Y$} so that the kernel $K_{t}(x,y,n)$ ($n$ the transverse parameter) of the leafwise bounded operator $De^{-tD^2}$ has the asymptotic expansion
\begin{equation}\label{scwe}K_{t}(x,x,n)\sim \sum_{m\geq 0}t^{(m-\operatorname{dim}\mathcal{F}-1)/2}\Psi_m(x,n).
\end{equation} 
Moreover $\Psi_m=0$ for $m$ even. The proof is an adaptation of the classical situation, for example can be found in \cite{Roe1} and \cite{Cos}. Now, thanks to the expansion \eqref{scwe}, since the operator $De^{-tD^2}$ is $\Lambda$ trace class and the trace is the integral of the Schwartz kernel against the transverse measure we get the corresponding expansion for the trace
\begin{equation}\label{23}\dfrac{1}{\Gamma(\frac{s+1}{2})}\int_0^1 t^{\frac{s-1}{2}}\operatorname
{tr}_{\Lambda}(De^{-tD^2})dt\sim \sum_{m\geq 0}\dfrac{2}{s+m-\operatorname{dim}{\mathcal{F}}}\int_Y \Psi_m d\lambda\end{equation}
where $\int \Psi_m d\lambda=\Lambda(\Psi_m dg)$ i.e. is the effect of the integration of the tangential measures 
$x\longmapsto {\Psi_m}_{|l_x}\times dg_{|l_x}$. From \eqref{23} we see that the eta function has a meromorphic continuation to the whole plane with simple (at most) poles at $(\operatorname{dim}{\mathcal{F}}-k)/2, \,\,k=0,1,2,....$
\item[Third step,] regularity at the origin.

 \noindent If $P=\operatorname{dim}{\mathcal F}$ is even we have said that the coefficients $\Psi_m$ of the development \eqref{scwe} are zero for $m$ even, then the eta function is regular at $0$. If $p$ is odd the regularity at zero follows from a very deep result of Bismut and Freed \cite{Bifree}. In fact they showed that the ordinary Dirac operator satisfies a remarkable cancellation property,
 $$\operatorname{tr}(De^{-tD^2})=O(t^{1/2}).$$ Since the $\Lambda$--trace can be, as pointed out by Connes (\cite{Cos}), locally approximated by the regular trace their result applies to our setting to give
 $$K_t(x,x,n)\sim \sum_{m\geq p+2}\underbrace{t^{(m-p-1)/2}\Psi(x,n)}_{\textrm{almost everywhere}},$$ and the regularity at the origin follows immediately.
 \end{itemize}
\end{proof}

\subsection{Eta invariant for perturbations of the Dirac operator}
\noindent Let 
\noindent Let us consider slightly more general operators
\begin{enumerate}
\item $P=D+K$ where $K\in \operatorname{Op^{-\ty}}$ is leafwise uniformly smoothing obtained by functional calculus, $K=f(D)$ where $f$ is bounded Borel function supported in  $(-a,a)$.

\noindent Start with the computation
\begin{align}\label{pa}
Qe^{-tQ^2}-De^{-tD^2}&=De^{-t(D+K)^2}+Ke^{-t(D+K)^2}-De^{-tD^2}
\\\nonumber &=D\int_0^1\partial_s e^{-s(D+K)^2-(t-s)D^2}dt+Ke^{-t(D+K)^2}\\&=
\nonumber Ke^{-t(D+K)^2}-D\int_0^1 e^{-s(D+K)^2}(KD+DK+K^2)e^{(t-s)D^2}ds.
\end{align} The family \eqref{pa} converges to $0$ as $t\rightarrow 0$ in the Frechet topology of kernels in $\operatorname{Op}^{-\ty}$ with uniform transverse control i.e. for kernels $K(x,y,n)$ ($n$ is the transverse parameter) one uses foliated charts to define seminorms that involve derivatives w.r.t. $x,y$. From \eqref{pa} one gets the development
$$\operatorname{tr}_{\Lambda}(Qe^{-tQ^2})\sim_{t\rightarrow 0}
\sum_{m=0}t^{\frac{m-\operatorname{dim}\mathcal{F}-1}{2}}\int_Y \Psi_jd\Lambda +\operatorname{tr}_{\Lambda}(K)+g(t)$$ where $g\in C[0,\ty)$ with $g(0)=0$. Then an asymptotic development for $\eta_{\Lambda}(Q)(0)_1$ as \eqref{23} follows. For the non finite integral $\eta_{\Lambda}(Q,0)^1$ no problem in carrying out the estimate \eqref{integralone}.
\item The smooth family $u\longmapsto Q_u:=D+K+u$. The function $\operatorname{tr}_{\Lambda}(Q_ue^{-tQ_u^2})$ is smooth ( same identical proof as \cite{vai}) then 
\begin{align}
\partial_u \operatorname{tr}_{\Lambda}(Q_u e^{-tQ_u^2})&=\operatorname{tr}_{\Lambda}(Q^{'}_u e^{-tQ_u^2}-tQ_u(Q^{'}_uQ_u+Q_uQ^{'}_u)e^{-tQ_u^2})\\
\nonumber &=(1+2t\partial_t)\operatorname{tr}_{\Lambda}(Q^{'}_u e^{-tQ_u^2})
\end{align} in fact $Q_u^{'}=\operatorname{I}.$ By integration 
\begin{align}\nonumber
\partial_u \eta_{\Lambda}(Q_u,s)_1&=\partial_u \int_0^1 \dfrac{t^{(s-1)/2}}{\Gamma(\frac{s+1}{2})}\operatorname{tr}_{\Lambda}(Q_ue^{-tQ_u^2})dt=
\int_0^1 \dfrac{t^{(s-1)/2}}{\Gamma(\frac{s+1}{2})}(1+2t\partial_t)\operatorname{tr}_{\Lambda}(Q^{'}_ue^{-tQ_u^2})dt \\\label{322} &= \int_0^1 \dfrac{t^{(s-1)/2}}{\Gamma{(\frac{s+1}{2})}}\operatorname{tr}_{\Lambda}(Q_u^{'}e^{-Q_u^2})-\dfrac{s}{\Gamma(\frac{s+1}{2})}\int_0^1 t^{(s-1)/2}\operatorname{tr}_{\Lambda}(Q_u^{'}e^{-tQ_u^2})dt.
\end{align}\noindent Now, from $Q_u^{'}=\operatorname{I}$ proceed as before using the asymptotic development of the heat kernel for $D+u$ \footnote{$(D+u)^2$ is a generalized Laplacian}
$$\operatorname{tr}_{\Lambda}(Q_u^{'}e^{-tQ_u^2}=\operatorname{tr}_{\Lambda}(Q_u^{'}e^{-tQ_u^2}\sim \sum_{m\geq 0}a_m(D+u)t^{(m-\operatorname{dim}\mathcal{F})/2}+g(t)$$ where $g\in C[0,\ty)$, $g(0)=0.$
\noindent We see that the integral in \eqref{322} admits a meromorphic expansion around zero in $\mathbb{C}$ with zero as a pole of almost first order. Then the derivative $\partial_u \eta_{\Lambda}(Q_u,s)_1$ is holomorphic around zero. The identity
$$\partial_u \operatorname{Res}_{|s=0} \eta_{\Lambda}(Q_u,s)_1=
 \operatorname{Res}_{|s=0}\partial_u \eta_{\Lambda}(Q_u,s)_1=0$$ says that $\operatorname{Res}_{|s=0} \eta_{\Lambda}(Q_u,s)_1$ is constant in $u$ then the function $\eta_{\Lambda}(Q_u,s)_1$ is holomorphic at zero since
 $\eta_{\Lambda}(Q_0,s)_1$ is holomorphic in $0$.
 \item Families in the form $Q_u=D+u+\Pi D$ for a spectral projection $\Pi=\chi_{(-a,a)}(D)$.
\begin{prop}\label{0348}The eta invariant for $Q_u$ exists and satisfies
 $$\eta_{\Lambda}(Q_u)=\operatorname{LIM}_{\delta \rightarrow 0}\int_{\delta}^{1}\dfrac{t^{-1/2}}{\gamma(1/2)}\operatorname{tr}_{\Lambda}(Q_ue^{-tQ_u^2})dt+\int_{1}^{\ty}\dfrac{t^{-1/2}}{\gamma(1/2)}\operatorname{tr}_{\Lambda}(Q_ue^{-tQ_u^2})dt$$ where $\operatorname{LIM}$ is the constant term in the asimptotic development in powers of $\delta$ as $t\rightarrow 0$. Moreover for every $u\in \R$ and $a>0$,
 
 \noindent a. $\eta_{\Lambda}(Q_u)-\eta_{\Lambda}(Q_0)=\operatorname{sign}(u)\operatorname{tr}_{\Lambda}(\Pi)$
 
 \noindent b. $\eta_{\Lambda}(Q_0)=1/2\eta_{\Lambda}(Q_u)+1/2\eta_{\Lambda}(Q_{-u})$
 
 \noindent c. $|\eta_{\Lambda}(D)-\eta_{\Lambda}(Q_0)|=|\eta_{\Lambda}(\Pi D)|\leq \mu_{\Lambda,D}((-a,a))$.
 \end{prop}
\end{enumerate}
\begin{proof}The first statement can be proved as above.
\noindent 
a. using the spectral measure we have to compute the difference  
\begin{align}\nonumber \int_0^{\ty}t^{-1/2}\int_{\R}(x+u-\chi x)e^{-t(x+u-\chi x)^2}&d\mu_{\Lambda,D}(x)\dfrac{dt}{\Gamma(1/2)}\\ \nonumber-&\int_0^{\ty}t^{-1/2}\int_{\R}(x-\chi x)e^{-t(x-\chi x)^2}d\mu_{\Lambda,D}(x)\dfrac{dt}{\Gamma(1/2)}\end{align} where $\chi=\chi_{(-a,a)}(x).$ Split the integral on $\R$ into two pieces, $|x|> a$ and $|x|\leq a$.

\noindent First case $|x|> a$ changing the integration order the first integral is $$\Gamma(1/2)^{-1}\int_{|x|>a}\int_0^{\ty}(x+u)t^{-1/2}e^{-t(x+u)^2}
dt d\mu_{\Lambda,D}(x)$$ and performing the substitution $\sigma:=t(x+u)^2$ in the second we see that the difference is zero.

\noindent Second case $|x|<a$, the second integral is zero, the first
\begin{align}&\nonumber
\int_0^{\ty}t^{-1/2}\int_{-a}^{a}ue^{-tu^2}d\mu_{\Lambda,D}(x)\dfrac{dt}{\Gamma(1/2)}=\int_0^{\ty}t^{-1/2}ue^{-tu^2}\dfrac{dt}{\Gamma(1/2)}\operatorname{tr}_{\Lambda}(\Pi)\\&
=\underbrace{\noindent \int_0^{\ty}u|u|\sigma^{-1/2}e^{-\sigma}\dfrac{d\sigma}{|u|^2}}_{tu^2=\sigma}\dfrac{\operatorname{tr}_{\Lambda}(\Pi)}{\Gamma(1/2)}=
\operatorname{sign}(u)\dfrac{\operatorname{tr}_{\Lambda}(\Pi)}{\Gamma(1/2)}\int_0^{\ty}\sigma^{-1/2}e^{-\sigma^2}d\sigma
\\ 
\nonumber& =\operatorname{sign}(u)\operatorname{tr}_{\Lambda}(\Pi).
\end{align}b. and c. follows easily from a.
\end{proof} 

\section{The index formula}\label{qww}
First we introduce the \underline{supertrace} notation. Since the bundle $E=E^+\oplus E^-$ is $\mathbb{Z}_2$--graded, there is a canonical  Random operator $\tau$ obtained by passing to the $\Lambda$--class of the family of involutions $\tau_x:L^2(L_x;E)\longrightarrow L^2(L_x;E)$ represented w.r.t. the splitting by matrices $$\tau_x:=\left(\begin{array}{cc}\operatorname{Id}_{L^2(L_x;E^+)} & 0 \\0 & -\operatorname{Id}_{L^2(L_x;E^-)}\end{array}\right).$$
 \begin{dfn}The \underline{$\Lambda$--{supertrace}} of $B\in \operatorname{End}_{\Lambda}(E)$ is
 $\operatorname{str}_{\Lambda}(B):=\operatorname{tr}_{\Lambda}(\tau B).$
 \end{dfn}
\bigskip

\noindent Now according to proposition \ref{312} for  $\eppu$ the perturbed operator $\deu$ is $\Lambda$--Breuer--Fredholm. 
Consider the heat operator $\hdeupx$ on the leaf $L_x$. This is a uniformly smoothing operator with a Schwartz kernel (remember that the metric trivializes densities and $[\bullet]$ means Schwartz kernel) $$[\hdeupx]\in UC^{\infty}(L_x\times L_x;\operatorname{End}(E)).$$ It is a well know fact the convergence for $t\longrightarrow \infty$ in the Frechet space of $UC^{\infty}$ sections of the heat kernel to the kernel of the projection on the $L^2$--Kernel,
$$\lim_{t\rightarrow \infty}[\hdeupx]=[\chi_{\{0\}}(\deux)].$$ This is explained in proposition \ref{regu} and is a consequence of continuity of the functional calculus $RB(\R)\longrightarrow UC^{\infty}(\operatorname{End(E)})$ applied to the sequence of functions $e^{-t\lambda^{2}}\longrightarrow \chi_{\{0\}}$ in $RB(\R).$
 Choose cut--off functions $\phi_k\in C_c^{\infty}(X)$ such that ${\phi_k}_{|X_k}=1$, ${\phi_k}_{|Z_{k+1}}=0.$ 
 \noindent The measurable family of bounded operators $\{\pk \hdeupx \pk\}_{x\in X}$ gives an intertwining operator $\pk \hdeut \pk\in \operatorname{End}_{\mathcal{R}}(L^2(E))$ hence a random operator $\pk \hdeut \pk\in \operatorname{End}_{\Lambda}(L^2(E)).$

\begin{lem}
The random operator $\pk \hdeut \pk\in \operatorname{End}_{\Lambda}(L^2(E))$ is $\Lambda$--trace class. The following formula (iterated limit) holds true 
\begin{equation}\label{limita}\indu(\deupp)=\stru(\chi_{\{0\}}(\deu))=
\lim_{k\rightarrow \infty}\lim_{t\rightarrow \infty}\stru(\pk \hdeut \pk).
\end{equation}
\end{lem}
\begin{proof}For the first statement there's nothing to proof, it is essentially the closed foliated manifold case. The local traces define a tangential measure that are $C^{\infty}$ in the leaves direction while Borel and uniformely bounded (by the uniform ellipticity of the operator) and we are integrating against the transverse measure on a compact set. More precisely we are evaluating the mass of a compact set through the measure $\Lambda_{h}$ where $h$ is the longitudinal measure that on the leaf $L_x$ is given by
$$A\longmapsto \int_{A}\operatorname{str}_{\operatorname{End}(E)}[e^{-tD^2_{\epsilon,u}}]_{\operatorname{diag}}dg_{|L_x},$$  with $\operatorname{str}_{\operatorname{End}(E)}$ the pointwise supertrace defined on the space of sections of $\operatorname{End}(E)\rightarrow X$ by $(\operatorname{str}_{\operatorname{End}(E)}\gamma)(x):=\operatorname{tr}_{\operatorname{end}(E_x)}(\tau(x)\gamma(x)).$

\noindent The limit formula \eqref{limita} is nothing but the Lebesgue dominated convergence theorem applied two times, first
$\stru(\chi_{\{0\}}(D_{\epsilon,u}))=\lim_{k\rightarrow \infty}
\stru(\pk \chi_{\{0\}}(D_{\epsilon,u})\pk)$ but for fixed $k$ one finds $\stru(\pk \chi_{\{0\}}(D_{\epsilon,u}))\pk)=\lim_{t\rightarrow \infty}\stru (\pk \hdeut \pk).$
The possibility to apply the dominated convergence theorem is given again by the integration process in fact as written above every tangential measure has smooth density w.r.t to the Riemannian metric and convergence is within the Frechet topology of $C^{\infty}$ functions. 

\end{proof}\noindent Now, Duhamel formula $d/dt \stru(\pk \hdeut \pk)=-\stru(\pk D^2_{\epsilon,u}\hdeut \pk)$ integrated between $s$ and $\infty$ leads to the identity
$$\lim_{t\rightarrow \infty}\stru(\pk \hdeut \pk)=
\stru(\pk \hdeups \pk)-\int_s^{\infty}
\stru(\pk D^2_{\epsilon,u}\hdeut\pk)dt.$$ Note that the right--hand side is independent from $s>0$. Then 
\begin{equation}\label{indint}\indu(\deupp)=\lim_{k\rightarrow \infty}
\left[
\stru (\pk \hdeups \pk)-\int_s^{\infty}\stru(\pk D^{2}_{\epsilon,u}\hdeut\pk) dt \right].\end{equation}
Split the integral into 
$$\int_s^{\infty}\stru(\pk D^{2}_{\epsilon,u}\hdeut\pk) dt =
\int_s^{\sqrt{k}}\stru(\pk D^{2}_{\epsilon,u}
\hdeut\pk) dt +\int_{\sqrt{k}}^{\infty}
\stru(\pk D^{2}_{\epsilon,u}\hdeut\pk) dt.$$
\noindent  Make the following definitions
$$\xymatrix{{\alpha}_0(k,s)=\stru(\pk \hdeups \pk), &  \beta_0(k,s)=\int_s^{\infty} 
\stru(\pk D^2_{\epsilon,u}\hdeut\pk)dt\\
\beta_{01}(k,s)=\int_s^{\sqrt k}\stru(\pk D^{2}_{\epsilon,u}
\hdeut\pk) dt,& 
\beta_{02}(k,s)=\int_{\sqrt{k}}^{\infty}\stru(\pk D^{2}_{\epsilon,u}
\hdeut\pk) dt
}$$

\noindent Then $\beta_0(k,s)=\beta_{01}(k,s)+\beta_{02}(k,s)$ and 
\begin{equation}
\label{indf}
\indu(\deupp)=\lim_{k\rightarrow \ty}[\alpha_0(k,s)-\beta_0(k,s)]=[\alpha_0(k,s)-\beta_{01}(k,s)-\beta_{02}(k,s)].\end{equation}
\noindent Let us start with $\beta_{01}$.
\begin{lem}\label{309}Let $\eta_{\Lambda}(\deuf)$ be the Ramachandran eta--invariant for the perturbed operator $\deuf$ on the foliation at the infinity. Then the following limit formula is true
$$\lim_{k\rightarrow \ty}\operatorname{LIM}_{s\rightarrow 0}\beta_{01}(k,s)=\lim_{k\rightarrow \infty}\operatorname{LIM}_{s\rightarrow 0}\int_s^{\sqrt k}\stru(\pk D^{2}_{\epsilon,u}e^{-t\deuq}\pk)ds,=1/2 \eta_{\Lambda}({\deuf})$$\noindent where as usual $\operatorname{LIM}_{s\rightarrow 0}g(s)$ is the constant term in the expansion of $g(s)$ in powers of $s$ near zero.
\end{lem}
\begin{proof}
The integrand can be written as follows
\begin{align}
\stru(\pk \deuss \pk)=& 1/2 \stru(\pk [\deu,\deus]\pk)\\ \label{integrando}
=& 1/2 \stru ([\deu,\pk \deus \pk]-[\deu,\pkd]\deus)\\ \nonumber
=& 1/2 \stru (-[\deu,\pkd]\deus)\\ \nonumber &=-1/2\stru(c(\partial_r)\partial_r(\pkd)\deus).
\end{align}
\noindent In the next we shall use the notation $[a,b]:=ab-(-1)^{|a|\cdot |b|}ba$ for the Lie--superbracket\footnote{everything we say about super--algebras can be found in \cite{BeGeVe}} on the Lie--superalgebra of $\mathbb{C}$--linear endomorphisms of 
$L^2(X,E^+\oplus E^-)$ while, when the standard bracket is needed we write $[a,b]_{\circ}:=ab-ba.$ notice that
$$[\alpha,ab]=[\alpha,a]b+(-1)^{|\alpha|\cdot|a|}a[\alpha,b].$$
\noindent Remember the definition of 
$\deu$, in the cylinder it can be written $$\deu=D+\dot{\theta}\Omega(u-\deuf)=c(\partial_r)\partial_r+Q$$ with the Clifford multiplication $c(\partial_r)=\left(\begin{array}{cc}0 & -1 \\1 & 0\end{array}\right)$ and $Q$ is $\R^+$--invariant in fact acts on the transverse section.
 The next identities are also useful
$$\deu=\left(\begin{array}{cc}0 & \deum \\\deupp & 0\end{array}\right),\quad \hdeut=\left(\begin{array}{cc}\eum & 0 \\0 & \eup\end{array}\right),$$
$$\deum \eup=\eum \deum,\quad \deupp \eum=\eup \deupp.$$
These are nothing but a rephrasing of the identity $$\deus=\hdeut \deu$$ granted by the spectral theorem. 
\noindent Now it's time to use the Cheeger--Gromov--Taylor relative estimates. Consider the leafwise operator 
\begin{equation}\label{sconfronta}S_{\epsilon,u}:=c(\partial_r)\partial_r+\Omega(u-\deuf)\end{equation} on the infinite foliated cylinder (in both directions) $Y=\partial X_0\times \R$ with the product foliation $\mathcal{F}_{\partial}\times \R$.
\noindent Choose some point $z_0=(x_0,r)$ on the cylinder. Estimate \eqref{asdl} says that we can compare the two kernels at the diagonal leaf by leaf for large $r$ and this estimate is uniform on the leaves,
\begin{equation}\label{cgt} \|[\deussxo]-[\dessxo]\|_{(z,z)}\leq Ce^{-(r-r_2)^2/(6t)}\end{equation}
for $\underline{z=(x,r)\in L_{z_0}}$. From \eqref{cgt}, since the derivatives of $\pk$ are supported on the cylindrical portion $Z_k^{k+1}=\partial X_0\times [k,k+1]$,
\begin{align}\nonumber
\sk |\stru&(\clib  \deus)-\stru(\clib \esm)|dt=
\int_s^{\sqrt{k}}\int_{Z_k^{k+1}}\Theta(z,t)d\Lambda_g dt
\end{align} 
where $\Lambda_g$ is the coupling of $\Lambda$ with the tangential  Riemannian measure and $\Theta(z,r)$ is the function
$$\Theta(z,r):=\|c(\partial_r)\partial_r \phi^2_k[D_{\epsilon,u,z}e^{-tD^2_{\epsilon,u,z}}-S_{\epsilon,u,z}e^{-tS^2_{\epsilon,u,z}}]\|_{(z,z)}.$$
Let $\mathcal{T}_k$ be a transversal of the foliation $\mathcal{F}_k$ induced on the slice $\{r=k\}$ then $\mathcal{T}_k$ is also transversal for $\mathcal{F}$  (since the boundary foliation has the same codimension of $\mathcal{F}$). The transverse measure $\Lambda$ defines also a transverse measure on the boundary foliation. Then the foliation $\mathcal{F}_{|Z_k^{k+1}}$ is fibering on $\mathcal{T}_k$ as in the diagram  $\partial{\mathcal{F}}\times[k,k+1] \longrightarrow \mathcal{T}_k$. Use this fibration to disintegrate the measure $\Lambda_{g}$. This is splitted into $d\Lambda_{\partial}\times dr$ where $\Lambda_{\partial}$ is the measure obtained applying the integration process of $\Lambda$ (restricted to $\mathcal{F}_k$ ) to the $g_{|\partial}$. In local coordinates $(r,x_1,...,x_{2p-1})\times (x_{2p},...,x_n)$ the transversal is decomposed into pieces $\mathcal{T}_k=\{(k,x_1^0,...,x_{2p-1}^0)\}\times \{(x_{2p},...,x_n)\}$ and we are taking integrals 
\begin{align}\label{itegralesplitt}\int_{\mathcal{T}_k\times\{x_1,...,x_{2p-1}\}}\int_{[k,k+1]} \Theta(r,x_1,...,x_{2p-1},x_{2p},...,x_n)dr \underbrace{dx_1\cdot\cdot \cdot dx_{2p-1}  d\Lambda(x_{2p},..,x_n)}_{\textrm{this is }d\Lambda_{\partial}}\\
\nonumber =:\int_{\mathcal{F}_k}\int_{[k,k+1]}\Theta(x,r)d\Lambda_{\partial}dr. \end{align}
\noindent Equation \eqref{itegralesplitt} can be taken as a definition of a notation that will be used next. Notice that $\int_{\mathcal{F}_k}$ contains a slight abuse of notation, in fact to follow rigorously the integration recipe 
one should write $\int_{\partial X_0\times \{k\}}$. We prefer the first to stress the fact that we are splitting 
w.r.t the foliation induced on the transversal.
With this notation in mind,
\begin{align}\nonumber
\sk |\stru&(\clib  \deus)-\stru(\clib \esm)|dt\\ \label{hjf}= &
\sk \int_{\mathcal{F}_z}\int_{[k,k+1]}\|\clib [\deus-\esm]\|_{((x,r),(x,r))}dr d\Lambda_{\partial} dt \\& \nonumber \leq C\int_s^{\sqrt{k}}\int_k^{k+1}e^{-(r-3)^2/6t}dr dt \leq C\int_s^{\sqrt{k}}e^{-(k-3)^2/6t}dt\\ \nonumber & \leq C\int_{1/\sqrt{k}}^{1/s}y^{-2}e^{-(k-3)^2y/6}dy 
\leq C(e^{-k^{3/2}/c_1}+e^{-c_2/s})
\end{align}
for sufficiently small\footnote{$\,\,y^s e^{-ay^2} \leq (\dfrac{s}{2ae})^{s/2}$ for $s,u,y,a>0$} $s$  and large $k$. This estimate says that 
$$\lim_{k\rightarrow +\ty}\operatorname{LIM}_{s\rightarrow 0}\beta_{01}(k,s)=\lim_{k\rightarrow +\ty}\operatorname{LIM}_{s\rightarrow 0}\int_s^{\sqrt k}\stru(\clib \esm)dt.$$

\noindent Now the second integral (on the cylinder) is explicitly computable
  in fact the Schwartz kernel of the operator $\dessxo$ on the diagonal is easily checked to be
\begin{align*}
\big[&\dessxo\big]_{(z,z)}\\=&
\big(\deufo\Omega+c(\partial_r)\partial_r\big)\Bigg(\big[e^{-t(\deufo\Omega)^2}\big]_{(x,y)}\dfrac{e^{-(r-s)^2/(4t)}}{\sqrt{4\pi t}}\Bigg)\Bigg{|}_{y=x,\,s=r}\\=&
\dfrac{1}{\sqrt{4\pi t}}\Omega\big[\deufo e^{-t\deufo}\big]_{(x,x)},\,\,\,z=(x,r)
\end{align*}i.e. it does not depend on the cylindrical coordinate $r$.
\noindent Now the pointwise supertrace on $\operatorname{End}(E)$ is related to the trace on the positive boundary eigenbundle $F$ via the identity (see the appendix on Clifford algebras)
$$\operatorname{str}^E(c(\partial_r)\Omega\bullet)=-2\operatorname{tr}^F(\bullet),$$ then
\begin{align*}
\int_{s}^{\sqrt{k}}\stru (c(\partial_r)\partial_r \pkd& S_{\epsilon,u}e^{-tS^2_{\epsilon,u}})dt\\=&-2\int_s^{\sqrt{k}}\int_k^{k+1} \partial_r \pkd dr \int_{\mathcal{F}_0}\dfrac{1}{\sqrt{4\pi t}}\operatorname{tr}^F[D^{\mathcal{F}_{\partial}}_{\epsilon,u,x}e^{-t (D^{\mathcal{F}_{\partial}}_{\epsilon,u,x} )^2   }]_{(x,x)}\cdot d\Lambda_{\partial} dt &\\=&2\int_s^{\sqrt{k}}\int_{\mathcal{F}_0}\dfrac{1}{\sqrt{4\pi t}}\operatorname{tr}^F[D^{\mathcal{F}_{\partial}}_{\epsilon,u,x}e^{-t (D^{\mathcal{F}_{\partial}}_{\epsilon,u,x} )^2   } ]_{(x,x)}\cdot d\Lambda_{\partial} dt\\=&
\int_s^{\sqrt{k}}\int_{\mathcal{F}_0}\dfrac{1}{\sqrt{\pi t}}\operatorname{tr}^F[D^{\mathcal{F}_{\partial}}_{\epsilon,u,x}e^{-t (D^{\mathcal{F}_{\partial}}_{\epsilon,u,x} )^2   }]_{(x,x)}\cdot d\Lambda_{\partial} dt, \end{align*}
with the same argument on the splitting of measures as above.
\noindent Finally it is clear from our discussion on the $\eta$--invariant 
 (exactly proposition \ref{0348})
\begin{align*}
\lim_{k\rightarrow \ty}&
\operatorname{LIM}_{s\rightarrow 0}\beta_{01}(k,s)\\&= 
\lim_{k\rightarrow \ty}\operatorname{LIM}_{s\rightarrow 0}
\int_s^{\sqrt{k}}\int_{\mathcal{F}_0}\dfrac{1}{\sqrt{\pi t}}\operatorname{tr}^F[D^{\mathcal{F}_{\partial}}_{\epsilon,u,x}e^{-t (D^{\mathcal{F}_{\partial}}_{\epsilon,u,x} )^2   }]_{(x,x)}\cdot d\Lambda_{\partial} dt
=1/2 \eta_{\Lambda}(D^{\mathcal{F}_{\partial}}_{\epsilon,u}).
\end{align*}
\end{proof}
\begin{lem}Since $\deu$ is $\Lambda$--Breuer--Fredholm for $\eppu$ then
$$\lim_{k\rightarrow \ty}\beta_{02}(k,s)=\lim_{k\rightarrow \ty}\int^{\ty}_{\sqrt{k}}\stru(\pk D^{2}_{\epsilon,u}\hdeut \pk)dt=0.$$
\end{lem}
\begin{proof}From the very definition of the $\Lambda$--essential spectrum ( see also lemma \ref{lemshu}) there exists some $\sigma=\sigma(u)>0$ such that the projection $\Pi_{\sigma}=\chi_{[-\sigma,\sigma]}(\deu)$ has finite $\Lambda$--trace. Then
\begin{align*}
|\beta_{02}(k,s)|=&\Big|\int^{\ty}_{\sqrt{k}}\stru(\pk D^{2}_{\epsilon,u}\hdeut \pk)dt \Big| \\ 
 &\leq \int_{\sqrt{k}}^{\infty}|\stru[\pk \deu e^{-\deuq/2}(1-\Pi_{\sigma})e^{-(t-1)\deuq}e^{-\deuq/2}\deu \pk]|dt\\
&+\int_{\sqrt{k}}^{\ty}|\stru[e^{-t\deuq/2}\Pi_{\sigma}\deu \pk^2 \deu \Pi_{\sigma}e^{-t\deuq/2}]|dt
 \\ &\leq \underbrace{{\int_{\sqrt{k}}^{\ty}e^{-(t-1)\sigma
}|\stru(\pk \deuq e^{-\deuq}\pk)|dt}}_{\beta_{021}(k,s)}+\underbrace{\int_{\sqrt{k}}^{\ty}|\stru(\deuq e^{-t\deuq}\Pi_{\sigma})|dt}_{\beta_{022}(k,s)}.
\end{align*}\noindent Now the Schwartz kernel of $(\deuq e^{-\deuq})_x$ is uniformly bounded in $x$ and varies in a Borel fashion transversally. When  forming the  $\Lambda$--supertrace we are integrating a longitudinal measure with $C^{\ty}$--density w.r.t. the longitudinal measure given by the Riemannian density. Let as usual $\Lambda_g$ the measure given by the integration of the Riemannian longitudinal measure with the transverse measure $\Lambda$. If $A$ is a uniform bound on the leafwise Schwartz kernels of $(\deuq e^{-\deuq})$, and $\mathcal{T}_0$ is a complete transversal contained in the normal section of the cylinder (the same in lemma \ref{309}), we can extimate 
 $$\beta_{021}(k,s)\leq \int_{\sqrt{k}}^{\ty}A(\Lambda_g(X_0)+\Lambda(\mathcal{\mathcal{T}_0})k)e^{-(t-1)\sigma}dt\longrightarrow_{k\rightarrow \ty}0. $$ \noindent For the second addendum, 
\begin{align*}\beta_{022}(k,s)=&\int_{\sqrt{k}}^{\ty}|\stru(\deuq e^{-t\deuq}\Pi_{\sigma})|dt\leq \intk\tsi x^2 e^{-tx^2}\dmd dt\\&=\tsi e^{-\sqrt{k}x^2}\int_{0}^{\ty}x^2 e^{-tx^2}dt\dmd \\&\leq C\tsi e^{-\sqrt{k}x^2}\dmd \leq C \mun([-\sigma,\sigma])\longrightarrow_{k\rightarrow \ty} 0 
\end{align*}\noindent since the $\Lambda$--essential spectrum of $D_{\epsilon,u}$ has a gap around zero and the normality property of the trace.
\end{proof}
\noindent It is time to update equation \eqref{indf},
\begin{align}\nonumber \indu(\deupp)=&\lim_{k\rightarrow \ty}[\alpha_0(k,s)-\beta_0(k,s)]=\lim_{k\rightarrow \ty}[\alpha_0(k,s)-\beta_{01}(k,s)-\beta_{02}(k,s)]\\ \label{passagg}=&\lim_{k\rightarrow \ty}\operatorname{LIM}_{s\rightarrow 0}\alpha_0(k,s)-1/2\eta_{\Lambda}(D_{\epsilon,u}^{\mathcal{F}_{\partial}}).
\end{align}
\begin{lem}There exists a function $g(u)$ with $\lim_{u\rightarrow 0}g(u)=0$ such that for $0<\epsilon<u$,
\begin{align*}\lim_{k\rightarrow \ty}\operatorname{LIM}_{s\rightarrow 0}\alpha_0(k,s)=\lim_{k\rightarrow \ty}\operatorname{LIM}_{s\rightarrow 0}\stru(\pk \hdeups \pk)=\langle \hat{A}(X)\operatorname{Ch(E/S)},C_{\Lambda}\rangle +g(u).\end{align*}
\noindent Here the leafwise characteristic form $\hat{A}(X)\operatorname{Ch(E/S)}$ is supported on $X_0$, in particular it belongs to the domain of the Ruelle--Sullivan current $C_{\Lambda}$ associated to the transverse measure $\Lambda.$\end{lem}
\begin{proof}This is the investigation of the behavior of the local supertrace of the family of the leafwise heat kernels $$\operatorname{str}^E[e^{-s D^2_{\epsilon,u}}]_{|\operatorname{diag}}$$ on the leafwise diagonals. We can do it dividing into three separate cases 
\begin{enumerate}
\item For $z\in X_0$ everything goes as in the classical computation by Gilkey \cite{gilkey} and Atiyah Bott and Patodi \cite{AtBo} 
$$\operatorname{LIM}_{s\rightarrow 0}\operatorname{str}^E[e^{-sD^2_{\epsilon,u,z}}]_{(x,x)}dg_z=\hat{A}(X,\nabla)\operatorname{Ch}(E/S,\nabla)(x),$$ where $dg_z$ is the Riemannian density on the leaf $L_z$.
\item In the middle, $z\in \partial X_0 \times [0,4]$ there's the cause of the presence of the defect function $g(u)$, more precisely we show that the asymptotic development of the local supertrace is the same for the comparison operator $S_{0,u}$ defined above 
$$\operatorname{str}^E([e^{-sD^2_{\epsilon,u,z}}])_{(z,z)}\simeq \sum_{j\in \mathbb{N}}a_j(S_{0,u})_{(z)}s^{(j-\operatorname{dim}{\mathcal{F}})/2}$$
 with coefficients 
$a_j(S_{0,u})$ smoothly depending on $u$ satisfying 
$a_j(S_{0,u})=0$ for $j\leq \operatorname{dim}\mathcal{F}/2$ 
\item Away from the base of the cylinder $z=(y,r)\in Z$ $r>4$ we find
$$[e^{-D^2_{\epsilon,u,z}}]_{(y,r)}=0.$$

\end{enumerate}Below the proofs of these facts.
\begin{enumerate}
\item We can consider the doubled manifold $2X_0$ so that we can apply the relative estimate of type Cheeger--Gromov--Taylor in the non--cylindrical case (the perturbation starts from the cylinder)
i.e. proposition \ref{rel12}
 shows that the two Schwartz kernels of the Dirac operator and the perturbed operator $D_{\epsilon,u}$ have the same development as $t\rightarrow 0$,
 $$\|[e^{-tD^2_{\epsilon,u}}-e^{-tD^2}]_{(x,x)}\|\leq K e^{-\alpha/(6t)}.$$
And the local computation of Atiyah Bott and Patodi, or the Getzler rescaling (\cite{Me},\cite{Ge}) can be performed as in the classical situation.

\item We are going to use an argument of comparison with the leafwise operator
$$S_{\epsilon,u}:=c(\partial_r)\partial_r+\Omega(D^{\mathcal{F}_{\partial}}+\dot{\theta}(u-\Pi_{\epsilon}D^{\mathcal{F}_{\partial}}))$$ on the infinite cylinder $\partial X_{0}\times \R$ equipped with the product foliation $\mathcal{F}_{\partial}\times \R$. Notice that, due to the presence of $\dot{\theta}$ this is a slightly different form of the operator \eqref{sconfronta}.  Choose some function $\psi_1$ supported in $\partial X_0\times [-1,5]$ and ${\psi_1}_{|\partial X_0 \times [0,4]}=1$. The first fact we show is
$$\lim_{s\rightarrow 0}\stru(\psi_1 (e^{-sS^2_{\epsilon,u}}-e^{-sS^2_{0,u}})\psi_1)=0.$$ Now, $S_{\epsilon,u}=S_{0,u}-\Omega \Pi_{\epsilon}D^{\mathcal{F}_{\partial}}=c(\partial_r)\partial_r+H$ with $H=\Omega D^{\mathcal{F}_{\partial}}+\Omega \dot{\theta}u$ hence 
\begin{align}\label{straccia}
\essp-\esspo=&-[S_{0,u},\Omega \dotto \Pi_{\epsilon}D^{\mathcal{F}_{\partial}}]+(\Omega \dotto \Pi_{\epsilon}D^{\mathcal{F}_{\partial}})^2\\ \nonumber=&-[c(\partial_r)\partial_r,\ome \dotto \piep \deffo]-[H,\ome \dotto \piep \deffo]+(\ome \dotto \piep \deffo)^2\\ \nonumber=&-\Phi \dotto \piep \deffo-2(\deffo+\dotto u)(\dotto \piep \deffo)+(\ome \dotto \piep \deffo)^2.
\end{align}Apply the Duhamel formula
\begin{align*}
|\stru(\psi_1(\essp &-\esspo)\psi_1|\\&=\big|\stru(\psi_1 e^{-\delta S^2_{0,u}}e^{-(s-\delta)S^2_{\epsilon,u}}\psi_1)_{(\delta=s)}-\stru(\psi_1 e^{-\delta S^2_{0,u}}e^{-(s-\delta)S^2_{\epsilon,u}}\psi_1)_{(\delta=0)}\big|\\&=
\Big|\int_0^s \stru(\psi_1^2 \Pi_{\epsilon})e^{-\delta S_{0,u}^2}(S^2_{\epsilon,u}-S^2_{0,u})\Pi_{\epsilon}e^{-(s-\delta)S^2_{\epsilon,u}}d\delta\Big|.
\end{align*} \noindent Again from the Cheeger--Gromov relative estimates \eqref{1333}
$$|\tru(\psi_1 e^{-\delta S_u^2}\Pi_{\epsilon}\psi_1)|\leq C {\delta^{-1/2}}$$
$$\|(S^2_{\epsilon,u}-S^2_{0,u})\Pi_{\epsilon}e^{-(s-\delta)S^2_{\epsilon,u}}\|\leq C(s-\delta)^{-1/2}$$ with the constants independent from $|u|<\epsilon.$
Then the integral of the supertrace $\eqref{straccia}$ can be estimated by
the function of s,  $h(s)=C\int^s_{0}(s-\delta)^{-1/2}\delta^{-1/2}d\delta\longrightarrow_{s\rightarrow 0}0.$ . To see this first split the integral into $\int_{0}^{s/2}+\int_{s/2}^s$ to prove finiteness then use the absolutely continuity of the integral for convergence to zero.
\noindent Now from the limit $\lim_{s\rightarrow 0}\stru(\psi_1 (e^{-sS^2_{\epsilon,u}}-e^{-sS^2_{0,u}})\psi_1)=0$ and the comparison argument we get that the asymptotic expansion for $s\rightarrow 0$ of $\operatorname{str}_{\Lambda}(\phi_ke^{-sD^2_{\epsilon,u}\phi_k})$ is the same of the comparison operator 
$$S_{0,u}=\underbrace{c(\partial_r)\partial_r+\Omega D^{\mathcal{F}_{\partial}}}_{D}+\underbrace{\quad \quad  \quad \dot{\vartheta}u\Omega \quad \,\quad \quad \,}_{\textrm{bounded perturbation}}$$ 
on the infinite cylinder. This is a very simple $u$--family of generalized laplacians (see \cite{BeGeVe} Chapter 2.7) and the
Duhamel formula
$$e^{-tS^2_{0,u}}-e^{-tS^2_{0,0}}=-\int_0^ut\dot{\vartheta}\Omega e^{-t S_{0,v}}dv ds $$
shows what is written in the statement i.e. $$\operatorname{str}^E([e^{-sD^2_{\epsilon,u,z}}])_{(z,z)}\simeq \sum_{j\in \mathbb{N}}a_j(S_{0,u})_{(z)}s^{(j-\operatorname{dim}{\mathcal{F}})/2}$$
 where the coefficients 
$a_j(S_{0,u})$ depend smoothly on $u$ and satisfy 
$a_j(S_{0,u})=0$ for $j\leq \operatorname{dim}\mathcal{F}/2$ since $S_{0,0}$ is the Cylindrical Dirac operator. One can take for the definition of $g$, $$g(u):=\sum_{j=0}^{ \operatorname{dim}{\mathcal{F}}/2}\int_{\partial X_0 \times [0,4]}   a_j(S_{0,u})_{(z)}s^{(j-\operatorname{dim}{\mathcal{F}})/2}d\Lambda_g.$$

\item This is done again by comparison with $S_{\epsilon,u}$ consider the $r$--depending family of tangential tangential measures $(y,r)\in \partial X_0 \times [a,b] \longmapsto \operatorname{str}^E{e^{-sD^2_{\epsilon,u,(x,r)}}dxdr}$ where $x\in L_{(y,r)}$, once coupled with $d\Lambda$ gives a measure on $X$
$\mu:=\operatorname{str}^E{e^{-sD^2_{\epsilon,u,(x,r)}}dxdr}\cdot d\Lambda$.
The Fubini theorem can certainly used during the integration process to find out that the mass of $\mu$ can be computed integrating first the $r$--depending tangential measures $y\longmapsto \operatorname{str}^E{e^{-sD^2_{\epsilon,u,(y,r)}}dy}$ against $\Lambda$ on the foliation at infinity $(\partial X_0,\mathcal{F}_{\partial})$ then the resulting function of $r$ on $[a,b],$
\begin{align*}
\operatorname{LIM}_{s\rightarrow 0}\int_{\partial X_0 \times [a,b]}d\mu&=
\operatorname{LIM}_{s\rightarrow 0}\int_a^b \int_{\partial X_0}\operatorname{str}^E([e^{-sS^2_{\epsilon,u}}])_{(y,r),(y,r)})dy\cdot d\Lambda dx\\&=  \operatorname{LIM}_{s\rightarrow 0}\dfrac{b-a}{\sqrt{4\pi s}}\stru(e^{-s(D_{\epsilon,u}^{\mathcal{F}_{\partial}})^2})=0
\end{align*}
in fact the boundary operator 
$D_{\epsilon,u}^{\mathcal{F}_{\partial}}$ is invertible and the well--known Mc--Kean--Singer formula for foliations on compact ambient manifolds (formula (7.39) in \cite{MoSc})
 says that 
$\operatorname{ind}_{\Lambda}(D_{\epsilon,u}^{\mathcal{F}_{\partial}})=\stru{e^{-s      
(D_{\epsilon,u}^{\mathcal{F}_{\partial}})^2}}$ independently from $s$.
\end{enumerate}
\end{proof} Finally \eqref{passagg} becomes 
\begin{equation}\label{terspass}
\operatorname{ind}_{\Lambda}(D_{\epsilon,u}^+)=\langle \hat{A}(X)\operatorname{Ch}(E/S),C_{\Lambda}\rangle-1/2 \eta_{\Lambda}(D_{\epsilon,u}^{\mathcal{F}_{\partial}})+g(u).
\end{equation}
\begin{thm}
The Dirac operator has finite dimensional $L^2-\Lambda$--index and the following formula holds\pecetta{
\begin{align}\label{2111}\operatorname{ind}_{L^2,\Lambda}(D^+)=\langle\hat{A}(X)\operatorname{Ch}(E/S),[C_{\Lambda}]\rangle +1/2[\eta_{\Lambda}(D^{\mathcal{F}_{\partial}})-h^+_{\Lambda}+h^{-}_{\Lambda}]\end{align}} where
\begin{equation}\label{1001}
h^{\pm}_{\Lambda}:=\operatorname{dim}_{\Lambda}(\operatorname{Ext}(D^{\pm})-\operatorname{dim}_{\Lambda}(\operatorname{Ker}_{L^2}(D^{\pm})\end{equation} 
with the dimension of the space of extended solutions as defined in the definition \ref{170} after the remark i.e.
$$\operatorname{dim}_{\Lambda}\operatorname{Ext}(D^{\pm}):=\operatorname{dim}_{\Lambda} \overline{\operatorname{Ext}(D^{\pm})}^{e^{u\theta}L^2}$$ independently from small $u>0$.
\end{thm}\begin{proof}
Start from 
\begin{equation}\label{234}\operatorname{ind}_{L^2,\Lambda}(D^+_{\epsilon})=\lim_{u\downarrow 0}1/2\{\operatorname{ind}_{\Lambda}(D^+_{\epsilon,u})+\operatorname{ind}_{\Lambda}(D^+_{\epsilon,-u})+h^{-}_{\Lambda,\epsilon}-h^+_{\Lambda,\epsilon}\},\end{equation} here $h^{\pm}_{\Lambda,\epsilon}=\operatorname{dim}_{\Lambda}(\operatorname{Ext}(D^{\pm}_{\epsilon}))-\operatorname{dim}_{\Lambda}(\operatorname{Ker}_{L^2}(D^{\pm}_{\epsilon}))$ for now proposition \ref{exx} says that 
$$\operatorname{Ext}(D^{\pm}_{\epsilon})=\operatorname{Ker}_{L^2}(D^{\pm}_{\epsilon,\pm})=\operatorname{Ker}_{e^{u\theta}L^2}(D^{\pm}_{\epsilon}).$$
Use the identity
 $\operatorname{ind}_{\Lambda}(D^{+}_{\epsilon,u})=\langle \hat{A(X)}\operatorname{Ch}(E/S),[C_{\Lambda}]\rangle +1/2\eta_{\Lambda}(D^{\mathcal{F}_{\partial}}_{\epsilon,u})+g(u)$ 
into \eqref{234},
\begin{align}
\operatorname{ind}_{L^2,\Lambda}(D^+_{\epsilon})&=\lim_{u\downarrow 0}1/2
\Big{\{}2\langle \hat{A}(X)\operatorname{Ch}(E/S),[C_{\Lambda}]\rangle +h^{-}_{\Lambda,\epsilon}-h^{+}_{\Lambda,\epsilon}+g(u)+g(-u) \\\nonumber
& 
\underbrace{ +1/2\eta_{\Lambda}(D^{\mathcal{F}_{\partial}}_{\epsilon,u})+1/2\eta_{\Lambda}(D^{\mathcal{F}_{\partial}}_{\epsilon,-u})}_{\eta_{\Lambda}(D_{\epsilon}^{\mathcal{F}_{\partial}})\textrm{ by proposition }\ref{0348}  }\Big{\}}
\\ \nonumber=&\langle \hat{A}(X)\operatorname{Ch}(E/S),[C_{\Lambda}]\rangle
+\dfrac{h^{-}_{\Lambda,\epsilon}-h^{+}_{\Lambda,\epsilon}}{2}+\dfrac{\eta_{\Lambda}(D_{\epsilon}^{\mathcal{F}_{\partial}})}{2}.\end{align}
It remains to pass to the $\epsilon$--limit remembering that:
\begin{itemize}
\item $\lim_{\epsilon\downarrow 0}\operatorname{ind}_{L^2,\Lambda}(D^+_{\epsilon})=\operatorname{ind}_{L^2,\Lambda}(D^+)$ (Proposition \ref{hj}),
\item $\lim_{\epsilon \downarrow 0}h^{-}_{\Lambda,\epsilon}-h^{+}_{\Lambda,\epsilon}=h^--h^+$ (again proposition \ref{hj})
\item $\lim_{\epsilon \downarrow 0}\eta_{\Lambda}(D_{\epsilon}^{\mathcal{F}_{\partial}})=\eta_{\Lambda}(D^{\mathcal{F}_{\partial}})$ (proposition \ref{0348}).
\end{itemize}
\end{proof}

\section{Comparison with Ramachandran index formula}
The Ramachandran index formula \cite{Rama} stands into index theory for foliations exactly as the Atiyah--Patodi--Singer formula stays in the classical theory. Our formula is in some sense the cylindrical point of view of this formula. In this section we prove that the two formulas are compatible and we do it exactly in the way it is done for the single leaf case by APS. First we recall the Ramachandran Theorem
\subsection{The Ramachandran index}
\noindent Since we have chosen an opposite orientation for the boundary foliation the Ramachandran index formula here written differs from the original in \cite{Rama} exactly for its sign (as in section \ref{aaps} for the APS formula). So let us consider the Dirac operator builded in section 
\ref{geom} but acting only on the foliation restricted to the compact manifold with boundary $X_0$.
To be precise with notation let us call $\mathcal{F}_0$ the foliation restricted to $X_0$ with leaves $\{L_x^{0}\}_x$, equivalence relation $\mathcal{R}_0$ and $D^{\mathcal{F}_0}$ the Dirac operator
acting on the field of Hilbert spaces 
$\{L^2(L_x^{0};E)\}_{x\in X_0}$. Near the boundary $$D^{\mathcal{F}_0}=\left(\begin{array}{cc}0 & D^{\mathcal{F}_0^-} \\D^{\mathcal{F}_0^+} & 0\end{array}\right)=\left(\begin{array}{cc}0 & -\partial_r+D^{\mathcal{F}_{\partial}} \\ 
 \partial_r+D^{\mathcal{F}_{\partial}}
& 0\end{array}\right)$$ with the boundary operator $D^{\mathcal{F}_{\partial}}.$ Let us consider the field of APS boundary conditions   $$B=\left(\begin{array}{cc}\chi_{[0,\ty)}(D^{\mathcal{F}_{\partial}}) & 0 \\0 & \chi_{(-\ty,0)}(D^{\mathcal{F}_{\partial}})\end{array}\right)=\left(\begin{array}{cc}P & 0 \\0 & \operatorname{I}-P\end{array}\right)$$ acting on the boundary foliation. In the order of ideas of Ramachandran paper (coming back from an idea of John Roe) this is a \underline{self adjoint boundary condition} i.e. its interacts with the Dirac operator in the following way:
\begin{enumerate}
\item $B$ is a field of bounded self--adjoint operators with $\sigma B+B\sigma=\sigma$ where $\sigma$ is Clifford multiplication by the unit (interior) normal.
\item If $b$ is the operator of restriction to the boundary acting on smooth sections then $(s_1,D^{\mathcal{F}_0}s_2)=(D^{\mathcal{F}_0} s_1,s_2)$ for every couple of smooth sections $s_1$ and $s_2$ such that $Bbs_1=0$ and $Bbs_2=0$.  
\end{enumerate}
Next Ramachandran proves using the generalized eigenfunction expansion of Browder and G{\aa}rding, that there's a field of restriction operators
$$H^k(X_0;E)\longrightarrow H^{k-1/2}(X_0;E)$$ extending $b$ where the Sobolev spaces are defined taking into account the boundary i.e. for a leaf $L^0_x$, the space $H^k(L_x^{0};E)$ is the completion of $C^{\ty}_{c}(L_x^{0};E)$ (compact support possibly meeting the boundary) w.r.t. the usual $L^2$--based Sobolev norms. It follows from the restriction theorem that one can define the domain of $D$ with boundary condition $B$ as
$H^{\ty}(X_0;E,B):=\{s \in H^{\ty}(X_0;E):Bbs=0\}.$
\begin{thm}\label{1122}(Ramachandran \cite{Rama})
The family of unbounded operators $D$ with domain $H^{\ty}(X_0;E,B)$ is essentially self--adjoint and Breuer--Fredholm in the Von Neumann algebra of the foliation with finite 
$\Lambda$--index in the sense of $\operatorname{ind}_{\Lambda}(D^{\mathcal{F}_0})=\operatorname{dim}_{\Lambda}(\operatorname{Ker}(D^{\mathcal{F}_0^{+}}))-\operatorname{dim}_{\Lambda}(\operatorname{Ker}(D^{\mathcal{F}_0^-}))$ given by the formula
\begin{align}\label{232}\operatorname{ind}_{\Lambda}(D^{\mathcal{F}_0})=\langle\hat{A}(X)\operatorname{Ch}(E/S),[C_{\Lambda}]\rangle +1/2[\eta_{\Lambda}(D^{\mathcal{F}}_0) -h]\end{align} 
\end{thm}
\noindent Now we are going to prove compatibility between formula \eqref{232} and \eqref{2111}. First of all we have to relate the two Von Neumann algebras in play. Denote (according to our notation) with 
$\operatorname{End}_{\mathcal{R}_0}(E)$ the space of intertwining operators of the representation of 
$\mathcal{R}_0$ on 
$L^2(E)$ and, only in this section 
$\operatorname{End}_{\mathcal{R}_0,\Lambda}(E)$ the resulting Von Neumann algebra with trace 
$\operatorname{tr}_{\mathcal{R}_0,\Lambda}$ in order to make distinction from 
$\operatorname{End}_{\mathcal{R},\Lambda}(E)$ the Von Neumann algebra of random operators associated with the representation of 
$\mathcal{R}$. Start with a measurable fields of bounded operators 
$X_0 \ni B_x\longmapsto B_x:L^2(L_x^{0};E)\longrightarrow L^2(L_x^{0};E)$ with $B_x=B_y$ a.e. if 
$(x,y)\in \mathcal{R}_0$. There's a natural way to extend $B$ to a field of operators in 
$\operatorname{End}_{\mathcal{R}}(E)$. 
\begin{enumerate}
\item If $x\in X_0$ simply let 
$\imath B_x$ act to $L^2(L_x;E)$ to be zero on the cylinder
$$\imath B_x:L^2(L_x^0;E)\oplus L^2(\partial L_x^0\times (0,\ty);E)\longrightarrow L^2(L_x^0;E)\oplus L^2(\partial L_x^0\times (0,\ty);E)$$
$\imath B_x(s,t):=(B_x s,0). $
\item If $x\in \partial X_0\times (0,\ty)$ define $\imath B_x:=\imath B_{p(x)}$ where $p:\partial X_0\times (0,\ty)\longrightarrow \partial X_0$ is the base projection and $\imath B_{p(x)}$ is defined by point $1.$
\end{enumerate}
\begin{prop}\label{987}
The map $\imath:\operatorname{End}_{\mathcal{R}_0}(E)\longrightarrow \operatorname{End}_{\mathcal{R}}(E)$ as defined above passes to the quotient to an injection 
$$\imath:\operatorname{End}_{\mathcal{R}_0,\Lambda}(E)\longrightarrow \operatorname{End}_{\mathcal{R},\Lambda}(E)$$ between the Von Neumannn algebras of Random operators preserving the two natural traces
$$\operatorname{tr}_{\mathcal{R},\Lambda}(\imath B)=\operatorname{tr}_{\mathcal{R}_0,\Lambda}(B).$$
\end{prop}
\begin{proof}The first part is clear. An intertwining operator $B=\{B_x\}_{x\in X_0}$ is zero $\Lambda$--a.e. in $X_0$ then also does $\imath B$ in $X$ for any transversal $T$ contained in the cylinder can slide by holonomy to a transversal contained in $X_0$. About the identity on traces remember the link between the direct integral algebras and the algebras of random operators i.e. Lemma \ref{11010}. Choose $\nu$ to be the longitudinal Riemannnian metric then $\Lambda_{\nu}$ is the integration of $\nu$ against $\Lambda$. Let $P_0$ be the Von Neumann algebra of $\Lambda_{\nu}$--a.e. classes of measurable fields of operators $X_0\ni x \longmapsto B_x\in B(L^2(L_x^0;E))$ and $P$ the corresponding algebra builded replacing $X_0$ with $X$ and $B(L^2(L_x^0;E))$ with $B(L^2(L_x^0;E))$. Notice that the family \begin{equation}\label{xxxy}X\ni y\longmapsto \int \imath B_x d\nu^y\end{equation} is bounded for $B$ in the domain of $\imath$ 
 then Lemma \ref{11010} says that 
 $$\operatorname{tr}_{\mathcal{R},\Lambda}(\imath B)=\int_{X}\operatorname{Trace}(B_x)d\Lambda_{\nu}(x)=\int_{X_0}\operatorname{Trace}(B_x)d\Lambda_{\nu}(x)=\operatorname{tr}_{\mathcal{R}_0,\Lambda}(B).$$
\end{proof}
 \begin{thm}
 Let $\operatorname{Pr}\operatorname{Ker}(D^{\mathcal{F}_0^{\pm}}) \in \operatorname{End}_{\mathcal{R}_0,\Lambda}(E)$ the projection on the Kernel of $D^{\mathcal{F}_0^{\pm}}$ with domain given by the boundary condition $Px=0,\,(\operatorname{I}-P=0)$ as in Ramachandran formula.
 Let also $\operatorname{Pr}\operatorname{Ker}_{L^2}(D^{\pm}) \in \operatorname{End}_{\mathcal{R},\Lambda}(E)$ 
 be the projection on the $L^2$--kernel of the leafwise operator on the foliation with the cylinder attached and 
 $\operatorname{Pr}\overline{\operatorname{Ext}(D^{\pm})} \in\operatorname{End}_{\mathcal{R},\Lambda}(e^{u\theta}L^2 E)$ 
 be the projection on the closure of the space of extended solution seen in 
 $e^{u\theta}$ for sufficiently small positive $u$.
 \begin{enumerate}
\item $\imath \operatorname{Pr}\kepp$ is equivalent to $\operatorname{Pr}\operatorname{Ker}_{L^2}(D^+)$ in $\operatorname{End}_{\mathcal{R},\Lambda}(E)$ i.e. there exists a partial isometry $u\in \operatorname{End}_{\mathcal{R},\Lambda}(E)$ such that $$u^*u=\imath \operatorname{Pr}\kepp,\quad uu^*=\operatorname{Pr}\operatorname{Ker}_{L^2}(D^+)$$. In particular
$$\operatorname{dim}_{\mathcal{R}_0,\Lambda}\kepp=\operatorname{dim}_{\mathcal{R},\Lambda} \operatorname{Ker}_{L^2}(D^+).$$
\item
$$\imath \operatorname{Pr}\operatorname{Ker}_{L^2}(D^{\mathcal{F}_0^-})
\sim \operatorname{Pr} \overline{\operatorname{Ext}(D^-)}^{e^{u\theta}L^2},$$ 
for sufficiently small $u$ and equivalence in 
$\operatorname{End}_{\Lambda}(e^{u\theta}L^2(E))$ 
with the inclusion $$\imath:\operatorname{End}_{\mathcal{R}_0,\Lambda}(E)\longrightarrow    \operatorname{End}_{\Lambda}(e^{u\theta}L^2(E))  $$ defined as in proposition \ref{987}. 

\noindent As a consequence $$\operatorname{dim}_{\Lambda}\operatorname{Ker}(D^{\mathcal{F}_0^-})=\operatorname{dim}_{\Lambda}\operatorname{Ext}(D^-).$$
\end{enumerate}
 \end{thm}
 \begin{proof}The idea is contained in A.P.S.
  \cite{AtPaSi1} when they prove the equivalence between the boundary value problem and the $L^2$ cylindrical problem. Their main instrument is the eigenfunction expansion of the operator at the boundary, now we use the Browder--Garding generalized expansion to see that any solution of the boundary value problems extends to a solution of the operator on the cylinder. 
 \begin{enumerate}
 \item Use the Browder--G{\aa}rding expansion as in the proof of the finiteness of the projection on the kernel \ref{2344}. For a single leaf, the isomorphism
 $$L^2(\partial L_x^0\times (-1,0])\longrightarrow \bigoplus_{j\in \mathbb{N}}L^2(\R,\mu_j)\otimes L^2((-1,0]) $$ represents a solution of the boundary value problem as $h_{j}(r,\lambda)=\chi_{(-\ty,0)}(\lambda)e^{-\lambda r}h_{j0}(r)$ hence the solution can be extended to the cylinder of the leaf $\partial L_x^0\times (0,\ty)$. This clearly gives a field of linear isomorphisms $T_x:\operatorname{Ker}(D_x^{\mathcal{F}_0^+})\longrightarrow \operatorname{Ker}_{L^2}(D_x^+)$ for $x\in X_0$, first extend $T_x$ to all $L^2(L_x^0;E)$ to be zero on $ \operatorname{Ker}(D^{\mathcal{F}_0^+})^{\bot}$ then let $x$ take values also in $X$ according to the method explained before i.e. put $T_x:=T_{p(x)}$ for $x$ in the cylinder. Take the polar decomposition $T_x=u_x|T_x|$, then $u_x$ is a partial isometry with initial space $\operatorname{Ker}(D_x^{\mathcal{F}_0^+})$ and range $\operatorname{Ker}(D_x^+),$ i.e 
 $$u_x^*u_x=\operatorname{Pr} \operatorname{Ker}(D_x^{\mathcal{F}_0^+}),\quad u_xu_x^*=\operatorname{Pr} \operatorname{Ker}(D_x^+).$$
 We have to look at this relation into the Von Neumann algebra of the foliation on $X$. Split every $L^2$ space of the leaves as $L^2(L_{p(x)}^0;E)\oplus L^2(\partial L_{p(x)}^0\times (0,\ty);E)$. With respect to the splitting, forgetting the indexes $x$ downstairs, we have $u=\left(\begin{array}{cc}u_{11} & 0 \\u_{21} & 0\end{array}\right)$ acting on the field of $L^2(X;E)$ spaces of the leaves. Then $u^*=\left(\begin{array}{cc}u_{11}^* & u_{21}^* \\0 & 0\end{array}\right)$ with conditions $u_{11}u_{21}^*=0$ and $u_{21}u_{11}^*=0.$
 Finally $$uu^*=\left(\begin{array}{cc}u_{11}u_{11}^*+ u_{21}u_{21}^*& 0 \\0 & 0\end{array}\right)
 \left(\begin{array}{cc}\operatorname{Pr}(D^{\mathcal{F}_0^+}) & 0 \\0 & 0\end{array}\right)=\imath \operatorname{Pr}(D^{\mathcal{F}_0^+})$$ and similarly $u^*u=\operatorname{Pr}(D^+).$
 \item It is very similar to statement 1. in fact writing the Browder--G{\aa}rding expansion and imposing the adjoint boundary condition one ends directly into the space of extended solutions.
 \end{enumerate}
 \end{proof}\noindent To conclude now we can compare Ramachandran index with our index; let's compare formula \eqref{232} with \eqref{2111} keeping in mind that, the index of Ramachandran is now our extended index (see section \ref{aaps} )
 $$\operatorname{ind}_{\Lambda}(D^{\mathcal{F}_0})={\operatorname{ind}_{\Lambda,L^2}(D^+)}=\operatorname{dim}_{\Lambda} \operatorname{Ker}_{L^2}(D^+)-\operatorname{Ker}_{L^2}(D^-)$$ to obtain the equation
 $$\operatorname{dim}_{\Lambda} \operatorname{Ext}(D^-)-\operatorname{dim}_{\Lambda} \operatorname{Ker}_{L^2}(D^-)=(h^{-}_{\Lambda}-h^{+}_{\Lambda})/2+h/2.$$
The same argument applied to the (formal) adjoint of $D^+$ leads to the equation
$$\operatorname{dim}_{\Lambda} \operatorname{Ext}(D^+)-\operatorname{dim}_{\Lambda} \operatorname{Ker}_{L^2}(D^+)=(h^{+}_{\Lambda}-h^{-}_{\Lambda})/2+h/2,$$ then 
\pecetta{$$h=h^{+}_{\Lambda}+h^{-}_{\Lambda}$$} as in A.P.S.

\section{The signature of a foliated manifold with boundary}
In this section we apply our index formula to the signature operator. First we recount the story of the signature operator and its relation with the signature of a closed manifold and a with a compact manifold with boundary as in A.P.S; then following the paper of Luck and Schick
\cite{lusc} about $L^2$ signatures of $\Gamma$--coverings of manifolds with boundary we propose three different definitions of the signature
for the foliation with boundary 
\begin{description}
\item[Analytical] (index theory)
\item[Topological ] (de Rham)
\item[Harmonic ] (Hodge)
\end{description}
and prove they all agree. 

\subsection{The Hirzebruch formula}
The reference for the notation about the signature operator is the book bt Berline Getzler and Vergne \cite{BeGeVe}.
Let $X$ be an oriented Riemannian manifold
and $|dvol|$ the volume the unique volume form compatible with the metric i.e. the one assuming the value 1 on each positive oriented orthonormal frame. In other words $|\operatorname{dvol}=|\sqrt g dx|.$
 One can define the Hodge $\ast$ operator in the usual way $$\ast e^{i_1}\wedge\cdot \cdot \cdot \wedge e^{i_k}= \operatorname{sign}({\sigma})e_{j_1}\wedge\cdot \cdot \cdot \wedge e_{i_{n-k}}$$
where $(e_1,...,e_n)$ is an oriented orthonormal basis, $(i_1,...,i_k)$ and $(j_i,...,j_k)$ are complementary multindices and $\sigma$ is the permutation $\sigma:=\left(\begin{array}{cccccc}1 & . & . & . & . & n \\i_1 & . & i_k & j_1 & . & j_{n-k}\end{array}\right).$

\noindent Since $\ast^2=(-1)^{|\cdot|(n-|\cdot|)}$ this is an involution on even dimensional manifolds.

\noindent The bundle $\Lambda{T^{*}X}$ of exterior algebras of $X$ is a natural Clifford module under the action defined by
\begin{equation}\label{cliffaction}c(e^{i}):=\epsilon(e_i)-\iota(e^{i})\end{equation} where  $\epsilon(e^i)\omega=e^{i}\wedge \omega$ is the exterior multiplication by $e^{i}$ and $\iota(e_i)$ is the contraction by the tangent vector $e_i$. In other words it is the metric adjoint of exterior multiplication, $\epsilon(e^i)^*=\iota(e_i)$. The chirality involution $$\tau:=i^{[(n+1)/2]}c(e_1)\cdot\cdot\cdot c(e_n)$$ is related to the Hodge duality operator by 
$$\tau=i^{[(n+1)/2]}\ast (-1)^{n|\cdot|+\frac{|\cdot|(|\cdot|-1)}{2}},$$
following from the identity (same deegree forms)
$$\int_X \alpha \wedge \tau\beta=(-1)^{n|\cdot|+|\cdot|(|\cdot|-1)/2}i^{[2n+1]/2}\int_{X}(\alpha,\beta)|dx|$$ while $\int_X \alpha \wedge \ast \beta=\int_X (\alpha,\beta)|dx|.$
 As a consequence one can write the adjoint of $d$ in two different ways,
$$d^*=-\ast d \ast (-1)^{n|\cdot|+n}=-(-1)^n \tau d \tau.$$ Sections of the positive and negative eigenbundles of $\tau$ are called the \underline{self--dual} and \underline{anti self--dual} differential forms respectively and denoted by $\Omega^{\pm}(X).$

\noindent Now suppose $n$ is even, and $X$ is compact. The bilinear form on the middle cohomolgy $H^{n/2}(X;\R)$ defined by $(\alpha,\beta)\longmapsto \int_X \alpha \wedge \beta$ satisfies the identity 
$$\int_X \alpha \wedge \beta =(-1)^{n/2}\int_X  \beta \wedge \alpha.$$
In particular if $n$ is divisible by four this is \underline{symmetric} and has a signature $\sigma(X)$ i.e. the number $p-q$ related to the representation
$$Q(x)=x_1^2+\cdot \cdot \cdot + x_p^2- x_{p+1}^2-\cdot \cdot \cdot -x_q^2$$
 of the associated quadratic form (this is independent by the choosen basis). In this situation the chiral Dirac operator $d+d^*$ acting on the space of differential forms is called the \underline{Signature operator}\footnote{it differs from the Gauss--Bonnet operator $d+d^*$ only for the choice of the involution}
 $$(d+d^*)=D^{\operatorname{sign}}=\left(\begin{array}{ccc}0 & D^{\operatorname{sign},-}  \\D^{\operatorname{sign},+} & 0  \end{array}\right):\Omega^+(X)\oplus \Omega^-(X)\longrightarrow \Omega^+(X)\oplus \Omega^-(X)$$
 \noindent The Atiyah--Singer index theorem in this case becomes the Hirzebruch signature theorem
 $$\operatorname{ind}(D^{\operatorname{sign},+})=\sigma(X)=\int_X \operatorname{L}(X)$$ where $\operatorname{L}(X)$ is the $\operatorname{L}$--genus, $\operatorname{L}(X)=(\pi i)^{-n/2}\operatorname{det}^{1/2}\Big{(}\dfrac{R}{\operatorname{tanh(R/2)}}\Big{)}$ for the Riemannian curvature $R$.
 The proof uses the Hodge theorem stating a natural isomorphism between the space of \underline{harmonic forms} $\mathcal{H}^q(X)$ i.e. the kernel of the forms laplacian $\Delta=(d+d^*)^2$ and the cohomology $H^q(X)$ together with Poincar\'e duality.
 \bigskip
 
 \noindent Now on a manifold with boundary with product structure near the bounday the situation is much more complicated. The signature formula is the most important application of the index theorem in the A.P.S. paper. The operator can be written on a collar around the boundary as 
 $$D^{\operatorname{sign},+}
 =\sigma(\partial_u+B)$$ where the isomorphism $\sigma:\Omega(\partial X)\longrightarrow \Omega^+(X)$ and $B$ is the self--adjoint operator on $\Omega(\partial X)$ defined by $$B\alpha=(-1)^{k+p+1}(\ast_{\partial}d-d\ast_{\partial})\alpha$$ where from here to the end \underline{$\operatorname{dim}(X)=4k$}, $\epsilon(\alpha)=\pm 1$ according to $\alpha$ is even or odd degree and $\ast_{\partial}$ is the Hodge duality operator on $\partial X$. Since $B$ commutes with $\alpha \longmapsto (-1)^{|\alpha|}\ast_{\partial}\alpha$ and preserves the parity of forms, $B=B^{\textrm{ev}}\oplus B^{\textrm{odd}}$ and the dimension of the kernel at the boundary as the $\eta$ invariant are twice that of $B^{\textrm{ev}}$. The A.P.S index theorem says
 $$\operatorname{ind}(D^{\operatorname{sign},+}
)=h^+-h^--h^-_{\ty}=\int_XL-h(B^{\operatorname{ev}})-\eta(B^{\operatorname{ev}})$$ or
 $$\operatorname{ind}_{L^2}(D^{\operatorname{sign},+}
)-h^-_{\ty}=\int_XL-h(B^{\operatorname{ev}})-\eta(B^{\operatorname{ev}})$$
 
  where $h^{\pm}$ are the dimensions of the $L^2$--harmonic forms on the manifold $\hat{X}$ with a cylinder attached and $h^{-}_{\ty}$ is the dimension of the limiting values of extended $L^2$ harmonic forms in $\Omega^-(X)$.

\noindent The identifications of all these numbers with topological quantities require some work. 
\begin{enumerate}
\item The space $\mathcal{H}(\hat{X})$ of $L^2$ harmonic forms on the manifold with a cylinder attached $\hat{X}$ is naturally isomorphic to the image $\hat{H}(X)$ of
$$H^{*}_{0}(\hat{X})\longrightarrow H^{*}(\hat{X}).$$ Equivalently  one can use the relative de Rham cohomology
$H^*(X,\partial X)\longrightarrow H^*(X)$
defined imposing boundary conditions $\omega_{|\partial X}=0$ on the de Rham complex. This statement plays in the $\partial$--case the role played by Hodge theory.
\item The signature $\sigma(X)$ of a manifold with boundary is defined to be the signature of the non--degenerate quadratic form on the middle--cohomology 
$\hat{H}^{2k}(X).$ This is induced by the degenerate quadratic form given by the cup--product on the relative cohomology 
$H^{2k}(X,\partial X)$. By Lefshetz duality the radical of this quadratic form is exactly the kernel of the mapping 
$H^{2k}(X,\partial X)\longrightarrow H^{2k}(X)$ then
$$\sigma(X)=h^+-h^-=\operatorname{ind}_{L^2}(A).$$
\item Then A.P.S get rid of the third number $h^-_{\ty}$ proving that $h^-_{\ty}=h^+_{\ty}=h(B^{\textrm{ev}})$ that together with $h^{+}_{\ty}+h^-_{\ty}=2h(B^{\textrm{ev}})$ gives the final signature formula
$$\sigma(X)=\int_XL-\eta(B^{\operatorname{ev}}).$$
\end{enumerate}

\subsection{Computations with the leafwise signature operator}
\noindent So let $X_0$ be a compact manifold with boundary equipped with an oriented $4k$--dimensional foliation transverse to the boundary and every geometric structure of product type near the boundary.
As usual attach an infinite cylinder $Z_0=\partial X_0\times [0,\ty)_r$ and extend everything.
 The leafwise signature operator corresponds to the leafwise Clifford action \eqref{cliffaction} on the leafwise exterior bundle $\Lambda T^*{\mathcal{F}}$. If $(e_1,...,e_{4k-1},\partial_r)$ is a leafwise positive orthonormal frame near the boundary, the leafwise chirality element \footnote{we omit simbols denoting leafwise action for ease of reading} satisfies
\begin{align*} 
\tau:=i^{2k}c(e^1)\cdot \cdot \cdot c(e^{4k-1})c(dr)&=i^{2k}\ast (-1)^{|\cdot|(|\cdot|-1)/2}\\&=-i^{2k}c(dr)c_{\partial}=-i^{2k}c(dr)\ast_{\partial}(-1)^{|\cdot|+|\cdot|(|\cdot|-1)/2}\end{align*}
where $\ast$ is leafwise Hodge duality operator, $c_{\partial}=c(e^1)\cdot \cdot \cdot c(e^{4k-1})$ is, a part for the $i^{2k}$ factor the leafwise boundary chirality operator and $\ast_{\partial}$ is the leafwise boundary Hodge operator. 
\noindent On the cylinder the leafwise bundle $\Lambda T^*\mathcal{F}$ is isomorphic to the pulled back bundle $\rho^*(\wedge T^*\mathcal{F}_{\partial X_0})$ (the projection on the base $\rho$ will be omitted throughout) while separating the $dr$ component on leafwise forms $\alpha=\omega + \beta \wedge dr$ yields an isomorphism
\begin{equation}\label{isobundles}(\Lambda T^*\mathcal{F})_{\partial X_0} \longrightarrow (\Lambda T^*\partial \mathcal{F})\oplus (\Lambda T^*\partial \mathcal{F}),\end{equation} sometimes we shall write $(\Lambda T^*\partial \mathcal{F})\wedge dr$ for the second addendum in \eqref{isobundles} to remember this isomorphism. An easy computation involving rules as $$d\omega =d_{\partial}\omega+(-1)^{|\omega|}\partial_r \omega\wedge dr$$ for $\omega \in C^{\ty}([0,\ty);\Lambda T^*\partial\mathcal{F})$ and $c(dr)(\omega+\alpha \wedge dr)=(-1)^{|\omega|}\omega \wedge dr-(-1)^{|\alpha|}\alpha$ shows that the operator can be written on the direct sum 
$(\Lambda T^*\partial \mathcal{F})\oplus (\Lambda T^*\partial \mathcal{F})$ as the matrix
\begin{equation}\label{repp1}D^{\operatorname{sign}}=\left(\begin{array}{cc}d_{\partial}+c_{\partial}d_{\partial}c_{\partial} & -(-1)^{|\cdot|}\partial_r \\(-1)^{|\cdot|} \partial_r & c_{\partial}d_{\partial}c_{\partial}\end{array}\right)=c(dr)\partial_r+(d_{\partial}+c_{\partial}d_{\partial}c_{\partial})\oplus (d_{\partial}+c_{\partial}d_{\partial}c_{\partial})
\end{equation} and 
\begin{equation}\label{tauchib}\tau=i^{2k}\left(\begin{array}{cc}0 & c_{\partial}(-1)^{|\cdot|} \\
-c_{\partial}(-1)^{|\cdot|}
 & 0\end{array}\right)
.\end{equation}
Since $d_{\partial}^*=\tau_{\partial}d_{\partial}\tau_{\partial}=c_{\partial}d_{\partial}c_{\partial}$ formula \eqref{repp1} is equivalent to
$$D^{\operatorname{sign}}=c(dr)\partial_r+(d_{\partial}+d_{\partial}^*)\oplus(d_{\partial}+d_{\partial}^*).$$
\noindent There's also another important formula corresponding to the fact that $d+d^*$ anticommutes with $\tau$. Denote $\Omega^{\pm}(\mathcal{F})$ the positive (negative) eigenbundles i.e. the bundles of leafwise auto--dual (anti auto--dual) forms. We can write the operator on the cylinder as an operator on sections of the direct sum $\rho^*(\Omega^+(\mathcal{F})_{\partial X_0}\oplus \Omega^+(\mathcal{F})_{\partial X_0} )$ as the matrix
\begin{align}\nonumber&
\left(\begin{array}{cc}0 & -(-1)^{|\cdot|}\partial_r+(\ast_{\partial} d_{\partial}-d_{\partial}\ast_{\partial})i^{2k} (-1)^{|\cdot|(|\cdot|-1)/2} \\(-1)^{|\cdot|}\partial_r+(\ast_{\partial}d_{\partial}-d_{\partial}\ast_{\partial})i^{2k}(-1)^{|\cdot|(|\cdot|-1)/2} & 0\end{array}\right)\\ \label{antidiagonal}&=
c(dr)\partial_r+(\ast_{\partial} d_{\partial}-d_{\partial}\ast_{\partial})i^{2k} (-1)^{|\cdot|(|\cdot|-1)/2}\Omega.
\end{align}
\noindent To pass from one representation to another we have to consider the following compositions
\begin{equation*}
\xymatrix{ \Lambda T^*\partial \mathcal{F} \ar[r]^-{i_1} & (\Lambda T^*\partial \mathcal{F})\bigoplus (\Lambda T^*\partial \mathcal{F})\wedge dr
\ar[r]^-{  1+\tau}
&
\Omega^+(\mathcal{F})
 \ar[r]^{d+d^*}&
 \Omega^-(\mathcal{F})\ar[r]^-{\operatorname{Pr}_2}&
\Lambda T^*\partial \mathcal{F}}.
\end{equation*}
and 
\begin{equation*}
\xymatrix{ \Lambda T^*\partial \mathcal{F} \ar[r]^-{i_2} & \Lambda (T^*\partial \mathcal{F})\bigoplus (\Lambda T^*\partial \mathcal{F})\wedge dr
\ar[r]^-{  1-\tau}
&
\Omega^-(\mathcal{F})
 \ar[r]^-{d+d^*}&
 \Omega^+(\mathcal{F})\ar[r]^-{\operatorname{Pr}_1}&
\Lambda T^*\partial \mathcal{F}.
}\end{equation*}
where $i_j$ is the inclusion on the $j$--th factor and $\operatorname{Pr}_{j}$ is the corresponding projection.
\subsection{The Analytic signature}
\noindent The first definition we give is simple. It is merely the $L^2$ index of the signature operator on the foliated manifold with a cylinder attached.
\begin{dfn}
The $\Lambda$--analytic signature of the foliated manifold with boundary $X_0$ is the measured $L^2$ index of the signature operator on the foliated manifold with a cylinder attached,
$$\sigma_{\Lambda,\operatorname{an}}(X_0,\partial X_0):=\operatorname{ind}_{L^2,\Lambda}(D^{\operatorname{sign},+}).$$
\end{dfn}
\noindent Now, by the standard identification of the Atiyah--Singer integrand for the signature operator \cite{BeGeVe}, formula \eqref{2111} becomes
$$\sigma_{\Lambda,\operatorname{an}}(X_0,\partial X_0)=\langle L(X),[C_{\Lambda}]\rangle +1/2[\eta_{\Lambda}(D^{\mathcal{F}_{\partial}})-h^+_{\Lambda}+h^{-}_{\Lambda}]$$ where $L(X)$ is the tangential $L$--characteristic class and the numbers $h^{\pm}_{\Lambda}$ and the foliation eta--invariant are referred to the boundary signature operator.

\noindent As in \cite{AtPaSi1} first we have to identify these numbers. Minor modifications of the proof of Vaillant \cite{Vai} are needed in order to prove the following.
\begin{prop}\label{primocalcolo}
For the foliated signature operator 
\begin{equation}\label{hacca}h^{+}_{\Lambda}=h^{-}_{\Lambda}.\end{equation} Consequently the formula for the analytical signature is
$$\sigma_{\Lambda,\,\operatorname{an}}(X_0,\partial X_0)=\langle L(X),[C_{\Lambda}]\rangle +1/2[\eta_{\Lambda}(D^{\mathcal{F}_{\partial}})].$$
\end{prop}
\begin{proof}
Use the representation \eqref{repp1} of the operator on the cylinder on the bundle $(\Lambda T^*\partial \mathcal{F})\oplus (\Lambda T^*\partial \mathcal{F}),$ here we can easily write the one parameter perturbation $$D^{sign}_{\epsilon}=c(dr)\partial_r+(d_{\partial}+d_{\partial}^*)\oplus(d_{\partial}+d_{\partial}^*)-\dot{\theta}\Pi_{\epsilon}[(d_{\partial}+d_{\partial}^*)\oplus(d_{\partial}+d_{\partial}^*)]$$ where $\Pi_{\epsilon}$ the spectral projection $\Pi_{\epsilon}=\chi_{(-\epsilon,\epsilon)}((d_{\partial}+d_{\partial}^*)\oplus (d_{\partial}+d_{\partial}^*))$ of the leafwise boundary (signature) operator and $\theta$ is the function considered above in \eqref{de2}. For clarity we make the position
$$d_{\partial}+d_{\partial}^*=D^{\operatorname{sign}}_{\partial}=\operatorname{S}_{\partial}$$ for the boundary signature operator.
Now pass to the antidiagonal form
\begin{equation}\label{asd}c(dr)\partial_r+{(\ast_{\partial} d_{\partial}-d_{\partial}\ast_{\partial})i^{2k} (-1)^{|\cdot|(|\cdot|-1)/2}}\Omega
.\end{equation}

\noindent It is a well known fact that only the middle dimension forms contribute to the index in fact the leafwise kernel of the signature operator is the space of leafwise harmonic forms and decompose
$$\operatorname{ker} \Delta_x=\oplus_{i=0}^p \operatorname{ker} \Delta^{(i)}_x$$ where $\Delta^{(i)}_x:\Omega^{i}(L_x)\longrightarrow \Omega^{i}(L_x).$ The subspace $\operatorname{ker} \Delta_x^{(r)}\oplus \Delta_x^{(n-r)}$ is $\tau$--invariant for each $0\leq r\leq n$
and there is a field of unitary equivalences 
$$[\operatorname{ker} \Delta_x^{(r)}\oplus \Delta_x^{(n-r)}]^+\longrightarrow [\operatorname{ker} \Delta_x^{(r)}\oplus \Delta_x^{(n-r)}]^-$$ given by $\omega+\tau \omega\longmapsto \omega-\tau \omega.$
\noindent Now choose a leaf and apply the Browder--G{\aa}rding expansion exactly as in section \ref{finitt} to the boundary operator in \eqref{asd}. We forget the subscript indicating we are on a single leaf and the isomorphisms coming from the eigenfunction expansion. A section $\xi \in \operatorname{Ext}(D^{\operatorname{sign},\pm}_{\epsilon,x})$ can be written on the cylinder $r\geq 3$,
$$\xi^{\pm}(\lambda,r)=\zeta^{\pm}(\lambda,i)[\chi_{(-\epsilon,\epsilon)}(\lambda)+(1-\chi_{(-\epsilon,\epsilon)}(\lambda))e^{\mp \lambda r}]$$ with the fundamental fact that the boundary datas $\zeta^{\pm}(\lambda,i) \in L^2(\pm[0,\ty)\times \mathbb{N},\mu)$ are univocally determined by $r=0$. Now there's a coefficient that's constant in $r$. It is precisely $\zeta^{\pm}(\lambda,i)\chi_{(-\epsilon,\epsilon)}(\lambda)$ and can be seen (under the spectral isomorphism) to belong to the image of the spectral projection $\chi_{(-\epsilon,\epsilon)}(\operatorname{S}_{\partial}\oplus \operatorname{S}_{\partial} ).$ This subspace of $L^2(\partial L_x;\Lambda T^*\partial L_x)$ is naturally 
$\mathbb{Z}_2$ 
graded in fact the chirality operator $\tau$ commutes with the boundary operator. 

In particular
$$\zeta^{\pm}(\lambda,i)\chi_{(-\epsilon,\epsilon)}(\lambda)\in [\chi_{(-\epsilon,\epsilon)}(\operatorname{S}_{\partial}\oplus \operatorname{S}_{\partial}) L^2]^{\pm}$$
The splitting becomes more evident looking at the decomposition \eqref{repp1} 
$$\chi_{(-\epsilon,\epsilon)}
(\operatorname{S}_{\partial}\oplus \operatorname{S}_{\partial})=\chi_{(-\epsilon,\epsilon)}
(\operatorname{S}_{\partial})\oplus \chi_{(-\epsilon,\epsilon)}(\operatorname{S}_{\partial})$$ 
with $\tau$ acting on the right--hand side according to 
$$\tau=\left(\begin{array}{ccc}0 & -\tau_{\partial}(-1)^{|\cdot|}  \\\tau_{\partial}(-1)^{|\cdot|} & 0 \end{array}\right),$$ exactly formula 
\eqref{tauchib}.
\noindent So we have defined a measurable family of maps
$$\mathcal{J}^{\pm}_{x}:\operatorname{Ext}(D^{\operatorname{sign},\pm}_{\epsilon,x})\longrightarrow [\chi_{(-\epsilon,\epsilon)}(\operatorname{S}_{\partial})L^2\oplus \chi_{(-\epsilon,\epsilon)}(\operatorname{S}_{\partial})L^2 ]^{\pm},\quad \xi^{\pm}\longmapsto \zeta^{\pm}(\lambda,i)\chi_{(-\epsilon,\epsilon)}(\lambda).$$
\noindent Now proposition \ref{exx} says that if we choose $\delta$ small, say $0<\delta<\epsilon$ then $\operatorname{Ker}_{e^{-\delta \theta}L^2}(D^{\operatorname{sign},\pm}_{\epsilon,x})$ is closed in each $e^{-\delta \theta}L^2$ and $\operatorname{Ext}(D^{\operatorname{sign},\pm}_{\epsilon,x})$ is closed into each $e^{\delta \theta}L^2$. It follows that
\begin{itemize}
\item We have a Borel family of continuous and middle--exact sequences
\begin{align}\label{primarigasucc}
(\operatorname{Ker}_{L^2}(D^{\operatorname{sign},\pm}_{\epsilon,x}),\|\cdot \|_{e^{-\delta \theta}L^2})&\longrightarrow (\operatorname{Ext}(D^{\operatorname{sign},\pm}_{\epsilon,x}),\|\cdot \|_{e^{\delta \theta}L^2})\\\nonumber &
\longrightarrow [\chi_{(-\epsilon,\epsilon)}(\operatorname{S}_{\partial})L^2\oplus \chi_{(-\epsilon,\epsilon)}(\operatorname{S}_{\partial})L^2 ]^{\pm}\end{align} where the last arrow is $\mathcal{J}_x^{\pm}$.
\item $h^{\pm}_{\Lambda,\epsilon}=\operatorname{dim}_{\Lambda}(\operatorname{range}(J^{\pm}))$.
\end{itemize}
Now join togheter $\mathcal{J}_x:=\mathcal{J}^+_x+\mathcal{J}^-_x$ assume that $$\operatorname{range}(\mathcal{J}_x)\subset \chi_{(-\epsilon,\epsilon)}(\operatorname{S}_{\partial})L^2_x\oplus \chi_{(-\epsilon,\epsilon)}(\operatorname{S}_{\partial})L^2_x$$ splits into a direct sum
\begin{equation}\label{splittaz}\operatorname{range}({\mathcal{J}_x})=\mathcal{V}_x\oplus \mathcal{W}_x.\end{equation} Then in this case the proof ends because 
the chirality element acts on $\operatorname{range}({\mathcal{J}_x})$ sending $\mathcal{V}_x$ into $\mathcal{W}_x$ and vice--versa then the $\pm$ eigenspaces must be isomorphic. 
\end{proof}
\noindent So it remains to prove \eqref{splittaz}.
First we need a lemma,
\begin{lem}\label{dim12}
If $0<\delta<\epsilon$ 
 The family of spaces ${\operatorname{range}_{{e^{\delta \theta}L^2}}(D_{\epsilon}^{\operatorname{sign}})}$ is \underline{$\Lambda$--closed} this property meaning that for every $\gamma>0$ there exists a Borel family of closed subspaces $M \subset {\operatorname{range}_{{e^{\delta \theta}L^2}}(D_{\epsilon}^{\operatorname{sign}})}$ such that
$$\operatorname{dim}_{\Lambda}{\overline{\operatorname{range}(D_{\epsilon}^{\operatorname{sign}})}^{e^{\delta \theta}L^2}}-\operatorname{dim}_{\Lambda}(M)<\gamma.$$
\end{lem}
\begin{proof}
The first is a direct consequence of the $\Lambda$--Fredholm of the perturbed operator $D^{\operatorname{sign}_{\epsilon}}$ on the field $e^{\delta \theta}L^2$ in fact the commutative diagram
\begin{equation}\label{ver}\xymatrix{e^{\delta \theta}L^2\ar[r]^-{D^{\operatorname{sign},\pm}_{\epsilon}}&e^{\delta \theta}L^2\\
L^2\ar[r]^-{D^{\operatorname{sign},\pm}_{\epsilon}+\delta \theta}\ar[u]^-{e^{\delta \theta}}&D\ar[u]^-{e^{\delta \theta}}&}
\end{equation} and lemma \ref{312} show that the operator on field of weighted spaces $e^{\delta \theta}L^2$ is Breuer--Fredholm than $0$ is not contained in the $\Lambda$--essential spectrum of $TT^*$ where $T=D^{\operatorname{sign},\pm}_{\epsilon}$ and $T^*$ is the adjoint w.r.t the $e^{\delta \theta}$ norm and the spaces $M_{\eta}:=\chi_{(-\ty,\eta)}(TT^*)\cup(\chi_{(\eta,+\ty)}(TT^*)$ are $\Lambda$--finite codimensional in the closure of the image of $T$ in $L^2$ ($L^2$ because the vertical arrows in \eqref{ver} are isomorphisms that preserve the $\Lambda$--dimension). 
\end{proof}
\begin{prop}\label{density}
For every $x$ the image of $\mathcal{J}$ splits,
\begin{equation*}\operatorname{range}({\mathcal{J}_x})=\mathcal{V}_x\oplus \mathcal{W}_x.\end{equation*}
\end{prop}
\begin{proof}
\noindent Consider the first row of \eqref{primarigasucc} i.e
$$(\operatorname{Ker}_{L^2}(D^{\operatorname{sign},\pm}_{\epsilon,x}),\|\cdot \|_{e^{-\delta \theta}L^2})\longrightarrow (\operatorname{Ext}(D^{\operatorname{sign},\pm}_{\epsilon,x}),\|\cdot \|_{e^{\delta \theta}L^2})$$ 
with the non--degenerate pairing 
$e^{-\delta \theta}\times e^{\delta \theta}
\longrightarrow \C$ on each leaf,
$$(\operatorname{Ker}_{L^2}(D^{\operatorname{sign},\pm}_{\epsilon,x}),\|\cdot \|_{e^{-\delta \theta}L^2})^{\bot}=(\operatorname{Ker}_{e^{-\delta \theta}L^2}(D^{\operatorname{sign},\pm}_{\epsilon,x}))^{\bot}=\overline{\operatorname{range}(D^{\operatorname{sign},\pm}_{\epsilon,x})}^{e^{\delta \theta}L^2}$$ then 
extend $\mathcal{J}$ to be zero on the $e^{\delta \theta}$--hortocomplement of $\operatorname{Ext}(D^{\operatorname{sign},\pm}_{\epsilon,x})$ then 
$$\operatorname{range}(\mathcal{J}_x)=\tilde{\mathcal{J}}_x\Big{(}\,\,\overline{\operatorname{range}(D^{\operatorname{sign},\pm}_{\epsilon,x})}^{e^{\delta \theta}L^2}\Big{)}.$$
\noindent Put
$\operatorname{range}(\mathcal{J})=\mathcal{J}(\mathcal{K})$ by the continuity of $\mathcal{J}$ we can restrict our attention to elements in $$\mathcal{K}_x^0:={\operatorname{range}_{{e^{\delta \theta}L^2}}(D_{\epsilon,x}^{\operatorname{sign}})}\cap \underbrace{\operatorname{Ext}(D^{\operatorname{sign}}_{\epsilon,x})}_{e^{\delta \theta}-\textrm{closed}}$$ for each $x$.
So let $\xi \in \mathcal{K}^0$, by definition there exist $\alpha \in e^{\delta \theta}L^2(\Lambda T^*L_x)$ such that 
$\xi=D^{\operatorname{sign}}_{\epsilon}$ and 
$(D^{\operatorname{sign}}_{\epsilon})^2\xi=0$. On the cylinder we can write 
$\alpha=\alpha_0+\alpha_1\wedge dr$ with 
$\alpha_i\in H^{\ty}(\partial L_x\times [0,\ty);\Lambda T^*L_x)$. 
Using again Browder--G{\aa}rding (or a spectral resulution, it's the same) of the boundary operator $\operatorname{S}_{\partial}$ we can see that in the region $r\geq 3$ these section satisfy the differential equation 
$$-(\partial_r)^2\alpha_l+(1-\chi_{(-\epsilon,\epsilon)}(\lambda))\lambda^2 \alpha_l=0$$
\noindent with solutions in the general form
$$\alpha_l(x,r)=r\beta_{l,1}(x)+\beta_{l,2}+O(e^{-\epsilon r})$$ and $\beta_{l,i}\in \chi_{(-\epsilon,\epsilon)}(\operatorname{S}_{\partial}).$
\noindent 
Keeping in mind the identities
$d+d^*=d_{\epsilon}+d_{\epsilon}^*$ with $d_{\epsilon}:=d-d\theta \Pi_{\epsilon}$ and $d_{\epsilon}^*:=d^*-d^*\theta \Pi_{\epsilon},$
using the identity $(1-\Pi_{\epsilon})\beta_{0,j}=0$
\begin{align*}
d_{\epsilon}\alpha_0(x,r)&=(\epsilon(dr)\partial_r+d(1-\Pi_{\epsilon}))(r\beta_{0,1}(x)+\beta_{(0,2)}(x)+O(e^{-\epsilon r})
\\&=dr\wedge \beta_{0,1}(x)+O(e^{-\epsilon r}).
\end{align*} 
\noindent The calculation to show that the second piece $d_{\epsilon}\alpha_1(x,r)\wedge dr=O(e^{-\epsilon r})$ can be performed in the same way.

\noindent For the second piece of the signature operator
\begin{align*}
d_{\epsilon}\alpha_1(x,r)\wedge dr =(-\iota(dr)\partial_r+d^*(1-\Pi_{\epsilon}))(r\beta_{1,1}\wedge dr+\beta_{1,2}(x)+O(e^{-\epsilon r}))\\=-(-1)^{|\beta_{1,1}|}\beta_{1,1}(x)+O(e^{-\epsilon r})
\end{align*} with $d_{\epsilon}\alpha_0(x,r)=e^{-\epsilon r}$.
This shows that
$$\mathcal{J}(\xi)=\mathcal{J}(d_{\epsilon}\alpha+d^*_{\epsilon}\alpha)=0\oplus (-1)^{|\beta_{0,1}|}\beta_{0,1}+(-1)^{|\beta_{1,1}|}\beta_{1,1}(x,r)\oplus 0$$ and concludes the proof.
\end{proof}\noindent It remains to apply \ref{dim12} to prove \eqref{hacca}.
\begin{oss}
Everythig works with coeficients on a rank $m$ leafwise flat bundle, the signature formula in this case becames
$$\sigma_{\nu,\,\operatorname{an}}(X_0,\partial X_0)=m\langle L(X),[C_{\Lambda}]\rangle +1/2[\eta_{\Lambda}(D^{\mathcal{F}_{\partial}})].$$

\end{oss}
\subsection{The Hodge signature}
\noindent Consider the measurable field of Hilbert spaces of $L^2$--harmonic forms
$$x\longmapsto \mathcal{H}_x:= \operatorname{ker}\{\Delta^q_x:L^2(\Lambda^q T^*L_x)\longrightarrow L^2(\Lambda^q T^*L_x)\}$$ 
where $L_x$ is a leaf of the foliation on the manifold $X$ with cylindrical ends. Since leafwise harmonic forms are closed this is a field of subspaces of the fields of de Rham cohomologies $H^*(L_x)$ hence inherits the structure of a measurable field of Hilbert spaces futhermore it makes sence to speak about the space of tangentially continuous sections 
$\mathcal{H}_{\tau}^q$. 

\noindent So if the dimension of the foliation is \underline{$\operatorname{dim}(\mathcal{F})=4k$} as above, we have a well defined bilinear form on the middle--degree leafwise transversally continuous (transversally measurable would be enough)
\begin{equation}\label{quadratic}\operatorname{s}^{\ty}_{\Lambda}:\mathcal{H}_{\tau}^{2k}\times \mathcal{H}_{\tau}^{2k}\longrightarrow \mathbb{C}, \,\,(\alpha,\beta)\longmapsto \int_{X}\alpha \wedge \beta d\Lambda=\int_{X}(\alpha,\ast\beta) d\Lambda.\end{equation}
given by the wedge product followed by integration against the transverse measure.
This bilinear form is defined on forms (and there is simmetric) with real coefficients and extended to be sesquilinear 
($\mathbb{C}$--antilinear in the second variable) on forms with complex coefficients in the usual way, 
$\operatorname{s}^{\ty}_{\Lambda}(\alpha,\beta \otimes \gamma):=\bar{\gamma}\operatorname{s}^{\ty}_{\Lambda}(\alpha,\beta \otimes \gamma)$. For sesquilinear forms to be simmetric means 
$$\overline{\operatorname{s}^{\ty}_{\Lambda}(\alpha,\beta)}=\operatorname{s}^{\ty}_{\Lambda}(\beta,\alpha).$$
 This field of bilinear forms corresponds, by Riesz Lemma to a continuous (measurable) field of self--adjoint bounded operators $A_x:\mathcal{H}_{\tau,x}^{2k}\longrightarrow \mathcal{H}_{\tau,x}^{2k}$ univocally defined by the property 
 $$\operatorname{s}^{\ty}_{\Lambda}(\alpha,\beta)=(\alpha, A\beta)$$ where at the right--hand side the scalar product of the field of Hilbert spaces i.e., the $L^2$ scalar product on forms. Now $A$ determines a field of orthogonal splittings $\mathcal{H}_{\tau,x}^{2k}=V_x^+\oplus V_x^0 \oplus V_x^-$ of Hilbert spaces where $ V_x^{\pm}$ is the image of the spectral projection $\chi{(0,\ty)}(A_x)$ $(\chi{(-\ty,0)}(A_x))$ and $V^0_x$ is the kernel of $A_x$. The pairing on the leaf passing trough $x$ is non degenerate if and only if $A^0_x=0$ but we are interested in the general behaviour using the transverse measure to integrate. 
\begin{dfn}
The signature on harmonic forms (The Hodge signature or the harmonic signature) on the foliated elongated manifold is 
$$\sigma^{\ty}_{\Lambda}(X):=\operatorname{dim}_{\Lambda}V^+-\operatorname{dim}_{\Lambda}V^-.$$
\noindent We shall use also the symbol $\sigma_{\Lambda,\operatorname{Hodge}}(X,X_0)$ to refer to the compact pair, to denote the same number.
\end{dfn}

\subsection{\underline{Analytical signature}=\underline{Hodge signature}}
\noindent All the computations made in Proposition \ref{primocalcolo}, Lemma \ref{dim12} and Proposition \ref{density} leads to the first equality promised.
\begin{thm}
The analytical signature of the compact manifold with boundary and the signature on harmonic forms on the manifold with cylinder attached do coincide,
\begin{equation}\label{anasign}\sigma_{\Lambda,\operatorname{an}}(X_0,\partial X_0)=\sigma^{\ty}_{\Lambda}(X)=\langle L(X),[C_{\Lambda}]\rangle +1/2[\eta_{\Lambda}(D^{\mathcal{F}_{\partial}})].\end{equation}
\end{thm}
\begin{proof}
The definition \eqref{quadratic} says that $B=\ast_{|\Omega^{2k}}$ but since the dimension of the foliation is $4k$ we have $\tau_{|\Omega^{2k}}=\ast_{|\Omega^{2k}}$. It follows that $$V^{\pm}=\operatorname{ker}_{L^2}(D^{\operatorname{sign},\pm}).$$
\end{proof}
\subsection{The $L^2$--de Rham signature} The goal of this section is to prepare the ground for the definition of the de Rham signature for the foliated manifold with boundary and the proof of its coincidence with the harmonic signature.
\subsubsection{manifolds with boundary with bounded geometry}\label{bbg}
\noindent The generic leaf of $(X_0,\mathcal{F})$ is a Riemannian manifold with boundary with bounded geometry as those examined by Schick \cite{scht,sch1,sch2}.
\begin{dfn}\label{deltabounded}
We say that a manifold with boundary equipped with a Riemannian metric has bounded geometry if the following holds
\begin{description}
\item[Normal collar]: there exists $r_C>0$ so that the geodesic collar
$$N:=[0,r_{C})\times \partial M:(t,x)\longmapsto \operatorname{exp}_x(t\nu_x)$$
is a diffeomorphism onto its image, where $\nu_x$ is the unit inward normal vector at $x\in \partial M$. Equip $N$ with the induced metric. In the sequel $N$ and its image will be identified. Denote $\operatorname{im}[0,r_C/3)\times \partial M$ by $N_{1/3}$ and similarly $N_{2/3}$.
\item[Injectivity radius of $\partial M$]: the injectivity radius of $\partial M$ is positive, $r_{\operatorname{inj}}(\partial M)>0$
\item[Injectivity radius of $M$]: there is $r_i>0$ so that for $x\in M- N_{1/3}$ the exponential mapping is a diffeomorphism on $B(0,r_1)\subset T_xM$. In particular if we identify $T_xM$ with $\R^m$ via an orthonormal frame we have Gaussian coordinates $\R^m \supset B(0,r_i)\longrightarrow M$ around any point in $M- N_{1/3}$
\item[Curvature bounds]: for every $K\in \mathbb{N}$ there is some $C_K>0$ so that $|\nabla^i R|\leq C_K$ and $|\nabla^{\partial}l|\leq C_K$, $0\leq i\leq K$. Here $\nabla$ is the Levi--Civita connection on $M$, $\nabla^{\partial}$ is the Levi--Civita connection on $\partial M$ and $l$ is the second fundamental form tensor with respect to $\nu$.
\end{description}
\end{dfn}
\noindent Choose some $0<r_1^C<r_{\operatorname{inj}}(\partial M)$, near points $x' \in \partial M$ on the boundary one can define {\bf{normal collar coordinates}} by iteration of the exponential mapping of $\partial M$ and that of $M$,
$$k_{x'}:\underbrace{B(0,r_i^C)}_{\subset \R^{m-1}}\times [0,r_C)\longrightarrow M, (v,t)\longmapsto \operatorname{exp}^{M}_{\operatorname{exp}_{x'}^{\partial M}(v)}(t\nu).$$ For points $x\in M-N_{1/3}$ standard {\bf{Gaussian coordinates}} are defined via the exponential mapping. In the following we shall call both {\bf{normal coordinates}}. It is a non trivial fact that the condition on curvature bounds in definition \ref{deltabounded} can be substituted by uniform control of each derivative of the metric tensor $g_{ij}$ and its inverse $g^{ij}$ on normal coordinates. 

\noindent The definition extends to bounded geometry vector bundles on $\delta$--manifolds with bounded geometry and each object of uniform analysis like i.e. uniformly bounded differential operators can be defined \cite{sch2}.
In particular, using a suitable partition of the unity adapted to normal coordinates one can define uniform Sobolev spaces (different coordinates give equivalent norms so we get hilbertable spaces) and every basic result continues to hold.
\begin{prop}\label{schick}
Let $E\longrightarrow M$ a bundle of bounded geometry over $M$. Suppose $F$ is bounded vector bundle over $\partial M$. Then the following hold for the Sobolev spaces $H^s(E), H^t(F)$, $s,t\in \R$ of sections.
\begin{enumerate}
\item $H^s(E), H^t(F)$ is an Hilbert space (inner product depending on the choices).
\item The usual (bounded) Sobolev embedding theorem holds with values on the Banach space $C_b^k(E)$ of all sections with the first $k$ derivatives uniformly bounded,
$$H^s(E)\hookrightarrow C^k_b(E),\quad \mbox{whenever} \quad s>m/2+k.$$
\item For the bundle of differential forms one can use as Sobolev norm the one coming from the integral of the norm of covariant differentials
$$\|\omega\|_k^2:=\sum_{i=0}^{k}\int_M\|\nabla^i \omega(x)\|^2_{T^{*}_x M \otimes \Lambda T^*M}|dx|.$$
\item For $s<t$ we have a bounded embedding with dense image $H^t(E)\subset H^s(E)$. The map is compact if and only if $M$ is compact. We define
$$H^{\ty}(E):=\bigcap_s H^s(E),\quad H^{-\ty}(E):=\bigcup_s H^s(E).$$
\item Let $p:C^{\ty}(E)\longrightarrow C^{\ty}(F)$ a $k$--bounded boundary differential operator i.e the composition of an order $k$ bounded differential operator on $E$ with the morphism of restriction to the boundary. Then $p$ extends to be a bounded operator
$$p:H^s(E)\longrightarrow H^{s-k-1/2}(F),\quad s>k+1/2.$$
\noindent In particular we have the bounded restriction map $H^s(E)\longrightarrow H^{s-1/2}(E_{|\partial M}), s>1/2$.
\item $H^s(E)$ and $H^{-s}(E)$ are dual to each other by extension of the pairing
$$(f,g)=\int_Mg(f(x))|dx|; \, f\in C^{\ty}_0(E), \,g\in C^{\ty}_0(E^*)$$ where $E^*$ is the dual bundle of $E$.
\noindent
If $E$ is a bounded Hermitian or Riemannian bundle, then the norm on $L^2(E)$ defined by charts is equivalent to the usual $L^2$--norm
$$|f|^2:=\int_M(f,f)_x|dx|,\,f\in C^{\ty}_0(E).$$
Moreover $H^s(E)$ and $H^{-s}(E)$ are dual to each other by extension of $(f,g)=\int_M(f,g)_x|dx|.$
\end{enumerate}
\end{prop}
\subsubsection{Random Hilbert complexes}\label{derham}
\noindent Now we define the de Rham $L^2$ complexes along the leaves. These are particular examples of Hilbert complexes studied in complete generality in \cite{brun}.

\noindent So let $x\in X_0$, consider the unbounded operator with \underline{Dirichlet boundary conditions} $$d_{L_x^0}:\Omega^k_{d,x}=\{\omega \in C^{\ty}_0(\Lambda T^kL_x^0); \omega_{|\partial M}=0\}\subset L^2_x(\Lambda T^kL_x^0)\longrightarrow L^2_x(\Lambda T^kL_x^0).$$
\noindent Being a differential operator it is closable, let $A^k_x(L_x^0,\partial L_x^0)$ the domain of its closure i.e the set of $L^2$ limits $\omega$ of sequences $\omega_n$ such that also the $d\omega_n$ converges in $L^2$ to some $\eta=:d\omega$. 
The graph norm $\|\cdot\|_A^2:=\|\cdot \|_{L^2}^2+\|d\cdot \|_{L^2}^2$ gives the graph the structure of an Hilbert space making $d$ bounded. 
It is easily checked that $d(A^k_x)\subset \operatorname{ker}(d:A^{k+1}_x)\longrightarrow L^2_x)$ then we have a Hilbert cochain complex
$$\cdot\cdot\cdot \longrightarrow A^{k-1}_x\longrightarrow A^k_x \longrightarrow A^{k+1}_x\longrightarrow \cdot \cdot \cdot$$
\noindent with cycles $Z^k_x(L_x^0,\partial L_x^0):=\operatorname{ker}(d:A^k_x\longrightarrow A^{k+1}_x)$ and boundaries $B^k_x(L_x^0,\partial L_x^0):=\operatorname{range}(d:A^{k-1}_x\longrightarrow A^{k}_x).$
\begin{dfn}
 The $L^2$ (reduced )\footnote{the word reduced stands for the fact we use the closure to make the quotient, also the non reduced cohomology can be defined. For a $\Gamma$ covering of a compact manifold the examination of the difference reduced/unreduced cohomology leads to the definition of the \underline{Novikov--Shubin invariants} \cite{luk}} relative de Rham cohomology of the leaf $L_x^0$
 is defined by the quotients
 $$H^{k,x}_{dR,(2)}(L_x^0,\partial L_x^0):=\dfrac{Z^k_x(L_x^0,\partial L_x^0)}{\overline{B^k_x(L_x^0,\partial L_x^0)}}.$$
\end{dfn}
\noindent Clearly the closure is taken in order to assure the quotient to be an Hilbert space. 
\noindent Similarly the $L^2$--de Rham cohomology of the whole leaf, $H^{k,x}_{dR,(2)}(L_x^0)$ is defined using no (Dirichlet) boundary conditions. In particular $A_x^k(L_x^0)$ will be used to denote the domain of the closure of the differential as unbounded operator on $L^2(L_x^0)$ defined on compactly supported sections (the support possibly meeting the boundary). The subscript $dR$ helps to make distinction with Sobolev spaces.
\noindent 
Each one of this spaces is naturally isomorphic to a corresponding space of harmonic forms. More precisely
\begin{dfn}
The space of $k$--$L^2$
harmonic forms which fulfill Dirichlet boundary conditions on $\dl$ is 
$$\mathcal{H}^k_{(2)}(L_x^0,\dl):=\{\omega \in C^{\ty}\cap L^2,\, \omega_{|\dl}=0,\, (\delta \omega)_{|\dl}=0,\,\underbrace{(d\omega)_{|\dl}=0}_{\mbox{aut. satisfied}}  \}$$
\end{dfn}  \noindent We shall see that the boundary conditions are exactly the square of the Dirichlet boundary condition on the Dirac operator $d+\delta$.
\noindent Since each leaf is complete a generalization of an idea of Gromov shows that these forms are closed and co--closed, \cite{scht,sch1}
$$\mathcal{H}^k_{(2)}(L_x^0,\dl)=\{\omega \in C^{\ty}\cap L^2(\Lambda^kL_x^0), d\omega=0,\,\delta \omega=0,\,\omega_{|\dl}=0\}.$$ 
\noindent Furthermore there's the $L^2$--orthogonal Hodge decomposition \cite{scht,sch1}
$$L^2(\Lambda^k T^*L_x^0)=\mathcal{H}^k_{(2)}(L_x^0,\dl)\oplus \overline{d^{k-1}\Omega_{d,x}^{k-1}(L_x^0,\dl)}^{L^2}\oplus \overline{\delta^{k+1}\Omega_{\delta,x}^{k+1}(L_x^0,\dl)}^{L^2} $$ 
where 
$\Omega_{d,x}^{k-1}:=\{\omega \in C^{\ty}_0(\Lambda^{k-1} T^*L_x^0),\,\omega_{|\dl}=0\}$
and the corresponding one for $\delta$ with no boundary conditions
$\Omega_{\delta,x}^{k+1}:=\{\omega \in C^{\ty}_0(\Lambda^{k+1} T^*L_x^0)\}.$
 These decompositions shows with a little work that the inclusion $\mathcal{H}^k(L_x^0,\dl)\hookrightarrow A^k_x$ induces isomorphism in cohomology (Hodge--de Rham Theorem)
$$\mathcal{H}^k(L_x^0,\dl)\cong H_{dR,(2)}^k(L_x^0,\dl).$$ This is a consequence of the fact that the graph norm (of $d$) and the $L^2$ norm coincide on the space of cycles $Z^k_x$.  Similar Hodge isomorphisms holds for the non--relative spaces and are well--known in literature.

\noindent As $x$ varies in $X_0$ they form measurable fields of Hilbert spaces. We discuss this aspect in a slightly more general way applicable to other situations. Remember that a measurable structure on a field of Hilbert spaces over $X_0$ is given by a fundamental sequence of sections, $(s_x)_{x\in X_0}$, $s_n(x)\in H_x$ such that $x\longmapsto \|s_n(x) \|_{H_x}$ is measurable and $\{s(x)\}_n$ is total in $H_x$ (see chapter IV in \cite{take} ). \begin{prop}If for a family of closed densely defined operators $(P_x)$ with minimal domain $\mathcal{D}(P_x)$ a fundamental sequence $s_n(x) \in \mathcal{D}(P_x)$ is a core for $P_x$ and $P_x s_n(x)$ is measurable for every $x$ and $n$ then the family $P_x$ is measurable in the sense of closed unbounded operators (definition \ref{messa} and the remark below ) i.e. the family of projections $\Pi_x^g$ on the graph is measurable in the square field $H_x\oplus H_x$ with product measurable structure.
\end{prop}
\begin{proof}\noindent It is trivial in fact the graph is generated by vectors $(s_n(x),P_x s_n(x))$ then the projections is measurable.
\end{proof}
\noindent The lemma above can be applied to the $(A^{k}_x(L_x^0,\dl))_x$ in fact in the appendix of \cite{hl} a fundamental sequence $\varphi_n$ of sections with the property that each $(\varphi_n(\cdot))_{|L_x^0}$ is smooth and compactly supported is showen to exist. Now the same proof works for manifold with boundary and, since the boundary has zero measure one can certainly require to each $\varphi_n$ to be zero on the boundary.

\bigskip

\noindent In particular we have defined complexes of square integrable representations. Reduction modulo $\Lambda$--a.e. gives complexes of random Hilbert spaces (with unbounded differentials) for which we introduce the following notations, 
\begin{itemize}
\item $(L^{2,\mathcal{F}}(\Omega^{\bullet}X_0),d)$ is the complex of Random Hilbert spaces obtained by $\Lambda$ a.e. reduction from the field of Hilbert complexes
\begin{equation}\label{sssd}
\xymatrix{
{\cdot\cdot\cdot} \ar[r]
&L^2(\Lambda^k T^*L_x^0)\ar[r]^-d
& L^2(\Lambda^{k+1}T^* L_x^0)\ar[r]&\cdot \cdot \cdot
}
\end{equation}
\item $(H^{\bullet,\mathcal{F}}_{dR,(2)}(X_0),d)$ is the complex of Random Hilbert spaces obtained by $\Lambda$ a.e. reduction from the reduced $L^2$ cohomology of
\eqref{sssd}
\item $(L^{2,\mathcal{F}}(\Omega^{\bullet}X_0,\partial X_0),d)$ is the complex of Random Hilbert spaces obtained by $\Lambda$ a.e. reduction from the field of Hilbert complexes with Dirichlet boundary condition
\begin{equation}\label{ssssd}
\xymatrix{
{\cdot\cdot\cdot} \ar[r]
&L^2(\Lambda^k T^*L_x^0)\ar[r]^-d
& L^2(\Lambda^{k+1}T^* L_x^0)\ar[r]&\cdot \cdot \cdot
}
\end{equation}
with differentials considered as unbounded operators with domains $A_x^{k}(L_x^0,\partial L_x^0).$
\item $(H^{\bullet,\mathcal{F}}_{dR,(2)}(X_0,\partial X_0),d)$ is the complex of Random Hilbert spaces of the cohomologies of the above complex.
\end{itemize}
\subsubsection{Definition of the de Rham signature}
\noindent Let $\operatorname{dim}(\mathcal{F})=4k$ Consider the measurable field of Hilbert spaces $A^k_{x}(L_x^0,\dl)$ of the minimal domains of the de Rham leafwise differential with Dirichlet boundary conditions $\omega_{| \dl}=0$ as above, with the graph Hilbert structure and the induce Borel structure. This square integrable representation of $\mathcal{R}_0$ carries a field of bounded symmetric sesquilinear forms defined by
$$s^0_x:A^{2k}_{x}(L_x^0,\dl)\times A^{2k}_{x}(L_x^0,\dl)\longrightarrow \mathbb{C},(\omega,\eta)\longmapsto \int_{L_x^0} \omega \wedge \overline{\eta}=\int_{L_x^0}(\omega,\ast \eta)d\nu^x$$ i.e. the $\mathbb{C}$--antilinear in the second variable extension of the wedge product on forms , $\overline{\sigma \otimes \gamma}=\sigma \otimes \bar{\gamma}$ is the complex conjugate and $\nu^x$ is the Leafwise Riemannian metric. Note that also the scalar product $(\cdot,\cdot)$ on forms is extended to be sesquilinear.
\begin{lem}
The sesquilinear form $s^0_x$ passes to the $L^2$ relative cohomology of the leaf $H^{2k}_{dR,(2)}(L_x^0,\dl)$ factorizing through the image of the map $H^{2k}_{dR,(2)}(L_x^0,\dl)\longrightarrow H^{2k}_{dR,(2)}(L_x^0)$ of the $L^2$ relative de Rham cohomology to the $L^2$ de Rham cohomology exactly as in the compact (one leaf) case.
\end{lem}
\begin{proof}
The first assertion is simply Stokes theorem, in fact let 
$\omega \in A^{2k}_x(L_x^0,\dl)$ i.e. 
$\xymatrix{{\omega}_n\ar[r]^{L^2}&\omega}$, 
$\xymatrix{{d\omega}_n\ar[r]^{L^2}&0       }$ and 
$\theta_m\in C^{\ty}_0(\Lambda T^{2k-1}L_x^0)$, 
$\xymatrix{d\theta_m \ar[r]^{L^2}&\varphi}$ then 
$$s_x^0(\omega,\varphi)=\lim_{n,m}\int_{L_x^0}\omega_n \wedge \overline{d\theta_m}=\lim_{n,m}\int_{L_x^0}d(\omega_n \wedge \overline{\theta_m} )=\lim_{n,m}\int_{\dl} (\omega_n\wedge \theta_m)_{|\dl}=0.$$ The second one is clear and follows exactly from the classical case i.e. if $\beta_1=\beta_2+\lim_n d\rho_n$ with $\rho_n$ compactly supported with no boundary conditions write $$s_x^0([\alpha],[\beta])=s_x^0([\alpha],[\beta_2])+\lim_n \int \alpha \wedge \rho_n,$$ represent $\alpha$ as a $L^2$ limit of forms with Dirichlet boundary conditions than apply Stokes theorem again.
\end{proof}
\noindent For every $x$ the sesquilinear form $s_x^0$ on the cohomology corresponds to a bounded selfadjoint operator $B_x\in B(H_{dR,(2)}^{2k}(L_x^0,\dl     ))$ (a proof is in \cite{Reed}) univocally determined by the condition $s_x^0(\alpha,\beta)=(\alpha, B_x \beta)$. Measurability properties of $(s_x^0)_{x\in X_0}$ are by definition (for us) measurability properties of the family $(B_x)_x.$ It is clear that everything varies in a Borel fashion (use again a smooth fundamental sequence of vector fields as in \cite{hl}) then the $B_x$'s define a self--adjoint random operator $B\in \operatorname{End}_{\Lambda}(H_{dR,(2)}^{2k}(X_0,\partial X_0))$.

\begin{dfn}The $\Lambda$--$L^2$ de Rham signature of the foliated manifold $X_0$ with boundary $\partial X_0$ is
$$\sigma_{\Lambda,dR}(X_0,\partial X_0):=\operatorname{tr}_{\Lambda}\chi_{(0,\ty)}(B)-\operatorname{tr}_{\Lambda}\chi_{(-\ty,0)}(B)$$ as random operators in $\operatorname{End}_{\Lambda}(H_{dR,(2)}^{2k}(X_0,\partial X_0)).$
\end{dfn}
\subsection{\underline{$L^2$ de Rham signature} =\underline{Harmonic signature} }
\noindent This is a very long proof. We need some new tools. The path to follow is clearly the one in the paper of of L{\"u}ck and Schick \cite{lus}. We shall show at the end of the section that we can reduce to the case in which
{\bf{every leaf meets the boundary}} or in other words {\bf{the boundary contains a complete transversal}}.
\subsubsection{The boundary foliation and $\mathcal{R}_0$ }\label{fundationtr}\label{primac}
\noindent We have denoted by $\mathcal{F}_{\partial}$ the foliation induced on the boundary $\partial X_0$ i.e. the foliation where a leaf is a connected component of the intersection of a leaf $L$ of $\mathcal{F}$ with the boundary. Let $\underline{\mathcal{R}_0}=\mathcal{R}(\mathcal{F}_{\partial})$ its equivalence relation with canonical inclusion $\underline{\mathcal{R}_0}\longrightarrow \mathcal{R}_0$. We are under the assumption that the boundary contains a complete transversal $T$. This is also a complete transversal for $\mathcal{F}_{\partial}$, Call $\nu_T$ its characteristic function on $\mathcal{R}_0$. Then every transverse measure $\Lambda$ for $\mathcal{R}_0$ is univocally determined by the measure $\Lambda_{\nu_{T}}$ supported on $T$. As a consequence (Theorem \ref{pocahontas}) one gets a transverse measure, continue to call $\Lambda$, on $\underline{\mathcal{R}_0}$. \noindent Let now $(H,U)$ be a square integrable representation of $\mathcal{R}_0\longrightarrow X_0$ and $H$ its corresponding random Hilbert space, it pulls back to a square integrable representation $(H',U')$ of $\underline{\mathcal{R}_0}$. Also a random operator $A\in \operatorname{End}_{\Lambda}(H)$ defines by restriction a random operator $A'$ in $\operatorname{End}_{\Lambda}(H')$.
We are going to show that 
\begin{equation}\label{andrea}
\operatorname{tr}_{\Lambda}(A)=\operatorname{tr}_{\Lambda}(A').
\end{equation}
\noindent This automatically proven if we think about the trace in terms of operator valued weight $\int \operatorname{tr}_{H_x}(\cdot)d\Lambda_{\nu}(x)$, of course we have to pay some care checking the domains of definitions of the two traces but from normality and square integrability the operators in form $\theta_{\nu}(\xi,\xi)$ as in the equation \eqref{novembre} furnish a sufficiently rich set to check the two. To see the problem under a slightly different point of view, first remember the trace is related to an integration of a Random variable. So suppose $F$ is a Random variable on $\mathcal{R}_0$ (see the definition on section \ref{paoletto}) the recipe Connes gives to integrate is the following: choose \underline{arbitrarily} some faithful transverse function $\nu \in \mathcal{E}^+$ then the integral is given by
$$\int F d\Lambda=\sup\Big{\{}\Lambda_{\nu}(\alpha(f)),f \in \mathcal{F}^+(X), \nu \ast f\leq 1\Big{\}}$$ where $X=\bigcup_{x\in X_0}F(x)$.
\noindent The point (leading to the treatment in Moore and Schochet \cite{MoSc}, for example) is that one can choose as $\nu$ the characteristic funtion $\nu_T$ of a complete transversal $T$. Then if $f\in \mathcal{F}^+(X)$ such that $\nu_T\ast f \leq 1$ also $\underline{f}$, the restriction of $f$ on the space $X'=\bigcup_{t\in T}F(t)$ satisfies $\nu_t \ast f\leq 1$ and 
$$\Lambda_T(\alpha(f))=(\mu \circ \nu)(\alpha(f)=\int_T\nu^t\alpha^t(f)d\mu(t)=\Lambda_T(\alpha(\underline{f}))$$ (where $\nu$ and $\mu$ decompose $\Lambda_{\nu_T}$ on $T$ \ref{pocahontas}) then 
$$\sup\Big{\{} \Lambda_{\nu}(\alpha(f)):\nu_T\ast f\leq 1,\,f\in \mathcal{F}^+(X)\Big{\}}\leq \sup\Big{\{} \Lambda_{\nu}(\alpha(f)):\nu_T\ast f\leq 1,\,f\in \mathcal{F}^+(X')\Big{\}}.$$ The reverse inequality is simpler starting with a function $\underline{f}\in \mathcal{F}^+(X')$, $\nu \ast \underline{f} \leq 1$  and considering the function $f:=\underline{f}$ on $\pi^{-1}(T)$ and $0$ elsewhere. ($\pi$ is the natural projection on the space $X$).

\noindent This simple argument allows ourselves to consider, as an instrument short sequences  
$$\xymatrix{
0\ar[r]&A_x^{k-1}(L_x^0,\partial L_x^0)\ar[r]^-i&A_x^{k-1}(L_x^0)\ar[r]^-r &A_x^{k-1}(\partial L_x^0)\ar [r]&0
}$$
where $x\in \partial X_0$ as sequences of Random Hilbert spaces associated to $\underline{\mathcal{R}_0}.$ In fact the third term seems not naturally defined without passing to the boundary relation. Formula \eqref{andrea} says that {\bf{traces of morphisms coming from the whole relation can be computed in the equivalence relation restricted to the boundary.}} Now the third term of the sequence is directly related to the boundary.
\noindent {\bf{Proper Functors}} Again some words on the relation between $\underline{\mathcal{R}_0}$ and ${\mathcal{R}_0}_{|\partial X_0}$ or, better its restriction to the boundary $(\mathcal{R}_0)_{|\partial X_0}$ . We shall investigate how $\underline{\mathcal{R}_0}$ sits inside $\underline{\mathcal{R}_0}$ and the traces on algebras associated on it. Consider a class of $\underline{\mathcal{R}_0}$ i.e. a leaf of the boundary foliation; this is a connected component of a class of $(\mathcal{R}_0)_{|\partial X_0}$. In other words each class of $(\mathcal{R}_0)_{|\partial X_0}$ is a denumerable union of classes of $\underline{\mathcal{R}_0}$ i.e. the bigger one seems like to be some sort of denumerable union of the smaller under the obvious natural functor
$$\underline{\mathcal{R}_0}\longrightarrow (\mathcal{R}_0)_{|\partial X_0.}$$ In the measure theory realm denumerability means that $(\mathcal{R}_0)_{|\partial X_0}$ is not so bigger than $\underline{\mathcal{R}_0}$.
Also if one makes use of a complete transversal for $\underline{\mathcal{R}_0}$ to integrate natural\footnote{i.e. given by $L^2\bullet L$, where $L$ is left traslation on $\mathcal{R}_0$} Random Hilbert spaces associated to $(\mathcal{R}_0)_{|\partial X_0}$ this transversal touches denumerably times classes of $\underline{\mathcal{R}_0}$ so in definitive the geometric intuition says that \emph{we are integrating (then taking traces) on the foliation induced on the boundary !} 
 the notion of \underline{properness} helps to understand this intuitive fact.
\begin{dfn}
\begin{itemize}
%$$\xymatrix{
%{\mathcal{G}}\ar[r]^{F}\ar[d]^r&  X:=\bigcap_{x\in \mathcal{G}^{(0)}}\mathcal{G}^x\ar[dl]^\pi\\
%\mathcal{G}^{(0)}}$$

\item A measurable functor $F:\mathcal{G}\longrightarrow M$ with values standard measure spaces is called \underline{proper} if w.r.t the diagram 
$$\xymatrix{
{\mathcal{G}}\ar[r]^-{F}\ar[d]^r&  X:=\bigcup_{x\in \mathcal{G}^{(0)}}\mathcal{G}^x\ar[dl]^\pi\\
\mathcal{G}^{(0)}}$$

$\mathcal{G}$ acts properly on $X$ i.e. there exist a strictly positive function $f\in\mathcal{F}^+(X)$ and a proper $\nu\in \mathcal{E}^+$ such that $\nu \ast f=1$. Here we recall the defining formula $$(\nu \ast f)(z):=\int_{\mathcal{G}^y} f(F({\gamma}^{-1})\cdot z)d\nu^y(\gamma).$$
\item We say that a Borel groupoid $\mathcal{G}$ is proper if it acts properly on itself. 
\end{itemize} 
\end{dfn}
There is a proposition (Lemme 2 \cite{Cos}) saying that properness is equivalent to each one of the following conditions
\begin{itemize}
\item There is one $\nu \in \mathcal{E}^+, \,f\in \mathcal{F}^+(X): \nu \ast f=1$.
\item For every faithful $\nu \in \mathcal{E}^+$ there is some $f\in \mathcal{F}^+(X): \nu \ast f=1$.
\item The kernel on $X$ defined by $z\longmapsto \rho^z$, $\rho^z(f):=\int f(F(\gamma^{-1})\cdot z)d\nu^{y}(\gamma)$ is proper for faithful $\nu$.
\end{itemize}
\noindent Properness is a very strict condition, there's a big literature on the subject, one can consult for example the paper of Renault and Delaroche \cite{RDL}. In some sense the existence of a strictly positive function $f$ as in the definition provides an embedding of $X\longmapsto \mathcal{G}$ at level of $L^\ty$ functions\footnote{for a good survey on non commutative integration theory and a physical application one can consult \cite{lenz} where properness of Random variables is essential}
 in fact it defines the fibrewise map $q:\mathcal{F}(\mathcal{G})\longrightarrow \mathcal{F}(X)$, $$q(u)(p):=\int_{\mathcal{G}^{\pi(p)}}f(  F(\gamma^{-1})\cdot p))u(\gamma)d\nu^{\pi(p)}(\gamma),$$ with the essential unital property $q(1_{\mathcal{G}})=1_X.$
The integration formula of a proper functor $F$ against some transverse measure $\Lambda$ becomes very simple
$$\int F d\Lambda=\Lambda_{\nu}(\alpha(f))$$ where $\nu$ and $f$ are as in the definition. This is also related to the natural trace on the Von Neumann algebra of Endomorphisms. of the square integrable representation $L^2\bullet F$.
We recall that $L^2\bullet F$ is nothing but the composition of the right translation functor with a measurable proper functor $F.$
 In fact, let $A\in \operatorname{End}_{\Lambda}(L^2\bullet F)^+$ a positive intertwining operator and consider the new proper functor $F_A$ associating to $x\in \go$ the measure space $F(x)$ and the measure, called \underline{local trace} $\mu_x(g):=\operatorname{Trace}_{L^2(F(x),\alpha)}(A^{1/2}gA^{1/2})$ while an arrow $\gamma \in \gr$ goes into the measure space isomorphism induced by left translation then Connes proves 
\begin{equation}\label{usadopo}\operatorname{Tr}_{\Lambda}(A)=\int F_Ad\Lambda=\Lambda_{\nu}(\mu(f))=\int_{\go}\mu_x(f)d\Lambda_{\nu}(x).\end{equation}
\begin{oss}
\noindent Almost by definition the measurable functor $L$ \underline{left traslation} given by $x\in \go \longmapsto \gr^x$ and $\gamma \longmapsto L_{\g}:\gr^{s(\g)}\longrightarrow \gr^{r(\g)}$ is \underline{proper}.
\end{oss}
\noindent We shall apply at once the remark above and the following proposition by Connes (Proposition 4. in \cite{Cos}).
\begin{prop}\label{circolo}
Let $F,F'$ be two measurable functors with values measure spaces as in the diagram
$$\xymatrix{
{\bigcup_{x\in \go}F'(x)}=X'\ar[dr]^\pi& \gr \ar[l]^-{F'}\ar[r]^-{F}\ar[d]&X=\bigcup_{x\in \go}F(x)\ar[dl]^\pi\\
{}&\go&{}
,}$$ suppose the existence of a measurable mapping associating to $z\in X$ a probability measure $\lambda^z$ on $X'$ supported on $F'(\pi(z))$ that's natural i.e. $\lambda^{\gamma \cdot z}=\gamma \lambda^z$, $\forall \gamma$, $\forall \gamma:s(\gamma)=\pi(z).$
Then \begin{enumerate}
\item If $F'$ is proper then $F$ is proper
\item If $F'$ is proper (then also $F$) then 
$$\int \lambda^z d\alpha^x(z)=\alpha'^x,\forall z \Longrightarrow \int Fd\Lambda=\int F' d\Lambda.$$
\end{enumerate}
\end{prop}
\noindent We shall apply this in the most simple way to some concrete situation of $L^2.$ So let $L'$ be the left multiplication functor $x\longmapsto (\mathcal{R}_0)_{|\partial X_0}^x$ while $L$ is left traslation in $\mathcal{R}_0$. In symbols,
$$\xymatrix{
{\underline{{\mathcal{R}_0}}}\ar[d]^{L'}\ar[r]^{L'}&X'=\bigcup_{x\in \partial X_0}(\mathcal{R}_0)^x_{|\partial X_0}\\
X=\bigcup_{x\in \partial X_0}\mathcal{R}_0^x}.$$ Both are proper because the first is the restriction of the multiplication of $\mathcal{R}_0$ the second is multiplication. Associate to $L$ and $L'$  some local trace of an intertwining operator $B$ of a s.i. representation, say $x\longmapsto L^2(\partial L_x^0)$. We are saying that the target space $L'(x)$ is $(\mathcal{R}_0)^x_{|\partial X_0}$ and the measure is $f\longmapsto \alpha(f)=\operatorname{tr}(B^{1/2}fB^{1/2})$; the same association is done for $L$. Note that the integral $\int L' d\Lambda$ is exactly $\operatorname{tr}_{\Lambda}(B)$ in $\operatorname{End}_{\Lambda}(L^2(\partial L_x^0)).$ Now the Borel map associating to $z\in X$ a probability measure on $X'$ in proposition \ref{circolo}
 can be taken as the Dirac mass i.e. $$z=(x,y)\longmapsto\lambda^z:=\delta_{(x,y)}.$$ Naturality and measurability are obiouvsly verified. Let's check the integral condition, so take a function $f$ in $(\mathcal{R}_0)^x_{\partial X_0}=\partial L_x^0$
\begin{align*}
\Big{\langle} \int \lambda(z)d{\alpha}^x(z),f\Big{\rangle} &=\Big{\langle} \int \delta_{(x,y)} d {\alpha}^x ((x,y)), f(x,y) \Big{\rangle}
= 
\sum_{C\in \mathcal{C}}\operatorname{tr}_{L^2(C)}(f^{1/2}_{|C}\,B_x\,f^{1/2}_{|C})\\&=
\operatorname{tr}_{L^2(\partial L_x^0)}(f^{1/2}\,B_x\,f^{1/2})={\alpha'}^x(f),
\end{align*}
\noindent in fact while $y$ runs trough $(\mathcal{R}_0^x)=\bigcup_{\mathcal{C}} C$ (connected components i.e classes of $\underline{\mathcal{R}_0}$) the factor $\delta_{(x,y)}$ selects exactly the connected component it belongs to. Summarizing: 
\pecetta{
\noindent {{a Random operator associated to}} $({\mathcal{R}_0})_{|\partial X_0}$ {{restricts to a Random operator associated to}} $\underline{\mathcal{R}_0}$ {{having the same trace.}}}

\subsubsection{Weakly exact sequences}
\noindent Consider for $x\in \partial X_0$ the Borel field of cochain complexes
$$\xymatrix{{}& {}\ar[d]^-d& {}\ar[d]^-d& {}\ar[d]^-d\\
0\ar[r]&A_x^{k-1}(L_x^0,\partial L_x^0)\ar[r]^-i\ar[d]^-d & A_x^{k-1}(L_x^0)\ar[r]^-r\ar[d]^-d & 
A_x^{k-1}(\partial L_x^0)\ar[r]\ar[d]^-d& 0\\
0\ar[r]& A_x^{k}(L_x^0,\partial L_x^0)\ar[r]^-i\ar[d]^d & A_x^{k}(L_x^0)\ar[r]^-r\ar[d]^-d & 
A_x^{k}(\partial L_x^0)\ar[r]\ar[d]^-d& 0\\
{}&{}&{}&{}}$$ where each morphism must be considered as an unbounded operator on the corresponding $L^2$, $i$ is bounded since is merely the restriction of the identity mapping on $L^2(L_x^0,\Lambda T^*L_x^0)$ and $r$ is restriction to the boundary.
\begin{prop}
\begin{enumerate}
\item For every $k$ the domain $A_x^{k}(L_x^0)$ is contained in the Sobolev space of forms $H^1(L_x^0,\Lambda T^*L_x^0)$. In particular the composition with $r$ makes sense.
\item The rows are (weakly) exact i.e. one has to consider the closure of the images of $i$ and $r$ in the $L^2$ topology in the $A_x^k$'s. 
\end{enumerate}
\end{prop}
\begin{proof}
1. An element $\omega$ in $A_x^{k}(L_x^0)$ is an $L^2$--limit of smooth compactly supported forms $\omega_n$ with differential also convergent in $L^2$. Then since the Hodge $\star$ is an isometry on $L^2$ also $\delta \omega_n=\pm \ast \omega \ast$ converges. In particular we can control the $L^2$--norm of $d\omega$ and $\delta \omega$; this means we have control of the first covariant derivative, in fact $d+\delta=c \circ \nabla$ where $c$ is the (unitary) Clifford action. Then the second term can made less that the norm of $\nabla$ by bounded geometry.
In particular we have control on the order one Sobolev norm by proposition \ref{schick}. The remaining part follows from the fact that the restriction morphism is bounded from $H^1$ to $H^{1/2}\hookrightarrow L^2$. 

\noindent 2. The only non--trivial point is exactness in the middle but as a consequence of the bounded geometry the boundary condition on the first space extends to $H^1$ (see proposition $5.4$ in the thesis of Thomas Schick \cite{scht}) that together with point $1.$ is exactness.
\end{proof}
\begin{oss}\label{oo}Note that the proof of the proposition above says also that the induced morphisms $i_*$ and $r_*$ are bounded.\end{oss}
\begin{dfn}We introduce the notations $L^{2,\mathcal{F}}(\Omega^{\bullet} X_0)$ and $L^{2,\mathcal{F}}(\Omega^{\bullet}X_0,\partial X_0)$ for the complexes of Random Hilbert spaces with unbounded differential introduced above.
\end{dfn}
\noindent Every arrow induces morphisms on the reduced $L^2$ cohomology. Miming the algebraic construction of the connecting morphism (everything works thanks to the remark above) we have, for every $x\in \partial X_0$ the long sequence of square integrable representations of the equivalence relation of the boundary foliation $\underline{\mathcal{R}_0}$
\begin{equation}\nonumber
\xymatrix{{\cdot \cdot \cdot}\ar[r]& H^{k,x}_{dR,(2)}(L_x^0,\partial L_x^0)\ar[r]^-{i_{*}} & 
H^{k,x}_{dR,(2)}(L_x^0)\ar[r]^-{r_{*}}&\\
{}&\ar[r]^-{r_{*}}&
H^{k,x}_{dR,(2)}(\partial L_x^0)\ar[r]^-{\delta}&
H^{k-1,x}_{dR,(2)}(L_x^0,\partial L_x^0)\ar[r]& \cdot \cdot \cdot}\end{equation}
\noindent Remove the dependence on $x$ to get a long sequence of Random Hilbert spaces over $\partial X_0$ with consistent notation with 
\eqref{sssd} and \eqref{ssssd}
\begin{equation}\label{success}
\xymatrix{{\cdot \cdot \cdot}\ar[r]& H^{k}_{dR,(2)}(X_0,\partial X_0)\ar[r]^-{i_{*}} & 
H^{k}_{dR,(2)}(X_0)\ar[r]^-{r_{*}}&\\
{}&\ar[r]^-{r_{*}}&
H^{k}_{dR,(2)}(\partial X_0)\ar[r]^-{\delta}&
H^{k-1}_{dR,(2)}(X_0,\partial X_0)\ar[r]& \cdot \cdot \cdot}
\end{equation}
\begin{dfn}\label{vonexact}
We say that a sequence of Random Hilbert spaces as \eqref{success} is $\Lambda$--weakly exact at some term if in the correspondig Von Neumann algebra of Endomorphisms the projection on the closure of the range of the incoming arrow coincide with the projection on the kernel of the starting one. These means i.e at point
$\xymatrix{{}\ar[r]^-{i^*}&H_{dR,(2)}^k(X_0)\ar[r]^-{r^*}&{}}$, 
$$\overline{\operatorname{range} i^*}=\operatorname{ker} i^*\in \operatorname{End}_{\Lambda}(H_{dR,(2)}^k(X_0)).$$
\end{dfn}
\noindent Of course such a sequence cannot be exact, just as in the case of Hilbert $\Gamma$--modules there are simple examples of non exacteness (see Example 1.19 in \cite{luk}, or the example on manifolds with cylindrical ends in the paper of Cheeger and Gromov \cite{ChGr}).
An necessary condition to weakly exactness is \underline{(left) fredholmness}, as in \cite{ChGr}.

\subsubsection{Spectral density functions and Fredholm complexes.}\label{1237}
\noindent Let $U$,$V$ two Random Hilbert spaces on $\mathcal{R}_0$ (for these consideration also the holonomy groupoid or, more generally a Borel groupoid would work) and an \underline{unbounded} 
\underline{Random} \underline{operator}
$f:\mathcal{D}(f)\subset U\longrightarrow V$ i.e start with a Borel family of closed densely defined operators $f_x:U_x\longrightarrow V_x$ intertwining the representation of $\mathcal{R}_0$. Since $f$ is closable, the question of measurability is addressed in definition \ref{messa}. For every 
$\mu\geq 0$ put 
$\mathcal{L}(f,\mu)$ as the set of measurable fields of subspaces $L_x\subset \mathcal{D}(f_x)\subset U_x$ (measurability is measurability of the family of the closures) such that for every $x\in X_0$ and $\phi\in L_x$, 
$
\|f_x(\phi)\|\leq \mu \|\phi\|$. After reduction modulo $\Lambda$ a.e. this becomes a set of Random Pre--Hilbert spaces we call 
$\mathcal{L}_{\Lambda}(f,\mu)$
\begin{dfn}\label{frett}
The $\Lambda$--spectral density function of the family $\{f_x\}_x$ is the monotone increasing function
$$\mu \longmapsto F_{\Lambda}(f,\mu):=\sup \{\operatorname{dim}_{\Lambda}:L\in \mathcal{L}_{\Lambda}(f,\mu)\}.$$ Here of course one has to pass to the closure in order to apply the $\Lambda$--dimension. We say $f$ is \underline{$\Lambda$ Fredholm} if for some $\epsilon>0$, $F_{\Lambda}(f,\epsilon)<\ty$
\end{dfn}
\noindent We want to show that this definition actually coincides with the definition given in term of the spectral measure of the positive self--adjoint operator $f^*f$.
\begin{lem}\label{LL}
In the situation above $$F_{\Lambda}(f,\mu)=\operatorname{tr}_{\Lambda}\chi_{[0,\mu^2]}(f^*f)=\operatorname{dim}_{\Lambda}\operatorname{range}(\chi_{[0,\mu^2]}(f^*f))$$
as a projection in $\operatorname{End}_{\Lambda}(U).$

\noindent Notice that since $f^*f$ is a positive operator $\chi_{[0,\mu^2]}(f^*f)=\chi_{(-\ty,\mu^2]}(f^*f)$ is the spectral projection associated to the spectral resolution $f^*f=\int_{-\ty}^{\ty}\mu d\chi_{(-\ty,\mu]}.$
\end{lem}
\begin{proof}
The spectral Theorem ( a parametrized measurable version) shows that the ranges of the family of projections $\chi_{[0,\mu^2]}(f^*f)$ belong to the class $\mathcal{L}(f,\mu)$, then $$\operatorname{dim}_{\Lambda}(\operatorname{range}(\chi_{[0,\mu^2)}(f^*f)))\leq F_{\Lambda}(f,\mu).$$ In fact it's clear that $\chi_{[0,\mu^2)}(f^*_xf_x)\omega =\omega \Rightarrow \|f \omega \|\leq \mu \|\omega \|$. But now for every $L\in \mathcal{L}(f,\mu)$ we get a family of injections 
$\chi_{\mu^2}(f_x^*f_x)_{|L_x}\longrightarrow \operatorname{range}(\chi_{\mu^2}(f^*_xf_x))$ that after reduction modulo $\Lambda$ and with the crucial property of the formal dimension 3 in lemma \ref{formaldimension} says $$\operatorname{dim}_{\Lambda}(L)\leq \operatorname{dim}_{\Lambda}(\operatorname{range}(\chi_{[0,\mu^2]}(f^*f)).$$
\end{proof}

\begin{dfn}\label{elll}
A complex of random Hilbert cochains as $({L^2}(\Omega^{\bullet}X_0),d)$ and its relative and boundary versions is said \underline{$\Lambda$--(left) Fredholm} in degree $k$ if the differential induced on the quotient
$$\xymatrix{{\dfrac{\mathcal{D}(d^k)}{\overline{\operatorname{range}(d^{k-1})   }      }}
 \ar[r]^-d & {
 L^2(\Omega^{k+1}X_0)
 }       }$$
 gives by $\Lambda$ a.e. reduction a \underline{left Fredholm} unbounded operator in the sence of definition \ref{frett}. In particular the condition  involving the spectrum distribution function is 
 \begin{equation}\label{messaggio}
 F_{\Lambda}\big{(}d|:\mathcal{D}(d^k)\cap \operatorname{range}(d^{k-1})^{\bot}\longrightarrow L^2(\Omega^{k+1}X_0),\mu\big{)}<\ty
 \end{equation} for some positive number $\mu$.

\noindent For this reason one calls the left hand--side of \eqref{messaggio} 
$$F_{\Lambda}\Big{(}L^2(\Omega^kX_0,\partial X_0),\mu\Big{)}:=F_{\Lambda}\big{(}d|:\mathcal{D}(d^k)\cap \operatorname{range}(d^{k-1})^{\bot}\longrightarrow L^2(\Omega^{k+1}X_0),\mu\big{)}$$
 the \underline{spectral density function} of the complex at point $k$.
\end{dfn}
\begin{oss}The definition above combined with lemma \ref{LL} says that we have to compute the formal dimension of $\chi_{[0,\mu^2]}(f^*f)$ where $f=d_{|\mathcal{D}(d)\cap \overline{\operatorname{range}(d^{k-1)}}^{\bot}}.$ But $f$ is an injective restriction of $d^k;$ then every spectral projection $\chi_B(f^*f)$ projects onto a subspace that's orthogonal to $\operatorname{ker}(d^k)$. This means
\begin{equation}\label{amn}F_{\Lambda}\big{(}d|:\mathcal{D}(d^k)\cap \operatorname{range}(d^{k-1})^{\bot}\longrightarrow L^2(\Omega^{k+1}X_0),\mu\big{)}=\sup \mathcal{L}_{\Lambda}^{\bot}(f,\mu)\end{equation} where $\mathcal{L}_{\Lambda}^{\bot}(f,\mu)$ is the set of Random fields of subspaces of $\mathcal{D}(d)\cap \operatorname{ker}(d)^{\bot}$ where $d$ is bounded by $\mu$ (see Definition \ref{frett} )
\end{oss}
\noindent Now return to the boundary foliation $\partial{F}_0$
with its equivalence relation $\underline{\mathcal{R}_{0}}$.
\begin{thm}All the three complexes of Random Hilbert spaces
$$\xymatrix{
{}&{}&L^{2,\mathcal{F}}(\Omega^\bullet X_0)\\
\underline{\mathcal{R}_0}\ar[urr]\ar[rr]\ar[drr]&{}&L^{2,\mathcal{F}}(\Omega^\bullet X_0,\partial X_0)\\
{}&{}&L^{2,\mathcal{F}}(\Omega^\bullet \partial X_0)}
$$ with unbounded differentials are $\Lambda$--Fredholm.
\end{thm}\label{longfre}
\begin{proof}
The proof follows by an accurate inspection of the relation between the differentials (with or without boundary conditions) and the Laplace operator trough the theory of selfadjoint boundary differential problems developed in \cite{scht}. 
To make the notation lighter let $M=L_x^0$ with $\partial M=\partial L_x^0$ the generic leaf. We concentrate on the relative sequence at point $d:A^{k}(M,\partial M)\longrightarrow A^{k+1}(M,\partial M)$ where the differential is an unbounded operator on $L^2$ with Dirichlet boundary conditions.
Let $\mathcal{D}(d)=A^{k+1}(M,\partial M)$. The following Lemma is inspired by Lemma 5.11 in \cite{lus} where in contrast Neumann boudary conditions are imposed.
\begin{lem}\label{1117}Let $\operatorname{ker}(d)$ be the kernel of $d$ as unbounded operator with Dirichlet boundary conditions, then
$$\mathcal{D}(d)\cap \operatorname{ker}(d)^{\bot}=H^1_{\textrm{Dir}}\cap \overline{\delta^{k+1} C^{\ty}_0(\Lambda^{k+1}T^*M)}^{L^2}$$ where $H^1_{\textrm{Dir}}$ is the space of order $1$ Sobolev $k$--forms $\omega$ such that $\omega_{|\partial M}=0$.
\end{lem}
\begin{proof}
First of all remember that the differential operator $d+\delta:C^{\ty}(\Lambda^{\bullet}T^*M)\longrightarrow C^{\ty}(\Lambda^{\bullet}T^*M)$ with either Dirichlet or Neumann boundary conditions is formally self--adjoint with respect to the greenian formula
$$(d^r \omega,\eta)-(\omega,\delta^{p+1}\eta)=\int_{\partial M}(\omega \wedge \ast^{p+1}\eta)_{|\eta}$$ and uniformly elliptic \cite{scht}.
 This means that this is an elliptic boundary value problem in the classical sense according to the definition of Lopatinski and Shapiro \cite{palais}, Appendix I, together with a uniform condition on the local fundamental solutions. Now let $\omega \in C^{\ty}_0$ and $\eta \in \operatorname{ker}(d)$ i.e. $\eta_n\in C^{\ty}_0$, $(\eta_n)_{|\partial M}=0$, 
 $\xymatrix{{\eta_n}\ar[r]^-{L^2}&\eta}$, 
 $\xymatrix{{d\eta_n}\ar[r]^-{L^2}&0}$ then
 $$(\eta,\delta \omega)=\lim_n(\eta_n,\delta \omega)=\underbrace{\lim_n(d\eta_n,\omega)}_{0}\pm \underbrace{\int_{\partial M}(\eta_n\wedge \ast \omega)_{|\partial M}}_{\eta_{|\partial M}=0}=0,$$ showing that 
 $\overline{\delta C^{\ty}_0}\subset \mathcal{D}(d)\cap \operatorname{ker}(d)^{\bot}.$ For the reverse inclusion take $\omega \in \dd \cap \kerd^{\bot}$ i.e. $\omega_n\in C^{\ty}_0$, $\xymatrix{{\omega_n}\ar[r]^-{L^2}&\omega}$, $\xymatrix{{d\omega_n}\ar[r]^-{L^2}&0}.$ For fixed $\eta \in C^{\ty}_0$,
 $$\underbrace{((d+\delta)\eta,\omega)}_{d\eta \in \kerd, \omega \in \kerd^{\bot}}=(\delta \eta,\omega)=\lim_n(\delta \eta,\omega_n)\underbrace{=}_{{\omega_n}_{|\partial M}=0}=\lim_n (\eta,d\omega).$$
 \noindent Then we can apply the adjoint regularity theorem of H{\"o}rmander \cite{scht} Lemma 4.19, cor 4.22 saying that $\omega \in H^1_{\textrm{loc}}$ then $(\delta \omega,\eta)=(\omega,d \eta)$ holds because for every $\eta \in C^{\ty}_0(M-\partial M)$, $d\eta \in \kerd$ then $\delta \omega=0$. It follows that for every $\sigma \in C^{\ty}_0$
 $$0\underbrace{=}_{d\sigma \in \kerd}(d\sigma,\omega)=\underbrace{(\sigma,\delta \omega)}_{0}\pm \int_{\partial M}(\sigma \wedge \ast \omega)_{|\partial M}=\pm \int_{\partial M}(\overline{\omega}\wedge \overline{\ast \sigma })_{|\partial M}.$$ The last passage coming from the definition of the Hodge $\ast$ operator, $\sigma \wedge \ast \omega=(\sigma,\omega)dvol=(\overline{\omega},\overline{\sigma})dvol=\overline{\omega}\wedge \overline{\ast \sigma}$, where $\overline{\cdot}$ is the complex conjugate in $\Lambda T^*M \otimes \mathbb {C}.$ Now from the density of $\{i^*(\overline{\ast \sigma})\}_{\sigma \in C^{\ty}_0}$ in $L^2(\partial M)$, $i:\partial M\hookrightarrow M$ the boundary condition $\omega_{|\partial M}=0$ follows in particular $\omega \in H^1_{\textrm{Dir}}$. Now it remains to apply the Hodge decomposition  
$$L^2(\Lambda^k T^*M)=\mathcal{H}^k_{(2)}(M,\partial M)\oplus \overline{d^{k-1}\Omega_{d}^{k-1}(M,\partial M)}^{L^2}\oplus \overline{\delta^{k+1}\underbrace{\Omega_{\delta}^{k+1}(M,\partial M)}_{
\mbox{no }\partial-\mbox{conditions}
}}^{L^2} $$ to deduce $\omega \in \overline{\delta^{k+1} C^{\ty}_0(\Lambda^{k+1}T^*M)}^{L^2}.$ \end{proof}
\noindent Consider again the formally selfadjoint boundary value problem $d+\delta$ with Dirichlet boundary conditions i.e $\mathcal{D}(d+\delta)=H^1_{\mbox{Dir}}$. Its square in the sense of unbounded operators on $L^2$ is the laplacian $\Delta$ with domain 
$$H^2_{\mbox{Dir}}:=\{\omega \in H^2: \omega_{|\partial M}=0, \,((d+\delta)\omega)_{|\partial M}=(\delta \omega)_{|\partial M}=0\}.$$
\noindent Let $\Delta_k^{\bot}$ the operator obtained from $\Delta$ on $k$--forms restricted to the orthogonal complement of its kernel, it is easy to see that the splitting
$$L^2(\Lambda^k T^*M)=\mathcal{H}^k_{(2)}(M,\partial M)\oplus \overline{d^{k-1}\Omega_{d}^{k-1}(M,\partial M)}^{L^2}\oplus \overline{\delta^{k+1}\underbrace{\Omega_{\delta}^{k+1}(M,\partial M)}_{
\mbox{no }\partial-\mbox{conditions}
}}^{L^2} $$
induces the following splitting on $\Delta_k$,
$$\Delta_k^{\bot}=(\delta^{k+1}d^p)_{|\overline{\delta^{k+1}\Omega_{\delta}^{k+1}}}\oplus (d^{k-1}\delta^k)_{|\overline{d^{k-1}\Omega^{k-1}_d}}.$$
\begin{lem}\label{1118}The following identies of unbounded operators hold

 \begin{align*}&(\delta^{k+1}d^p)_{|\overline{\delta^{k+1}\Omega_{\delta}^{k+1}}}=(d^k_{|\overline{\delta^{k+1}\Omega_{\delta}^{k+1}}})^*(d^k_{|\overline{\delta^{k+1}\Omega_{\delta}^{k+1}}}),\\
&(d^{k-1}\delta^k)_{|\overline{d^{k-1}\Omega^{k-1}_d}}=(d^{k-1}_{|\overline{\delta^{k}\Omega_{\delta}^{k}}})(d^{k-1}_{|\overline{\delta^{k}\Omega_{\delta}^{k}}})^*\end{align*}
where the $d^k_{|\overline{\delta^{k+1}\Omega_{\delta}^{k+1}}}$ is the unbounded operator on the subspace $\overline{\delta^{k+1}\Omega_{\delta}^{k+1}}$ of $L^2$ with domain $H^1_{\textrm{Dir}}\cap \overline{\delta^{k+1}\Omega_{\delta}^{k+1}}$ and range $\overline{d^{k+1}\Omega^{k
+1}_d}$.
\end{lem}
\begin{proof}
This is again the \emph{dual} (in the sense of boundary conditions) statement of Lemma 5.16 in \cite{lus}.
We first state that the Hilbert space adjoint of the operator $d^k$ with domain $H^1_{\textrm{Dir}}\cap \overline{\delta^{k+1}\Omega_{\delta}^{k+1}}$ and range $\overline{d^{k+1}\Omega^{k+1}_d}$ is exactly $\delta^{k+1}$ with domain $H^1_{\textrm{Dir}}\cap \overline{d^{k}\Omega_d^{k}}$. We shall omit degrees of forms and call $d$ this restricted operator. Thanks to the intersection with $H^1$ this is also the restriction of $d+\delta$ to the same subspace, in particular $\omega \in \mathcal{D}(d^*)\subset \overline{dC^{\ty}_0}$ implies $\omega \in \dd$ and $d\omega=0.$ Take arbitrary $\eta \in H^1_{\textrm{Dir}}\cap \overline{\delta C^{\ty}_0}$, then since $\delta \eta=0$, $((d+\delta)\eta,\omega)=(d\eta,\omega)=(\eta,d^*\omega)$ and if $\eta \in H^1_{\textrm{Dir}}\cap \overline{d\Omega_d}$, $((d+\delta)\eta,\omega)=(\delta \eta,\omega)=0.$ Since $\delta H^1_{\textrm{Dir}}\bot \overline{d\Omega_d}$ this is immediately checked, 

\noindent $\sigma \in d\Omega_d$, $\sigma=d\lambda$, $\lambda_{|\partial M}=0$, $(\sigma,\delta \gamma)=\underbrace{(d\sigma,\gamma)}_{=0}+\int_{|\partial M}\underbrace{(\sigma \wedge \ast \gamma)_{|\partial M}}_{=0}.$

\noindent Also $(\eta,d^* \omega)=0$ since $d^*\omega \in \overline{\delta C^{\ty}_0}$ and $d\Omega_{\textrm{Dir}}\bot \delta C^{\ty}_0$. Then we can apply again the adjoint regularity theorem \cite{scht}, Lemma 4.19 to deduce $\omega \in H^1_{\textrm{loc}}$.
The next goal is to show $\omega \in H^1_{\textrm{Dir}}$ i.e. $d\omega,\delta \omega \in L^2$, $\omega_{|\partial M}=0$ but $dx=0\in L^2$, $\delta \omega=(d+\delta)\omega=d^* \omega \in L^2$ and

\noindent $(\omega, d\delta \eta)=(d^*\omega,\delta \eta)=(\delta \omega,\delta \eta)=(\omega,d\delta \eta)\pm \overline{\int_{\partial M}(\delta \eta \wedge \ast \omega)_{|\partial M}}$ for every $\eta \in C^{\ty}_0$. Then $0=\int_{\partial M}(\delta \eta \wedge \ast \omega)_{|\partial M}=\int_{\partial M}(\bar{\omega}\wedge \overline{\ast \delta \eta})_{|\partial M}\underbrace{=}_{=0}\int_{\partial M}({\omega}\wedge {\ast \delta \eta})_{|\partial M}$ for every $\eta$. The boundary condition follows by density.
\noindent Finally it is clear that $\delta d_{|{\mathcal{D}(d^*d)}}=\Delta=\Delta^{\bot}$ but we have to prove the coincidence of the domains
$$\mathcal{D}(\Delta)\cap \overline{\delta C^{\ty}_0}=\mathcal{D}(d^*(d_{|\overline{\delta C^{\ty}_0}})),$$ now $\mathcal{D}(\Delta)=H^2_{\textrm{Dir}}=\{\omega \in H^2,\,\omega_{|\partial M},\,(\delta \omega)_{|\partial M}=0\}\subset \mathcal{D}(d^*d_{|\overline{\delta C^{\ty}_0}}).$ Clearly $$\omega \in \mathcal{D}(d^*d_{|\overline{\delta C^{\ty}_0}})\Rightarrow \omega \in H^1_{\textrm{Dir}}\cap \overline{\delta C^{\ty}_0},$$ $d\omega \in H^1_{\textrm{Dir}}$ then $(d+\delta)\omega \in H^1$ and since $\omega_{|\partial M}=0$ by elliptic regularity (for the boundary value problem $(d+\delta)$ with Dirichlet conditions \cite{scht}) $\omega \in H^2$. We have just checked the boundary conditions, finally $\omega\in H^2_{\textrm{Dir}}=\mathcal{D}(\Delta)$. The second equality in the statement is proven in a very similar way.
\end{proof}
\noindent Now that the relation of $d$ with Dirichlet boundary condition restricted to the complement of its kernel with the Laplacian ($\Delta^{\bot}$) is clear we can use elliptic regularity to deduce that the relative Random Hilbert complex is $\Lambda$--Fredholm. 
This has to be done in two steps, the first is to show that the spectral function of the Laplacian controls the spectral function of the complex 
\begin{equation}\label{fdg}
F_{\Lambda}(\Delta^{\bot}_k,\sqrt{\mu})=F_{\Lambda}(L^{2,\mathcal{F}}(\Omega^k X_0,\partial X_0),\mu)+F_{\Lambda}(L^{2,\mathcal{F}}(\Omega^{k-1} X_0,\partial X_0),\mu)
\end{equation}
in fact \begin{align}\nonumber F_{\Lambda}(\Delta^{\bot}_k,\sqrt{\mu})&=F_{\Lambda}\Big{(}(\delta^{k+1}d^k)_{|\overline{\delta^{k+1}\Omega^{k+1}_{\delta}}  }),\sqrt{\mu}\Big{)}+F_{\Lambda}\Big{(}(d^{k-1}{\delta}^k)_{|\overline{d^{k-1}\Omega^{k-1}_{d}}  }),\sqrt{\mu}\Big{)}\\ \nonumber
&=
F_{\Lambda}\Big{(}(d^k_{|\overline{\delta^{k+1}\Omega_{\delta}^{k+1}}})^*(d^k_{|\overline{\delta^{k+1}\Omega_{\delta}^{k+1}}}),\sqrt{\mu}\Big{)}+F_{\Lambda}\Big{(}(d^{k-1}_{|\overline{\delta^{k}\Omega_{\delta}^{k}}})(d^{k-1}_{|\overline{\delta^{k}\Omega_{\delta}^{k}}})^*,\sqrt{\mu}\Big{)}\\ \nonumber
&=F_{\Lambda}\Big{(}d^{k}_{|\overline{\delta^{k+1}\Omega_{\delta}^{k+1}}},\mu\Big{)}+F_{\Lambda}(d^{k-1}_{|\delta^k\Omega_{\delta}^k},\mu\Big{)}
\end{align} where, in the first step we have used the obvious fact that the spectral functions behave additively under direct sum of operators togheter with the remark after \eqref{elll} , at the second step there are lemmas \ref{1117} and \ref{1118} together with the following properties of the spectral functions 
\begin{itemize}
\item $F_{\Lambda}\Big{(}f^*f,\sqrt{\lambda}\Big{)}=F_{\Lambda}(f,\lambda)$
\item $F_{\Lambda}(\phi,\lambda)=F_{\lambda}(\phi^*,\lambda)$
\end{itemize} that can be adapted 
to hold in our situation with unbounded operators. Good references are the paper of Lott and L{\"u}ck \cite{lolu} and the paper of L{\"u}ck and Schick \cite{lus} that inspired completely this treatment.

\noindent Let us firs recall the equation
\begin{equation}\nonumber
F_{\Lambda}(\Delta^{\bot}_k,\sqrt{\mu})=F_{\Lambda}(L^{2,\mathcal{F}}(\Omega^k X_0,\partial X_0),\mu)+F_{\Lambda}(L^{2,\mathcal{F}}(\Omega^{k-1} X_0,\partial X_0),\mu). 
\end{equation}
It says that
 we have only to show that $\Delta_k^{\bot}$ is left $\Lambda$--Fredholm to have control of both Fredholmness at degree $k$ and $k-1$. We can use the heat kernel, in fact by elliptic regularity for each leaf the heat kernel 
$e^{-t {\Delta_k,x}^{\bot}}(z,z')$ is smooth and uniformly bounded along the leaf on intervals $[t_0,\ty)$ \cite{scht} Theorem 2.35. As $x$ varies in $\partial X_0$ these bounds can made uniform by the uniform geometry (in fact the constants depend on the metric tensor, its inverse and a finite number of their derivatives in normal coordinates) and we get a family of smooth kernels that varies transversally in a measurable fashion since it is obtained by functional calculus from a measurable family of operators. Then they give a $\Lambda$--trace class element in the Von neumann algebra. Now the projections $\chi_{[0,\mu]}(f^*f)$ in definition \ref{elll} where $f$ is the differential restricted to the complement of its kernel are obtained from the heat kernel as $$\chi_{[0,\mu]}(f^*f)=\underbrace{\chi_{[0,\mu]}(\Delta_k^{\bot})e^{\Delta_k^{\bot}}}_{\textrm{bounded}}
\underbrace{\chi_{[0,\mu]}(\Delta_k^{\bot})e^{-\Delta_k^{\bot}}}_{\Lambda-\textrm{trace class}}<\ty.$$
\end{proof}
\begin{oss}
The same argument of elliptic regularity for b.v. problems togheter with the various Hodge decompositions shows that each term of the long sequence \eqref{success} \underline{is a finite} \underline{Random Hilbert space.}
\end{oss}
\begin{thm}The long sequence \eqref{success}
\begin{equation*}
\xymatrix{{\cdot \cdot \cdot}\ar[r]& H^{k,\mathcal{F}}_{dR,(2)}(X_0,\partial X_0)\ar[r]^-{i_{*}} & 
H^{k,\mathcal{F}}_{dR,(2)}(X_0)\ar[r]^-{r_{*}}&\\
{}&\ar[r]^-{r_{*}}&
H^{k,\mathcal{F}}_{dR,(2)}(\partial X_0)\ar[r]^-{\delta}&
H^{k-1,\mathcal{F}}_{dR,(2)}(X_0,\partial X_0)\ar[r]& \cdot \cdot \cdot}
\end{equation*} is $\Lambda$--weakly exact (definition \ref{vonexact} )
\end{thm}
\begin{proof}
This is exactly the same proof of Cheeger and Gromov \cite{ChGr}; we present it here, in the form appearing in the book by L{\"u}ck (Theorem 1.21 in \cite{luk}). In fact the crucial final step there, that is based on the property of formal dimension of Hilbert $\Gamma$--modules
$$\operatorname{dim}_{\Gamma}\Big{(}\bigcap_{i\in I}V_i\Big{)}=\inf_{i\in I} \operatorname{dim}_{\Gamma} V_i,$$ is replaced here by proposition \ref{decrescenza}.

\noindent We prove only exactness in the middle i.e. at $H_{dR,(2)}^k(X_0)$. At each point $x\in \partial X_0$ we have $\overline{\operatorname{range}(i_{\ast,x})}\subset \operatorname{ker}(r_{\ast,x})$. So let $U_x$ the (Borel) field of orthogonal complements in $\operatorname{ker}(r_{\ast,x})$, seen as a projection in the corresponding algebra, we have only to prove (the trace is faithful) $\operatorname{tr}_{\Lambda} U=0$. Let $V_x$ the family of closed subspaces corresponding to the various $U_x$ under the baby $L^2$--Hodge--de Rham isomorphism or abstract Hodge is. (see Theorem 1.18 in \cite{luk}) \begin{equation}\label{hodgeoggi}\operatorname{ker}(d_x^k)\cap \operatorname{ker}(\gamma^k_x)\longrightarrow H^k_{dR,(2)}(L_x^0).\end{equation} where $\gamma^k$ is the Hilbert space adjoint of $d^k$. Note that this varies in measurable fashion since the is. is induced by inclusion. 
Let $d^{k-1}_x$ the de Rham differential on the leaf $d^{k-1}_x:L^2(\Omega^{k-1} L_x^0)\longrightarrow L^2(\Omega^K  L_x^0)$ then we have $r_x^k(V_x)\subset \overline{\operatorname{range}e_x^{k-1} }$ since $r_{\ast,x}(V_x)=0$. The operator $d_x^{k-1}\circ (d_x^{k-1})^{\ast}:L^2(\Omega^{k}  L_x^0)\longrightarrow L^2(\Omega^{k}  L_x^0)$ is positive so let $\{E_{\lambda,x}|\lambda \in \R\}$ its right continuous spectral family (to ensure measurability use a parametrized version of the spectral theorem for measurable fields of unbounded operators). Each projection $E_{\lambda,x}$ commutes with $d_x^{k-1}\circ (d_x^{k-1})^{\ast}$ then sends $\overline{\operatorname{range}[ d_x^{k-1}\circ (d_x^{k-1})^{\ast}]}$ to itself. We have
$$\overline{ d_x^{k-1}\circ (d_x^{k-1})^{\ast} }=
\operatorname{ker}[d_x^{k-1}\circ (d_x^{k-1})^{\ast}]^{\bot}=
[\operatorname{ker}[(d_x^{k-1})^{\ast}]^\bot=
\overline{ \operatorname{range} d_x^{k-1}   },$$ the second following from $\langle d^{k-1}\circ (d^{k-1})^\ast) v,v\rangle=|(d^{k-1})^\ast v|^2.$ From this chain of equalities follows 
$$E_{\lambda,x}\circ r_x^k\subset \overline{\operatorname{range}d_x^{k-1}    }.$$ The next goal is to show that $E_{\lambda,x}\circ r_x^k$ is an injection on $V_x$ for every $\lambda>0$. So pick a vector $v_x\in V_x: E_{\lambda,x}\circ r_x^k v_x=0$ then $r_x^k v_x \in \operatorname{range}(E_{\lambda,x})^\bot.$ But for every $\lambda>0$ the operator $\di \circ \dii$ is invertible from $\ra (E_{\lambda,x})^\bot$ to itself, the inverse given by the spectral function $\int_{\lambda}^ty\mu^{-1}d E_{\mu,x}$, in particular we can find $w_x\in \ki$ such that 
$r_x^k v_x=\di w_x$. A little diagram chasing shows there exist 
$\xi_x\ L^2(\Omega^{k-1}L_x^0):r_x^k\xi_x=w_x$ and some 
$\sigma \in L^2(\Lambda^k T^*L_x^0,\partial L_x^0):i_x^k \sigma=\di \xi_x-v_x$. Then 
$\sigma$ is closed and 
$[\sigma]\in H^k_{dR,(2)}(L_x^0,\partial L_x^0)$ is the class whose image under 
$i_{\ast,x}^k$ is the element in 
$U_x$ corresponding to $v_x$. Since $U_x$ is orthogonal to 
$\ra(i_{\ast,x^k})$, $v_x=0$. Now we know $E_{\lambda,x}\circ r_x^k$ is an injection the following chain of equalities collect all the $V_x$ to form the projection $V$ at level of the Von Neumann algebra and, by left Fredholmness, right continuity of the spectral family $E_\lambda$ (leaf by leaf) and proposition \ref{decrescenza}
\begin{align*}
\dn (V) =&\dn (\overline{E_\lambda \circ r^k(V)})\\
=&\dn(\overline{E_\lambda \circ r^k(V)}\cap \overline{\ra d^{k-1}})\\
=&\lim_{\lambda\rightarrow 0} \dn (\overline{E_\lambda \circ r^k(V)}\cap \overline{\ra d^{k-1}})\\
=&\dn \Big{(}\bigcap_{\lambda>0}\ra E_\lambda \cap \overline {(\ra d^{k-1})}\Big{)}\\=&\dn (\ra E_0 \cap \overline{(\ra d^{k-1})}\\
=& \dn (\ker (d^{k-1}\circ (d^{k-1})^*)\cap \overline{\ra d^{k-1}})\\
=& \dn(\ker(d^{k-1})^\ast \cap \ra \overline{\ra d^{k-1}} )=\dn 0=0
\end{align*}
\end{proof}
\subsubsection{The proof}
\begin{thm}We have
$$\sigma_{\Lambda,dR}(X_0,\partial X_0)=\sigma_{\Lambda,\textrm{an}}(X,\partial X_0)$$ thus together with formula \eqref{anasign} w.r.t. the manifold with cylinder attached $X$ all the three signatures we have defined agree
$$\sigma_{\Lambda,dR}(X_0,\partial X_0)=\sigma_{\Lambda,\operatorname{an}}(X_0,\partial X_0)=\sigma^{\ty}_{\Lambda}(X)=\langle L(X),[C_{\Lambda}]\rangle +1/2[\eta_{\Lambda}(D^{\mathcal{F}_{\partial}})]$$
\end{thm}
\begin{proof}
We pass through different intermediate results, sometimes doing leafwise considerations. Our model is of course the work of L{\"u}ck and Schick \cite{lusc} whose our work is only an adaptation. The proof of L{\"u}ck and Schick in turn is inspired by the classical argument of Atiyah Patodi and Singer \cite {AtPaSi1} with the great issue that at $L^2$ level long sequences are only weakly exact and the spectrum of the boundary operator is not discrete.
\bigskip

\noindent {\bf{First step}}. This is done. We have proved, following the method of Vaillant the equality $$\sigma_{\Lambda,\operatorname{an}}(X_0,\partial X_0)=\sigma^{\ty}_{\Lambda}(X)$$ where at right--hand side the signature on harmonic leafwise $L^2$--forms on the elonged manifold with elonged foliation i.e. the $\Lambda$ signature of the Poincar\`e product on leafwise harmonic forms. Our reference is then the harmonic signature. 
\bigskip

\noindent {\bf{Second step}}.
We shall prove $\sigma_{\Lambda,dR}(X_0,\partial X_0)=\sigma^{\ty}_{\Lambda}(X)$. We explain now the strategy
\pecetta{
We have to measure the $+/-$ eigenspaces of the intersection form on the field of Hilbert spaces $H^{2k}_{dR,(2)}(X_0,\partial X_0)$ as square integrable representations of 
$\mathcal{R}_0$ (the whole foliation on $X_0$). Now thanks to the fundamental note on section \ref{fundationtr} it is sufficient to measure the corresponding projections in the von Neumann algebra arising by restriction of the Random Hilbert spaces to 
$\underline{\mathcal{R}_0}$ (the equivalence relation of the foliation induced on the boundary). This is a consequence of the very definition of the trace as an integral of a functor with values measure spaces and the fact the boundary contains a complete transversal. The passage to 
$\underline{\mathcal{R}_0}$ has the great vantage we can write boundary problems and sequences of random Hilbert spaces, in particular the third term  in 
$$\xymatrix{
0\ar[r]&A_x^{k-1}(L_x^0,\partial L_x^0)\ar[r]^-i&A_x^{k-1}(L_x^0)\ar[r]^-r &A_x^{k-1}(\partial L_x^0)\ar [r]&0
}$$ is natural as representations of $\underline{\mathcal{R}_0}.$
}
\noindent Remember the notation $x\in \partial X_0$, $L_x^0$ is the leaf of the compact foliated manifold with boundary, $L_x$ is the leaf of the foliation on the manifold $X$ with a cylinder attached.
Consider the random Hilbert space $H^{2k}_{dr,(2)}(X_0)$ obtained from the various $L^2$ cohomologies of the leaves with no boundary conditions (this is called in \cite{lusc} the $L^2$--homology since it naturally pairs with forms with Dirichlet boundary conditions). We have a family of restriction maps $\partial X_0\ni x \longmapsto r_x^p:\mathcal{H}^{2k}(L_x)\longrightarrow H^{2k}_{dR,(2)}(L_x^0)$ and intertwining operators   $(\mathcal{H}^{2k}(L_x))_{x\in X_0}:\longmapsto H^{2k}_{dR,(2)}(L_x^0)$. There are also natural mappings $i_x^{2k}:H^{2k}_{dR,(2)}(L_x^0,\dl)\longrightarrow H^{2k}_{dR,(2)}(L_x)$. And the mapping $q$ coming from the long sequence in cohomology as in the following diagram
\begin{equation}\label{prty}\xymatrix{H^{2k}_{dR,(2)}(L_x^0,\dl)\ar[r]^-{i_x^{2k}}&H^{2k}_{dR,(2)}(L_x^0)\ar[d]^-{q^{2k}_x}\\
\mathcal{H}^{2k}(L_x)\ar[ur]^-{r_x^{2k}}&               H^{2k}_{dR,(2)}(\dl)}.
\end{equation}

\noindent Following the program of L{\"u}ck and Schick we shall prove 
\begin{enumerate}
\item  \begin{equation}\label{gess}
\overline{\operatorname{range}(r^{2k})}=\overline{\operatorname{range}(i^{2k})}\textrm{ as projections in } \operatorname{End}_{\Lambda,\underline{\mathcal{R}_0}}\Big{[}H_{dR,(2)}^{2k,\mathcal{F}}(X_0)\Big{]}\end{equation}  and the signature can be computed looking at the fields of sesquilinear Poincar\`e products on the images of $i_x^{2k}$ as square integrable representations of $\mathcal{R}_0$,
\begin{equation}\label{closure}\xymatrix{H^{2k}_{dR,(2)}(L_x^0,\dl)\ar[r]^-{i_x^{2k}}&H^{2k}_{dR,(2)}(L_x^0)\\
\mathcal{H}^{2k}(L_x)\ar[ur]^-{r_x^{2k}}&{}                           }.\end{equation}
\item The signature of the field of products on the image of $i_x^{2k}$ concides with the signature of the fields of Poincar\'e products on $(\mathcal{H}_x)_{x\in X_0}$ as square integrable representations of $\underline{\mathcal{R}_0}$ that in turn coincides with those computed tracing in $\mathcal{R}_0$
\end{enumerate}
\noindent Notice about \eqref{gess} that $\overline{\operatorname{range}(i^{2k})}=\ker q^{2k}$ by the long exact sequence.
\bigskip

\noindent 1.
\noindent L{\"u}ck and Schick (lemma 3.12 in \cite{lus}) prove the following result 
\begin{lem}
$$q_x^{2k} \circ r_x^{2k}=0,\,\, x\in \partial X_0$$
\end{lem}\label{lafine}
\begin{proof}
\noindent We write the proof only for the sake of completeness. Look at   the diagram \eqref{prty}
 and choose an harmonic form $\omega \in \mathcal{H}^{2k}(L_x)$ then by elliptic regularity $\omega$ is in every Sobolev space $H^s(\Omega^{2k}(L_x))$ and the family of continuos restriction morphisms on submanifolds of codimension 1, $s>1/2,\,\,H^s(L_x)\longrightarrow H^{s-1/2}(\dl)\hookrightarrow  L^2(\dl\times \{t\})$ gives ($\dl\times{t}=\dl$) a family of continuous mappings 
 $$q[t]^{2k}_x:\mathcal{H}^{2k}(L_x)\longrightarrow L^2(\Lambda^{2k}T^*\dl)$$ with $q[0]^{2k}_x=q^{2k}_x\circ r^{2k}_x$. Now the manifolds $\dl \times [0,\ty)$ and $\dl \times [t,\ty)$ are all isometric (choose the most simple family of is. $\phi_t(x)=x+t$) then, since the sequence of restrictions $\omega_t$ to $\dl \times [t,\ty)$ have all Sobolev norms going to zero the sequence $q[t]^{2k}_x (\omega)$ goes to zero in $L^2(\dl)$ with $t\longrightarrow \ty)$. More precisely $q[t]_x^{2k}\omega=\omega(t)={\omega_t}_{|\dl \times {t}}=q[0]_x^{2k}\phi_t^{*}\omega_t$ and $\phi_t^{*}\omega_t \longrightarrow 0$ in $H^s$, $s>1/2.$ Now we declare that all the forms $q[t]_x^{2k}(\omega)$ represent the same class in the reduced $L^2$ cohomology of the boundary. In fact on the cylinder $\omega=\omega_1(r)+\omega_2(r)\wedge dr$ with $\omega_{1,2}\in L^2([0,\ty),L^2(\Lambda^{2k}T^*\dl))$, since $\omega$ is closed
 $$d\omega=0=d_{\partial}\omega_1(r)\pm 
 \dfrac{\partial \omega_1(r)}{\partial r}\wedge dr+d_{\partial}\omega_2(r)\wedge dr \underbrace{\Longmapsto}_{\textrm{contraction}} \pm \dfrac{\partial \omega_1(r)}{\partial r}=d_{\partial}\omega_2(r).$$ Integrating this last equation $\omega_1(t)-\omega_1(0)=\pm d\Big{(}\int_0^t\omega_2(r)dr\Big{)}$ since the term $\omega_1(t)$ is the pullback of $\omega$ to $\dl \times{t}$,
 $$q[t]^{2k}_x(\omega)-q[0]^{2k}_x(\omega)=\pm d\Big{(}\int_0^t\omega_2(r)dr\Big{)}.$$ \noindent
 \noindent Since $\xymatrix{q[t]^{2k}_x \omega \ar[r]^{L^2}&0}$ the proof is concluded if we show that $\int_0^t\omega_2(r)dr$ is an $L^2$--form. Now since $\omega_{1,2}$ are $L^2$ functions on $[0,\ty]$ with values $L^2$ forms on the boundary we can write the Cauchy--Schwartz inequality on compact pieces (use the constant function 1)
\begin{align*}
|(\omega_2(r),1)_{L^2([0,t];L^2(\Lambda T^*\dl)}|^2=\Big{|}\int_0^t\omega_2(r)dr\Big{|}
&\leq \int_0^t 1^2 dr \cdot \int_0^t\|\omega_2(r)\|_{L^2(\dl)}^2dr\\
& \leq t\cdot \int_0^t\|\omega_2(r)\|_{L^2(\dl)}^2dr.
\end{align*}
\end{proof}
 \noindent By definition of the algebra of intert. operators $q_x^{2k}\circ r^{2k}_x=0\,\forall x\in \partial X_0 \Longrightarrow \overline{\ra r_x^{2k}}\subset \ker q_x^2k$ then  $$\overline{\ker q^{2k}}\cdot \overline{\ra r^{2k}}=\overline{\ra i^{2k}}\cdot \overline{\ra r^2k}=\overline{\ra r^{2k}}\in \operatorname{End}_{\Lambda,\underline{\mathcal{R}_0}}\Big{[}H_{dR,(2)}(X_0)\Big{]}$$
\noindent Now v.n. dimentions comes in play in a fundamental way. Consider the field of unbounded boundary differentials $d_x:L^2(\Omega^{2k-1}\partial L_x^0)\longrightarrow L^2(\Omega^{2k-1}\partial L_x^0)$ exactly as in \cite{lus} (and essentially by elliptic regularity and the fact trace=trace on the boundary foliation) they define a left Fredholm affiliated operator so the image of the field of the spectral projection $\chi_{(0,\gamma]}(\delta d)$ has dimension tending to zero for $\gamma\rightarrow 0$. Given $\epsilon>0$ define the following field of subspaces,
$$E_{\epsilon,x}^{2k}:=\ra (d\circ \chi_{(\gamma,\ty)}(\delta d ))\subset L^2(\Omega^{2k}\partial L_x^0).$$ Properties of $E_\epsilon^{2k}$
\begin{enumerate}
\item \underline{it is measurable}, in fact is obtained by functional calculus from a natural Borel family.
\item \underline{It has codimension less that} $\epsilon$ in $\overline{\ra(d)}$ in fact $d\circ \chi_{(-\ty,0]}(\delta d)=0.$
\item \underline{It is closed}, because the restriction of $\delta d$ to the subspace corresponding to $(0,\ty)$ satisfies $\delta d \geq \gamma$ than is invertible (this automatically seen using the polar decomposition).
\end{enumerate}
\noindent Now we have to invoke the leafwise Hodge decomposition with (Neumann)  boundary condition, 
\begin{equation}
\label{hokka}
L^2(\Omega^{2k-1}(L_x^0))=\overline{\ra d^{2k-2}}\oplus \overline{ \ra \delta^{2k-2}_{|\{\omega_{|\partial}=0\}}}\oplus 
 \ker \Delta^{2k}_{|\{ (\ast \omega)_{|\partial}=0=(\delta \omega)_{|\partial}      \}}.
\end{equation} 
The methods of Schick \cite{scht}  surely applies to the generic leaf $L_x^0$ in fact this is bounded geometry and has a collar so the fact its boundary has infinite connected components (complete in the induced metric) plays no role. So the space $\clis$ can be canonically identified with the third addendum in \eqref{hokka} and pull back to the boundary gives a well defined measurable
\footnote{inverse image of a measurable field of subspaces by a unif. bounded measurable family of bounded operators is measurable, one can split the domain space as $\operatorname{Ker}\oplus\operatorname{Ker}^{\bot}$ and apply the well known fact that inverses of isom. are measurable \cite{Dix}} 
family of (uniformely) bounded mappings $\beta_x^{2k}:\clis \longrightarrow \cali$. Define, by pull--back the following measurable field of closed subspaces
$$K_{\epsilon,x}^{2k}\subset \clis.$$
\noindent Properties of $K_{\epsilon,x}^{2k}:$
\begin{enumerate}
\item $K_{\epsilon,x}^{2k}\subset \clis$
\item $K_{\epsilon,x}^{2k}\subset (\beta^{2k}_x)^{-1}(\overline{\ra d_{\partial}})$
\item The field $K_{\epsilon,x}^{2k}$ defines a projection having codimension in $\ker q^{2k}$ that's less than $\epsilon.$
\end{enumerate}
\noindent Then there's another density lemma in \cite{lusc} (Lemma 3.16) 
stating another property,
$$K^{2k}_{\epsilon}\subset \ra(r^{2k}_x:\mathcal{H}^{2k}_{2}(L_x)\longrightarrow \overline{\ra i^{2k}}).$$
\noindent All of this properties certainly say that \eqref{gess} is true (by normality of the trace we can reach $\overline{\ra (i^{2k})}$ with a family of subprojections whose codimension tends to zero).
\bigskip

\noindent 2. 

\noindent Again following \cite{lusc}, $q[0]^{2k}$ (notation of the proof of Lemma \ref{lafine}) defines a bounded family of mappings from $\mathcal{H}^{2k}_{(2)}(L_x)$ to $\overline{\ra d_{\partial}}$. So let $\mathcal{H}_{\epsilon,x}^{2k}\subset \mathcal{H}^{2k}_{(2)}(L_x)$ be as before the inverse image of $E_{\epsilon,x}^{2k}$. Since we are using harmonic forms the pull--back is (uniformly) bounded in the $L^2$ norm so $\mathcal{H}_{\epsilon,x}^{2k}$ is a field of closed subspaces giving projection of codimension in $\mathcal{H}^{2k}_{(2)}(X)$ not greater than $\epsilon$. Now if $$L_{\epsilon,x}^{2k}\subset \overline{\ra i^{2k}_x}$$ is the closure of the image of $L_{\epsilon,x}^{2k}$ under the mapping $r^{2k}_x:\mathcal{H}^{2k}_{(2)}(L_x)\longrightarrow \overline{\ra i^{2k}_x},$ its codimension into $\overline{\ra i^{2k}}$ is less than $\epsilon$ exactly because of \eqref{gess} since the codimension of $\mathcal{H}^{2k}_\epsilon$ in $\mathcal{H}^{2k}_{(2)}(X)$ is less than $\epsilon$.

\noindent The leafwise intersection form $$s_x^0:H^{2k}_{(2)}(L_x^0,\partial L_x^0)\times H^{2k}_{(2)}(L_x^0,\partial L_x^0)\longrightarrow \C$$ descends into a pairing on $\overline{\ra i^{2k}_x}$  which restricts to a pairing
$$\eta_x^0:L_{\epsilon,x}^{2k}\times L_{\epsilon,x}^{2k}\longrightarrow \C.$$
\noindent But the codimension of $L_{\epsilon}^{2k}\subset \overline{\ra i^{2k}}$ is less than $\epsilon$ one gets
$$|\operatorname{sign}_{\Lambda}(s^0)-\operatorname{sign}_{\Lambda}(\eta)|\leq \epsilon,$$
remember that $\operatorname{sign}_{\Lambda}(s^0)=\sigma_{\Lambda,dR}(X_0,\partial X_0).$

\noindent Now it's a quite amazing computation performed by Luck and Schick \cite{lusc} that the leafwise Hodge intersection form we called $s^{\ty}_x:\mathcal{H}_{(2)}(L_x)\times \mathcal{H}_{(2)}(L_x)\longrightarrow \C$ \underline{descends} to a pairing on each $\mathcal{H}_{\epsilon,x}^{2k}$ and in turn to exactly the pairing $\eta_x^0$ defined above. Again since the codimension of $\mathcal{H}_{\epsilon}^{2k}$ in $\mathcal{H}^{2k}_{(2)}(X)$ is $\leq \epsilon$ we get $|\operatorname{sign}_{\Lambda}(s^{\ty})-\operatorname{sign}_{\Lambda}(\eta^0)|\leq \epsilon$ then
$$|\operatorname{sign}_{\Lambda}(s^\ty)-\operatorname{sign}_{\Lambda}(s^0)|\leq 2\epsilon.$$ The theorem is proved since $\epsilon$ is arbitrary.

\begin{oss}{\bf{On the assumption of the complete transversal contained into the boundary.}}
The assumption $\operatorname{Saturation}(\partial X_0)=X_0$ is really simply 
avoidable in fact one can write the sequence
$$\xymatrix{
0\ar[r]&A_x^{k-1}(L_x^0,\partial L_x^0)\ar[r]^-i&A_x^{k-1}(L_x^0)\ar[r]^-r &A_x^{k-1}(\partial L_x^0)\ar [r]&0
}$$
for $x$ also in the interior but the last arrow is null for $\partial L_x^0=0$ so everything works in the exactly same way.
\end{oss}

\end{proof}
\newpage
\appendix
\section{Analysis on Manifolds with bounded geometry}\label{oper}
Hereafter we review some essential results about differential operators, and the Dirac one in particular, on manifolds with bounded geometry. This theory was developed by J. Roe \cite{Roel,Roe1,Roe2}, M. Shubin \cite{Shu} and J. Lott \cite{Lo-h} among others.

Let $M$ be an oriented Riemannian manifold of bounded geometry, by definition,
\begin{enumerate}
\item the injectivity radius of $M$, $\textrm{inj}(M)$, defined as the infimum on $M$ of radii of regular geodesic balls is finite.
\item The Riemann curvature tensor is uniformly bounded with every covariant derivative.
\end{enumerate}
\begin{dfn}
\noindent For an vector bundle to be of bounded geometry will mean that it is given a connection with uniformly bounded curvature together with every covariant derivative.
\end{dfn}  
\noindent Natural examples are, compact manifolds, Galois covering of compact manifolds, the interior of a compact manifold with boundary equipped with a $b$--metric and finally leaves of a compact foliated manifold. An obvious but important property is that compact perturbations, i.e. connected sum preserve bounded geometry. 
\noindent Note that a non--compact manifold with bounded geometry has infinite volume.
\bigskip
Directly from the definition one finds that if $\textrm{dim}(M)=n$ there exists a positive number $r$ such that the eclidean ball $B(0,r)\subset \R^n$ is the domain of exponential coordinates for every point in $M$.
The Christoffel symbols of $M$ regarded as a family of smooth functions depending on $i,j,k$ and points $m$ in $B$ are a bounded subset of the Fr\'echet space $C^{\infty}(B)$.
These geodetic balls can be used also to trivialize bundles by parallel traslation along geodesic rays of a fixed orthonormal basis at the center. Such frames are called \emph{synchronous}. With a "\emph{good coordinate ball}" one refers to this situation.

We shall consider till the end of this section  Clifford modules of bounded geometry with $\mathbb{Z}_2$ graduated structure denoted generally by $S$ and call $D$ the associated Dirac operator.
\begin{dfn}
\begin{enumerate}
\item For $k\in \mathbb{N}$ the Sobolev space of sections of $H^k(S)$ is the completion of $C^{\infty}_c(S)$ under the norm
$$\|s\|_k=(\|s\|_{L^2}^2+\|\nabla s\|_{L^2}^2+\cdot \cdot \cdot \|\nabla^k s\|_{L^2}^2)^{1/2}.$$
\item For negative $k$, $H^k(S)$ is the dual space of $H^{-k}(S)$ regarded as a distributional sections space.
\item Put $H^{\infty}(S)=\bigcap_k H^k(S)$ equipped with its natural Fr\'echet topology,
$H^{\infty}(S)=\bigcup H^k(S)$ with the weak topology that it inherits as as the dual of $H^{\infty}(S)$.
\end{enumerate}
\end{dfn}
\begin{dfn}
\begin{enumerate}
\item Let $r\in \mathbb{N}$, the uniform $C^r$ space is the Banach space of all $C^r$ sections $s$ of $S$ such that the norm
$$\||s\||_r=\sup\big\{| \nabla_{v_1}\cdot \cdot \cdot \nabla_{v_r}s(m)|\big \}$$ is finite, supremum taken over points $m\in M$ and collections $\{v_1,...,v_q, 0\leq q\leq r\}$ of unit vectors at $m$.
\item Also, $UC^{\infty}(S)$ is the Fr\'echet space $\bigcap_r UC^r(S)$.
\end{enumerate}
\end{dfn}
The algebra of differential operators $\textrm{Diff}^*(M,S)$ acting on $S$ contains the subalgebra $\textrm{UDiff}^*(M,S)$ of \emph{uniform differential operators} generated by 
the uniform space $UC^{\infty}(\textrm{End}(S))$ together with covariant derivatives $\nabla^S_X$ (as differential operators) along uniform vector fields $X\in UC^{\infty}(TM)$. 

It turns out that for a differential operator to be uniformly elliptic is necessary and sufficient to have every derivative (also 0 order of course) of its symbol uniformly bounded on every good coordinate ball. A $k$--order uniform differential operator naturally defines continuous mappings, $H^r(M,S)\longrightarrow H^{r-k}(M,S)$ and $UC^l(M,S)\longrightarrow UC^{l-k}(M,S)$.
\begin{dfn}
An uniform differential operator $P\in \udif^*(M,S)$ is uniformly elliptic if its principal symbol
$$\sigma_{\textrm{pr}}(P)\in \uc(T^*M,\pi^*(\textrm{End}(S))$$ has an uniform inverse in an $\epsilon$--neightborhood of the zero section in $T^*M$.
\end{dfn}
\begin{thm}(uniform G\aa rding's inequality)
For an uniformly elliptic operator $T\in \udif^k(M,S)$, for every $l$ there exists a positive constant $C(l)$ such that
\begin{equation}\label{gar}
\|s\|_{H^{s+k}}\leq C(l)\{\|s\|_{H^s}+\|Ps\|_{H^s}\},\end{equation} 
for every $s\in C^{\infty}_c(M,S).$
\end{thm}
\begin{proof}
A straightforward generalization of compact case.
\end{proof}
Here a list of properties

In this framework the Sobolev embedding theorem reads as follows,
\begin{thm}For $k,s\in \mathbb{N}$, $s>k+(\textrm{dim}(M))/2$ There is a continuous inclusion $H^s(M,S)\longrightarrow UC^k(M,S)$ hence also a continuous inclusion of Fr\'echet spaces
$$H^{\infty}(S)\longrightarrow UC^{\infty}(S)$$.
\end{thm}
\begin{proof}
As observed by J. Roe, this is an adaption of the standard compact case, in fact thanks to bounded geometry assumption the family of local Sobolev constant on good balls is bounded.
\end{proof}
\noindent Now by Schwartz kernel theorem a continuous linear operator\footnote{If $T$ is not a pseudo--differential operator it is customary to require that it respects all the connected components of $M$.
} $T:C^{\infty}_c(M,S)\longrightarrow C^{-\infty}(M,s)$ is univocally represented by its Schwartz kernel, the unique distribution--section 
$K_T\in C^{-\infty}(M\times M,\textrm{END}(S)\otimes \textrm{Pr}_1^*\Omega(M))$ satisfying the distributional equation
$$\langle K_T u,v \rangle=\langle K_T, v \XBox u\rangle$$ for every $u,v\in C^{\infty}_c(M,S)$. Here the \emph{big endomorphism bundle} $\textrm{END}(S)\longrightarrow M\times M$ has fiber $\textrm{Hom}(S_x,S_y)$ over $(x,y)$. 
the following is a group of definitions.
\begin{dfns}
\begin{enumerate}
\item We say that $T$ has order $k \in \mathbb{Z}$ if it extends to an operator in $B(H^s(M,S),H^{s-k}(M,s))$ for every $s$. 
\item The space of $k$--order operators is denoted by $\op^k(M,S)$.
with seminorms given by $B(H^s(M,S),H^{s-k}(M,s)).$
\item The space $\op^{-\infty}(M,S)=\bigcap_{k<0}\op^k(M,S)$ is called the space of uniformly smoothig operators. In fact we shall see it is the space of operators with uniformly smooth kernels.
\item An element $T\in \op^k(M,S)$, $k\geq 1$ is called elliptic if it satisfies the uniform G\aa rding inequality \eqref{gar}.
\end{enumerate}
\end{dfns}
Below a list of properties that can be found in the papers cited at the beginning.
\begin{prop}
\begin{itemize}
\item Ellipticity is stable under order 0 perturbations, if $T\in \op^k(M,S)$ elliptic and $Q\in \op^0(M,S)$ then $T+Q$ is elliptic.
\item If $\in \op^k(M,S)$ is elliptic and formally self--adjoint then every its spectral projection belongs to $\op^0(M,s)$.
\item It follows from the completeness of $M$ that an elliptic and formally self--adjoint element $T\in \op^k(M,S)$ ($k\geq 1 $ as required by the definition of elliptic element) is essentially selfadjoint on $L^2(M,S)$.

If $T$ denotes its closure also one finds that $\textrm{dom}(T)=H^k(M,S)$.
In particular this is true for the Dirac operator $D$.
 \end{itemize}
\end{prop}

 \subsection{Spectral functions of elliptic operators}
Last theorem says that an uniformly elliptic operator on a manifold with bounded geometry is essentially self--adjoint. We need some considerations about spectral functions of $T$.
Let $$RB(\R):=\{f:\R\longrightarrow \C, \textrm{ Borel};\quad |(1+x^2)^{k/2}f(x)|_{\infty}<\infty \quad \forall k \in \mathbb{N}\}$$ be the space of rapidly decreasing Borel functions with Fr\'echet structure induce by the seminorms $|(1+x^2)^{k/2}\cdot|_{\infty}$ 

\noindent 
Let $RC(\R)$ denote the closed subspace of continuous functions. 

\begin{prop}{
For an elliptic element $T$ and $l\in \mathbb{N}$ and rapid Borel functions $f$, $T^lf(T)$ is bounded in $L^2$ and the following G\aa rding inequality holds true,
\begin{equation}\label{ordine positivo}
\|f(T)\psi\|_{H^l}\leq  C(l) \sum_{i=0}^l\|T^if(T)\psi\|_{L^2}\leq C(l) \|\psi \|_{L^2} \sum_{i=0}^{l}|x^{i}f|_{\infty}\end{equation}
for every $\psi \in C^{\infty}_c(M,S)$. 
Suppose now, by simplicity of writing that $T$ has order 1, making use of the duality $$(H^s)^*=H^{-s}$$ one finds, for $k,l \in \mathbb{Z}$, $l\geq k$, 
\begin{equation}\label{ordine negativo}
\|f(T)\psi\|_{H^l}\leq C(l,k)\sum_{i=0}^{l-k}\|T^if(T)\psi\|_{H^k}\leq C(l,k)\|\psi\|_{H^k}\sum_{i=0}^{l-k}|x^{i}f|_{\infty}.
\end{equation}}
\end{prop}
\begin{proof}
Observe first that the operator $T^lf(T)$ is the spectral function of $T$ corresponding to the function $x^lf(x)$ on $\R$ hence is bounded. Again, since $f$ is bounded no problem here in commuting relations, in particular $T^lf(T)=f(T)T^l$ (equality in the sense of unbounded operators) in particular $f(T):L^2\longrightarrow H^{l+k}$.
Now from G\aa rding's inequality for $T$,
\begin{equation*}
\|f(T)\psi\|_{H^l}\leq C(l) \sum_{i=0}^l\|T^if(T)\psi\|_{L^2}\leq C(l) \|\psi \|_{L^2} \sum_{i=0}^{l}|x^{i}f|_{\infty}.
\end{equation*}
Inequality \eqref{ordine negativo} follows at once 
from the first one \eqref{ordine positivo} in fact the first step is to consider the transpose of $T^lf(T):H^{-l}\longrightarrow H^{-k}$ while the second step is based on our very dual definition of Sobolev space of order negative.
\end{proof}

Hence, we get continuity of the functional calculus $RB(\R)\longrightarrow B(H^l(M,S),H^k(M,S))$ for each $l,k$ then continuity of $RB(\R)\longrightarrow \op^{-\infty}(M,S).$
With a little work, using Sobolev embedding one can prove the following theorem.\begin{thm}\label{regu}
Let $T\in \op^k(M,S)$ uniformly elliptic and formally selfadjoint.
\begin{itemize}
\item If $L=[n/2+1]$, $n=\textrm{dim}M$ and $l\in \mathbb{N}$ then the \emph{kernel mapping}
$$\op^{-2L-l}(M,S)\longrightarrow UC^{l}(M\times M,\textrm{END}(S)\otimes \textrm{Pr}_1^*\Omega(M)), T\longmapsto K_T,$$ is continuous.
\item For $f\in RB(\R)$ the kernel of $f(T)$ is uniformly smoothing, $$K_T \in UC^{\infty}(M\times M,\textrm{END}(S)\otimes \textrm{Pr}_1^*\Omega(M)).$$
and the kernel mapping $RB(\R)\longrightarrow UC^{\infty}(M\times M,\textrm{END}(S)\otimes \textrm{Pr}_1^*\Omega(M))$ is continuous.
\end{itemize}
\end{thm}
\begin{oss}
Combining \ref{regu} and \ref{localtr} we see that every spectral projection $\Pi_A$ of the Dirac operator obtained by a bounded Borel set $A \subset \R$ is represented by a uniformly smoothing kernel hence is locally traceable (in the usual sense on $L^2(M,S)$ w.r.t the Abelian Von Neumann algebra $L^{\infty}(M)$). This means that for every Borel set $B\subset M$ with compact closure the operator $ \chi_B \Pi_A \chi_B$ is trace class, one gets a Radon measure $B\longmapsto \operatorname{trace} \chi_B \Pi_A \chi_B$ called the \underline{local trace} of $\Pi_A$.
\end{oss}

\subsection{Some computations on Clifford algebras}\label{cliffo}
Let ${\C}{l}(k)$ the (complex) Clifford algebra over the euclidean space
${\R}^k,$ with generators ${\bold c}_1,\ldots,{\bold c}_k$ and relations (${\bold c}_j$ orthonormal basis) 
$${\bold c}_i{\bold c}_j+{\bold c}_j{\bold c}_i=-2\delta_{ij}.$$ The algebra ${\C}l(k)$
is ${\Z}_2$-graded: ${\C}l(k)={\C}l^+(k)\oplus 
{\C}l^-(k),$ being ${\C}l^+(k)$ the subalgebra spanned by products of even sets of generators.

\noindent The map ${\bold c}_i\longmapsto{\bold c}_i
{\bold c}_{k+1}$ defines an isomorphism ${\C}l(k)\overset{\sim}{\to}
{\C}l^+(k+1)$. 

\noindent The  volume element $\tau_k:=i^{[(k+1)/2]}{\bold c}_1
\ldots{\bold c}_k\in{\C}l(k)$ satisfies $\tau^2_k=1$ and thus induces a ${\Z}_2$-grading 
on each representation of ${\C}l(k).$
Due to the fact $$\tau_k{\bold c}=-(-1)^k{\bold c}\tau_k$$ for ${\bold c}\in{\R}^k\subset 
{\C}l(k)$ this induced grading is trivial if $k$ is odd.
\noindent ${\C}l(2l)$ has a unique irreducible representation, called its \underline{spinor space} and we denote it by
$S(2l).$ Its dimension is $\dim S(2l)=2^l.$ Decomposing into the  $\pm1$-Eigenspaces
of $\tau_{2l}$ we write $S(2l)=S^+(2l)\oplus S^-(2l).$ Via the identification ${\C}l
(2l-1)\cong {\C}l^+(2l)$ the spaces $S^+(2l)$, $S^-(2l)$ are non-equivalent
irreducible representations of ${\C}l(2l-1)$, which can be considered as being
isomorphic representations of ${\C}l(2l-2)\cong {\C}l^+(2l-1)$ via the map $S^+(2l)
\overset{c_{2l}}{\to}S^-(2l)$. This of course is then just the representation $S(2l-2)$
of  ${\C}l(2l-2)$.

\noindent \underline{Notation}: for  $S^\pm(2l)$ we also write  $S^\pm(2l-1)$ when these spaces are seen as 
representations of ${\C}l(2l-1)$.
$$\xymatrix{{\C}l(2l-1)\ar@{<->}[r]&{\C}l^+(2l)\ar@{<->}[r]&\operatorname{End}^{+}(S^+(2l)\bigoplus S^-(2l))\ar@{=>}[r]&\operatorname{End}(S^{\pm}(2l))=:\operatorname{End}(S^{\pm}(2l-1))}.$$

\noindent It is easily seen that ${\C}l(2l)$ acts injectively on $S(2l)$. 
Comparison of dimensions then yields ${\C}(2l)\cong \operatorname{End}(S(2l)),$ and,
using ${\C}l(2l-1)\cong{\C}l^+(2l)$ also 
$${\C}l(2l-1)\cong {\C}l^+(2l) \cong
\operatorname{End}^+(S(2l)).$$ The identification ${\C}l(2l-1)\longrightarrow \operatorname{End}(S^\pm(2l-1))$ maps
$\tau_{2l-1}$ to $\pm 1$ and one can show that the null space is $(1\mp\tau_{2l-1}){\C}(2l-1).$
$$\xymatrix{\operatorname{End}(S^+(2l)\bigoplus S^-(2l))\ar@{<->}[r]&      {\C}l(2l)=&{\underbrace{{\C}l^+(2l)}}&\bigoplus &{\C}l^-(2l)\\
{\overbrace{{\C}l(2l-1)}}=\ar@{<->}[urr] &{\underbrace{{\C}l^+(2l-1)}}&\bigoplus& {\C}l^-(2l-1)\\
{\overbrace{{\C}l(2l-2)}}\ar@{<->}[ur]
            }$$

\noindent The traces $\operatorname{tr^\pm}$ on $\operatorname{End}(S^\pm(2l-1))$ and the graded trace $\operatorname{str}$ on
$\operatorname{End}(S(2l))$ then induce traces on ${\C}(2l-1)$ and 
${\C}(2l).$  On elements of the form ${\bold c}_I:={\bold c}_{i1}
\ldots{\bold c}_{i|I|}$ where $I=\{i_1\le\ldots\le i_{|I|}\}\subset \{1,\ldots,
k\}$ these can be computed as follows
\begin{lem}\label{A.1} 
\begin{enumerate}
\item[(a)] In ${\C}l(2l)$ we have ${\operatorname{str}}(\tau_{2l})=2^l$ 
and ${\operatorname{str}}(1)={\operatorname{str}}({\bold c}_I)=0$ for $I\neq\{1,\ldots,k\}.$

\item[(b)] In ${\C}l(2l-1)$ we have $\operatorname{str}(\tau_{2l-1})=-\operatorname{tr}^-(\tau_{2l-1})=
\operatorname{tr}^\pm(1)=2^{l-1}$ and for $I\neq\{1,\ldots,k\}$ we have $\operatorname{tr}^\pm({\bold c}_1)=
0.$
\end{enumerate}
\par

On $({\C}l(2l-1)-{\C})\subset{\C}l(2l)$ therefore $\operatorname{tr}^\pm
(\bullet
)=\mp\frac{1}{2} \operatorname{str}({\bold c}_{2l}\bullet)$ and on ${\C}l(2l)\subset
{\C}l(2l+1)$ we have $\operatorname{str}(\bullet)=\pm i \operatorname{tr}^\pm({\bold c}_{2l+1}\bullet)$
\end{lem}
\begin{proof}
 Cf. \cite{BeGeVe}, Proposition 3.21
\end{proof}
\noindent The map $S^+(2l)\overset{c_{2l}}{\longrightarrow} S^-(2l)$ gives an identification
$S(2l)\cong S^\pm(2l-1)\oplus S^\pm(2l-1).$ In this representation,  ${\C}l(2l)$ 
acts on $S(2l)$ as follows

\begin{equation*}
\begin{split}
&  {\bold c}_i\in{\C}l(2l-1)\simeq
\begin{pmatrix}
0  &  \pm{\bold c}_i  \\
\pm {\bold c}_i  &  0
\end{pmatrix}\;
{\bold c}_{2l}\simeq
\begin{pmatrix}
0  &  -1  \\
1  &  0
\end{pmatrix}\\
&  \mbox{and}\;
\operatorname{str}
\begin{pmatrix}
\phi_1 & \phi_2 \\
\phi_3 & \phi_4
\end{pmatrix}=\operatorname{tr}^\pm(\phi_1)-\operatorname{tr}^\pm(\phi_4)
\end{split}
\end{equation*}

\end{document}